\documentclass[leqno]{amsart}
\usepackage{amssymb}
\usepackage{mathrsfs}
\usepackage{amsmath, amsfonts, vmargin, enumerate}
\usepackage{graphics}
\usepackage{color}
\usepackage{verbatim}
\usepackage{amsthm}
\usepackage{latexsym, bm}
\usepackage{euscript}
\usepackage{dsfont}

\input xypic
\xyoption {all}

 \makeatletter

\newtheorem{thm}{Theorem}[part]
\newtheorem{prop}[thm]{Proposition}
\newtheorem{cor}[thm]{Corollary}
\newtheorem{lem}[thm]{Lemma}
\newtheorem{defi}[thm]{Definition}
\newtheorem{remark}[thm]{Remark}
\newtheorem{example}[thm]{Example}
\newtheorem{pb}[thm]{Problem}
\newtheorem{conj}[thm]{Conjecture}

\newenvironment{rk}{\begin{remark}\rm}{\end{remark}}
\newenvironment{Def}{\begin{defi}\rm}{\end{defi}}

\newenvironment{problem}{\begin{pb}\rm}{\end{pb}}
\newenvironment{conjecture}{\begin{conj}\rm}{\end{conj}}

\numberwithin{equation}{part}
\numberwithin{section}{part}

\newcommand{\real}{{\mathbb R}}
\newcommand{\nat}{{\mathbb N}}
\newcommand{\ent}{{\mathbb Z}}
\newcommand{\com}{{\mathbb C}}
\newcommand{\un}{{\mathds {1}}}

\newcommand{\T}{{\mathbb T}}
\newcommand{\I}{{\mathbb I}}

\newcommand{\A}{{\mathcal A}}

\newcommand{\F}{{\mathcal F}}
\renewcommand{\H}{{\mathcal H}}

\newcommand{\M}{{\mathcal M}}
\newcommand{\N}{{\mathcal N}}

\newcommand{\BMO}{{\rm BMO}}

\renewcommand{\a}{\alpha}
\renewcommand{\b}{\beta}

\newcommand{\D}{\Delta}
\renewcommand{\d}{\delta}
\newcommand{\e}{\varepsilon}
\newcommand{\f}{\varphi}
\newcommand{\p}{\psi}
\renewcommand{\t}{\theta}

\renewcommand{\l}{\lambda}
\newcommand{\s}{\sigma}
\renewcommand{\o}{\omega}

\newcommand{\ot}{\otimes}

\newcommand{\8}{\infty}
\newcommand{\el}{\ell}

\newcommand{\la}{\langle}
\newcommand{\ra}{\rangle}
\newcommand{\wt}{\widetilde}
\newcommand{\wh}{\widehat}
\newcommand{\n}{\noindent}

\newcommand{\les}{\lesssim}

\newcommand{\Om}{\Omega}
\newcommand{\be}{\begin{eqnarray*}}
\newcommand{\ee}{\end{eqnarray*}}
\newcommand{\beq}{\begin{equation}}
\newcommand{\eeq}{\end{equation}}
\newcommand{\beqn}{\begin{equation*}}
\newcommand{\eeqn}{\end{equation*}}

\begin{document}

\title{Sobolev, Besov and Triebel-Lizorkin spaces on quantum tori}

\thanks{{\it 2000 Mathematics Subject Classification:} Primary: 46L52, 46L51, 46L87. Secondary: 47L25, 47L65, 43A99}

\thanks{{\it Key words:} Quantum tori, noncommutative $L_p$-spaces, Bessel and Riesz potentials, (potential) Sobolev spaces, Besov spaces, Triebel-Lizorkin spaces, Hardy spaces,  characterizations, Poisson and heat semigroups, embedding inequalities, interpolation, (completely) bounded Fourier multipliers}

\author{Xiao Xiong}
\address{Laboratoire de Math{\'e}matiques, Universit{\'e} de Franche-Comt{\'e}, 25030 Besan\c{c}on Cedex, France; current address: Department of Mathematics and Statistics, University of Saskatchewan, Saskatoon, Saskatchewan, S7N 5E6, Canada}
\email{ufcxx56@gmail.com}

\author{Quanhua Xu}
\address{Institute for Advanced Study in Mathematics, Harbin Institute of Technology, Harbin 150001, China; and Laboratoire de Math{\'e}matiques, Universit{\'e} de Bourgogne Franche-Comt{\'e}, 25030 Besan\c{c}on Cedex, France; and Institut Universitaire de France}
\email{qxu@univ-fcomte.fr}

\author{Zhi Yin}
\address{Institute for Advanced Study in Mathematics, Harbin Institute of Technology, Harbin 150001, China}
\email{hustyinzhi@163.com}

\date{}
\maketitle

\markboth{X. Xiong, Q. Xu, and Z. Yin}%
{Sobolev, Besov and Triebel-Lizorkin spaces on quantum tori}

%%%%%%%%%%%%%%%%%%%%%%%%%%%%%%%%%%%%%%%%%%%%%%%%%%%%%%%%%%%%%%%%%%%%%%%%
%%%%%%%%%%%%%%%%%%%%%%%%%%%%%%%%%%%%%%%%%%%%%%%%%%%%%%%%%%%%%%%%%%%%%%%%

\begin{abstract}
This paper  gives a systematic study of  Sobolev, Besov and Triebel-Lizorkin spaces  on a noncommutative $d$-torus $\T^d_\t$ (with $\t$ a skew symmetric real $d\times d$-matrix). These spaces share many properties with their classical counterparts. We prove, among other basic properties, the lifting theorem for all these spaces and a Poincar\'e type inequality for Sobolev spaces. We also show that the Sobolev space $W^k_\8(\T^d_\t)$ coincides with the Lipschitz space of order $k$, already studied by Weaver in the case $k=1$. We establish the embedding inequalities of all these spaces, including the Besov and Sobolev embedding theorems. We obtain Littlewood-Paley type characterizations for Besov and Triebel-Lizorkin spaces in a general way, as well as the concrete ones in terms of the Poisson, heat semigroups and  differences. Some of them are new even in the commutative case, for instance, our Poisson semigroup characterizations  improve the classical ones. As a consequence of the characterization of the Besov spaces by differences, we extend to the quantum setting the recent results of  Bourgain-Br\'ezis -Mironescu and Maz'ya-Shaposhnikova on the limits of Besov norms. The same characterization implies that the Besov space $B^\a_{\8,\8}(\T^d_\t)$ for $\a>0$ is the quantum analogue of the usual Zygmund class of order $\a$.  We investigate the interpolation of all these spaces, in particular, determine explicitly the K-functional of the couple $(L_p(\T^d_\t), \, W^k_p(\T^d_\t))$, which is the quantum analogue of a classical result due to Johnen and Scherer. Finally, we show that the completely bounded Fourier multipliers on all these spaces do not depend on the matrix $\t$, so coincide with those on the corresponding spaces on the usual $d$-torus. We also give a quite simple description of (completely) bounded Fourier multipliers on the Besov spaces in terms of their behavior on the $L_p$-components in the Littlewood-Paley decomposition.
\end{abstract}

%%%%%%%%%%%%%%%%%%%%%%%%%%%%%%%%%%%%%%%%%%%%%%%%%%%%%%%%%%%%%%%%%%%%%%%%
%%%%%%%%%%%%%%%%%%%%%%%%%%%%%%%%%%%%%%%%%%%%%%%%%%%%%%%%%%%%%%%%%%%%%%%%

\newpage

\tableofcontents

%%%%%%%%%%%%%%%%%%%%%%%%%%%%%%%%%%%%%%%%%%%%%%%%%%%%%%%%%%%%%%%%%%%%%%%%
%%%%%%%%%%%%%%%%%%%%%%%%%%%%%%%%%%%%%%%%%%%%%%%%%%%%%%%%%%%%%%%%%%%%%%%%
 \newpage

{\Large\part{Introduction}}

%%%%%%%%%%%%%%%%%%%%%%%%%%%%%%%%%%%%%%%%%%%%%%%%%%%%%%%%%%%%%%%%%%%%%%%%
%%%%%%%%%%%%%%%%%%%%%%%%%%%%%%%%%%%%%%%%%%%%%%%%%%%%%%%%%%%%%%%%%%%%%%%%

This paper is the second part of our project about analysis on quantum tori. The previous one \cite{CXY2012} studies several subjects of harmonic analysis on these objects, including maximal inequalities, mean and pointwise convergences of Fourier series, completely bounded Fourier multipliers on $L_p$-spaces and the theory of Hardy spaces. It was directly inspired by the current  line of investigation on noncommutative harmonic analysis. As pointed out there, very little had been done about the analytic aspect of quantum tori before \cite{CXY2012}; this situation is in strong contrast with their geometry  on which there exists a considerably long list of publications. Presumably, this deficiency is due to numerous difficulties one may encounter when dealing with noncommutative $L_p$-spaces, since these spaces come up unavoidably if one wishes to do analysis.   \cite{CXY2012} was made possible by the recent developments  on noncommutative martingale/ergodic inequalities and the Littlewood-Paley-Stein theory for quantum Markovian semigroups, which had been achieved  thanks to the efforts of many researchers; see, for instance,  \cite{PX1997, Junge2002, JX2003, JX2007,  Ran2002, Ran2007, PR2006}, and \cite{JLX2006, Mei2007, Mei2008, JM2010, JM2011}.

 \medskip

This second part intends to study  Sobolev, Besov and Triebel-Lizorkin spaces on quantum tori.  In the classical setting, these spaces are fundamental for many branches of mathematics such as harmonic analysis, PDE, functional analysis and approximation theory. Our references for the classical theory are \cite{Ad1975, Ma1980, Ni1975, Pee1976, HT1983, HT1992}.  However, they have never been investigated so far in the quantum setting, except two special cases to our best knowledge. Firstly, Sobolev spaces with the $L_2$-norm  were studied by Spera \cite{Spera1992} in view of applications to the Yang-Mills theory for quantum tori \cite{Spera1992b} (see also \cite{GL2004, Lu2006, Pol2006, Rosenberg2008} for  related works). On the other hand, inspired by Connes' noncommutative geometry \cite{Connes1994}, or more precisely, the part on noncommutative metric spaces, Weaver \cite{Weaver1996, Weaver1998} developed the Lipschitz classes of order $\a$ for $0<\a\le1$ on quantum tori. The fact that only these two cases have been studied so far illustrates once more the above mentioned difficulties related to  noncommutativity.

 \medskip
 
 Among these difficulties, a specific one is to be emphasized: it is notably relevant to this paper,  and is the lack of a noncommutative analogue of the usual pointwise maximal function.  However,  maximal function techniques play a paramount role in the classical theory of Besov and Triebel-Lizorkin spaces (as well as in the theory of Hardy spaces). They are no longer available in the quantum setting, which forces us to invent new tools, like in the previously quoted works on noncommutative martingale inequalities and the quantum Littlewood-Paley-Stein theory where the same difficulty already appeared.

 \medskip

One powerful tool used in  \cite{CXY2012} is the transference method. It consists in transferring problems on quantum tori to  the corresponding ones in the case of operator-valued functions on the usual tori, in order to use existing results in the latter case or adapt classical arguments. This method is efficient for several problems studied in \cite{CXY2012}, including the maximal inequalities and Hardy spaces. It is still useful for some parts of the present work; for instance,  Besov spaces can be investigated through the classical vector-valued Besov spaces by means of transference, the relevant Banach spaces being the noncommutative $L_p$-spaces on a quantum torus. However, it becomes inefficient for  others. For example, the Sobolev or Besov embedding inequalities cannot be proved by transference. On the other hand, if one wishes to study  Triebel-Lizorkin spaces on quantum tori via transference, one should first develop the theory of operator-valued Triebel-Lizorkin spaces on the classical tori. The latter is as hard as the former. Contrary to  \cite{CXY2012} , the transference method will play a very limited role in the present paper. Instead, we will use Fourier multipliers in a crucial way, this approach is of interest in its own right.  We thus develop an intrinsic differential analysis on quantum tori, without frequently referring to the usual tori via transference as in \cite{CXY2012}. This is a major advantage of the present methods over those of \cite{CXY2012}. We hope that the study carried out here  would open new perspectives of applications and motivate more future research works on quantum tori or in similar circumstances. In fact, one of our main objectives of developing analysis on quantum tori is to gain more insights on the geometrical structures of these objects, so ultimately to return back to their differential geometry.

\medskip

To describe the content of the paper, we need  some definitions and notation (see the respective sections below for more details).  Let $d \ge 2$ and $\theta = (\theta_{k j})$ be a real skew-symmetric
$d \times d$-matrix. The $d$-dimensional noncommutative
torus $\mathcal{A}_{\theta}$ is the universal C*-algebra generated by $d$
unitary operators  $U_1, \ldots, U_d$ satisfying the following commutation
relation
 $$U_k U_j = e^{2 \pi \mathrm{i} \theta_{k j}} U_j U_k,\quad1\le j, k\le d.$$
 Let  $U=(U_1,\cdots, U_d)$. For $m=(m_1,\cdots,m_d)\in\ent^d$, set
 $$U^m=U_1^{m_1}\cdots U_d^{m_d}.$$
 A polynomial in $U$ is a finite sum:
  $$ x =\sum_{m \in \mathbb{Z}^d}\alpha_{m} U^{m}\,,\quad  \alpha_{m} \in \mathbb{C}.$$
 For such a polynomial $x$, we define
 $\tau (x) = \alpha_{0}$.
Then $\tau$ extends to a  faithful  tracial state on $\A_{\theta}$.  Let  $\mathbb{T}^d_{\theta}$ be the w*-closure of $\A_{\theta}$ in  the GNS representation of $\tau$. This is our $d$-dimensional quantum torus.
It is to be viewed as a deformation of the usual $d$-torus $\T^d$, or more precisely, of the commutative algebra $L_\8(\T^d)$. The noncommutative $L_p$-spaces associated to  $(\T^d_\t,\tau)$ are denoted by $L_p(\T^d_\t)$. The Fourier transform of an element $x\in L_1(\T^d_\t)$ is defined by
  $$\wh x(m)=\tau\big((U^m)^*x\big),\quad m\in\ent^d.$$
The formal Fourier series of $x$ is
$$x\sim \sum_{m \in \mathbb{Z}^d}\wh x (m) U^{m}\;.$$

The differential structure of $\T^d_\t$ is modeled on that of $\T^d$. Let
 $$\mathcal{S}(\T^d_\t)=\big\{ \sum_{m\in\ent^d} a_m U^m : \{a_m\}_{m\in\ent^d} \; \mbox{rapidly decreasing}\big\}.$$
This is the deformation of the space of infinitely differentiable functions on $\T^d$; it is the Schwartz class of $\T^d_\t$. Like in the commutative case, $\mathcal{S}(\T^d_\t)$ carries  a natural  locally convex topology. Its topological dual  $\mathcal{S}'(\T^d_\t)$ is the space of distributions on $\T^d_\t$.
The partial derivations on $\mathcal{S}(\T^d_\t)$ are determined by
 $$\partial_j(U_j)=2\pi  \mathrm{i} U_j\;\text{ and }\; \partial_j (U_k)=0, \quad k\neq j, \;1\le j, k\le d.$$
Given $m =(m_1,\ldots,m_d)\in \mathbb{N}_0^d$ ($\mathbb{N}_0$ denoting the set of nonnegative integers), the associated partial derivation $D^m$ is defined to be
$\partial_1^{m_1}\cdots \partial_d^{m_d}$. The order of $D^m$ is   $|m|_1=m_1+\cdots+ m_d$.  Let $\D=\partial_1^2+\cdots+\partial_d^2$ be the Laplacian.
By duality, the derivations and Fourier transform transfer to  $\mathcal{S}'(\T^d_\t)$  too.

Fix a Schwartz function $\f$ on $\real^d$ satisfying the usual Littlewood-Paley decomposition property expressed in \eqref{LP dec}. For each $k\ge0$ let $\f_k$ be the function whose Fourier transform is equal to $\f(2^{-k}\cdot)$. For a distribution $x$ on $\T^d_\t$, define
 $$\wt\f_k*x=\sum_{m\in\ent^d}\f(2^{-k}m)\wh x(m)U^m\,.$$
So $x\mapsto \wt\f_k*x$ is the Fourier multiplier with symbol $\f(2^{-k}\cdot)$.

We can now define the four families of function spaces on $\T^d_\t$ to be studied . Let $1 \le p, q\le\8$ and $k\in\nat, \a\in\real$, and let $J^\a$ be the Bessel potential of order $\a$:  $J^\a=(1-(2\pi)^{-2}\D)^{\frac \a 2}$.

 \medskip\begin{enumerate}[$\bullet$]
 \item {\em Sobolev spaces}:
 $$W_p^k(\T_\t^d)= \big\{ x\in\mathcal{S}'(\T^d_\t) :  D^{m} x \in L_p(\mathbb{T}_{\theta}^d) \textrm{ for each }m\in \mathbb {N}_0^d \textrm{ with } |m|_1\leq k \big\}.$$
 \item {\em  Potential or fractoinal Sobolev spaces}:
 $$H_p^\a(\T^d_\t)=\big\{ x\in\mathcal{S}'(\T^d_\t) : J^\a x\in L_p(\mathbb{T}_{\theta}^d) \big\}.$$
 \item {\em  Besov spaces}:
  $$B^\a_{p,q} (\T^d_\t)=\big\{x\in \mathcal{S}'(\T^d_\t) : \big(|\wh x(0)|^q + \sum_{k\ge 0} 2^{qk\a} \| \wt\f_k * x\|_p^q\big)^{\frac{1}{q}} < \8 \big\}.$$
\item {\em  Triebel-Lizorkin spaces for $p<\8$} :
  $$F^{\a, c}_{p} (\T^d_\t)=\big\{x\in \mathcal{S}'(\T^d_\t) : \big\|\big(|\wh x(0)|^2 + \sum_{k\ge 0} 2^{2k \a  } | \wt\f_k * x|^2\big)^{\frac{1}{2}}\big\|_p < \8 \big\}.$$
\end{enumerate}
Equipped with their natural norms, all these spaces become Banach spaces.

\medskip

Now we can describe the main results proved in this paper by classifying them into five  families.

\medskip\noindent{\bf Basic properties}. A common basic property of potential Sobolev, Besov and Triebel-Lizorkin spaces is a reduction theorem by the Bessel potential. For example, $J^\b$ is an isomorphism from $B^\a_{p,q} (\T^d_\t)$ onto $B^{\a-\b}_{p,q} (\T^d_\t)$ for all $1\le p, q\le\8$ and $\a, \b\in\real$; this is the so-called lifting or reduction theorem. Specifically to Triebel-Lizorkin spaces, $J^\a$ establishes an isomorphism between $F^{\a,c}_{p} (\T^d_\t)$ and the Hardy space  $\H_p^c(\T^d_\t)$ for any $1\le p<\8$. As a consequence, we deduce that the potential Sobolev space $H_p^\a(\T^d_\t)$ admits a Littlewood-Paley type characterization for $1<p<\8$.

Another type of reduction for Besov and Triebel-Lizorkin spaces is  that for any positive integer $k$, $x\in F^{\a, c}_{p} (\T^d_\t)$ (resp. $B^\a_{p,q} (\T^d_\t)$)  iff all its partial derivatives of order $k$ belong to $F^{\a-k, c}_{p} (\T^d_\t)$ (resp. $B^{\a-k}_{p,q} (\T^d_\t)$).

Concerning Sobolev spaces,  we obtain a Poincar\'e type inequality:  For any $x\in W_p^1(\T_{\t}^d)$ with $1\le p\le\8$, we have
  $$\|x-\wh x(0)\|_p\les\|\nabla x\|_{ p}\,.$$
Our proof of this inequality  greatly differs with standard arguments for such results in the commutative case.

We also show  that $W_\8^k(\T^d_\t)$ is the analogue for $\T^d_\t$ of the classical Lipschitz class of order $k$.  For $u\in\real^d$,  define $\D_u x=\pi_z(x)-x$, where $z=(e^{2\pi\mathrm{i}u_1}, \cdots, e^{2\pi\mathrm{i}u_d})$ and $\pi_z$ is the automorphism of $\T^d_\t$ determined by $U_j\mapsto z_jU_j$ for $1\le j\le d$. Then for a positive integer $k$,  $\D_u^k$ is the  $k$th difference operator on $\T^d_\t$ associated to $u$. Note that $\D_u^k$ is also  the Fourier multiplier with symbol $e_u^k$, where $e_u(\xi)=e^{2\pi\mathrm{i}u\cdot\xi}-1$.
The $k$th order modulus of $L_p$-smoothness of an  $x\in L_p(\T^d_\t)$  is defined to be
 $$\o_{p}^k(x,\e) =\sup_{0<|u|\leq \e}\big\|\D_u^k x\big\|_p\,.$$
We then  prove that for any $1\le p\le\8$ and $k\in\nat$,
 $$\sup_{\e>0}\frac{\o_{p}^k(x,\e)}{\e^k}\approx \sum_{m\in\nat_0^d,\,|m|_1=k} \|D^mx\|_{p}\,.$$
In particular, we recover  Weaver's results \cite{Weaver1996, Weaver1998} on  the Lipschitz class on $\T^d_\t$ when $p=\8$ and $k=1$.

\medskip\noindent{\bf Embedding}. The second family of results  concern the embedding of the preceding  spaces. A typical one is the analogue of the classical Sobolev embedding inequality for $W_p^k(\T_\t^d)$: If $1<p<q<\8$ such that $\frac1q=\frac1p- \frac{k}d$, then
 $$W_p^k(\T_\t^d)\subset L_q(\T_\t^d)\;\text{ continuously}.$$
Similar embedding inequalities hold for the other spaces too. Combined with real interpolation, the embedding inequality of $B^\a_{p,q} (\T^d_\t)$ yields the above Sobolev embedding. Our proofs of these embedding inequalities are based on Varopolous' celebrated semigroup approach \cite{Va1985} to the Littlewood-Sobolev theory, which has also been developed by Junge and Mei \cite{JM2011} in the noncommutative setting for the study of BMO spaces on quantum Markovian semigroups. Thus the characterization of Besov spaces by Poisson or heat semigroup described below is essential for the proof of our embedding inequalities.

We also establish compact embedding theorems. For instance, the previously mentioned Sobolev embedding becomes a compact one $W_p^k(\T_\t^d)\hookrightarrow L_{q^*}(\T_\t^d)$ for any $q^*$ with $1\le q^*<q$.

\medskip\noindent{\bf Characterizations.} The third family of results are  various characterizations of  Besov and Triebel-Lizorkin spaces. This is the most difficult and technical part of the paper. In the classical case, all existing proofs of these characterizations that we know use maximal function techniques in a crucial way. As pointed out earlier, these  techniques are no longer available. Instead, we use frequently Fourier multipliers. We would like to emphasize that our results are  better than those in literature even in the commutative case. Let us illustrate this by stating the characterization of Besov spaces in terms of the circular Poisson semigroup.

Given a distribution $x$ on $\T^d_\t$ and $k\in\ent$,  let
 $$\mathbb{P}_r(x) = \sum_{m \in \mathbb{Z}^d } \wh{x} ( m ) r^{|m|} U^{m}$$
and
 $$\mathcal{J}^k_r\,\mathbb{P}_r(x) = \sum_{m \in \ent^d}C_{m, k} \wh{x} (m) r^{|m|-k} U^{m}, \quad 0 \le r < 1\,,$$
where $|\cdot|$ denotes the Euclidean norm of $\real^d$ and
 $$C_{m, k}=|m|\cdots(|m|-k+1)\;\text{ if }\;k\ge0 \quad\text{ and }\quad C_{m, k}=\frac1{(|m|+1)\cdots(|m|-k)}\;\text{ if }\;k<0.$$
Note that $\mathcal{J}^k_r$ is the $k$th derivation operator relative to $r$ if $k\ge0$, and the $(-k)$th integration operator  if $k<0$. Then our characterization  asserts that for $1\le p, q\le\8$ and $\a\in\real, k\in\ent$ with $k>\a$,
  $$
 \|x\|_{B_{p,q}^\a}\approx
 \Big(\max_{|m|<k}|\wh x(m)|^q+\int_0^1(1-r)^{(k-\a)q}\big\|\mathcal{J}^k_r\,{\mathbb{P}}_r(x_k)\big\|_p^q\,\frac{d r}{1-r}\Big)^{\frac1q}\,,
 $$
where $\displaystyle x_k=x-\sum_{|m|<k}\wh x(m)U^m$.

The use of the integration operator (corresponding to negative $k$) in the above statement is completely new even in the case $\t=0$ (the commutative case). This is very natural, and consistent with the fact that the smaller $\a$ is, the lower smoothness the elements of $B_{p,q}^\a(\T^d_\t)$ have.  This is also consistent with the previously mentioned lifting theorem. A similar result holds for Triebel-Lizorkin spaces  too. But its proof is much subtler. For the latter spaces, another improvement of our characterization over the classical one lies on the assumption on $k$: in the classical case, $k$ is required to be greater than $d+\max(\a,0)$, while we  only need to assume  $k>\a.$

The classical characterization of Besov spaces by differences is also extended to the quantum setting. This result resembles the previous one in terms of the derivations of the Poisson semigroup.
For $1\le p, q\le\8$ and $\a\in\real, k\in\nat$ with $0<\a<k$, let
 $$\|x\|_{B^{\a,\o} _{p,q}} =\Big(\int_0^1 \e^{-\a q} \o_{p}^k(x,\e)^q \frac{d\e}{\e}\Big)^{\frac1q}\,.$$
Then $x\in B^{\a} _{p,q}(\T^d_\t)$ iff $\|x\|_{B^{\a,\o} _{p,q}} <\8$.

The difference characterization of Besov spaces shows that $B_{\8,\8}^\a(\T^d_\t)$ is the quantum analogue of the classical Zygmund class. In particular, for $0<\a<1$, $B_{\8,\8}^\a(\T^d_\t)$ is the H\"older class of order $\a$, already studied by Weaver \cite{Weaver1998}.

In the commutative case, the limit behavior of the quantity $\|x\|_{B^{\a,\o} _{p,q}}$ as $\a\to k$ or $\a\to0$ are object of a recent series of publications. This line of research was initiated by Bourgain, Br\'ezis and Mironescu \cite{BBM2002, BBM2001}  who considered the case $\a\to 1$ ($k=1$). Their work was later simplified and extended by Maz'ya and  Shaposhnikova \cite{MS2002}. Here, we obtain the  following analogue for $\T^d_\t$ of their results: For $1\le p\le\8$, $1\le q<\8$ and $0<\a<k$ with $k\in\nat$,
 \be\begin{split}
 \lim_{\a\to k} (k-\a)^{\frac1q}\|x\|_{B^{\a,\o} _{p,q}} &\approx q^{-\frac1q}\sum_{m\in\nat_0^d,\,|m|_1=k} \|D^mx\|_{p}\,,\\
 \lim_{\a\to 0} \a^{\frac1q}\|x\|_{B^{\a,\o} _{p,q}} &\approx q^{-\frac1q}\|x\|_{p}
 \end{split}\ee
with relevant constants depending only on $d$ and $k$.

\medskip\noindent{\bf Interpolation}.  Our fourth family of results deal with interpolation. Like in the usual case, the interpolation of Besov spaces is quite simple, and that of Triebel-Lizorkin spaces can be easily reduced to the corresponding problem of Hardy spaces. Thus the really hard task here concerns the interpolation of Sobolev spaces for which we have obtained only partial results. The most interesting couple is $\big(W^k_1(\T^d_\t),\,W^k_\8(\T^d_\t)\big)$. Recall that the complex interpolation problem of this couple remains always unsolved even in the commutative case (a well-known longstanding open problem which is explicitly posed by P. Jones in \cite[p.~173]{HN}), while its real interpolation spaces were completely determined by DeVore and Scherer \cite{DS1979}.  We do not know, unfortunately, how to prove the quantum analogue of DeVore and Scherer's theorem.

However, we are able to extend to the quantum tori the K-functional formula of  the couple $\big(L_p(\real^d),\,W^k_p(\real^d)\big)$ obtained by  Johnen and Scherer \cite{JS1976}. This result reads as follows:
 $$K(x,\e^k;\, L_p(\T^d_\t),W^k_p(\T^d_\t)) \approx \e^k|\wh x(0)|+ \o^k_p(x,\e),\quad 0<\e\le1.$$
As a consequence, we determine the real interpolation spaces of $\big(L_p(\T^d_\t),\,W^k_p(\T^d_\t)\big)$, which are Besov spaces.

The real interpolation of $\big(L_p(\mathbb{T}_{\theta}^d),\,W^k_p(\mathbb{T}_{\theta}^d) \big)$ is closely related to the limit behavior of Besov norms described previously. We show that it implies the optimal order (relative to $\a$) of the best constant in the embedding of $B_{p,p}^\a(\T^d_\t)$ into $L_q(\T^d_\t)$ for $\frac1q=\frac1p-\frac\a d$ and $0<\a<1$, which is the quantum analogue of a result of Bourgain, Br\'ezis and Mironescu. On the other hand, the latter result is equivalent to the Sobolev embedding $W_p^1(\T_\t^d)\subset L_q(\T_\t^d)$ for $\frac1q=\frac1p-\frac1d$.

\medskip\noindent{\bf Multipliers}. The last family of results of the paper describe Fourier multipliers on the preceding spaces. Like in the $L_p$ case treated in \cite{CXY2012}, we are mainly concerned with completely bounded Fourier multipliers. All spaces in consideration carry a natural operator space structure in Pisier's sense. We show that the completely bounded Fourier multipliers on $W^k_p(\T^d_\t)$ are independent of $\t$, so they coincide with those on the usual Sobolev space $W^k_p(\T^d)$. This is the Sobolev analogue of the corresponding result for $L_p$ proved in \cite{CXY2012}. The main tool is  Neuwirth-Ricard's transference between Fourier multipliers and Schur multipliers in \cite{NR2011}. A similar result holds for the Triebel-Lizorkin spaces too.

The situation for Besov spaces is very satisfactory since it is  well known that Fourier multipliers behave much better on Besov spaces than on $L_p$-spaces (in the commutative case). We prove that a function  $\phi$ on $\ent^d$ is a (completely) bounded Fourier multiplier on $B_{p, q}^\a(\T^d_\t)$ iff the $\phi \f(2^{-k}\cdot)$'s are (completely) bounded Fourier multipliers on $L_p(\T^d_\t)$ uniformly in $k\ge0$. Consequently, the Fourier multipliers on $B_{p, q}^\a(\T^d_\t)$ are completely determined by the Fourier multipliers on $L_p(\T^d_\t)$ associated to  their components in the Littlewood-Paley decomposition. So the completely bounded multipliers on $B_{p, q}^\a(\T^d_\t)$ depend solely on $p$. In the case of $p=1$, a multiplier is bounded on $B_{1, q}^\a(\T^d_\t)$ iff it is completely bounded iff it is the Fourier transform of an element of $B_{1, \8}^0(\T^d)$. Using a classical example of Stein-Zygmund \cite{SZ1967}, we show that there exists a $\phi$ which is a completely bounded Fourier multiplier on $B_{p, q}^\a(\T^d_\t)$ for all $p$ but  bounded on $L_p(\T^d_\t)$ for no $p\neq2$.

\medskip

We will frequently use the notation $A\les B$, which is an inequality up to a constant: $A\le c\, B$ for some constant $c>0$. The relevant constants in all such inequalities may depend on the dimension $d$, the test function $\f$ or $\p$, etc. but never on the functions $f$ or distributions $x$ in consideration. The main results of this paper have been announced in \cite{XXY}.

%%%%%%%%%%%%%%%%%%%%%%%%%%%%%%%%%%%%%%%%%%%%%%%%%%%%%%%%%%%%%%%%%%%%%%%%
%%%%%%%%%%%%%%%%%%%%%%%%%%%%%%%%%%%%%%%%%%%%%%%%%%%%%%%%%%%%%%%%%%%%%%%%

{\Large\part{Preliminaries}}
\label{Preliminaries}
 \setcounter{section}{0}

%%%%%%%%%%%%%%%%%%%%%%%%%%%%%%%%%%%%%%%%%%%%%%%%%%%%%%%%%%%%%%%%%%%%%%%%
%%%%%%%%%%%%%%%%%%%%%%%%%%%%%%%%%%%%%%%%%%%%%%%%%%%%%%%%%%%%%%%%%%%%%%%%

This chapter collects the necessary preliminaries for the whole paper. The first two sections  present the definitions and some basic facts about noncommutative $L_p$-spaces and quantum tori which are the central objects of the paper. The third one contains some results on Fourier multipliers that will play a paramount role in the whole paper.  The last section gives the definitions and some fundamental results on operator-valued Hardy spaces on the usual and quantum tori. This section will be needed only starting from chapter~\ref{Triebel-Lizorkin spaces} on Triebel-Lizorkin spaces.
\bigskip

%%%%%%%%%%%%%%%%%%%%%%%%%%%%%%%%%%%%%%%%%%%%%%%%%%%%%%%%%%%%%%%%%%%%%%%%
%%%%%%%%%%%%%%%%%%%%%%%%%%%%%%%%%%%%%%%%%%%%%%%%%%%%%%%%%%%%%%%%%%%%%%%%

\section{Noncommutative $L_p$-spaces}

%%%%%%%%%%%%%%%%%%%%%%%%%%%%%%%%%%%%%%%%%%%%%%%%%%%%%%%%%%%%%%%%%%%%%%%%
%%%%%%%%%%%%%%%%%%%%%%%%%%%%%%%%%%%%%%%%%%%%%%%%%%%%%%%%%%%%%%%%%%%%%%%%

Let $\M$ be a von Neumann algebra equipped with a
normal semifinite faithful trace $\tau$ and
$S^+_{\M}$ be the set of all positive elements $x$
in $\M$ with $\tau(s(x))<\infty$, where $s(x)$ denotes the support of $x$, i.e.,
the smallest projection $e$ such that $exe=x$. Let
$S_{\M}$ be the linear span of
$S^+_{\M}$. Then every $x\in
S_{\M}$ has finite trace, and
$S_{\M}$ is a w*-dense $*$-subalgebra of
$\M$.

Let $0< p<\infty$. For any $x\in S_{\M}$, the
operator $|x|^p$ belongs to $S^+_{\M}$
(recalling $|x|=(x^*x)^{\frac{1}{2}}$). We define
$$\|x\|_p=\big(\tau(|x|^p)\big)^{\frac{1}{p}}.$$
One can check that $\|\cdot\|_p$ is a norm or $p$-norm on
$S_{\M}$ according to $p\ge1$ or $p<1$. The completion of
$(S_{\M},\|\cdot\|_p)$ is denoted by $L_p(\M)$,
which is the usual noncommutative $L_p$-space associated to
$(\M,\tau)$. For convenience, we set $L_{\infty}(\M)=\M$
equipped with the operator norm $\|\cdot\|_{\M}$. The norm of  $L_p(\M)$ will be often denoted simply by $\|\cdot\|_p$. But if different  $L_p$-spaces  appear in a same context, we will sometimes precise their norms in order to avoid possible ambiguity. The reader is referred to \cite{PX2003} and  \cite{Xu2007} for more information on noncommutative
$L_p$-spaces.

The elements of
$L_p(\M)$ can be described as closed densely defined operators
on $H$ ($H$ being the Hilbert space on which $\M$ acts). A closed densely defined operator $x$ on $H$ is said to be affiliated with $\M$ if $ux = xu$ for any unitary $u$ in the commutant $\mathcal{M}'$ of $\M$. An operator $x$ affiliated with $\M$ is said to be measurable with respect to $(\M,\tau)$ (or simply measurable) if for any $\delta > 0$ there exists a projection $e\in B(H)$ such that
$$e(H)\subset Dom(x)\;\;\mbox{and}\;\; \tau(e^\perp)\leq \delta, $$
where $Dom(x)$ defines the domain of $x.$
We denote by $L_0(\M,\tau)$, or simply $L_0(\M)$ the family of all measurable operators. For such an operator $x$, we define
$$\lambda_s(x)=\tau(e^\perp_s(|x|)),\;\;s>0$$
where $e^\perp_s(x)=\un_{(s,\infty)}(x)$ is the spectrum projection of $x$ corresponding to the interval $(s,\8),$ and
$$\mu_t(x)=\inf \{s>0: \lambda_s(x)<t\},\;\;t>0.$$
The function $s \mapsto \lambda_s(x)$ is called the distribution function of $x$ and the $\mu_t(x)$ the generalized singular numbers of $x$. Similarly to the classical case, for $0<p<\8, 0<q\leq \8,$ the noncommutative Lorentz space $L_{p,q}(\M)$ is defined to be the collection of all measurable operators $x$ such that
$$\|x\|_{p,q}=\big(\int_0^\8 (t^{\frac1p} \mu_t(x))^q \frac{dt}{t}\big)^{\frac1q} <\8.$$
Clearly, $L_{p,p}(\M)=L_{p}(\M)$. The space $L_{p,\8}(\M)$ is usually called a weak $L_p$-space, $0 < p < \8,$ and
$$\|x\|_{p,\8}=\sup_{s>0}s\lambda_s(x)^{\frac1p}.$$

Like the classical $L_p$-spaces, noncommutative $L_p$-spaces behave well with respect to interpolation. Our reference for interpolation theory is \cite{BL1976}.   Let $1\le p_0<p_1\leq \infty$, $1\le q\le \8$ and $0<\eta<1$. Then
 \beq\label{interpolation of Lp}
 \big( L_{p_0}(\M),\, L_{p_1}(\M) \big)_{\eta}= L_{p}(\M)\;\text{ and }\;
\big( L_{p_0}(\M),\, L_{p_1}(\M) \big)_{\eta, q}= L_{p,q}(\M),
\eeq
where $\frac{1}{p}=\frac{1-\eta}{p_0}+\frac{\eta}{p_1}$.

\medskip

Now we introduce noncommutative Hilbert space-valued
$L_p$-spaces $L_p(\M; H^c)$ and $L_p(\M; H^r)$, which are studied at length in \cite{JLX2006}.
Let $H$ be a Hilbert space  and $v$ a norm one element of $H$. Let $p_v$ be the orthogonal projection onto
the one-dimensional subspace generated by $v$. Then define the following row and column noncommutative $L_p$-spaces:
 \be\begin{split}
 &L_p(\M; H^{r})=(p_v\ot1_{\M}) L_p(B(H)\overline{\ot}\M),\\
 &L_p(\M; H^{c})= L_p(B(H)\overline{\otimes}\M)(p_v\ot1_{\M}),
 \end{split}\ee
where the tensor product $B(H)\overline{\ot}\M$ is equipped with the tensor trace while $B(H)$ is equipped with the usual trace. For $f\in L_p(\M; H^c)$,
 $$\|f\|_{L_p(\M;H^c)}=\|(f^*f)^{\frac{1}{2}}\|_{L_p(\M)}.$$
A similar formula holds for the row space by passing to adjoints: $f\in L_p(\M; H^{r})$ iff $f^*\in L_p(\M; H^c)$, and $\|f\|_{L_p(\M;H^r)}=\|f^*\|_{L_p(\M;H^c)}$. It is clear that $ L_p(\M; H^c)$ and $ L_p(\M; H^r)$ are 1-complemented subspaces of $L_p(B(H)\overline{\ot}\M)$ for any $p$. Thus they also form interpolation scales with respect to both complex and real interpolation methods:
Let $1\leq p_0, p_1\leq\8$ and $0<\eta<1$. Then
 \beq\label{column interpolation}
 \begin{split}
 \big( L_{p_0}(\M; H^c),\, L_{p_1}(\M; H^c) \big)_{\eta}&= L_{p}(\M; H^c),\\
\big( L_{p_0}(\M; H^c),\, L_{p_1}(\M; H^c) \big)_{\eta, p}&= L_{p}(\M; H^c),
\end{split}\eeq
where $\frac{1}{p}=\frac{1-\eta}{p_0}+\frac{\eta}{p_1}$.
The same formulas hold for row spaces too.

%%%%%%%%%%%%%%%%%%%%%%%%%%%%%%%%%%%%%%%%%%%%%%%%%%%%%%%%%%%%%%%%%%%%%%%%
%%%%%%%%%%%%%%%%%%%%%%%%%%%%%%%%%%%%%%%%%%%%%%%%%%%%%%%%%%%%%%%%%%%%%%%%

\section{Quantum tori}

%%%%%%%%%%%%%%%%%%%%%%%%%%%%%%%%%%%%%%%%%%%%%%%%%%%%%%%%%%%%%%%%%%%%%%%%
%%%%%%%%%%%%%%%%%%%%%%%%%%%%%%%%%%%%%%%%%%%%%%%%%%%%%%%%%%%%%%%%%%%%%%%%

Let $d\ge2$ and $\theta=(\theta_{kj})$ be a real skew symmetric $d\times d$-matrix. The associated $d$-dimensional noncommutative
torus $\mathcal{A}_{\theta}$ is the universal $C^*$-algebra generated by $d$
unitary operators $U_1, \ldots, U_d$ satisfying the following commutation
relation
 \beq \label{eq:CommuRelation}
 U_k U_j = e^{2 \pi \mathrm{i} \theta_{k j}} U_j U_k,\quad j,k=1,\ldots, d.
 \eeq
We will use standard notation from multiple Fourier series. Let  $U=(U_1,\cdots, U_d)$. For $m=(m_1,\cdots,m_d)\in\ent^d$ we define
 $$U^m=U_1^{m_1}\cdots U_d^{m_d}.$$
 A polynomial in $U$ is a finite sum
  $$ x =\sum_{m \in \mathbb{Z}^d}\alpha_{m} U^{m}\quad \text{with}\quad
 \alpha_{m} \in \mathbb{C},$$
that is, $\alpha_{m} =0$ for all but
finite indices $m \in \mathbb{Z}^d.$ The involution algebra
$\mathcal{P}_{\theta}$ of all such polynomials is
dense in $\A_{\theta}.$ For any polynomial $x$ as above we define
 $$\tau (x) = \alpha_{0},$$
where $0=(0, \cdots, 0)$.
Then, $\tau$ extends to a  faithful  tracial state on $\A_{\theta}$.  Let  $\mathbb{T}^d_{\theta}$ be the $w^*$-closure of $\A_{\theta}$ in  the GNS representation of $\tau$. This is our $d$-dimensional quantum torus. The state $\tau$ extends to a normal faithful tracial state on $\mathbb{T}^d_{\theta}$ that will be denoted again by $\tau$. Recall that the von Neumann algebra $\mathbb{T}^d_{\theta}$ is hyperfinite.

Any $x\in L_1(\T^d_\theta)$ admits a formal
Fourier series:
 $$x \sim \sum_{m \in\mathbb{Z}^d} \wh{x} ( m ) U^{m},$$
where \
 $$\wh x( m) = \tau((U^m)^*x),\quad m \in \mathbb{Z}^d$$
are the Fourier coefficients of $x$. The operator $x$ is, of course, uniquely determined by its Fourier series.

\medskip

We introduced in  \cite{CXY2012}  a transference method to overcome the full noncommutativity of quantum tori and use methods of operator-valued harmonic analysis. Let $\mathbb{T}^d$ be the usual $d$-torus equipped with normalized Haar measure $dz$.  Let $\mathcal{N}_{\theta} = L_{\infty} (\mathbb{T}^d) \overline{\otimes}\mathbb{T}^d_{\theta}$, equipped with the tensor trace $\nu = \int dz \otimes \tau$. It is well known that for every $0 < p < \8,$
 \be
 L_p(\mathcal{N}_{\theta}, \nu ) \cong L_p (\mathbb{T}^d; L_p(\mathbb{T}^d_{\theta})).
 \ee
The space on the right-hand side is the space of Bochner $p$-integrable functions from $\mathbb{T}^d$ to $L_p (\mathbb{T}^d_{\theta})$. In general, for any Banach space $X$ and any measure space $(\Om, \mu)$, we use $L_p(\Om; X)$ to denote the space of Bochner $p$-integrable functions from $\Om$ to $X$. For each $z
\in \mathbb{T}^d,$ define $\pi_{z}$ to be the isomorphism of
$\mathbb{T}^d_{\theta}$ determined by
\beq \label{trans-pi}
\pi_{z} (U^{m}) = z^{m} U^{m } = z_1^{m_1} \cdots z_d^{m_d}
U_1^{m_1} \cdots U_d^{m_d}. \eeq
Since $\tau (\pi_{z} (x)) = \tau
(x)$ for any $x \in \mathbb{T}^d_{\theta},$ $\pi_{z}$ preserves the
trace $\tau.$ Thus for every $0 < p < \8,$
\beq
\label{eq:TransLpIso}\|\pi_{z} (x) \|_p = \|x\|_p,\; \forall x\in
L_p(\mathbb{T}^d_{\theta}).\eeq

Now we state the transference method as follows (see \cite{CXY2012}).

\begin{lem}\label{prop:TransC}
 For any $x \in L_p (\mathbb{T}^d_{\theta})$, the function
$\wt{x} : z \mapsto \pi_z(x)$ is continuous from
$\mathbb{T}^d$ to $L_p(\mathbb{T}^d_{\theta})$ $($with respect to the w*-topology for $p=\8)$. If $x\in\A_\theta$, it is continuous from $\mathbb{T}^d$ to $\A_\theta$.
\end{lem}

\begin{cor}\label{prop:TransLp}
 \begin{enumerate}[{\rm (i)}]

\item Let $0 < p \leq \8.$ If $x\in L_p (\mathbb{T}^d_{\theta}),$ then $\wt{x} \in L_p (\N_{\theta})$ and $\|\wt{x}\|_p = \| x\|_p,$ that is, $x \mapsto \wt{x}$ is an isometric embedding from $L_p (\mathbb{T}^d_{\theta})$ into $L_p (\N_{\theta}).$ Moreover, this map is also an isomorphism from $\A_\theta$ into $C(\T^d; \A_\theta)$.

\item Let $\widetilde{\mathbb{T}^d_{\theta}} = \{\wt{x} : x \in \mathbb{T}^d_{\theta}\}.$ Then $\widetilde{\mathbb{T}^d_{\theta}}$ is a von Neumann subalgebra of $\mathcal{N}_{\theta}$ and the associated conditional expectation is given by
 $$\mathbb{E} (f)(z) = \pi_{z} \Big ( \int_{\mathbb{T}^d} \pi_{\overline{w}}
 \big [ f( w )\big ] dw \Big ),\quad z\in\T^d, \; f \in \N_{\theta}.$$
Moreover, $\mathbb{E}$ extends to a contractive projection from $L_p(\N_{\theta})$ onto $L_p(\widetilde{\mathbb{T}^d_{\theta}})$ for $1\leq p\leq \infty.$
\item $L_p (\mathbb{T}^d_{\theta})$ is isometric to $L_p (\widetilde{\mathbb{T}^d_{\theta}})$ for every $0 < p \le \8.$

\end{enumerate}

\end{cor}

%%%%%%%%%%%%%%%%%%%%%%%%%%%%%%%%%%%%%%%%%%%%%%%%%%%%%%%%%%%%%%%%%%%%%%%%
%%%%%%%%%%%%%%%%%%%%%%%%%%%%%%%%%%%%%%%%%%%%%%%%%%%%%%%%%%%%%%%%%%%%%%%%

\section{Fourier multipliers}
\label{Fourier multipliers}

%%%%%%%%%%%%%%%%%%%%%%%%%%%%%%%%%%%%%%%%%%%%%%%%%%%%%%%%%%%%%%%%%%%%%%%%
%%%%%%%%%%%%%%%%%%%%%%%%%%%%%%%%%%%%%%%%%%%%%%%%%%%%%%%%%%%%%%%%%%%%%%%%

Fourier multipliers will be the most important tool for  the whole work. Now we present some known results on them for later use.  Given a function $\phi :\ent^d\to\com$, let $M_{\phi}$ denote the associated Fourier multiplier on $\T^d$, namely, $\wh{M_{\phi}f} (m) = \phi(m)\wh f (m)$ for any trigonometric polynomial $f$ on $\T^d$. We call $\phi$ a multiplier on $L_p(\T^d)$ if $M_{\phi}$ extends to a bounded
map on $L_p (\T^d)$. Fourier multipliers on  $\T^d_\t$ are defined exactly in the same way, we still use  the same symbol $M_\phi$ to denote the corresponding multiplier on $\T^d_\t$. Note that the isomorphism $\pi_z$ defined in \eqref{trans-pi} is the Fourier multiplier associated to the function $\phi$ given by $\phi(m)=z^m$.

It is natural to ask if the boundedness of  $M_\phi$ on $L_p(\T^d)$ is equivalent to that on $L_p(\T^d_\t)$. This is open until a negative answer given by Ricard \cite{Ricard16} recently. However, it is proved in \cite{CXY2012} that the answer is affirmative if ``boundedness" is replaced by ``complete boundedness", a notion from operator space theory for which we refer to  \cite{ER2000} and  \cite{Pisier2003}. All noncommutative $L_p$-spaces are equipped with  their natural operator space structure introduced by Pisier \cite{Pisier1998, Pisier2003}.

We will use the following fundamental property of completely bounded (c.b. for short) maps due to Pisier  \cite{Pisier1998}. Let  $E$ and $F$ be operator spaces. Then a linear map $T:E\to F$ is c.b. iff ${\rm Id}_{S_p}\ot T: S_p[E]\to S_p[F]$ is bounded for some $1\le p\le\8$. In this case,
 $$\|T\|_{\rm cb}=\big\|{\rm Id}_{S_p}\ot T: S_p[E]\to S_p[F]\big\|.$$
Here $S_p[E]$  denotes the $E$-valued  Schatten $p$-class. In particular, if $E=\com$, $S_p[\com]=S_p$ is the noncommutative $L_p$-space associated to $B(\el_2)$, equipped with the usual trace. Applying this criterion to the special case where $E=F =L_p(\M)$, we see that a map $T$ on  $L_p(\M)$ is c.b. iff ${\rm Id}_{S_p}\ot T: L_p(B(\el_2)\overline{\ot}\M)\to L_p(B(\el_2)\overline{\ot}\M)$ is bounded. The readers unfamiliar with operator space theory can take this property as the definition of c.b. maps between $L_p$-spaces.

Thus $\phi$ is a c.b. multiplier on $L_p(\T^d_\t)$ if $M_{\phi}$ is c.b. on $L_p (\T^d_\t)$, or equivalently, if ${\rm Id}_{S_p}\ot M_\phi$ is bounded on $L_p(B(\el_2)\overline\ot\T^d_\t)$.  Let $\mathsf{M}(L_p (\mathbb{T}^d_{\theta}))$ (resp. $\mathsf{M}_\mathrm{cb}(L_p (\mathbb{T}^d_{\theta}))$) denote the space of Fourier multipliers (resp. c.b. Fourier multipliers) on $L_p(\T^d_\t)$,  equipped with the natural norm. When $\theta=0$, we recover the (c.b.) Fourier multipliers on the usual $d$-torus $\T^d$. The corresponding multiplier spaces are denoted by $\mathsf{M}(L_p (\mathbb{T}^d))$ and $\mathsf{M}_\mathrm{cb}(L_p (\mathbb{T}^d))$. Note that in the latter case ($\t=0$), $L_p(B(\el_2)\overline\ot\T^d)=L_p(\T^d; S_p)$, thus $\phi$ is a c.b. multiplier on $L_p(\T^d)$ iff $M_{\phi}$ extends to a bounded map on $L_p(\T^d; S_p)$.

The following result is taken from  \cite{CXY2012}.

\begin{lem}\label{cb-multiplier}
 Let $1\le p \le \8$. Then  $\mathsf{M}_{\mathrm{cb}} (L_p (\mathbb{T}^d_{\theta}))
 =\mathsf{M}_{\mathrm{cb}} (L_{p} (\mathbb{T}^d))$ with equal norms.
\end{lem}

\begin{rk}\label{multiplier-rk}
 Note that $\mathsf{M}_{\mathrm{cb}} (L_{1} (\mathbb{T}^d))=\mathsf{M}_{\mathrm{cb}} (L_{\8} (\mathbb{T}^d))$ coincides with the space of the Fourier transforms of bounded measures on $\T^d$, and $\mathsf{M}_{\mathrm{cb}} (L_{2} (\mathbb{T}^d))$ with the space of bounded functions on $\ent^d$.
 \end{rk}

The most efficient criterion for Fourier multipliers on $L_p(\T^d)$ for $1<p<\8$ is  Mikhlin's condition. Although it can be formulated in the periodic case, it is more convenient to state this condition in the case of $\real^d$. On the other hand, the Fourier multipliers on $\T^d$ used later will be the restrictions to $\ent^d$ of continuous Fourier multipliers on $\real^d$. As usual, for $m=(m_1,\cdots, m_d)\in \mathbb{N}_0^d$ (recalling that $\nat_0$ denotes the set of nonnegative integers), we set
 $$D^m=\partial_1^{m_1}\cdots \partial_d^{m_d}\,,$$
where $\partial_k$ denotes the $k$th partial derivation on $\real^d$.  Also put $|m|_1=m_1+\cdots +m_d$. The Euclidean norm of $\real^d$ is denoted by $|\cdot|$: $|\xi|=\sqrt{\xi_1^2+\cdots+\xi_d^2}$.

\begin{Def}\label{Mikhlin-def}
 A function $\phi:\real^d\to\com$ is called a Mikhlin multiplier if it is $d$-times differentiable on $\real^d\setminus\{0\}$ and satisfies the following condition
 $$
 \|\phi\|_{\rm M}=\sup\big\{|\xi|^{|m|_1}|D^m\phi(\xi)|: \xi\in\real^d\setminus\{0\},\; m\in\nat^d_0, |m|_1\le d\big\}<\8.
 $$
\end{Def}

Note that the usual Mikhlin condition requires only partial derivatives up to order $[\frac d2]+1$ (see, for instance, \cite[section~II.6]{gar-rubio} or \cite[Theorem~4.3.2]{Stein1970}). Our requirement above up to order $d$  is imposed by the boundedness of these multipliers on UMD spaces.  We refer to section~\ref{A multiplier theorem} for the usual Mikhlin condition and more multiplier results on $\T^d_\t$.

It is a classical result that every Mikhlin multiplier is a Fourier multiplier on $L_p(\real^d)$ for $1<p<\8$ (cf. \cite[section~II.6]{gar-rubio} or \cite[Theorem~4.3.2]{Stein1970}), so its restriction $\phi\big|_{\ent^d}$ is a Fourier multiplier on $L_p(\T^d)$ too. It is, however, less classical that such a multiplier is also c.b. on $L_p(\real^d)$ or $L_p(\T^d)$. This follows from a general result on UMD spaces. Recall that a Banach space $X$ is called a UMD space if the $X$-valued  martingale differences are unconditional in $L_p(\Omega; X)$ for any $1<p<\8$ and any probability space $(\Omega, P)$. This is equivalent to the requirement that the Hilbert transform be bounded on $L_p(\real^d; X)$ for $1<p<\8$. Any noncommutative $L_p$-space with $1<p<\8$ is a UMD space. We refer to \cite{bourg-UMD, burk-UMD, burk-HT} for more information.

\smallskip

Before proceeding  further, we make a convention used throughout the paper:

\smallskip\n{\bf Convention.} To simplify the notational system, we will use the same derivation symbols for $\real^d$ and $\T^d$. Thus for a multi-index $m\in\nat_0^d$, $D^m=\partial_1^{m_1}\cdots \partial_d^{m_d}$, introduced previously, will also denote the partial derivation of order $m$ on $\T^d$. Similarly, $\D=\partial_1^{2}+\cdots +\partial_d^{2}$ will denote the Laplacian on both $\real^d$ and $\T^d$.  In the same spirit,  for a function $\phi$ on $\real^d$, we will call $\phi$ rather than  $\phi\big|_{\ent^d}$ a Fourier multiplier on  $L_p(\T^d)$ or $L_p(\T^d_\t)$. This should not cause any ambiguity in concrete contexts.  Considered as a map on  $L_p(\T^d)$ or $L_p(\T^d_\t)$, $M_\phi$ will be often denoted by $f\mapsto \wt\phi*f$ or $x\mapsto \wt\phi*x$.

\smallskip

Note that the notation $\wt\phi*f$ coincides with the usual convolution when $\phi$ is good enough. Indeed, let $\wt\phi$ be the $1$-periodization of the inverse Fourier transform of $\phi$ whenever it exists in a reasonable sense:
 $$\wt\phi(s)=\sum_{m\in\ent^d}\F^{-1}(\phi)(s+m),\quad s\in\real^d\,.$$
Viewed as a function on $\T^d$, $\wt\phi$ admits the following Fourier series:
 $$\wt\phi(z)=\sum_{m\in\ent^d}\phi(m)z^m.$$
Thus for any trigonometric polynomial $f$,
 $$\wt\phi*f(z)=\int_{\T^d}\wt\phi(zw^{-1})f(w) dw,\quad z\in\T^d\,.$$

The following lemma is proved in  \cite{McC1984, Zimm1989} (see also \cite{bourg-CZ} for the one-dimensional case).

\begin{lem}\label{UMD-multiplier}
 Let $X$ be a UMD space and $1<p<\8$. Let $\phi$ be a Mikhlin multiplier. Then $\phi$ is a Fourier multiplier on  $L_p(\T^d; X)$.
 Moreover, its norm  is controlled by $\|\phi\|_{\rm M}$, $p$ and the UMD constant of $X$.
  \end{lem}

The following lemma will play a key role in the whole paper.

\begin{lem}\label{q-multiplier}
 Let $\phi$ be a function on $\real^d$.
 \begin{enumerate}[\rm(i)]
 \item If $\F^{-1}(\phi)$ is integrable on $\real^d$, then $\phi$ is a c.b. Fourier multiplier on  $L_p(\T^d_\t)$ for $1\le p\le\8$. Moreover, its c.b. norm  is not greater than $\big\|\F^{-1}(\phi)\big\|_{1}$.
 \item If $\phi$ is a Mikhlin multiplier, then $\phi$ is a c.b. Fourier multiplier on  $L_p(\T^d_\t)$ for $1< p<\8$. Moreover, its c.b. norm  is controlled by $\|\phi\|_{\rm M}$ and $p$.
 \end{enumerate}
\end{lem}

\begin{proof}
 (i) Since $\F^{-1}(\phi)$ is integrable,  $\phi$ is a c.b. Fourier multiplier on  $L_1(\T^d)$, so on  $L_p(\T^d)$ for $1\le p\le\8$ (see Remark \ref{multiplier-rk}). Consequently, by Lemma~\ref{cb-multiplier}, $\phi$ is a c.b. Fourier multiplier on  $L_p(\T^d_\t)$ for $1\le p\le\8$.
 
 (ii) It is well known that $S_p$  is a UMD space for $1<p<\8$. Thus, by Lemma~\ref{UMD-multiplier}, $\phi$ is a Fourier multiplier on  $L_p(\T^d; S_p)$; in other words, it is a c.b. Fourier multiplier on  $L_p(\T^d_\t)$.
\end{proof}

%%%%%%%%%%%%%%%%%%%%%%%%%%%%%%%%%%%%%%%%%%%%%%%%%%%%%%%%%%%%%%%%%%%%%%%%
%%%%%%%%%%%%%%%%%%%%%%%%%%%%%%%%%%%%%%%%%%%%%%%%%%%%%%%%%%%%%%%%%%%%%%%%

\section{Hardy spaces}
\label{Hardy spaces}

%%%%%%%%%%%%%%%%%%%%%%%%%%%%%%%%%%%%%%%%%%%%%%%%%%%%%%%%%%%%%%%%%%%%%%%%
%%%%%%%%%%%%%%%%%%%%%%%%%%%%%%%%%%%%%%%%%%%%%%%%%%%%%%%%%%%%%%%%%%%%%%%%

We now present some preliminaries on operator-valued Hardy spaces on $\T^d$ and Hardy spaces on $\T^d_\t$.  Motivated by the developments  of noncommutative martingale inequalities in \cite{PX1997,JX2003} and  quantum Markovian semigroups in \cite{JLX2006}, Mei \cite{Mei2007} developed the theory of  operator-valued Hardy spaces on $\real^d$. More recently, Mei's work was extended to the torus case in  \cite{CXY2012} with the objective of developing the Hardy space theory in the quantum torus setting. We now recall the  definitions and results that will be needed later.  Throughout this section, $\M$ will denote a von Neumann algebra equipped with a normal faithful tracial state $\tau$ and $\N=L_\8(\T^d)\overline\ot\M$ with the tensor trace. In our future applications, $\M$ will be $\T^d_\t$.

A cube of $\T^d$ is a product $Q=I_1\times\cdots\times I_d$, where each $I_j$ is an interval (= arc) of $\T$. As in the Euclidean case, we use $|Q|$ to denote the normalized volume (= measure) of $Q$. The whole $\T^d$ is now a cube too (of volume $1$).

We will often identify $\T^d$ with the unit cube $\mathbb{I}^d=[0,\, 1)^d$ via $(e^{2 \pi \mathrm{i} s_1},\cdots , e^{2 \pi \mathrm{i} s_d})\leftrightarrow(s_1, \cdots, s_d)$. Under this identification, the addition in $\mathbb{I}^d$ is the usual addition modulo $1$ coordinatewise; an interval of $\mathbb{I}$ is either a subinterval of $\mathbb{I}$ or a union  $[b,\, 1)\cup[0,\, a]$ with $0<a<b<1$, the latter union being the interval $[b-1,\, a]$ of $\I$ (modulo $1$). So the cubes of $\mathbb{I}^d$ are exactly those of $\T^d$.  Accordingly, functions on $\T^d$ and $\mathbb{I}^d$ are identified too; they are considered as $1$-periodic functions on $\real^d$. Thus  $\N=L_\8(\T^d)\overline\ot\M= L_\8(\I^d)\overline\ot\M$.

\medskip

We define $\BMO^c(\T^d, \M)$ to be the space of all $f\in L_2(\N)$ such that
 $$\|f\|_{\BMO^c}=\max\big\{\big\|f_{\T^d}\big\|_\M,\;
 \sup_{Q\subset \T^d  \text{cube}} \Big\|\frac1{|Q|}\int_Q\big|f(z)-\frac1{|Q|}\int_Qf(w)dw\big|^2dz\Big\|_\M^{\frac12}\big\}<\8.$$
The row  $\BMO^r(\T^d, \M)$  consists of all $f$ such that $f^*\in \BMO^c(\T^d, \M)$, equipped with $\|f\|_{\BMO^r}=\|f^*\|_{\BMO^c}$. Finally, we define mixture space $\BMO(\T^d, \M)$ as the intersection of  the column and row BMO spaces:
 $$\BMO(\T^d, \M)=\BMO^c(\T^d, \M)\cap \BMO^r(\T^d, \M),$$
 equipped with  $\|f\|_{\BMO}=\max(\|f\|_{\BMO^c}\,,\; \|f\|_{\BMO^r}).$

As in the Euclidean case, these spaces can be characterized  by the circular Poisson semigroup.  Let $\mathbb{P}_r$ denote the circular Poisson kernel of  $\T^d$:
 \beq\label{circular P}
 \mathbb{P}_r(z) = \sum_{m \in \ent^d }r^{|m|} z^{m}, \quad z\in\T^d,\; 0 \le r < 1.
 \eeq
The Poisson integral of $f\in L_1(\N)$ is
 $$\mathbb{P}_r(f) (z)=\int_{\T^d}\mathbb{P}_r(zw^{-1})f(w)dw=
 \sum_{m \in \mathbb{Z}^d } \wh f(m) r^{|m|} z^{m}.$$
Here $\wh f$ denotes, of course, the Fourier transform of $f$:
 $$\wh f(m)=\int_{\T^d}f(z)\, z^{-m}dz.$$
It is proved in \cite{CXY2012} that
 \beq\label{Poisson-BMO}
 \sup_{Q\subset \T^d  \text{cube}} \Big\|\frac1{|Q|}\int_Q\big|f(z)-\frac1{|Q|}\int_Qf(w)dw\big|^2dz\Big\|_\M\approx
 \sup_{0\le r<1}\big\|\mathbb{P}_r(|f-\mathbb{P}_r(f)|^2)\big\|_{\N}
 \eeq
with relevant constants depending only on $d$. Thus
 $$
 \|f\|_{\mathrm{BMO}^c}\approx \max\big\{\|\wh f(0)\|_\M,\;
 \sup_{0\le r<1}\big\|\mathbb{P}_r(|f-\mathbb{P}_r(f)|^2)\big\|_{\N}^{\frac12}\big\}.
 $$

Now we turn to  the operator-valued Hardy spaces on $\T^d$ which are defined by the Littlewood-Paley  functions associated to the circular Poisson kernel. For $f\in L_1(\N)$ define
 $$
 s^c(f) (z)= \Big( \int_0^1 \big | \partial_r\mathbb{P}_r(f)(z)\big |^2(1-r)dr\Big)^{\frac12}\,,\quad z\in\T^d.
$$
For $1\leq p <\infty$, let
 $$\H^c_p(\T^d, \M)=\{f\in L_1(\N): \|f\|_{\H^c_p}<\8\},$$
where
 $$\|f\|_{\H^c_p}=\|\wh f(0)\|_{L_p(\M)} +\|s^c(f)\|_{L_p(\N)}.$$
The row Hardy space  $\H^r_p(\T^d, \M)$ is defined to be the space of all $f$ such that $f^*\in\H^c_p(\T^d, \M)$, equipped with the natural norm. Then we define
  \be
  \H_p(\T^d, \M) =
 \left \{ \begin{split}
 &\H_p^c(\T^d, \M)+ \H_p^r(\T^d, \M)& \textrm{ if }\; 1\le p<2,\\
 &\H_p^c(\T^d, \M) \cap \H_p^r(\T^d, \M)  & \textrm{ if }\; 2\le  p<\8,
 \end{split} \right.
 \ee
equipped with the sum and intersection norms, respectively:
  \be
  \|f\|_{\H_p} =
 \left \{ \begin{split}
 &\inf\big\{\|g\|_{\H_p^c} +\|h\|_{\H_p^r} : f=g+h \big\} & \textrm{ if }\; 1\le p<2,\\
 &\max\big(\|f\|_{\H_p^c}\,, \;\|f\|_{\H_p^r}\big)   & \textrm{ if }\; 2\le  p<\8.
 \end{split} \right.
 \ee

The following is the main results of \cite[Section~8]{CXY2012}

\begin{lem}\label{H-BMO}
  \begin{enumerate}[\rm (i)]
   \item  Let $1<p<\8$. Then $\H_p(\T^d, \M)=L_p(\N)$ with equivalent norms.
 \item The dual space of $\H_1^c(\T^d, \M)$  coincides isomorphically with $\BMO^c(\T^d, \M)$.
 \item  Let $1<p<\8$. Then
 \be\begin{split}
 (\BMO^c(\T^d, \M),\; \H^c_1(\T^d, \M))_{\frac1p}&=\H^c_p(\T^d, \M) \\
 (\BMO^c(\T^d, \M),\; \H^c_1(\T^d, \M))_{\frac1p\,,p}&=\H^c_p(\T^d, \M).
\end{split}\ee
 \end{enumerate}
 Similar statements hold for the row and mixture spaces too.
  \end{lem}

By transference, the previous results can be transferred to the quantum torus case. The Poisson integral of an element  $x$ in $L_1(\T^d_{\theta})$ is defined by
 $$\mathbb{P}_r(x) = \sum_{m \in \mathbb{Z}^d } \wh{x} ( m ) r^{|m|} U^{m}, \quad 0 \le r < 1.$$
Its associated Littlewood-Paley $g$-function is
 $$s^c(x) =  \Big ( \int_0^1 \big |\partial_r\mathbb{P}_r(x) \big |^2(1-r)dr \Big )^{\frac12}.$$
For $1\leq p <\infty$ let
 $$\|x\|_{\H^c_p}=|\wh x(0)| +\|s^c(x)\|_{L_p(\T^d_{\t})}.$$
The column Hardy space $\H^c_p(\T^d_\t)$ is then defined to be
 $$\H^c_p(\mathbb{T}^d_{\theta})=\big\{x\in L_1(\T^d_{\theta}): \|x\|_{\H^c_p}<\8 \big\}.$$
On the other hand, inspired by \eqref{Poisson-BMO}, we define
 $$
 \BMO^c(\T^d_\t) =  \big\{x\in L_2(\T^d_\t)\;:\; \sup_{0\le r<1}
  \big\|\mathbb{P}_{r} \big(|x-\mathbb{P}_r(x)|^2\big) \big\|_{\T^d_\t}<\infty \big\},
 $$
equipped with the norm
 $$\|x\|_{\mathrm{BMO}^c}=\max\big\{|\wh x(0)|,\;
 \sup_{0\le r<1}  \big\|\mathbb{P}_r \big(|x-\mathbb{P}_r(x)|^2 \big) \big\|_{\mathbb{T}^d_{\theta}}^{\frac12}\,\big\}.$$
The corresponding row and mixture spaces are defined similarly to the preceding torus setting.

\smallskip

Lemma~\ref {H-BMO} admits the following quantum analogue:

\begin{lem}\label{q-H-BMO}
  \begin{enumerate}[\rm (i)]
   \item  Let $1<p<\8$. Then $\H_p(\T^d_\t)=L_p(\T^d_\t)$ with equivalent norms.
 \item The dual space of $\H_1^c(\T^d_\t)$  coincides isomorphically with $\BMO^c(\T^d_\t)$.
 \item  Let $1<p<\8$. Then
 $$(\BMO^c(\T^d_\t),\; \H^c_1(\T^d_\t))_{\frac1p}
 =\H^c_p(\T^d_\t)=(\BMO^c(\T^d_\t),\; \H^c_1(\T^d_\t))_{\frac1p\,,p}\,.$$
 \end{enumerate}
 Similar statements hold for the row and mixture spaces too.
  \end{lem}

In the above definition of $\H_p^c(\T^d_\t)$, the Poisson kernel can be replaced by any reasonable smooth function.  Let $\p$ be a Schwartz function on $\real^d$ and $\p_j$ be the function whose Fourier transform is $\p(2^{-j}\cdot)$. The map  $x\mapsto \wt\p_j*x$ is the Fourier multiplier on $\T^d_\t$ associated to $\p_j$. Now  we define the square function  associated to $\p$ of an element $x\in L_1(\T^d_\t)$ by
  $$s^{c}_\p(x)=\Big(\sum_{j\ge0}\big|\wt \p_j*x\big|^2\Big)^{\frac12}.$$
The following lemma is the main result of \cite{XXX}. We will need it essentially in the case of $p=1$.

 \begin{lem}\label{Hp-discrete}
 Let $1\le p<\8$, and let $\p$ be a Schwartz function that does not vanish in $\{\xi: 1\le|\xi|<2\}$.  Then $x\in \H_p^c(\T^d_\t)$ iff $s^{c}_\p(x)\in L_p(\T^d_\t)$. In this case, we have
$$\|x\|_{\mathcal{H}_p^c}\approx |\wh x(0)|+ \| s^{c}_\p(x)\|_{L_p(\T^d_\t)}\,,$$
where the equivalence constants depend only on $d, p$ and $\p$.
\end{lem}

%%%%%%%%%%%%%%%%%%%%%%%%%%%%%%%%%%%%%%%%%%%%%%%%%%%%%%%%%%%%%%%%%%%%%%%%
%%%%%%%%%%%%%%%%%%%%%%%%%%%%%%%%%%%%%%%%%%%%%%%%%%%%%%%%%%%%%%%%%%%%%%%%

{\Large\part{Sobolev spaces}}
 \setcounter{section}{0}
%%%%%%%%%%%%%%%%%%%%%%%%%%%%%%%%%%%%%%%%%%%%%%%%%%%%%%%%%%%%%%%%%%%%%%%%
%%%%%%%%%%%%%%%%%%%%%%%%%%%%%%%%%%%%%%%%%%%%%%%%%%%%%%%%%%%%%%%%%%%%%%%%

This chapter starts with a brief introduction to distributions on quantum tori. We then pass to the definitions of Sobolev spaces on $\T^d_\t$ and give some fundamental properties of them. Two families of Sobolev spaces are studied: $W^k_p(\T^d_\t)$ and the fractional Sobolev spaces $H^\a_p(\T^d_\t)$. We  prove a Poincar\'e type inequality for $W^k_p(\T^d_\t)$ for any $1\le p\le\8$. Our approach to this inequality seems very different from existing proofs for such an inequality in the classical case. We show that $W^k_\8(\T^d_\t)$  coincides with  the Lipschitz class of order $k$, studied by Weaver \cite{Weaver1996, Weaver1998}. We conclude the chapter with a section on the link between the quantum Sobolev spaces and the vector-valued Sobolev spaces on the usual $d$-torus $\T^d$.

\bigskip
%%%%%%%%%%%%%%%%%%%%%%%%%%%%%%%%%%%%%%%%%%%%%%%%%%%%%%%%%%%%%%%%%%%%%%%%
%%%%%%%%%%%%%%%%%%%%%%%%%%%%%%%%%%%%%%%%%%%%%%%%%%%%%%%%%%%%%%%%%%%%%%%%

\section{Distributions on quantum tori}
\label{Distributions on quantum tori}

%%%%%%%%%%%%%%%%%%%%%%%%%%%%%%%%%%%%%%%%%%%%%%%%%%%%%%%%%%%%%%%%%%%%%%%%
%%%%%%%%%%%%%%%%%%%%%%%%%%%%%%%%%%%%%%%%%%%%%%%%%%%%%%%%%%%%%%%%%%%%%%%%

In this section we give an outline of the distribution theory on quantum tori. Let
 $$\mathcal{S}(\T^d_\t)=\big\{ \sum_{m\in\ent^d} a_m U^m : \{a_m\}_{m\in\ent^d} \; \mbox{rapidly decreasing}\big\}.$$
This is a w*-dense $*$-subalgebra of $\mathbb{T}_{\theta}^d$ and contains all polynomials. We simply write  $\mathcal{S}(\mathbb{T}^d_0)= \mathcal{S}(\mathbb{T}^d)$, the algebras of infinitely differentiable functions on $\T^d$. Thus for a general $\t$,  $\mathcal{S}(\mathbb{T}^d_\t)$ should be viewed as a noncommutative deformation of  $\mathcal{S}(\mathbb{T}^d)$. We will need the differential structure on $\mathcal{S}(\mathbb{T}^d_\t)$, which is similar to that on $\mathcal{S}(\mathbb{T}^d)$.

According to our convention made in section~\ref{Fourier multipliers} and in order to lighten the notational system, we will use the same derivation notation on $\T^d_\t$ as on $\T^d$. For every $1\le j\le d$, define the following derivations, which are operators on $\mathcal{S}(\T^d_\t)$:
 $$\partial_j(U_j)=2\pi {\rm i}\, U_j \; \text{ and }\; \partial_j (U_k)=0\; \text{ for } \; k\neq j.$$
These operators $\partial_j$ commute with the adjoint operation $*$,
and play the role of the partial derivations in the classical analysis on the usual $d$-torus. Given $m =(m_1,\ldots,m_d)\in \mathbb{N}_0^d$, the associated partial derivation $D^m$ is
$\partial_1^{m_1}\cdots \partial_d^{m_d}$. We also use $\Delta$ to denote the Laplacian:
 $\D=\partial_1^2+\cdots+\partial_d^2$. The elementary fact expressed in the following remark will be frequently used later on.

Restricted to $L_2(\T_{\theta}^d)$, the partial derivation $\partial_j$ is a densely defined closed (unbounded) operator whose adjoint is equal to $-\partial_j$. This is an immediate consequence of the following obvious fact (cf. \cite{Rosenberg2008}):

\begin{lem}\label{lem:integrationbypart}
If $x,y\in\mathcal{S}(\T^d_\t)$, then $\tau(\partial_j(x)y)=-\tau(x \partial_j(y))$ for $j=1,\cdots, d$.
\end{lem}

Thus $\Delta= - (\partial_1^* \partial_1+\cdots+\partial_d^* \partial_d)$, so $-\Delta$  is a positive operator on $L_2(\T_{\theta}^d)$ with spectrum  equal to $\{4\pi ^2 |m|^2: m\in \mathbb{Z}^d\}$.

 \begin{rk}\label{derivation}
 Given $u\in\real^d$ let $e_u$  be the function on $\real^d$ defined by $e_u(\xi)=e^{2\pi{\rm i} u\cdot\xi}$, where $u\cdot\xi$ denotes the inner product of $\real^d$. The Fourier multiplier on $\T^d_\t$ associated  to $e_u$ coincides with $\pi_z$ in  \eqref{trans-pi}  with $z=(e^{2\pi{\rm i} u_1},\cdots\, e^{2\pi{\rm i} u_d})$. This Fourier multiplier will play an important role in the sequel. By analogy with the classical case, we will call it the translation by $u$ and denote it by $T_u$: $T_u(x)=\pi_z(x)$ for any $x\in\mathcal{S}(\T^d_\t)$. Then it is clear that
 \beq\label{derivation}
   \frac{\partial}{\partial u_j}T_u(x)=T_u(\partial_j x)\,,\;\text {so }\; \frac{\partial}{\partial u_j}T_u(x)\Big|_{u=0}=\partial_j x.
 \eeq

 \end{rk}

Following the classical setting as in \cite{Ed1979}, we now endow $\mathcal{S}(\T^d_\t)$ with an appropriate topology. For each $k\in\nat_0$ define a norm $p_k$ on $\mathcal{S}(\T^d_\t)$ by
 $$p_k(x)= \sup_{0\leq |m|_1 \leq k} \|D^{m}x\|_\8.$$
The sequence $\{p_k\}_{k\ge0}$ induces a locally convex topology on  $\mathcal{S}(\T^d_\t)$. This topology is metrizable by the following distance:
 $$d(x, y) = \sum_{k=0}^{\8} \frac{2^{-k} p_k(x-y)}{1+p_k(x-y)}$$
with respect to which $\mathcal{S}(\T^d_\t)$ is complete, an easily checked fact. The following simple proposition describes the convergence in $\mathcal{S}(\T^d_\t)$.

\begin{prop}\label{convergeA}
A sequence $\{x_n\} \subset \mathcal{S}(\T^d_\t)$ converges to $x\in\mathcal{S}(\T^d_\t)$ if and only if for every $m \in \mathbb{N}_0^d$, $D^m x_n \rightarrow D^m x$ in $\mathbb{T}_\theta ^d.$
\end{prop}

\begin{proof}
Without loss of generality, we assume $x=0$. Suppose that $x_n \rightarrow 0 $ in $\mathcal{S}(\T^d_\t)$. Then for $m \in \mathbb{N}_0^d$ and $\e>0$, there exists an integer $N$ such that for every $n\geq N,$
 $$d(x_n, 0) = \sum_{k=0}^{\8} \frac{2^{-k} p_k(x_n)}{1+p_k(x_n)} \leq 2^{-|m|_1} \frac{\e}{1+\e}.$$
Thus, $p_{|m|_1} (x_n) \leq \e$, so $\|D^m x_j\|_\8 \leq \e$, which means $D^m x_n\rightarrow 0$ in $\mathbb{T}_\theta ^d.$

Conversely, assume that for every $m \in \mathbb{N}_0^d$, $D^m x_n \rightarrow 0$ in $\mathbb{T}_\theta ^d$.  For  $\e >0$ let $N_0$ be an integer such that $\sum_{k>N_0} 2^{-k} <\frac{\e}{2}.$ Since $D^m x_n \rightarrow 0$ for $|m|_1 \leq N_0,$ there exists  $N\in \mathbb{N}$ such that for $n > N,$
$$\sum_{k=0}^{N_0} \frac{2^{-k} p_k(x_n)}{1+p_k(x_n)} < \frac{\e}{2}.$$
Consequently, $d(x_n, 0)  < \e$.
 \end{proof}

\begin{Def}
 A distribution on $\T_\theta ^d$ is a continuous linear functional on $\mathcal{S}(\T^d_\t)$.  $\mathcal{S}'(\T^d_\t)$ denotes the space of distributions.
\end{Def}

As a dual space,  $\mathcal{S}'(\T^d_\t)$ is endowed with the w*-topology. We will use the bracket $\la\,,\,\ra$ to denote the duality between   $\mathcal{S}(\T^d_\t)$ and $\mathcal{S}'(\T^d_\t)$: $\la F,\,  x\ra=F(x)$ for $x\in \mathcal{S}(\T^d_\t)$ and $F\in \mathcal{S}'(\T^d_\t)$.   We list some elementary properties of distributions:
\begin{enumerate}[{\rm(1)}]
\item For $1\le p\le\8$, the space $L_p(\T_{\theta}^d)$ naturally embeds into $\mathcal{S}'(\T^d_\t)$: an element $y\in L_p(\T_{\theta}^d)$ induces a continuous functional on $\mathcal{S}(\T^d_\t)$ by $x\mapsto \tau(yx)$.
  \item $\mathcal{S}(\T^d_\t)$ acts as a bimodule on $\mathcal{S}'(\T^d_\t)$: for $a, b \in \mathcal{S}(\T^d_\t)$ and  $F\in\mathcal{S}'(\T^d_\t)$, $aF b$ is the distribution defined by $x\mapsto \la F, \, bxa\ra$.
 \item The partial derivations $\partial_j$ extend to $\mathcal{S}'(\T^d_\t)$ by duality: $\la \partial_j F,\,x\ra=- \la F,\,\partial_j x\ra$.
For $m \in \mathbb{N}_0^d$, we use again $D^{m}$ to denote the associated partial derivation on $\mathcal{S}'(\T^d_\t)$.
 \item The Fourier transform extends to $\mathcal{S}'(\T^d_\t)$: for $F\in \mathcal{S}'(\T^d_\t)$ and $m\in\ent^d$, $\wh F(m)=\la F,\, (U^m)^*\ra$. The Fourier series  of $F$  converges to $F$ according to any (reasonable) summation method:
  $$F=\sum_{m\in\ent^d}\wh F(m) U^m\,.$$
\end{enumerate}

%%%%%%%%%%%%%%%%%%%%%%%%%%%%%%%%%%%%%%%%%%%%%%%%%%%%%%%%%%%%%%%%%%%%%%%%
%%%%%%%%%%%%%%%%%%%%%%%%%%%%%%%%%%%%%%%%%%%%%%%%%%%%%%%%%%%%%%%%%%%%%%%%

\section{Definitions and basic properties}
\label{Definitions and basic properties: Sobolev}

%%%%%%%%%%%%%%%%%%%%%%%%%%%%%%%%%%%%%%%%%%%%%%%%%%%%%%%%%%%%%%%%%%%%%%%%
%%%%%%%%%%%%%%%%%%%%%%%%%%%%%%%%%%%%%%%%%%%%%%%%%%%%%%%%%%%%%%%%%%%%%%%%

We begin this section with a simple observation on Fourier multipliers on $\mathcal{S}(\T^d_\t)$ and $\mathcal{S}'(\T^d_\t)$. Let $\phi: \ent^d\to\com$ be a function of polynomial growth. Then its associated Fourier multiplier $M_\phi$ is a continuous map on both $\mathcal{S}(\T^d_\t)$ and $\mathcal{S}'(\T^d_\t)$. Here and in the sequel, a generic element of $\mathcal{S}'(\T^d_\t)$ is also denoted by $x$. Two specific Fourier multipliers will play a key role later: they are the Bessel and Riesz potentials.

\begin{Def}
 Let $\a\in\real$. Define $J_\a$ on $\real^d$ and $I_\a$ on $\real^d\setminus\{0\}$ by
 $$J_\a(\xi)=(1+|\xi|^2)^{\frac\a2}\;\text{ and }\; I_\a(\xi)=|\xi|^{\a}\,.$$
Their associated Fourier multipliers are the Bessel and Riesz potentials of order $\a$, denoted by $J^a$ and $I^\a$, respectively.
 \end{Def}

By the above observation,  $J^\a$  is a Fourier multiplier on $\mathcal{S}'(\T^d_\t)$, and $I^\a$ is also a Fourier multiplier on the subspace of $\mathcal{S}'(\T^d_\t)$ of all $x$ such that $\wh x(0)=0$. Note that
 $$J^\a=(1-(2\pi)^{-2}\D)^{\frac\a2}\;\text{ and }\; I^\a=(-(2\pi)^{-2}\D)^{\frac{\a}{2}}.$$

\begin{Def}\label{Sobolev def}
 Let $1\leq p\leq \infty$, $k\in \mathbb{N}$  and $\a\in\real$.
 \begin{enumerate}[\rm(i)]
 \item The Sobolev space of order $k$  on $\mathbb{T}_{\theta}^d$ is defined to be
 $$W_p^k(\mathbb{T}_{\theta}^d)= \big\{ x\in\mathcal{S}'(\T^d_\t) :  D^{m} x \in L_p(\mathbb{T}_{\theta}^d) \textrm{ for each }m\in \mathbb {N}_0^d \textrm{ with } |m|_1\leq k \big\},$$
equipped with the  norm
 $$\|x\|_{W_p^k}=\Big(\sum_{0\leq |m|_1\leq k}\|D^{m}x\|_{p}^p\Big)^{\frac{1}{p}}.$$
 \item The potential (or fractional)  Sobolev space of order $\a$  is defined to be
 $$H_p^\alpha(\mathbb{T}_{\theta}^d)=\big\{ x\in\mathcal{S}'(\T^d_\t) : J^\a x\in L_p(\mathbb{T}_{\theta}^d) \big\},$$
equipped with the  norm
 $$\|x\|_{H_p^\alpha}=\|J^\a x\|_p\,.$$
 \end{enumerate}
\end{Def}

In the above definition of $\|x\|_{W_p^k}$, we have followed the usual convention for $p=\8$ that  the right-hand side is replaced by the corresponding supremum. This convention will be always made in the sequel. We collect some basic properties of these spaces in the following:

 \begin{prop}\label{Sobolev-P}
 Let $1\le p\le\8$, $k\in\nat$ and $\a\in\real$.
\begin{enumerate}[\rm(i)]
\item $W_p^k(\mathbb{T}_{\theta}^d)$ and $H_p^\alpha (\mathbb{T}_{\theta}^d)$ are Banach spaces.
\item The polynomial subalgebra $\mathcal{P}_{\theta}$ of $\mathbb{T}_{\theta}^d$ is dense in $W^{k}_{p} (\mathbb{T}_{\theta}^d)$ and $H_p^\alpha (\mathbb{T}_{\theta}^d)$ for $1\le p<\infty$. Consequently, $\mathcal{S}(\T^d_\t)$ is dense in $W_p^k(\mathbb{T}_{\theta}^d)$ and $H_p^\alpha (\mathbb{T}_{\theta}^d)$.
\item For any $\b\in\real$, $J^\b$ is an isometry from $H_p^\a(\mathbb{T}_{\theta}^d)$ onto $H_p^{\a-\b} (\mathbb{T}_{\theta}^d)$. In particular, $J^{\a}$ is an isometry from $H_p^\a(\mathbb{T}_{\theta}^d)$ onto $L_p(\mathbb{T}_{\theta}^d)$.
\item $H_p^\a(\mathbb{T}_{\theta}^d)\subset H_p^\b(\mathbb{T}_{\theta}^d)$ continuously whenever $\b<\a$.
\end{enumerate}
\end{prop}

\begin{proof}
 (iii) is obvious. It implies (i) for $H_p^\alpha (\mathbb{T}_{\theta}^d)$.

 (i)  It suffices to show that $W_p^k(\mathbb{T}_{\theta}^d)$ is complete. Assume that $\{x_n\}_n \subset W_p^k(\mathbb{T}_{\theta}^d)$ is a Cauchy sequence. Then for every $|m|_1\leq k,$ $\{D^m x_n\}_n$ is a Cauchy sequence in $L_p(\mathbb{T}_{\theta}^d)$, so $D^m x_n\rightarrow y_{m}$ in $L_p(\mathbb{T}_{\theta}^d)$.
Particularly, $\{D^m x_n\}_n$ converges to $y_{m}$ in $\mathcal{S}'(\T^d_\t)$.
On the other hand, since $x_n \rightarrow y_0$ in $L_p(\mathbb{T}_{\theta}^d)$, for every $x\in\mathcal{S}(\T^d_\t)$ we have
 $\tau(x_n D^m x)\rightarrow \tau(y_0 D^m x).$
Thus $\{D^m x_n\}_n$ converges to $D^m y_0$ in $\mathcal{S}'(\T^d_\t)$.  Consequently, $D^m y_0 = y_m$ for $|m|_1 \leq k.$ Hence, $y_0 \in W_p^k(\mathbb{T}_{\theta}^d)$ and $x_n \rightarrow y_0$ in $W_p^k(\mathbb{T}_{\theta}^d)$.

(ii) Consider the  square Fej\'er mean
 $$F_N (x) = \sum_{m \in \mathbb{Z}^d,\, \max_j|m_j |\leq N} \Big ( 1-\frac{|m_1|}{N+1} \Big )
 \cdots \Big( 1-\frac{|m_d|}{N+1} \Big ) \wh{x} ( m) U^{m}.$$
By \cite[Proposition~3.1]{CXY2012}, $\lim_{N\to\8}F_N (x) =x$ in $L_p(\mathbb{T}_{\theta}^d)$.
On the other hand, $F_N$ commutes with $D^m$: $F_N (D^mx) =D^m F_N (x)$. We then deduce that $\lim_{N\to\8}F_N (x) =x$ in $W^k_p(\mathbb{T}_{\theta}^d)$ for every $x\in W^k_p(\mathbb{T}_{\theta}^d)$. Thus $\mathcal{P}_{\theta}$ is dense in $W^{k}_{p} (\mathbb{T}_{\theta}^d)$. On the other hand,  $F_N$ and $J^\a$  commute; so by (iii), the density of $\mathcal{P}_{\theta}$ in $L_p(\mathbb{T}_{\theta}^d)$ implies its density in $H_p^\a(\mathbb{T}_{\theta}^d)$.

(iv) It is well known that if $\gamma<0$, the inverse Fourier transform of $J_\gamma$ is an integrable function on $\real^d$ (see \cite[Proposition~V.3]{Stein1970}). Thus,  Lemma~\ref{q-multiplier} implies that  $J^{\b-\a}$ is a bounded map on $L_p(\mathbb{T}_{\theta}^d)$ with norm majorized by  $\big\|\F^{-1}(J_{\b-\a})\big\|_{L_1(\real^d)}$. This is the desired assertion.
 \end{proof}

The following shows that the potential Sobolev spaces can be also characterized by the Riesz potential.

\begin{thm}
 Let $1\le p\le\8$. Then
 $$\|x\|_{H_p^\a}\approx \Big(|\wh x(0)|^p+\|I^\a (x-\wh x(0))\|_p^p\Big)^{\frac1p},$$
where the equivalence constants depend only on $\a$ and $d$.
\end{thm}

\begin{proof}
  By changing $\a$ to $-\a$, we can assume  $\a>0$. It suffices to show $\|I^\a x\|_p\approx \|J^\a x\|_p$ for $\wh x(0)=0$.  By \cite[Lemma~V.3.2]{Stein1970}, $\frac{I_\a} {J_\a}$ is the Fourier transform of a bounded measure on $\real^d$, which, together with Lemma~\ref{q-multiplier}, yields  $\|I^\a x\|_p\les \|J^\a x\|_p$.

To show the converse inequality, let $\eta$ be an infinite differentiable function on $\real^d$ such that $\eta (\xi)=0$ for $|\xi|\le\frac12$ and $\eta (\xi)=1$ for $|\xi|\ge1$, and let $\phi=J_\a I_{-\a}\eta$. Then the Fourier multiplier with symbol $\phi I_\a$ coincides with $J^\a$ on the subspace of distributions on $\T^d_\t$ with vanishing Fourier coefficients at the origin. Thus we are reduced to proving $\F^{-1}(\phi)\in L_1(\real^d)$. To that end, first observe that for any $m\in\nat_0^d$,
 $$\big|D^m\phi(\xi)\big|\les  \frac{1}{|\xi|^{|m|_1+2}}\,.$$
Consider first the case $d\ge3$. Choose  positive integers $\el$ and $k$ such that $\frac{d}2-2<\el<\frac{d}2$ and $k>\frac{d}2$. Then by the Cauchy-Schwarz inequality and the Plancherel theorem,
 \be\begin{split}
 \Big(\int_{|s|<1} |\F^{-1}\phi(s)|ds \Big)^2
 &\le \int_{|s|<1} |s|^{-2\el}  ds\, \int_{|s|<1} |s|^{2\el}|\F^{-1}\phi(s)|^2ds\\
 &\les\sum_{m\in\nat_0^d, |m|_1=\el} \int_{\real^d} |D^{m}\phi(\xi)|^2d\xi \\
 &\les \int_{|\xi|\ge \frac12}   \frac{1}{|\xi|^{2(\el+2)}} d\xi \les 1.
 \end{split}\ee
On the other hand,
 \be\begin{split}
 \Big(\int_{|s|\geq 1} |\F^{-1}\phi(s)|ds \Big)^2
 &\leq \int_{|s|\geq 1} |s|^{-2k}  ds \, \int_{|s|\geq 1} |s|^{2k}|\F^{-1}\phi(s)|^2ds\\
 &\les\sum_{m\in\nat_0^d, |m|_1=k} \int_{\real^d} |D^{m}\phi(\xi)|^2d\xi\les1.
 \end{split}\ee
Thus $\F^{-1}(\phi)$ is integrable for $d\ge3$.

If $d\le2$, the second part above remains valid, while the first one should be modified since the required positive integer $\el$ does not exist for $d\le2$. We will consider $d=2$ and $d=1$ separately. For $d=2$, choosing $0<\e<\frac12$, we have
 $$
 \int_{|s|<1} |\F^{-1}\phi(s)|ds
 \leq \Big(\int_{|s|<1} |s|^{-2\e}  ds \Big)^{\frac12} \Big( \int_{\real^d} |s|^{2\e} |\F^{-1}\phi(s)|^{2}d s\Big)^{\frac12}
 \les \|I^{\e}\phi\|_2\,.$$
Writing $I^{\e}=I^{\e-1}\,I^{1}$ and using the classical  Hardy-Littlewood-Sobolev inequality (see \cite[Theorem~V.1]{Stein1970}) and  the Riesz transform, we obtain
  $$\|I^{\e}\phi\|_2\les  \|I^{1}\phi\|_q\approx \|\nabla\phi\|_q\les \Big(\int_{|\xi|\ge\frac12} \frac{1}{|\xi|^{3q}} d\xi \Big)^{\frac1q} \les1\;,$$
  where $\frac1q =1-\frac\e2$ (so $1<q<\8$).
Thus we are done in the case $d=2$.

It remains to deal with the one-dimensional case. Write
 $$\phi (\xi)=(1+\xi^{-2})^{\frac\a2}\eta(\xi)=\eta(\xi)+\rho(\xi)\eta(\xi),\quad \xi\in\real\setminus\{0\},$$
where $\rho(\xi)={\rm O}(\xi^{-2})$ as $|\xi|\to\8$. Since $\eta-1$ is infinitely differentiable and supported by $[-1,\,1]$, its inverse Fourier transform is integrable. So $\eta$ is the Fourier transform of a finite measure on $\real$. On one hand, as $\rho\eta\in L_1(\real^d)$, $\F^{-1}(\rho\eta)$ is a bounded continuous function, so it is integrable on  $[-1,\,1]$. On the other hand, by the second part of the preceding argument for $d\ge3$, we see that $\F^{-1}(\rho\eta)$ is integrable outside  $[-1,\,1]$ too, whence $\F^{-1}(\rho\eta)\in L_1(\real^d)$. We thus deduce that $\phi$ is the Fourier transform of a finite measure on $\real$.  Hence the assertion is completely proved.
\end{proof}

\begin{thm}\label{q-Sobolev-Bessel}
 Let $1<p<\8$. Then $H_p^k(\T_\t^d)=W_p^k(\mathbb{T}_{\theta}^d)$ with equivalent norms.
\end{thm}

\begin{proof}
 This proof is based on Fourier multipliers by virtue of Lemma~\ref{q-multiplier}.  For any $m\in\nat_0^d$ with $|m|_1\le k$, the function $\phi$, defined by $\phi(\xi)=(2\pi{\rm i})^{|m|_1}\xi^m(1+|\xi|^2)^{-\frac k2}$, is clearly a Mikhlin multiplier. Then for any $x\in\mathcal{S}'(\T^d_\t)$,
  $$\|D^mx\|_p=\|M_\phi(J^kx)\|_p\les \|J^kx\|_p\,,$$
 whence $\|x\|_{W^k_p}\les \|x\|_{H^k_p}$. To prove the converse inequality, choose an infinite differentiable function $\chi$ on $\real$ such that $\chi=0$ on $\{\xi: |\xi|\le 4^{-1}\}$ and $\chi=1$ on $\{\xi: |\xi|\ge 2^{-1}\}$. Let
  $$\phi(\xi)=\frac{(1+|\xi|^2)^{\frac k2}}{1+\chi(\xi_1)|\xi_1|^k+\cdots+\chi(\xi_d)|\xi_d|^k}\;\text{ and }\;
  \phi_j(\xi)=\frac{\chi(\xi_j)|\xi_j|^k}{(2\pi{\rm i}\,\xi_j)^k}\,,\quad 1\le j\le d.$$
 These are Mikhlin multipliers too,  and
  $$J^kx =M_\phi(x+M_{\phi_1}\partial_1^kx+\cdots+M_{\phi_d}\partial_d^kx).$$
 It then follows that
  $$\|x\|_{H^k_p}\les\Big( \|x\|_p^p+\sum_{j=1}^d\|\partial_j^kx\|_p^p\Big)^{\frac1p}\le \|x\|_{W^k_p}\,.$$
 The assertion is thus proved.
\end{proof}

\begin{rk}
Incidentally, the above proof shows that if $1<p<\8$, then
 $$\|x\|_{W^k_p}\approx \Big( \|x\|_p^p+\sum_{j=1}^d\|\partial_j^kx\|_p^p\Big)^{\frac1p}$$
with relevant constants depending only on $p$ and $d$.
\end{rk}

However, if one allows the above sum to run over all partial derivations of order $k$, then $p$ can be equal to $1$ or $\8$. Namely, for any  $1\le p\le\8$,
 $$\|x\|_{W^k_p}\approx \Big( \|x\|_p^p+\sum_{m\in\nat_0^d,\,|m|_1=k}\|D^m x\|_p^p\Big)^{\frac1p}$$
with relevant constants depending only on $d$. This equivalence can be proved by standard arguments (see  Lemma~\ref{0-k} below and its proof).  In fact, we have a  nicer result, a Poincar\'e-type inequality:
 $$\|x\|_p\les \sum_{j=1}^d\|\partial_jx\|_p$$
for any $x\in W_p^1(\T^d_\t)$ with $\wh x(0)=0$. So $\|x\|_p$ can be removed from the right-hand side of the above equivalence. This inequality will be proved in the next section.

\medskip

We conclude this section with an easy description of the dual space of $W_p^k(\mathbb{T}_{\theta}^d)$.
Let $N= N(d,k)=\sum_{m\in\nat_0^d,\, 0\leq|m|_1\leq k }1$ and
 $$L_p^N=\prod_{j=1}^N L_p(\mathbb{T}_{\theta}^d)\;\text{ equipped with the norm }\;
 \|x\|_{L_p^N}=\Big( \sum_{j=1}^N\|x_j\|_p^p\Big)^{\frac{1}{p}}\,.$$
The map $x\mapsto (D^mx)_{0\leq|m|_1\leq k }$ establishes an isometry from $W_p^k(\mathbb{T}_{\theta}^d)$ into $L_p^N$. Therefore, the dual   of $W_p^k(\mathbb{T}_{\theta}^d)$ with $1\le p<\8$ is identified with a quotient  of $L_{p'}^N$, where $p'$ is the conjugate index of $p$. More precisely,  for every $\ell\in (W^k_p(\mathbb{T}_{\theta}^d))^{*}$  there exists an element $y=(y_{m})_{m\in\nat_0^d,\, 0\leq|m|_1\leq k }\in L_{p'}^N$ such that
\beq\label{dualparity}
\ell(x)=\sum_{0\leq|m|_1\leq k }\tau(y_{m} D^{m}x),\;\;\forall x\in W_p^k(\mathbb{T}_{\theta}^d),
\eeq
and
 $$
\|\ell\|_{(W^k_p)^{*}}=\inf \|y\|_{L_{p'}^N},
 $$
the infimum running over all $y \in L_{p'}^N$ as above.

$(W_p^k(\mathbb{T}_{\theta}^d))^*$ can be described as a space of distributions. Indeed, let  $\ell\in (W^k_p (\mathbb{T}_{\theta}^d))^*$ and $y \in L_{p'}^N$ be a representative of $\ell$ as in \eqref{dualparity}. Define $\el_y\in\mathcal{S}'(\T^d_\t)$ by
 \beq\label{dual}
\el_y= \sum_{0\leq |m|_1\leq k} (-1)^{|m|_1}D^{m}y_m.
 \eeq
Then
$$\el_y(x)=\sum_{0\leq |m|_1 \leq k}\tau(y_m D^{m}x) = \ell(x),\quad x\in \mathcal{S}(\T^d_\t).$$
So  $\ell$ is  an extension of $\el_y$; moreover,
 $$\|\ell\|_{(W^k_p)^*}= \min\{ \|y\|_{L_{p'}^N}: \; \ell \; \text{ extends } \el_y  \text{ given by } \eqref{dual}\}.$$
Conversely, suppose $\el$ is an element of $\mathcal{S}'(\T^d_\t)$ of the above form $\el_y$  for some
$y\in L_{p'}^N$. Then by the density of $\mathcal{S}(\T^d_\t)$ in $W^k_p (\mathbb{T}_{\theta}^d)$,  $\el$ extends uniquely to a continuous functional on  $W^k_p (\mathbb{T}_{\theta}^d)$. Thus we have proved the following

\begin{prop}
Let $1\leq p<\infty$ and  $W^{-k}_{p'}(\mathbb{T}_{\theta}^d)$ be the space of those distributions $\el$ which admit a representative  $\el_y$ as above, equipped with the norm $\inf \{ \|y\|_{L_{p'}^N}: y \; \text{as in}\; \eqref{dual}\}$. Then $(W^k_p(\mathbb{T}_{\theta}^d))^*$ is isometric to  $W^{-k}_{p'}(\mathbb{T}_{\theta}^d)$.
\end{prop}

Note that the duality problem for the potential Sobolev spaces is trivial. Since $J^{\a}$ is an isometry between $H_p^\a(\T^d_\t)$ and  $L_p(\T^d_\t)$, we see that for $1\le p<\8$ and $\a\in\real$, the dual space of $H_p^\a(\T^d_\t)$  coincides with $H_{p'}^{-\a}(\T^d_\t)$  isometrically.

%%%%%%%%%%%%%%%%%%%%%%%%%%%%%%%%%%%%%%%%%%%%%%%%%%%%%%%%%%%%%%%%%%%%%%%%
%%%%%%%%%%%%%%%%%%%%%%%%%%%%%%%%%%%%%%%%%%%%%%%%%%%%%%%%%%%%%%%%%%%%%%%%

\section{A Poincar{\'e}-type inequality}

%%%%%%%%%%%%%%%%%%%%%%%%%%%%%%%%%%%%%%%%%%%%%%%%%%%%%%%%%%%%%%%%%%%%%%%%
%%%%%%%%%%%%%%%%%%%%%%%%%%%%%%%%%%%%%%%%%%%%%%%%%%%%%%%%%%%%%%%%%%%%%%%%

For $x\in W_p^k(\T_{\t}^d)$ let
 $$|x|_{ W^k_p}=\Big(\sum_{m\in\nat_0^d,\,|m|_1=k}\|D^mx\|_p^p\Big)^{\frac1p}\,.$$

\begin{thm}\label{k-Sobolev}
 Let $1\le p\le\8$. Then for any $x\in W_p^1(\T_{\t}^d)$,
  $$\|x-\wh x(0)\|_p\les |x|_{ W^1_p}\,.$$
  More generally, if  $k\in\nat$ and  $x\in W_p^k(\T_{\t}^d)$ with $\wh x(0)=0$, then
  $$|x|_{ W^j_p}\les |x|_{ W^k_p}\,,\quad\forall \,0\le j<k.$$
 Consequently, $|\wh x(0)|+|x|_{ W^k_p}$ is an equivalent norm on $W_p^k(\mathbb{T}_{\theta}^d)$.
 \end{thm}

The proof given below is quite different from standard approaches to the Poincar\'e inequality. We will divide it into several lemmas, each of which might be interesting in its own right. We start with the following definition which will be frequently used later. Note that the function $e_u$ and the translation operator $T_u$ have been defined in Remark~\ref{derivation}.

\begin{Def}\label{diff}
 Given $u\in\real^d$ let $d_u=e_u-1$.   The Fourier multiplier on $\T^d_\t$ with symbol  $d_u$ is called  the  difference operator by $u$ and denoted by $\D_u$.
 \end{Def}

\begin{rk}
Note that $e_u$ is the Fourier transform of the Dirac measure $\d_u$ at $u$. Thus  considered as operators on $L_p(\T^d_\t)$  with $1\le p\le\8$,  $T_u$ is an isometry and $\D_u$ is of norm 2.
\end{rk}

\begin{lem}\label{0-k}
Let $1\le p\le\8$, and $j, k\in\nat$ with $j<k$. Then for any $x\in W_p^k(\T^d_\t)$,
 $$\,|x|_{W_p^j}\les \|x\|_p^{1-\frac{j}k}\,|x|_{W_p^k}^{\frac{j}k}\,.$$
 \end{lem}

\begin{proof}
Fix $x\in W_p^k(\T^d_\t)$ with $\wh x(0)=0$.
 For any $u,\xi\in\real^d$ we have
 \be\begin{split}
 d_u(\xi)-\frac{\partial}{\partial r}d_{ru}(\xi)\big|_{r=0}
 &=\int_0^1\big(\frac{\partial}{\partial r}d_{ru}(\xi)-\frac{\partial}{\partial r}d_{ru}(\xi)\big|_{r=0}\big)dr\\
 &=\int_0^1\int_0^r\frac{\partial^2}{\partial s^2}d_{su}(\xi)\,ds\,dr.\\
 \end{split}\ee
Since
 $$\frac{\partial}{\partial r}d_{ru}(\xi)=e_{ru}(\xi)(2\pi{\rm i} u\cdot\xi)\;\text{ and }\;
 \frac{\partial^2}{\partial s^2}d_{su}(\xi)=e_{su}(\xi)(2\pi{\rm i} u\cdot\xi)^2\,,$$
letting $u=t\mathbf{e}_j$ with $t>0$ and $\mathbf{e}_j$ the $j$th canonical basic vector of $\real^d$, we deduce
 $$\D_ux-t\partial_jx=\int_0^1\int_0^r T_{su}(t^2\partial_j^2x)ds\,dr.$$
Thus
 $$t\|\partial_jx\|_p\le \|\D_ux\|_p+t^2\int_0^1\int_0^r \|T_{su}(\partial_j^2x)\|_pds\,dr
 \le 2\|x\|_p+\frac{t^2}2 \,\|\partial_j^2x\|_p\,.$$
Dividing by $t$ and taking the infimum over all $t>0$, we get
 \beq\label{0-1}
 \|\partial_jx\|_p\le 2\sqrt{\|x\|_p\|\partial_j^2x\|_p}.
 \eeq
This gives the assertion for the case $j=1$ and $k=2$. An iteration argument yields the general case.
\end{proof}

 \begin{lem}\label{0-2}
Let $j\in\{1,\cdots,d\}$ and $x\in W_p^2(\T^d_\t)$ such that $m_j\neq0$ whenever $\wh x(m)\neq0$ for $m\in\ent^d$. Then
 $$\|x\|_p\le c  \|\partial_j^2 x\|_p\,,$$
where $c$ is a universal constant. More generally, for any $x\in W_p^2(\T^d_\t)$ with $\wh x(0)=0$
 $$\|x\|_p\le c \sum_{j=1}^d \|\partial_j^2 x\|_p\,.$$
 \end{lem}

\begin{proof}
 Assume $j=1$.  Define $\phi: \ent\to\real$ by
  $$\phi(m_1)=\frac1{m_1^2}\,\text{ for } m_1\in\ent\setminus\{0\}\;\text{ and }\; \phi(m_1)=0\;\text{ for } m_1=0.$$
We also view $\phi$  as a function on $\ent^d$, independent of $(m_2,\cdots, m_d)$. Then the inequality to prove amounts to showing that $\phi$ is a bounded Fourier multiplier on $L_p(\T^d_\t)$ for any $1\le p\le\8$. This is easy. Indeed, let $\psi:\real\to\real$ be the $2\pi$-periodic even function determined by 
 $$\psi(s)=\frac{(\pi-s)^2}2-\frac{\pi^2}6\;\text{ for }\; s\in[0,\,\pi).$$
Then 
 $$\wh\psi=\phi\;\text{  and  }\; \|\psi\|_{L_1(\T)}=\frac{2\pi^2}{9\sqrt3}\,.$$
Thus by Lemma~\ref{cb-multiplier},  $\phi$ is a bounded Fourier multiplier on $L_p(\T^d_\t)$ with norm $\frac{2\pi^2}{9\sqrt3}$, which proves the first inequality of the lemma.

The second one is an immediate consequence of the first. Indeed,  let $\mathbb{E}_{U_{1},\cdots, U_{d-1}}$ be the trace preserving conditional expectation from $\T_\t^d$ onto the subalgebra generated by $(U_{1},\cdots, U_{d-1})$. Let $x'=\mathbb{E}_{U_{1},\cdots, U_{d-1}}(x)$ and $x_d=x-x'$. Then $m_d\neq0$ whenever $\wh x_{d}(m)\neq0$ for $m\in\ent^d$. Thus
  $$\|x_{d}\|_p\le c \|\partial_d^2x_{d}\|_p=c \|\partial_d^2x\|_p\,.$$
 Since $x'$ depends only on  $(U_{1},\cdots, U_{d-1})$, an induction argument then yields the desired inequality.
\end{proof}

 \begin{lem}\label{increase-sob}
The sequence $\{|x|_{W_p^k}\}_{k\ge1}$ is increasing, up to constants. More precisely,  for any $k\ge1$ there exists a constant $c_{d,k}$ such that
 $$|x|_{W_p^k}\le c_{d,k}\,|x|_{W_p^{k+1}}\,,\quad \forall x\in W_p^k(\T^d_\t).$$
 \end{lem}

 \begin{proof}
  The proof is done easily by induction with the help of the previous two lemmas. Indeed, we have (assuming $\wh x(0)=0$)
  $$|x|_{W_p^1}\les\sqrt{\|x\|_p\,|x|_{W_p^2}}\les |x|_{W_p^2}\,.$$
 Thus the assertion is proved for $k=1$. Then induction gives the general case by virtue of Lemma~\ref{0-k}.
 \end{proof}

  \begin{proof}[Proof of Theorem~\ref{k-Sobolev}]
   By the preceding lemma, it remains to show $\|x\|_p\les |x|_{W_p^1}$ for any $x\in W_p^1(\T_\t^d)$ with $\wh x(0)=0$. By approximation, we can assume that $x$ is a polynomial. We proceed by induction on the dimension $d$. Consider first the case $d=1$. Then
  $$x=\sum_{m_1\in\ent\setminus\{0\}}\wh x(m_1)U_1^{m_1}\,.$$
 Define
  $$y=\sum_{m_1\in\ent\setminus\{0\}}\frac1{2\pi{\rm i}m_1}\,\wh x(m_1)U_1^{m_1}\,.$$
 Then $\partial_1y=x$ and $\partial_1^2y=\partial_1x$. Thus Lemma~\ref{0-2} implies $\|x\|_p\les \|\partial_1x\|_p$.

 Now consider a polynomial $x$ in $(U_1,\cdots, U_d)$. As in the proof of Lemma~\ref{0-2}, let $x'=\mathbb{E}_{U_{1},\cdots, U_{d-1}}(x)$ and $x_d=x-x'$. The induction hypothesis implies
  $$\|x'\|_p\les |x'|_{W_p^1}=\Big(\sum_{j=1}^{d-1}\|\mathbb{E}_{U_{1},\cdots, U_{d-1}}(\partial_jx)\|_{p}^p\Big)^{\frac1p}
  \le |x|_{W_p^1}\,,$$
 where we have used the commutation between $\mathbb{E}_{U_{1},\cdots, U_{d-1}}$ and the partial derivations.

To handle the term $x_{d}$, recalling that $m_d\neq0$ whenever $\wh x_{d}(m)\neq0$ for $m\in\ent^d$,  we introduce
 $$y_{d}=\sum_{m\in\ent^d}\frac1{2\pi{\rm i}m_d}\,\wh x_d(m)U^{m}\,.$$
Then $\partial_d y_{d}=x_{d}$.
So by \eqref{0-1} and Lemma~\ref{0-2},
 $$\|x_{d}\|_p\les \sqrt{\|y_{d}\|_p\,\|\partial_d^2y_{d}\|_p}\les \|\partial_d^2y_{d}\|_p=\|\partial_d x_{d}\|_p\,.$$
Thus we are done, so the theorem is proved.
  \end{proof}

%%%%%%%%%%%%%%%%%%%%%%%%%%%%%%%%%%%%%%%%%%%%%%%%%%%%%%%%%%%%%%%%%%%%%%%%
%%%%%%%%%%%%%%%%%%%%%%%%%%%%%%%%%%%%%%%%%%%%%%%%%%%%%%%%%%%%%%%%%%%%%%%%

\section{Lipschitz classes}
\label{Lipschitz classes}

%%%%%%%%%%%%%%%%%%%%%%%%%%%%%%%%%%%%%%%%%%%%%%%%%%%%%%%%%%%%%%%%%%%%%%%%
%%%%%%%%%%%%%%%%%%%%%%%%%%%%%%%%%%%%%%%%%%%%%%%%%%%%%%%%%%%%%%%%%%%%%%%%

In this section we show that $W_\8^k(\T^d_\t)$ is the quantum analogue of the classical Lipschitz class of order $k$. We will use the translation and difference operators introduced in Remark~\ref{derivation} and Definition~\ref{diff}. Note that for any positive integer $k$, $T_u^k=T_{ku}$ and $\D_u^k$ is the $k$th difference operator by $u\in\real^d$.

\begin{Def}\label{def diff}
 Let $k$ be a positive integer and $1\le p\le\8$.  
 \begin{enumerate}[\rm(i)]
 \item The $k$th order modulus of $L_p$-smoothness of an element $x\in L_p(\T^d_\t)$  is defined by
 $$\o_{p}^k(x,\e) =\sup_{0<|u|\leq \e}\big\|\D_u^k x\big\|_p\,.$$
 \item An element $x$ is said to be $L_p$-Lipschitzian of order $k$  if 
  $$\sup_{\e>0}\frac{\o^k_p(x,\e)}{\e^k} <\8.$$
 Let ${\rm Lip}_p^k(\T^d_\t)$ denote the class of all elements that are $L_p$-Lipschitzian of order $k$, equipped with the norm
  $$\|x\|_{{\rm Lip}_p^k}=|\wh x(0)|+\sup_{\e>0}\frac{\o^k_p(x,\e)}{\e^k}\,. $$
 \item The little Lipschitz class ${\rm Lip}_{p,0}^k(\T^d_\t)$ of order $k$  is defined to be  the subspace of ${\rm Lip}_p^k(\T^d_\t)$ consisting of all elements $x$ such that
  $$\lim_{\e\to0} \frac{\o^k_p(x,\e)}{\e^k}=0.$$
  \end{enumerate}
\end{Def}

The spaces  ${\rm Lip}_\8^1(\T^d_\t)$ and ${\rm Lip}_{\8,0}^1(\T^d_\t)$ were introduced by Weaver \cite{Weaver1996, Weaver1998}.

\begin{rk}\label{rk diff}
  It is clear that $\o_{p}^k(x,\e)\le 2^k\|x\|_p$ and $\o_{p}^k(x,\e)$ is nondecreasing in $\e$. On the other hand, $\o_{p}^1(x,\e)$ is subadditive in $\e$; for $k\ge2$, $\o_{p}^k(x,\e)$ is quasi subadditive in the sense that there exists a constant $c=c_k$ such that $\o_{p}^k(x,\e+\eta)\le c\,(\o_{p}^k(x,\e)+\o_{p}^k(x,\eta))$.
\end{rk}

The following is the main result of this section. 

\begin{thm}\label{lim=nabla}
  Let $k$ be a positive integer and $1\le p\le\8$.  Then  $W_p^k(\T^d_\t)={\rm Lip}_p^k(\T^d_\t)$ with equivalent norms. More precisely, 
   $$\sup_{\e>0}\frac{\o^k_p(x,\e)}{\e^k} \approx \ |x|_{W_p^k}\,,\quad \forall\, x\in W_p^k(\T^d_\t)\,,$$
 where the equivalence constants depend only on $d$ and $k$.
\end{thm}

We require the following lemma for the proof.

\begin{lem}\label{sup=lim}
 For any $x\in L_p(\T^d_\t)$,
  $$\lim_{\e\to0}\frac{\o^k_p(x,\e)}{\e^k}=\sup_{0<\e\le1}\frac{\o^k_p(x,\e)}{\e^k}\,.$$
\end{lem}

\begin{proof}
 The assertion for $k=1$ is a common property of increasing and subadditive functions (in $\e$), and easy to check. Indeed, for any $0<\e, \d<1$, choose $n\in\nat$ such that $n\d\le \e<(n+1)\d$. Then
  $$\frac{\o^1_p(x,\e)}\e\le \frac{n+1}n\, \frac{\o^1_p(x,\d)}\d\,,$$
 whence the result for $k=1$. The general case is treated in the same way. Instead of being subadditive, $\o^k_p(x,\e)$ is quasi subadditive in the sense that $\o^k_p(x,n\e)\le n^k\o^k_p(x,\e)$ for any $n\in\nat$. The latter follows immediately from
  $$d_{nu}^k=\big(\sum_{j=0}^{n-1}e_{ju}\big)^k\, d_u^k,\;\text{ so }\;
  \D_{nu}^k=\big(\sum_{j=0}^{n-1}T_{ju}\big)^k\, \D_u^k.$$
 Thus the lemma is proved. 
  \end{proof}

\begin{proof}[Proof of Theorem~\ref{lim=nabla}]
 If the assertion is proved for all $p<\8$ with constants independent of $p$, the case $p=\8$ will follow by letting $p\to\8$. So we will assume $p<\8$.

 We first consider the case $k=1$ whose proof contains all main ideas. As in the proof of Lemma~\ref{0-k}, for any $u\in\real^d$, we have
 $$d_u(\xi)=\int_0^1\frac{\partial}{\partial t}d_{tu}(\xi)dt=\int_0^1e_{tu}(\xi)\,(2\pi{\rm i}u\cdot\xi)dt,\quad\xi\in\real^d\,.$$
In terms of Fourier multipliers, this yields
$$
 \D_ux=\int_0^1T_{tu}(u\cdot\nabla x)dt,
$$
where $u\cdot\nabla x=u_1\partial_1x+\cdots +u_d\partial_dx$. Since the translation $T_{tu}$ is isometric, it then follows that
  \beq\label{diff vs gradient-1}
  \|\D_ux\|_p\le |u|\,\big\|\big(|\partial_1x|^2+\cdots+|\partial_dx|^2\big)^{\frac12}\big\|_p\;{\mathop =^{\rm def}}\;|u|\,\|\nabla x\|_p\,,
  \eeq
whence
$$ \lim_{\e\to0}\frac{\o^1_p(x,\e)}\e\le \|\nabla x\|_p\,.$$
To show the converse inequality, by the density of $\mathcal{P}_\t$ in $ W_p^k(\T^d_\t)$ (see Proposition~\ref{Sobolev-P}), we can assume that $x$ is a polynomial. Given $u\in\real^d$ define $\phi$ on $\real^d$ by
 $$\phi(\xi)=d_u(\xi)-\frac{\partial}{\partial t}d_{tu}(\xi)\Big|_{t=0}\,.$$
Then the Fourier multiplier on $\T^d_\t$ associated to $\phi$ is
 $$\wt\phi*x=\D_u x-u\cdot\nabla x.$$
Thus if $|u|=\e$,
 $$\|u\cdot\nabla x\|_p\le \o_p(x,\e)+\sup_{|u|=\e} \|\D_u x-u\cdot\nabla x\|_p\,.$$
Since $x$ is a polynomial,
 $$\lim_{\e\to0}\sup_{|u|=\e}\frac{\|\D_u x-u\cdot\nabla x\|_p}\e=0.$$
For $u=(\e,0,\cdots,0)$, we then deduce
 $$\|\partial_1x\|_p\le \lim_{\e\to0}\frac{\o_p(x,\e)}\e\,.$$
Hence the desired assertion for $k=1$ is proved.

Now we consider the case $k>1$. \eqref{diff vs gradient-1} can be easily iterated as follows:
 \be\begin{split}
  \|\D_u^kx\|_p
  &\le |u|\,\sum_{j=1}^d\|\partial_j\D_u^{k-1}x\|_p= |u|\,\sum_{j=1}^d\|\D_u^{k-1}\partial_jx\|_p\\
  &\le |u|^k\,\sum_{|m|_1=k}\|\D^mx\|_p\approx  |u|^k\,|x|_{W_p^k}\,.
  \end{split}\ee
So
 $$\sup_{\e>0}\frac{\o_p^k(x,\e)}{\e^k}\les |x|_{W_p^k}\,.$$
The converse inequality is proved similarly to the case $k=1$. Let $m\in\nat_0^d$ with $|m|_1=k$.  For each $j$ with $m_j>0$, using the Taylor formula of the function $d_{\e\mathsf{e}_j}$ (recalling that $(\mathsf{e}_1\,,\cdots, \mathsf{e}_d)$ denotes the canonical basis of $\real^d$), we get
 $$\D_{\e\mathsf{e}_j}^{m_j} x=\e^{m_j}\,\partial_j^{m_j} x+ {\rm o}(\e^{m_j})\,,$$
which implies
 $$\prod_{j=0}^d\D_{\e\mathsf{e}_j}^{m_j} x=\e^{k} D^m x+ {\rm o}(\e^{k}) \;\text { as }\; \e\to0.$$
Thus by the next lemma, we deduce
  $$ \|D^mx\|_p\le \e^{-k}\big\|\prod_{j=0}^d\D_{\e\mathsf{e}_j}^{m_j} x\big\|_p+ {\rm o}(1)
  \les \e^{-k}\o_p^k(x,\e)+ {\rm o}(1) ,$$
whence the desired converse inequality by letting $\e\to0$. So the theorem is proved modulo the next lemma.
  \end{proof}

\begin{lem}\label{k-diff}
 Let $u_1,\cdots, u_k\in\real^d$. Then
 $$\D_{u_1}\cdots\D_{u_k}=\sum_{D\subset\{1,\cdots,k\}}(-1)^{|D|}\,T_{\overline{u}_D}\D^k_{u_D}\,,$$
where the sum runs over all subsets of $\{1,\cdots,k\}$, and where
 $$\overline{u}_D=\sum_{j\in D}u_j,\quad u_D=\sum_{j\in D}\frac1j\,u_j\,.$$
Consequently, for $\e>0$ and $x\in L_p(\T^d_\t)$,
 $$\sup_{|u_1|\le\e, \cdots, |u_k|\le\e}\big\|\D_{u_1}\cdots\D_{u_k}x\big\|_p\approx \o^k_p(x, \e).$$
\end{lem}

\begin{proof}
 This is a well-known lemma in the classical setting (see \cite[Lemma~5.4.11]{BS1988}).
We outline  its proof for the convenience of the reader.  The above equality is equivalent to  the corresponding one with $\D_{u}$ and $T_u$ replaced by $d_u$ and $e_u$, respectively. Setting
 $$v=\sum_{\el\in D}\el u_\el\;\text{ and }\; w=-\sum_{\el\in D}u_\el\,,$$
 for each $0\le j \le k$, we have
 \be\begin{split}
 \prod_{\el=1}^kd_{(\el-j)u_\el}
 &=\sum_{D\subset\{1,\cdots, k\}}(-1)^{k-|D|}\prod_{\el\in D} e_{(\el-j)u_\el}\\
 &=\sum_{D\subset\{1,\cdots, k\}}(-1)^{k-|D|} e_{v}^{} \,(e_{w})^j\,.
 \end{split}\ee
The left hand side is nonzero only for $j=0$. Multiply by $(-1)^{k-j}\left(\begin{array}{c}  k\\j \end{array}\right)$ and sum over $0\le j\le k$; then replacing $u_\el$ by $\frac{u_\el}{\el}$ gives the desired equality.
 \end{proof}

\begin{rk}
It might be interesting to note that in the commutative case, the proof of Theorem~\ref{lim=nabla} shows
 $$\sup_{0<\e\le1}\frac{\o_p(x,\e)}\e=\lim_{\e\to0}\frac{\o_p(x,\e)}\e= \|\nabla x\|_p\,.$$
So Lemma~\ref{sup=lim} is not needed in this case.
\end{rk}

%%%%%%%%%%%%%%%%%%%%%%%%%%%%%%%%%%%%%%%%%%%%%%%%%%%%%%%%%%%%%%%%%%%%%%%%
%%%%%%%%%%%%%%%%%%%%%%%%%%%%%%%%%%%%%%%%%%%%%%%%%%%%%%%%%%%%%%%%%%%%%%%%

\section{The link with the classical Sobolev spaces}
\label{The link with the classical Sobolev spaces}

%%%%%%%%%%%%%%%%%%%%%%%%%%%%%%%%%%%%%%%%%%%%%%%%%%%%%%%%%%%%%%%%%%%%%%%%
%%%%%%%%%%%%%%%%%%%%%%%%%%%%%%%%%%%%%%%%%%%%%%%%%%%%%%%%%%%%%%%%%%%%%%%%

The transference enables us to establish a strong link between the quantum Sobolev spaces defined previously and the vector-valued Sobolev spaces on $\T^d$. Note that the theory of  vector-valued Sobolev spaces  is well-known and can be found in many books on the subject (see, for instance, \cite{Amann1996}). Here we just recall some basic notions. In the sequel, $X$ will always denote a (complex) Banach space.

Let $\mathcal{S}(\T^d; X)$ be the space of  $X$-valued infinitely differentiable functions on $\T^d$ with the standard Fr\'{e}chet topology, and $\mathcal{S}'(\T^d; X)$  be the space of continuous linear maps from $\mathcal{S}(\T^d)$ to $X$. All operations on $\mathcal{S}(\T^d)$ such as derivations, convolution and Fourier transform transfer to $\mathcal{S}'(\T^d; X)$ in the usual way. On the other hand, $L_p(\T^d; X)$ naturally embeds into $\mathcal{S}'(\T^d; X)$ for $1\le p\le\8$, where $L_p(\T^d; X)$ stands for the space of strongly $p$-integrable functions from $\T^d$ to $X$.

\begin{Def}
 Let $1\leq p\leq \infty$, $k\in \mathbb{N}$  and $\a\in\real$.
 \begin{enumerate}[\rm(i)]
 \item The $X$-valued Sobolev space of order $k$  is
 $$W_p^k(\T^d; X)= \big\{ f\in \mathcal{S}'(\T^d; X):  D^{m} f \in L_p(\T^d; X) \textrm{ for each }m\in \mathbb {N}_0^d \textrm{ with } |m|_1\leq k \big\}$$
equipped with the  norm
 $$\|f\|_{W_p^k}=\Big(\sum_{0\leq |m|_1\leq k}\|D^{m}f\|_{L_p(\T^d; X) }^p\Big)^{\frac{1}{p}}.$$
\item The $X$-valued potential Sobolev space of order $\a$  is
 $$H_p^\alpha(\T^d; X)=\big\{ f\in \mathcal{S}'(\T^d; X): J^\a f \in L_p(\T^d; X) \big\}$$
equipped with the  norm
 $$\|f\|_{H_p^\alpha}=\|J^\a f\|_{L_p(\T^d; X) }\,.$$
 \end{enumerate}
 \end{Def}

\begin{rk}
 There exists a parallel theory of vector-valued Sobolev spaces on $\real^d$. In fact, a majority of the literature on the subject is devoted to the case of $\real^d$ which is simpler from the pointview of treatment. The corresponding spaces are  $W_p^k(\mathbb{R}^d; X)$ and $H_p^\alpha(\mathbb{R}^d; X)$. They are subspaces of $\mathcal{S}'(\mathbb{R}^d; X)$. The latter is the space of $X$-valued distributions on $\real^d$, that is, the space of continuous linear maps from  $\mathcal{S}(\real^d)$ to $X$. We will sometimes use the space $\mathcal{S}(\mathbb{R}^d; X)$ of $X$-valued Schwartz functions on $\real^d$. We set $W_p^k(\mathbb{R}^d)=W_p^k(\mathbb{R}^d; \com)$ and $H_p^\a(\mathbb{R}^d)=H_p^\a(\mathbb{R}^d; \com)$.
 \end{rk}

The properties of the Sobolev spaces on $\T^d_\t$ in the previous sections also hold for the present setting. For instance, the proof of Proposition~\ref{q-Sobolev-Bessel} and Lemma~\ref{UMD-multiplier} give  the following well-known result:

\begin{rk}\label{Sobolev-Bessel}
 Let $X$ be a UMD space. Then $W_p^k(\T^d; X)=H_p^k(\T^d; X)$ with equivalent norms for $1<p<\8$.
\end{rk}

Let us also mention that Theorem~\ref{k-Sobolev}, the Poincar\'e inequality,  transfers to the vector-valued case too.  It seems that this result does not appear in literature but it should be known to experts. We record it explicitly here.

\begin{thm}
 Let $X$ be a Banach space and $1\le p\le\8$, $k\in\nat$.  Then
   $$\Big(\|\wh f(0)\|^p_X+\sum_{m\in\nat_0^d,\,|m|_1=k}\|D^m f\|^p_{ L_p(\T^d;X)}\Big)^{\frac1p}$$
 is an equivalent norm on $W_p^k(\mathbb{T}^d; X)$ with relevant constants depending only on $d$ and $k$.
 \end{thm}

Now we use the transference method in Corollary \ref{prop:TransLp}. It is clear  that the map $x\mapsto \wt x$  there commutes with $\partial_j$, that is,  $\partial_j\wt x=\wt{\partial_j x}$ (noting that the $\partial_j$ on the left-hand side is the $j$th partial derivation on $\T^d$ and that on the right-hand side is the one on $\T^d_\theta$). On the other hand, the expectation in that corollary commutes with $\partial_j$ too. We then deduce the following:

\begin{prop}\label{TransSobolev}
Let $1\leq p\leq \8$. The map $x \mapsto \wt{x}$ is an isometric embedding from $W_p^k(\T^d_{\theta})$ and $H_p^\alpha (\T^d_{\theta})$ into $W_p^k (\T^d; L_p(\T_{\theta}^d))$ and $H_p^\alpha (\T^d; L_p(\T_{\theta}^d))$, respectively. Moreover, the ranges of these embeddings are $1$-complemented in their respective spaces.
\end{prop}

This result allows us to reduce many problems about $W_p^k(\T^d_{\theta})$ to the corresponding ones about $W_p^k (\T^d; L_p(\T_{\theta}^d))$. For example, we could deduce the properties of $W_p^k(\T^d_{\theta})$ in the preceding sections from those of $W_p^k (\T^d; L_p(\T_{\theta}^d))$. But we have chosen to work directly on $\T^d_\t$ for the following reasons. It is more desirable to develop an intrinsic quantum theory, so we work directly on $\T^d_\t$ whenever possible. On the other hand, the existing literature on vector-valued Sobolev spaces often concerns the case of $\real^d$, for instance, there exist few publications on periodic Besov or Triebel-Lizorkin spaces. In order to use existing results, we have to transfer them from $\real^d$ to $\T^d$. However, although it  is  often easy, this transfer may not be obvious at all, which is the case for Hardy spaces treated in \cite{CXY2012} and recalled in section~\ref{Hardy spaces}. This difficulty will appear again later for  Besov  and Triebel-Lizorkin spaces.

\begin{rk}\label{os-Sobolev}
 The preceding discussion on vector-valued Sobolev spaces on $\T^d$ can be also transferred to the quantum case.  We have seen in section~\ref{Fourier multipliers} that all noncommutative $L_p$-spaces are equipped with their natural operator space structure. Thus  $W_p^k(\T^d_\t)$ and $H_p^\a(\T^d_\t)$ becomes operator spaces in the natural way. More generally, given an operator space $E$, following Pisier \cite{Pisier1998}, we define the $E$-valued noncommutative $L_p$-space $L_p(\T^d_\t; E)$ (recalling that $\T^d_\t$ is an injective von Neumann algebra). Similarly, we define the $E$-valued distribution space $\mathcal{S}'(\T^d_\t; E)$ that consists of continuous linear maps from $\mathcal{S}(\T^d_\t) $ to $E$. Then as in Definition~\ref{Sobolev def}, we define the corresponding Sobolev spaces $W_p^k(\T^d_\t; E)$ and $H_p^\a(\T^d_\t; E)$. Almost all previous results remain valid in this vector-valued setting since all Fourier multipliers used in their proofs are c.b. maps. For instance, Theorem~\ref{q-Sobolev-Bessel} (or Remark~\ref{Sobolev-Bessel}) now becomes $H_p^k(\T^d_\t; E)=W_p^k(\mathbb{T}^d_\t;E)$ for any $1<p<\8$ and any OUMD space $E$ (OUMD is the operator space version of UMD; see \cite{Pisier1998}). Note that  we recover $W_p^k(\T^d; E)$ and $H_p^\a(\T^d; E)$ if $\t=0$ and if $E$ is equipped with its minimal operator space structure.
 \end{rk}

%%%%%%%%%%%%%%%%%%%%%%%%%%%%%%%%%%%%%%%%%%%%%%%%%%%%%%%%%%%%%%%%%%%%%%%%
%%%%%%%%%%%%%%%%%%%%%%%%%%%%%%%%%%%%%%%%%%%%%%%%%%%%%%%%%%%%%%%%%%%%%%%%

{\Large\part{Besov spaces}}
  \setcounter{section}{0}

%%%%%%%%%%%%%%%%%%%%%%%%%%%%%%%%%%%%%%%%%%%%%%%%%%%%%%%%%%%%%%%%%%%%%%%%
%%%%%%%%%%%%%%%%%%%%%%%%%%%%%%%%%%%%%%%%%%%%%%%%%%%%%%%%%%%%%%%%%%%%%%%%

We study Besov spaces  on $\T^d_\t$ in this chapter. The first section presents the relevant definitions and  some basic properties of these spaces. The second one shows a general characterization of them. This is the most technical part of the chapter. The formulation of our characterization is very similar to Triebel's classical theorem. Although modeled on Triebel's pattern, our proof is harder than his. The main difficulty is due to the unavailability in the noncommutative setting of the usual techniques involving maximal functions which play an important role in the study of  the classical Besov and Triebel-Lizorkin spaces. Like for the Sobolev spaces in the previous chapter, Fourier multipliers are our main tool. We then concretize this general characterization by means of Poisson, heat kernels and differences. We would like to point out that when $\t=0$ (the commutative case), these characterizations (except that by differences) improve the corresponding ones in the classical case. Using the characterization  by differences, we extend a recent result of Bourgain, Br\'ezis and Mironescu on the limits of Besov norms to the quantum setting. The chapter ends with a section on vector-valued Besov spaces on $\T^d$.

\bigskip
%%%%%%%%%%%%%%%%%%%%%%%%%%%%%%%%%%%%%%%%%%%%%%%%%%%%%%%%%%%%%%%%%%%%%%%%
%%%%%%%%%%%%%%%%%%%%%%%%%%%%%%%%%%%%%%%%%%%%%%%%%%%%%%%%%%%%%%%%%%%%%%%%

\section{Definitions and basic properties}
\label{Definitions and basic properties: Besov}

%%%%%%%%%%%%%%%%%%%%%%%%%%%%%%%%%%%%%%%%%%%%%%%%%%%%%%%%%%%%%%%%%%%%%%%%
%%%%%%%%%%%%%%%%%%%%%%%%%%%%%%%%%%%%%%%%%%%%%%%%%%%%%%%%%%%%%%%%%%%%%%%%

We will use Littlewood-Paley decompositions as in the classical theory. Let $\f$ be a  Schwartz function on $\real^d$ such that
  \beq\label{LP dec}
  \left \{ \begin{split}
  \displaystyle &{\rm supp}\,\f \subset \{\xi: 2^{-1}\leq |\xi|\leq 2\},\\
  \displaystyle&\f>0 \text{ on }  \{\xi: 2^{-1}< |\xi|< 2\},\\
  \displaystyle&\sum_{k\in\mathbb{Z}}\f(2^{-k}\xi)=1, \; \xi\neq0.
  \end{split} \right.
  \eeq
Note that if $m\in\ent^d$ with $m\neq0$, $\f(2^{-k} m)=0$ for all $k<0$, so
 $$\sum_{k\ge0}\f(2^{-k} m)=1, \quad m\in\ent^d\setminus\{0\}\,.$$
 On the other hand, the support of the function $\f(2^{-k}\cdot)$ is equal to $\{\xi: 2^{k-1}\leq |\xi|\leq 2^{k+1}\}$, thus  ${\rm supp}\,\f(2^{-k}\cdot)\cap{\rm supp}\,\f(2^{-j}\cdot)=\emptyset$ whenver $|j-k|\ge2$; consequently,
 \beq\label{3-supports}
 \f(2^{-k}\cdot)=\f(2^{-k}\cdot)\,\sum_{j=k-1}^{k+1}\f(2^{-j} \cdot),\quad k\ge0.
 \eeq
Therefore, the sequence $\{\f(2^{-k}\cdot)\}_{k\ge0}$ is  a  Littlewood-Paley decomposition of $\T^d$, modulo constant functions.

The Fourier multiplier on $\mathcal{S}'(\T^d_\t)$ with symbol $\f(2^{-k}\cdot)$ is denoted by $x\mapsto \wt\f_k*x$:
 $$\wt\f_k*x=\sum_{m\in\ent^d}\f(2^{-k} m)\wh x(m)U^m.$$
As noted in section~\ref{Fourier multipliers}, the convolution here has the usual sense. Indeed, let $\varphi_k$ be the function whose Fourier transform is equal to $\f(2^{-k}\cdot)$, and let $\wt\f_k$ be its $1$-periodization, that is,
 $$\wt\f_k(s)=\sum_{m\in\ent^d}\f_k(s+m).$$
Viewed as a function on $\T^d$ by our convention, $\wt\f_k$ admits the following Fourier series:
 $$\wt\f_k(z)=\sum_{m\in\ent^d}\f(2^{-k} m)z^m.$$
Thus for any distribution $f$ on $\T^d$,
 $$\wt\f_k*f(z)=\sum_{m\in\ent^d}\f(2^{-k} m)\wh f(m)z^m.$$
We will maintain the above notation throughout the remainder of the paper.

\smallskip

We can now start our study of quantum Besov spaces.

\begin{Def}\label{q-Besov}
Let $1\leq p,q \leq \8$ and $\a\in \real$. The associated Besov space on $\T^d_\t$ is defined by
$$B^\a_{p,q} (\T^d_\t)=\big\{x\in \mathcal{S}'(\T^d_\t) : \|x\|_{B^\a_{p,q}} < \8 \big\},$$
where
 $$\|x\|_{B^\a_{p,q}} =\Big(|\wh x(0)|^q + \sum_{k\ge 0} 2^{qk\a} \| \wt\f_k * x\|_p^q\Big)^{\frac{1}{q}}\,.$$
Let $B^\a_{p,c_0} (\T^d_\t)$ be the subspace of $B^\a_{p,\8} (\T^d_\t)$ consisting of all $x$ such that $2^{k\a} \| \wt\f_k * x\|_p \to 0$ as $k\to\8$.
 \end{Def}

\begin{rk}\label{independence of phi}
The Besov spaces defined above are independent of the choice of the function $\f$, up to equivalent norms. More generally, let $\{\p^{(k)}\}_{k\ge0}$ be a sequence of Schwartz functions such that
 \be
  \left \{ \begin{split}
  \displaystyle &{\rm supp}\,\p^{(k)}\subset\{\xi: 2^{k-1}\leq |\xi|\leq 2^{k+1}\},\\
   \displaystyle &\sup_{k\ge0}\big\|\F^{-1}(\p^{(k)})\big\|_1<\8, \\
  \displaystyle&\sum_{k\ge0}\p^{(k)}(m)=1, \; \forall\, m\in\ent^d\setminus\{0\}.
  \end{split} \right.
  \ee
 Let $\p_k=\F^{-1}(\p^{(k)}) $ and $\wt\p_k$ be the periodization of $\p_k$. Then
  $$\|x\|_{B^\a_{p,q}}\approx\Big(|\wh x(0)|^q + \sum_{k\ge 0} 2^{qk\a} \| \wt\p_k * x\|_p^q\Big)^{\frac{1}{q}}\,.$$
\end{rk}

Let us justify this remark. By the discussion leading to \eqref{3-supports}, we have (with $\wt\f_{-1}=0$)
 $$\wt\psi_k*x=\sum_{j=k-1}^{k+1}\wt\psi_k*\wt\f_j*x.$$
By Lemma~\ref{q-multiplier},
 $$\|\wt\psi_k*x\|_p\les \sum_{j=k-1}^{k+1}\|\wt\f_j*x\|_p.$$
It then follows that
 $$\Big(|\wh x(0)|^q + \sum_{k\ge 0} 2^{qk\a} \| \wt\psi_k * x\|_p^q\Big)^{\frac{1}{q}}
 \les \Big(|\wh x(0)|^q + \sum_{k\ge 0} 2^{qk\a} \| \wt\f_k * x\|_p^q\Big)^{\frac{1}{q}}\,.$$
The reverse inequality is proved similarly.

\begin{prop}\label{Besov-P}
Let $1\leq p,q \leq \8$ and $\a\in \mathbb{R}$.
\begin{enumerate}[\rm(i)]
\item $B_{p,q}^\a (\T^d_\t)$ is a Banach space.
\item $B_{p,q}^\a(\T^d_\t)\subset B_{p,r}^\a(\T^d_\t)$ for $r>q$ and $B_{p,q}^{\a}(\T^d_\t)\subset B_{p,r}^{\b}(\T^d_\t)$  for $\b<\a$ and $1\le r\le\8$.
\item $\mathcal{P}_{\theta}$ is dense in $B_{p,q}^\a (\T^d_\t)$ and $B_{p,c_0}^\a (\T^d_\t)$ for $1\le p\le\8$ and $1\le q<\infty$.
\item The dual space of $B_{p,q}^{\a}(\T^d_\t)$ coincides  isomorphically with $B_{p',q'}^{-\a}(\T^d_\t)$ for  $1\le p\le\8$ and $1\le q<\8$, where $p'$ denotes the conjugate index of $p$. Moreover, the dual space of $B_{p,c_0}^{\a}(\T^d_\t)$ coincides  isomorphically with $B_{p',1}^{-\a}(\T^d_\t)$.
\end{enumerate}
\end{prop}

\begin{proof}
 (i) To show the completeness of $B_{p,q}^\a (\mathbb{T}^d_{\theta})$,  let $\{x_n\}_n$ be a Cauchy sequence in $B_{p,q}^\a(\mathbb{T}^d_{\theta})$. Then $\{\wh x_n(0)\}_n$ converges to some $\wh y(0)$ in $\com$, and  for every $k\ge0$,  $\{\wt\f_k*x_n\}_n$  converges to some $y_k$ in $L_p(\mathbb{T}_{\theta}^d)$. Let
  $$y=\wh y(0)+\sum_{k\ge0}y_k\,.$$
Since $\wh y_k$ is supported in $\{m\in\ent^d : 2^{k-1}\le|m|\le2^{k+1}\}$, the above series converges in $\mathcal{S}'(\T^d_\t)$. On the other hand, by \eqref{3-supports},  as $n\to\8$, we have
 $$\wt\f_k*x_n=\sum_{j=k-1}^{k+1} \wt\f_k*\wt\f_j*x_n\,\to\, \sum_{j=k-1}^{k+1} \wt\f_k*y_j=\wt\f_k*y.$$
 We then deduce that $y\in B_{p,q}^\a(\mathbb{T}^d_{\theta})$ and $x_n\to y$ in $B_{p,q}^\a (\mathbb{T}^d_{\theta})$.

 (ii) is obvious.

(iii)  We only show the density of $\mathcal{P}_{\theta}$  in $B_{p,q}^\a (\T^d_\t)$ for finite $q$. For $N\in\nat$ let
 $$x_N=\wh x(0)+\sum_{j=0}^N\wt\f_j*x.$$
Then by \eqref{LP dec}, $\wt\f_k*(x-x_N)=0$ for $k\le N-1$, $\wt\f_k*(x-x_N)=\wt\f_k*x$ for $k> N+1$, and $\wt\f_N*(x-x_N)=\wt\f_N*x-\wt\f_N*\wt\f_N*x$, $\wt\f_{N+1}*(x-x_N)=\wt\f_{N+1}*x-\wt\f_{N+1}*\wt\f_N*x$. We then deduce
 $$\|x-x_N\|_{B_{p,q}^{\a}}^q\le 2\sum_{k\ge N}2^{qk\a}\|\wt\f_k*x\|_p^q\to0\;\text{ as }\; N\to\8.$$

(iv) In this part, we view $B_{p,q}^{\a}(\T^d_\t)$ as $B_{p,c_0}^{\a}(\T^d_\t)$ when $q=\8$. Let $y\in B_{p',q'}^{-\a}(\T^d_\t)$. Define  $\el_y(x)=\tau(xy)$ for $x\in\mathcal{P}_{\theta}$. Then
  \be\begin{split}
  |\el_y(x)|
  &=\big|\wh x(0)\wh y(0)+\sum_{k\ge0}\tau\big(\wt\f_k*x\sum_{j=k-1}^{k+1}\wt\f_j*y\big)\big|\\
  &\le |\wh x(0)\wh y(0)|+ \sum_{k\ge0}\big\|\wt\f_k*x\big\|_p\,\big\|\sum_{j=k-1}^{k+1}\wt\f_j*y\big\|_{p'}\\
  &\les \|x\|_{B_{p,q}^{\a}}\, \|y\|_{B_{p',q'}^{-\a}}\,.
  \end{split}\ee
 Thus by the density of $\mathcal{P}_{\theta}$ in $B_{p,q}^{\a}(\T^d_\t)$, $\el_y$ defines a continuous functional on $B_{p,q}^{\a}(\T^d_\t)$.

 To prove the converse, we need a more notation. Given a Banach space $X$, let $\el_q^\a(X)$ be the weighted direct sum of $(\com, X, X, \cdots)$ in the $\el_q$-sense, that is, this is the space of all sequences $(a, x_0, x_1, \cdots)$ with $a\in\com$ and $x_k\in X$, equipped with the norm
  $$\Big(|a|^q+\sum_{k\ge0}2^{qk\a}\|x_k\|^q\Big)^{\frac1q}\,.$$
If $q=\8$, we replace $\el_q^\a(X)$ by its subspace $c_0^\a(X)$ consisting of sequences $(a, x_0, x_1, \cdots)$ such that $2^{k\a}\|x_k\|\to0$ as $k\to\8$. Recall that the dual space of $\el_q^\a(X)$ is  $\el_{q'}^{-\a}(X^*)$.
By definition,  $B_{p,q}^{\a}(\T^d_\t)$ embeds into $\el_q^\a(L_p(\T^d_\t))$ via $x\mapsto (\wh x(0), \wt\f_0*x, \wt\f_1*x, \cdots)$.
Now let $\el$ be a continuous functional on $B_{p,q}^{\a}(\T^d_\t)$ for $p<\8$. Then by the Hahn-Banach theorem, $\el$ extends to a continuous functional on $\el_q^\a(L_p(\T^d_\t))$ of unit norm, so there exists a unit element $(b, y_0, y_1, \cdots)$ belonging to $\el_{q'}^{-\a}(L_{p'}(\T^d_\t))$ such that
 $$\el(x)=b\wh x(0)+ \sum_{k\ge0}\tau(y_k\wt\f_k*x),\quad x\in B_{p,q}^{\a}(\T^d_\t).$$
Let
 $$y=b+ \sum_{k\ge0}(\wt\f_{k-1}*y_{k}+\wt\f_{k}*y_{k}+\wt\f_{k+1}*y_{k})\,.$$
Then clearly $y\in B_{p',q'}^{-\a}(\T^d_\t)$  and $\el=\el_y$ when $p<\8$. The same argument works for $p=\8$ too. Indeed, for $\el$ as above, there exists a unit element $(b, y_0, y_1, \cdots)$ belonging to $\el_{q'}^{-\a}(L_{\8}(\T^d_\t)^*)$ such that
 $$\el(x)=b\wh x(0)+ \sum_{k\ge0}\la y_k,\, \wt\f_k*x\ra,\quad x\in B_{p,q}^{\a}(\T^d_\t).$$
Let $y$ be defined as above. Then $y$ is still a distribution and
 $$\Big(|\wh y(0)|^{q'}+\sum_{k\ge0}2^{q'k\a}\|\wt\f_k*y\|_{L_{\8}(\T^d_\t)^*}^{q'}\Big)^{\frac1{q'}}<\8\,.$$
Since it is a polynomial, $\wt\f_k*y$ belongs to $L_1(\T^d_\t)$; and we have
 $$\|\wt\f_k*y\|_{L_{\8}(\T^d_\t)^*}=\|\wt\f_k*y\|_{L_{1}(\T^d_\t)}.$$
Thus we are done for $p=\8$ too.
 \end{proof}

To proceed further, we require some elementary lemmas. Recall that $J_\a(\xi)=(1+|\xi|^2)^{\frac\a2}$ and $I_\a(\xi)=|\xi|^\a$.

\begin{lem}\label{IJ}
Let  $\a\in\real$ and $k\in\nat_0$. Then
  $$\big\|\F^{-1}(J_\a\f_k)\big\|_{1}\les  2^{\a k}\;\text{ and }\;
  \big\|\F^{-1}(I_\a\,\f_k)\big\|_{1}\les  2^{\a k}\,.
  $$
 where the constants depend only on $\f$, $\a$ and $d$. Consequently,
 for $x \in L_p(\T^d_\t)$ with $1\le p\le\8$,
 $$\|J^\a(\wt\f_k\ast x)\|_{p} \les  2^{\a k} \|\wt\f_k\ast x\|_{p}\;\text{ and }\;
  \|I^\a(\wt\f_k\ast x)\|_{p} \les  2^{\a k} \|\wt\f_k\ast x\|_{p}\,.
  $$
  \end{lem}

\begin{proof}
 The first part is well-known and easy to check. Indeed,
 $$\big\|\F^{-1}(J_\a\f_k)\big\|_{1}=2^{\a k}\big\|\F^{-1}((4^{-k}+|\cdot|^2)^{\frac\a2}\f)\big\|_{1}\,;$$
the function $(4^{-k}+|\cdot|^2)^{\frac\a2}\f$ is a Schwartz function supported by $\{\xi:2^{-1}\le |\xi| \le 2\}$, whose all partial derivatives, up to a fixed order, are bounded uniformly in $k$, so
 $$\sup_{k\ge0}\big\|\F^{-1}((4^{-k}+|\cdot|^2)^{\frac\a2}\f)\big\|_{1}<\8.$$
 Similarly,
  $$\big\|\F^{-1}(I_\a\f_k)\big\|_{1}=2^{\a k}\big\|\F^{-1}(I_\a\f)\big\|_{1}\,.$$
Since $\f_k=\f_k(\f_{k-1}+\f_k+\f_{k+1})$, by Lemma~\ref{q-multiplier}, we obtain the second part.
 \end{proof}

Given $a\in\real_+$, we define $D_{i, a}(\xi)=(2\pi{\rm i}\xi_i)^a$ for $\xi\in\real^d$, and $D_i^a$ to be the associated Fourier multiplier on $\T^d_\t$. We set  $D_a=D_{1,a_1}\cdots D_{d, a_d}$ and $D^a=D_1^{a_1}\cdots D_d^{a_d}$ for any $a=(a_1,\cdots, a_d)\in\real_+^d$. Note that if $a$ is a positive integer, $D_i^a=\partial_i^a$, so there does not exist any conflict of notation.  The following lemma is well-known. We include a sketch of proof for the reader's convenience (see  the proof of Remark~1 in Section~ 2.4.1 of \cite{HT1992}).

\begin{lem}\label{weighted Bessel}
Let $\rho$ be a compactly supported infinitely differentiable function on $\real^d$. Assume $\s, \b\in\real_+$ and $a\in\real_+^d$ such that $\s>\frac{d}2$, $\b>\s-\frac{d}2$ and $|a|_1>\s-\frac{d}2$. Then the functions $I_\b\rho$ and $D_a \rho$ belong to $H^\s_2(\real^d)$.
\end{lem}

\begin{proof}
If $\s$ is a positive integer, the assertion clearly holds in view of $H^\s_2(\real^d)=W^\s_2(\real^d)$. On the other hand, $I_\b\rho \in L_2(\real^d)=H^0_2(\real^d)$ for $\b>-\frac{d}2$. The general case follows by complex interpolation. Indeed, under the assumption on $\s$ and $\b$, we can choose $\s_1\in\nat$, $\b_1, \b_0\in\real$ and $\eta\in (0\,, 1)$ such that
 $$\s_1>\s,\;\, \b_1>\s_1-\frac d2,\;\, \b_0>-\,\frac d2, \;\, \s=\eta \s_1\,, \; \b=(1-\eta)\b_0+\eta \b_1\,.$$
 For a complex number $z$ in the strip $\{z\in\com: 0\le{\rm Re}(z)\le 1\}$ define
  $$F_z(\xi)=e^{(z-\eta)^2}\,|\xi|^{\b_0(1-z)+\b_1z}\,\rho(\xi).$$
 Then
 $$\sup_{b\in\real}\big\|F_{{\rm i}b}\big\|_{L_2}\les \big\|I_{\b_0}\,\rho\big\|_{L_2}\;\text{ and }\;
 \sup_{b\in\real}\big\|F_{1+{\rm i}b}\big\|_{H^{\s_1}_2}\les \big\|I_{\b_1}\,\rho\big\|_{H^{\s_1}_2}\,.$$
 It thus follows that
  $$I_\b\rho=F_\eta\in (L_2(\real^d),\, H^{\s_1}_2(\real^d))_\eta\,.$$
The second assertion is proved in the same way.
\end{proof}

The usefulness of the previous lemma relies upon the following well-known  fact.

\begin{rk}\label{Bessel multiplier}
 Let $\s>\frac{d}2$ and $f\in H_2^\s(\real^d)$. Then
  $$\big\|\F^{-1}(f)\big\|_1\les \big\| f\big\|_{H_2^\s}\,.$$
 \end{rk}

 The verification is extremely easy:
  \be\begin{split}
  \big\|\F^{-1}(f)\big\|_1
  &=\int_{|s|\le 1}\big|\F^{-1}(f)(s)\big|ds +\sum_{k\ge0}\int_{2^{k}<|s|\le 2^{k+1}}\big|\F^{-1}(f)(s)\big|ds \\
  &\les \Big(\int_{|s|\le 1}\big|\F^{-1}(f)(s)\big|^2ds +\sum_{k\ge0}2^{2k \s}\int_{2^{k}<|s|\le 2^{k+1}}\big|\F^{-1}(f)(s)\big|^2ds\Big)^{\frac12} \\
  &\approx \big\| f\big\|_{H_2^\s}\,.
  \end{split}\ee

\smallskip

The following is the so-called reduction (or lifting) theorem of Besov spaces.

\begin{thm}\label{Besov-isom}
Let $1\leq p,q \leq \8$, $\a\in\real$.
 \begin{enumerate}[\rm(i)]
 \item For any $\b\in\real$, both $J^{\b}$ and $I^\b$ are  isomorphisms between $B_{p,q}^\a (\T^d_\t)$ and $B_{p,q}^{\a -\b}(\T^d_\t)$.
 \item Let $a\in\real_+^d$. If $x\in B_{p,q}^\a (\T^d_\t)$, then $D^ax\in B_{p,q}^{\a-|a|_1} (\T^d_\t)$ and
   $$\|D^ax\|_{B_{p,q}^{\a-|a|_1}}\les \|x\|_{B_{p,q}^\a}\,.$$
 \item Let $\b>0$. Then $x\in B_{p,q}^\a (\T^d_\t)$ iff $D_i^\b x\in B_{p,q}^{\a-\b} (\T^d_\t)$ for all $i=1,\cdots,  d$. Moreover, in this case,
  $$\|x\|_{B_{p,q}^\a}\approx |\wh x(0)|+\sum_{i=1}^d \|D_i^\b x\|_{B_{p,q}^{\a-\b}}\,.$$
  \end{enumerate}
\end{thm}

\begin{proof}
  (i)  Let $x \in B_{p,q}^s (\T^d_\t)$ with $\wh x(0)=0$. Then by Lemma~\ref{IJ},
 \be\begin{split}
 \|J^{\b}x\|_{B^{\alpha -\b}_{p,q}}
 &=\Big(\sum_{k\ge0} \big(2^{k(\alpha -\b)}\|J^{\b}(\wt\f_k\ast x)\|_p\big)^q\Big)^{\frac{1}{q}}\\
 &\les\Big(\sum_{k\ge0} \big(2^{k\a}\|\wt\f_k\ast x\|_p\big)^q\Big)^{\frac{1}{q}}
 = \|x\|_{B^\alpha _{p,q}}\,.
 \end{split}\ee
Thus $J^{\b}$ is bounded from $B_{p,q}^\a (\T^d_\t)$ to $B_{p,q}^{\a -\b}(\T^d_\t)$, and its inverse, which is $J^{-\b}$, is bounded too. The case of $I^\b$ is treated similarly.

(ii) By Lemma~\ref{weighted Bessel} and Remark~\ref{Bessel multiplier}, we have
 $$\big\|\F^{-1}(D_a\f_k)\big\|_{1}=2^{k|a|_1}\big\|\F^{-1}(D_a\f)\big\|_{1}\les 2^{k|a|_1}\,.$$
Consequently, by Lemma~\ref{q-multiplier},
 $$\|\wt\f_k*D^ax\|_p\les 2^{k|a|_1}\|\wt\f_k*x\|_p\,, \forall j\ge0,$$
whence
 $$\|D^ax\|_{B_{p,q}^{\a-|a|_1}}\les \|x\|_{B_{p,q}^\a }\,.$$

(iii) One implication is contained in (ii). To show the other, choose an infinitely differentiable function $\chi:\real\to \real_+$ such that $\chi(s)=0$ if $|s|<\frac1{4\sqrt d}$ and  $\chi(s)=1$ if $|s|\ge \frac1{2\sqrt d}$. For $i=1,\cdots, d$, let $\chi_i$ on $\real^d$ be defined by
 $$\chi_i(\xi)=\frac1{\chi(\xi_1)|\xi_1|^\b+\cdots+\chi(\xi_d)|\xi_d|^\b}\,\frac{\chi(\xi_i)|\xi_i|^\b}{(2\pi{\rm i}\xi_i)^\b}$$
whenever the first denominator is positive, which is the case when $|\xi|\ge 2^{-1}$. Then for any $k\ge0$, $\chi_i\f_k$ is a well-defined  infinitely differentiable function on $\real^d\setminus\{\xi: \xi_i=0\}$. We have
 $$\big\|\F^{-1}(\chi_i\f_k)\big\|_{1}= 2^{-k\b}\big\|\F^{-1}(\p\f)\big\|_{1}\,,$$
where
 $$\p(\xi)=\frac1{\chi(2^k\xi_1)|\xi_1|^\b+\cdots+\chi(2^k \xi_d)|\xi_d|^\b}\,\frac{\chi(2^k\xi_i)|\xi_i|^\b}{(2\pi{\rm i}\xi_i)^\b}\,.$$
The function $\p\f$ is   supported in $\{\xi:2^{-1}\le|\xi|\le2\}$. An inspection reveals that all its partial derivatives of order less than a fixed integer are bounded uniformly in $k$. It then follows that the $L_1$-norm of $\F^{-1}(\p\f)$ is majorized by a constant independent of $k$, so
 $$\big\|\F^{-1}(\chi_i\f_k)\big\|_{1}\les 2^{-k\b}\,,$$
 and by Lemma~\ref{q-multiplier},
  $$\|\wt\chi_i*\wt\f_k*D_i^\b x\|_p\les 2^{-k\b}\|\wt\f_k*D_i^\b x\|_p\,.$$
Since
  $$\f_k= \sum_{i=1}^d\chi_i D_{i,\b}\f_k,$$
we deduce
$$\|\wt\f_k*x\|\les 2^{-k\b}\sum_{i=1}^d\|\wt\f_k*D_i^\b x\|_p\,,$$
which implies
 $$\|x\|_{B_{p,q}^\a}\les |\wh x(0)|+\sum_{i=1}^d \|D_i^\b x\|_{B_{p,q}^{\a-\b}}\,.$$
Thus (iii) is proved.
 \end{proof}

The following result relates the Besov and potential Sobolev spaces.

\begin{thm}\label{Sobolev-Besov}
 Let $1\le p\le \8$ and $\a\in\real^d$. Then we have the following continuous inclusions:
 $$ B_{p,\min(p, 2)}^\a(\T^d_\t) \subset H_p^\a(\T^d_\t) \subset B_{p,\max(p,2)}^\a(\T^d_\t).$$
\end{thm}

\begin{proof}
 By Propositions~\ref{Sobolev-P} and \ref{Besov-isom}, we can assume $\a=0$. In this case,  $H_p^0(\T^d_\t) =L_p(\T^d_\t)$. Let $x$ be a distribution on $\T^d_\t$ with $\wh x(0)=0$. Since
 $$x=\sum_{k\ge0}\wt\f_k \ast x,$$
we see that the inclusion $B_{p,1}^0(\T^d_\t) \subset L_p(\T^d_\t)$ follows immediately from triangular inequality.
On the other hand, the inequality
 $$\|\wt\f_k \ast x\|_p \les \|x\|_p, \;k\ge0$$
yields the inclusion $L_p(\T^d_\t) \subset B_{p,\8}^0(\T^d_\t)$.
Both inclusions can be improved in the range $1<p<\8$.

Let us consider only the case $2\le p<\8$. Then the inclusion $L_p(\T^d_\t) \subset B_{p,p}^0(\T^d_\t)$ can be easily proved by interpolation. Indeed, the two spaces coincide isometrically when $p=2$. The other extreme case $p=\8$ has been already proved. We then deduce the case $2<p<\8$ by complex interpolation and by embedding $B_{p,\8}^0(\T^d_\t)$ into $\el_\8(L_p(\T^d_\t))$.

The converse inclusion $B_{p,2}^0(\T^d_\t) \subset L_p(\T^d_\t)$ is subtler. To show it, we use Hardy spaces and the equality $L_p(\T^d_\t) =\H_p(\T^d_\t)$ (see Lemma~\ref{q-H-BMO}). Then we must show
 $$\max(\|x\|_{\H_p^c}\,,\,\|x\|_{\H_p^r})\les \|x\|_{B_{p,2}^0}\,.$$
To this end, we appeal to Lemma~\ref{Hp-discrete}. The function $\p$ there is now equal to $\f$. The associated square function of $x$  is thus given by
 $$s_\f^{c}(x)=\big(\sum_{k\ge0}|\wt\f_k*x|^2\big)^{\frac12}\,.$$
Recall the following well-known inequality
 $$\big\|\big(\sum_{k\ge0}|x_k|^2\big)^{\frac12}\big\|_p\le \big(\sum_{k\ge0}\|x_k\|_p^2\big)^{\frac12}$$
for $x_k\in L_p(\T^d_\t)$ and $2\le p\le\8$. Note that this inequality is proved simply by the triangular inequality in $L_{\frac{p}2}(\T^d_\t)$. Thus
 $$\|x\|_{\H_p^c}\approx \|s_\f^{c}(x)\|_p\le \big(\sum_{k\ge0}\|\wt\f_k*x\|_p^2\big)^{\frac12}
 =\|x\|_{B_{p,2}^0}\,.$$
Passing to adjoints, we get $\|x\|_{\H_p^r}\les \|x\|_{B_{p,2}^0}$. Therefore,  the desired inequality follows.
\end{proof}

%%%%%%%%%%%%%%%%%%%%%%%%%%%%%%%%%%%%%%%%%%%%%%%%%%%%%%%%%%%%%%%%%%%%%%%%
%%%%%%%%%%%%%%%%%%%%%%%%%%%%%%%%%%%%%%%%%%%%%%%%%%%%%%%%%%%%%%%%%%%%%%%%

\section{A general characterization}
\label{A general characterization-Besov}

%%%%%%%%%%%%%%%%%%%%%%%%%%%%%%%%%%%%%%%%%%%%%%%%%%%%%%%%%%%%%%%%%%%%%%%%
%%%%%%%%%%%%%%%%%%%%%%%%%%%%%%%%%%%%%%%%%%%%%%%%%%%%%%%%%%%%%%%%%%%%%%%%

In this and next sections we extend some characterizations of the classical Besov spaces to the quantum setting. Our treatment follows Triebel \cite{HT1992} rather closely.

We give a general characterization in this section. We have observed in the previous section that the definition of the Besov spaces is independent of the choice of  $\f$ satisfying \eqref{LP dec}. We now show that  $\f$ can be replaced by more general functions. To state the required conditions, we introduce an auxiliary Schwartz function $h$ such that
 \beq\label{hH}
   {\rm supp}\,h\subset \{\xi\in\real^d: |\xi|\leq 4\} \quad\text { and }\quad  h=1 \text{ on }  \{\xi\in\real^d: |\xi|\leq 2\}.
    \eeq
 Let $\a_0, \a_1\in\real$. Let $\p$ be an infinitely differentiable function on $\real^d\setminus\{0\}$ satisfying the following conditions
   \beq\label{psi}
  \left \{ \begin{split}
  &\displaystyle |\p|>0 \;\text{ on }\; \{\xi: 2^{-1}\leq |\xi|\leq 2\},\\
  &\displaystyle \F^{-1}(\p hI_{-\a_1})\in L_1(\real^d), \\
  &\displaystyle \sup_{j\in\nat_0}2^{-\a_0j}\big\|\F^{-1}(\p(2^{j}\cdot)\f)\big\|_1<\8.
  \end{split} \right.
  \eeq
The first nonvanishing condition above on $\p$ is a Tauberian condition.  The integrability of the inverse Fourier transforms can be reduced to a handier criterion by means of the potential Sobolev space $H_2^\s(\real^d)$ with $\s>\frac{d}2$; see Remark~\ref{Bessel multiplier}.

We will use the same notation for $\p$ as for $\f$. In particular, $\p_k$ is the inverse Fourier transform of $\p(2^{-k}\cdot)$ and $\wt\p_k$ is the Fourier multiplier on $\T^d_\t$ with symbol $\p(2^{-k}\cdot)$. It is to note that compared with \cite[Theorem~2.5.1]{HT1992}, we need not require $\a_1>0$ in the following theorem. This will have  interesting consequences in the next section.

\begin{thm}\label{g-charct-Besov}
 Let $1\le p, q\le\8$ and $\a\in\real$. Assume $\a_0<\a<\a_1$. Let $\p$ satisfy the above assumption. Then a distribution $x$ on $\T^d_\t$ belongs to $B_{p,q}^\a(\T^d_\t)$ iff
  $$\Big(\sum_{k\ge0}\big(2^{k\a}\|\wt\p_k*x\|_p\big)^q\Big)^{\frac1q}<\8.$$
 If this is the case, then
 \beq\label{phi-psi}
 \|x\|_{B_{p,q}^\a}\approx
 \Big(|\wh x(0)|^q+\sum_{k\ge0}\big(2^{k\a}\|\wt\p_k*x\|_p\big)^q\Big)^{\frac1q}
 \eeq
with relevant constants depending only on $\f, \p, \a, \a_0,\a_1$ and $d$.
 \end{thm}

\begin{proof}
 We will follow the pattern of the proof of \cite[Theorem~2.4.1]{HT1992}. Given a function $f$ on $\real^d$, we will use  the notation that $f^{(k)}=f(2^{-k}\cdot)$ for $k\ge0$  and  $f^{(k)}=0$ for $k\le -1$.   Recall that $f_k$ is the inverse Fourier transform of $f^{(k)}$ and $\wt f_k$ is the $1$-periodization of $f_k$:
 $$\wt f_k(s)=\sum_{m\in\ent^d}f_k(s+m).$$
In the following, we will fix a distribution  $x$ on $\T^d_\t$. Without loss of generality, we assume $\wh x(0)=0$. We will denote the right-hand side of \eqref{phi-psi} by $\|x\|_{B_{p,q}^{\a,\p}}$ when it is finite. For clarity, we divide the proof into several steps.

\smallskip\n{\it Step~1.} In the first two steps, we assume $x\in B_{p,q}^\a(\T^d_\t)$.  Let $K$ be a positive integer to be determined later in step~3. By \eqref{LP dec}, we have
 $$\p^{(j)}=\sum_{k=0}^\8\p^{(j)}\f^{(k)}
 =\sum_{k=-\8}^K\p^{(j)}\f^{(j+k)}+ \sum_{k=K}^\8\p^{(j)}\f^{(j+k)}\;\text{ on }\; \{\xi: |\xi|\ge1\}.$$
Then
 \beq\label{split f}
 \wt\p_j*x=\sum_{k\le K}\wt\p_j*\wt\f_{j+k}*x+ \sum_{k>K}\wt\p_j*\wt\f_{j+k}*x\,.
 \eeq
For the moment, we do not care about the convergence issue of the second series above, which is postponed to the last step. Let $a_{j, k}=2^{j\a}\|\wt\p_j*\wt\f_{j+k}*x\|_p$. Then
 \beq\label{split}
 \|x\|_{B_{p,q}^{\a,\p}}
 \le\Big(\sum_{j=0}^\8\big[\sum_{k\le K}a_{j, k}\big]^q\Big)^{\frac1q}
 + \Big(\sum_{j=0}^\8\big[\sum_{k>K}a_{j, k}\big]^q\Big)^{\frac1q}\,.
 \eeq
We will treat the two sums on the right-hand side separately. For the first one,  by  the support assumption on $\f$ and $h$,  for $k\le K$, we can write
 \beq\label{h}
 \begin{split}
 \p^{(j)}(\xi)\f^{(j+k)}(\xi)
 &=2^{k\a_1}\,\frac{\p^{(j)}(\xi)}{|2^{-j}\xi|^{\a_1}}\,h^{(j+K)}(\xi) |2^{-j-k}\xi|^{\a_1}\f^{(j+k)}(\xi)\\
 &=2^{k\a_1}\eta^{(j)}(\xi)\rho^{(j+k)}(\xi),
 \end{split}\eeq
where $\eta$  and $\rho$ are defined by
 $$\eta(\xi)=\frac{\p(\xi)}{|\xi|^{\a_1}}\,h^{(K)}(\xi) \;\text{ and }\; \rho(\xi)=|\xi|^{\a_1}\f(\xi).$$
Note that $\F^{-1}(\eta)$ is integrable on $\real^d$. Indeed, write
 \beq\label{diff psi}
 \eta(\xi)=\frac{\p(\xi)}{|\xi|^{\a_1}}\,h(\xi)+\frac{\p(\xi)}{|\xi|^{\a_1}}\,(h^{(K)}(\xi)-h(\xi)).
 \eeq
By \eqref{psi}, the inverse Fourier transform of the first function on the right-hand side is integrable. The second one is an infinitely differentiable function with compact support, so its inverse Fourier transform  is also integrable with $L_1$-norm controlled by a constant depending only on $\psi$, $h$, $\a_1$ and $K$. Therefore,  Lemma~\ref{q-multiplier} implies that each $\eta^{(j)}$ is a Fourier multiplier on $L_p(\T^d_\t)$ for all $1\le p\le\8$ with norm controlled by a constant $c_1$, depending only on $\psi$, $h$, $\a_1$  and $K$. Therefore,
 \beq\label{Kk}
 a_{j,k}\le c_1 2^{j\a +k\a_1} \|\wt\rho_{j+k}*x\|_p
 =c_12^{k(\a_1-\a)}\big(2^{(j+k)\a} \|\wt\rho_{j+k}*x\|_p\big)\,.
 \eeq
Thus by triangular inequality and Lemma~\ref{IJ}, we deduce
 \be\begin{split}
 \Big(\sum_{j=0}^\8\big[\sum_{k\le K}a_{j, k}\big]^q\Big)^{\frac1q}
 &\le c_1\sum_{k\le K}2^{k(\a_1-\a)}\Big(\sum_{j=-\8}^\8 \big(2^{(j+k)\a} \|\wt\rho_{j+k}*x\|_p\big)^q\Big)^{\frac1q}\\
 &=c_1\sum_{k\le K}2^{k(\a_1-\a)}\Big(\sum_{j=0}^\8 \big(2^{j\a} 2^{-j\a_1}\|I^{\a_1}\wt\f_{j}*x\|_p\big)^q\Big)^{\frac1q}\\
 &\le c_1' \|x\|_{B_{p, q}^\a}\,,
  \end{split}\ee
 where $c_1' $ depends only on $\psi$, $h$, $K$, $\a$ and $\a_1$.

\smallskip\n{\it Step~2.}  The second sum on the right-hand side of \eqref{split} is treated similarly.  Like in step~1 and by \eqref{3-supports}, we write
  \beq\label{H}\begin{split}
 \p^{(j)}(\xi)\f^{(j+k)}(\xi)
 &=\frac{\p^{(j)}(\xi)}{|2^{-j-k}\xi|^{\a_0}}\,(\f^{(j+k-1)}+\f^{(j+k)}+\f^{(j+k+1)})(\xi)|2^{-j-k}\xi|^{\a_0}\f^{(j+k)}(\xi)\\
 &=\big[\frac{\p(2^{-j-k}\cdot 2^k\xi)}{|2^{-j-k}\xi|^{\a_0}}\,H(2^{-j-k}\xi)\big]\,\rho^{(j+k)}(\xi),
  \end{split}\eeq
where $H=\f^{(-1)}+\f+\f^{(1)}$, and where $\rho$ is now defined by
 $$\rho(\xi)=|\xi|^{\a_0}\f(\xi).$$
The $L_1$-norm of the inverse Fourier transform of the function
 $$\frac{\p(2^{-j-k}\cdot 2^k\xi)}{|2^{-j-k}\xi|^{\a_0}}\,H(2^{-j-k}\xi)$$
is equal to $\big\|\F^{-1}(I_{-\a_0} H \p(2^k\cdot))\big\|_1$. Using Lemma~\ref{IJ}, we see that the last norm is majorized by
 $$\big\|\F^{-1}(\psi(2^{k}\cdot) H)\big\|_1\les \big\|\F^{-1}(\psi(2^{k}\cdot) \f)\big\|_1\,.$$
Then, using \eqref{psi}, for $k>K$ we get
 \beq\label{kK}
 a_{j,k}\le c_22^{k(\a_0-\a)}\big(2^{(j+k)\a} \|\wt\rho_{j+k}*x\|_p\big)\,,
 \eeq
where $c_2$ depends only on $\f$, $\a_0$  and the supremum in \eqref{psi}. Thus as before, we get
 $$\Big(\sum_{j=0}^\8\big[\sum_{k> K}a_{j, k}\big]^q\Big)^{\frac1q}
 \le c_2'\,\frac{2^{K(\a_0-\a)}}{1- 2^{\a_0-\a}}\, \|x\|_{B_{p, q}^\a}\,,$$
which, together with the inequality obtained in step~1, yields
  $$\|x\|_{B_{p, q}^{\a, \p}}
  \les  \|x\|_{B_{p, q}^\a}\,.$$

\smallskip\n{\it Step~3.} Now we prove the inequality reverse to the previous one.  We first assume that $x$ is a polynomial. We write
 \beq\label{hbis}
 \f^{(j)}=\f^{(j)}h^{(j+K)}=\frac{\f^{(j)}}{\psi^{(j)}}\,h^{(j+K)} \psi^{(j)}\,.
 \eeq
The function $\f\p^{-1}$ is an infinitely differentiable  function with compact support, so its inverse Fourier transform belongs to $L_1(\real^d)$. Thus by Lemma~\ref{q-multiplier},
 $$\|\wt\f_j*x\|_p\le c_3 \|\wt h_{j+K}*\wt\p_j*x\|_p\,,$$
where $c_3=\big\|\F^{-1}(\f\p^{-1})\big\|_1$. Hence,
 $$
  \|x\|_{B_{p, q}^\a}
  \le c_3 \Big(\sum_{j=0}^\8\big(2^{j\a} \|\wt h_{j+K}*\wt\p_j*x\|_p\big)^q\Big)^{\frac1q}\,.
  $$
To handle the right-hand side, we let $\l=1-h$ and write $h^{(j+K)} \p^{(j)}=\p^{(j)}-\l^{(j+K)} \p^{(j)}$. Then
 $$\Big(\sum_{j=0}^\8\big(2^{j\a} \|\wt h_{j+K}*\wt\p_j*x\|_p\big)^q\Big)^{\frac1q}
 \le \|x\|_{B_{p, q}^{\a, \p}}
 +\Big(\sum_{j=0}^\8\big(2^{j\a} \|\wt\l_{j+K}*\wt\p_j*x\|_p\big)^q\Big)^{\frac1q}\,.$$
Thus it remains to deal with the last sum. We do this as in the previous steps with $\psi$ replaced by $\l\p$, by writing
 $$\l^{(j+K)}\p^{(j)}=\sum_{k=-\8}^\8\l^{(j+K)}\p^{(j)}\f^{(j+k)}\,.$$
The crucial point now is the fact that $\l^{(j+K)}\f^{(j+k)}=0$ for all $k\le K$ and all $j$. So
 $$\l^{(j+K)}\p^{(j)}=\sum_{k>K}\l^{(j+K)}\p^{(j)}\f^{(j+k)}\,,$$
that is, only the second sum on the right-hand side of \eqref{split} survives now:
 $$\Big(\sum_{j=0}^\8 (2^{j\a}\|\wt\l_{j+K}*\wt\p_j*x\|_p)^q\Big)^{\frac1q}
 \le\Big(\sum_{j=0}^\8\big[2^{j\a}\sum_{k>K}\|\wt\l_{j+K}*\wt\p_j*\wt\f_{j+k}*x\|_p\big]^q\Big)^{\frac1q}\,.$$
Let us reexamine the argument of step~2 and formulate \eqref{H} with $\l^{(K)}\psi$ instead of $\p$. We then arrive at majorizing the norm
$\big\|\F^{-1}\big(\l(2^{k-K}\cdot)\p(2^{k}\cdot) \f\big)\big\|_1$:
 $$\big\|\F^{-1}\big(\l(2^{k-K}\cdot)\p(2^{k}\cdot) \f\big)\big\|_1\le
 \big\|\F^{-1}(\l)\big\|_1\,\big\|\F^{-1}\big(\p(2^{k}\cdot) \f\big)\big\|_1\,.$$
Keeping the notation of step~2 and as for \eqref{kK}, we get
 $$\|\wt\l_{j+K}*\wt\p_j*\wt\f_{j+k}*x\|_p\le cc_2 2^{k(\a_0-\a)}\big(2^{(j+k)\a} \|\wt\rho_{j+k}*x\|_p\big)\,,$$
where $c=\big\|\F^{-1}(\l)\big\|_1$. Thus
 $$\Big(\sum_{j=0}^\8\big[2^{\a j}\sum_{k>K}\|\wt\l_{j+K}*\wt\p_j*\wt\f_{j+k}*x\|_p\big]^q\Big)^{\frac1q}
 \le cc_2'\,\frac{2^{K(\a_0-\a)}}{1- 2^{\a_0-\a}}\, \|x\|_{B_{p, q}^{\a}}\,.$$
Combining the preceding inequalities, we obtain
 $$\|x\|_{B_{p, q}^{\a}}\le c_3 \|x\|_{B_{p, q}^{\a, \p}} +
 cc_2'  \,\frac{2^{K(\a_0-\a)}}{1- 2^{\a_0-\a}}\, \|x\|_{B_{p, q}^{\a}}\,.$$
Choosing $K$ so that
 $$c\,c_2' \,\frac{2^{K(\a_0-\a)}}{1- 2^{\a_0-\a}}\le\frac12,$$
we finally deduce
 $$\|x\|_{B_{p, q}^{\a}}\le 2c_3\|x\|_{B_{p, q}^{\a, \p}}\,,$$
which shows \eqref{phi-psi} in  case $x$ is a polynomial.

The general case can be easily reduced to this special one. Indeed, assume $\|x\|_{B_{p, q}^{\a, \p}}<\8$. Then using the Fej\'er means $F_N$ as in the proof of  Proposition~\ref{Sobolev-P}, we see that
 $$\|F_N(x)\|_{B_{p, q}^{\a, \p}}\le \|x\|_{B_{p, q}^{\a, \p}}\,.$$
Applying the above part already proved for polynomials yields
 $$\|F_N(x)\|_{B_{p, q}^{\a}}\le 2c_3\|F_N(x)\|_{B_{p, q}^{\a, \p}} \le 2c_3\|x\|_{B_{p, q}^{\a, \p}}\,.$$
However, it is easy to check that
 $$\lim_{N\to\8} \|F_N(x)\|_{B_{p, q}^{\a}}=\|x\|_{B_{p, q}^{\a}}\,.$$
We thus deduce \eqref{phi-psi} for general $x$, modulo the convergence problem on  the second series of \eqref{split f}.

\smallskip\n{\it Step~4.} We now settle up the convergence issue left previously. Each term $\wt\p_j*\wt\f_{j+k}*x$ is a polynomial, so a distribution on $\T^d_\t$. We must show that the series converges in $\mathcal{S}'(\T^d_\t)$. By \eqref{kK}, for any $L>K$, by the H\"older inequality (with $q'$ the conjugate index of $q$), we get
  \be\begin{split}
  2^{j\a}\sum_{k=K+1}^L\|\wt\p_j*\wt\f_{j+k}*x\|_p
  &\le c_2'\sum_{k=K+1}^L2^{k(\a_0-\a)}\big(2^{(j+k)\a}\|\wt\f_{j+k}*x\|_p\big)\\
  &\le c_2'R_{K, L}\Big(\sum_{k=K+1}^L\big(2^{(j+k)\a}\|\wt\f_{j+k}*x\|_p\big)^q\Big)^{\frac{1}{q}}\\
  &\les c_2'R_{K, L}\|x\|_{B_{p, q}^{\a}} \,,
  \end{split}\ee
 where
  $$R_{K, L}=\Big(\sum_{k=K+1}^L2^{q'k(\a_0-\a)}\Big)^{\frac{1}{q'}} \,.$$
 Since $\a_0<\a$, $R_{K, L}\to 0$ as $K$ tends to $\8$. Thus the series $\sum_{k>K}\wt\p_j*\wt\f_{j+k}*x$ converges in $L_p(\T^d_\t)$, so in $\mathcal{S}'(\T^d_\t)$ too. In the same way, we show that the series also converges in $B_{p, q}^{\a}(\T^d_\t)$. Hence, the proof of the theorem is complete.
 \end{proof}

\begin{rk}
 The infinite differentiability of $\p$ can be substantially relaxed without changing the proof. Indeed, we have used this condition only once to insure that the inverse Fourier transform of the second term on the right-hand side of \eqref{diff psi} is integrable. This integrability is guaranteed when $\p$ is continuously differentiable up to order $[\frac d2]+1$. The latter condition can be replaced by the following slightly weaker one: there exists $\s>\frac{d}2 +1$ such that $\p \eta\in H_2^\s(\real^d)$ for any compactly supported infinitely differentiable function $\eta$ which vanishes in a neighborhood of the origin.
 \end{rk}

The following is the continuous version of Theorem~\ref{g-charct-Besov}. We will use similar notation for continuous parameters: given $\e>0$, $\p_\e$ denotes the function with Fourier transform $\p^{(\e)}=\p(\e\cdot)$, and $\wt\p_\e$ denotes the Fourier multiplier on $\T^d_\t$ associated to $\p^{(\e)}$.

 \begin{thm}\label{g-charct-Besov-cont}
 Keep the assumption of the previous theorem. Then for any distribution $x$ on $\T^d_\t$,
  \beq\label{phi-psi-cont}
 \|x\|_{B_{p,q}^\a}\approx
 \Big(|\wh x(0)|^q+\int_0^1\e^{-q\a}\big\|\wt\p_\e*x\big\|_p^q\,\frac{d\e}\e\Big)^{\frac1q}\,.
 \eeq
The above equivalence is understood in the sense that if one side is finite, so is the other, and the two are then equivalent with constants independent of $x$.
 \end{thm}

\begin{proof}
 This proof is very similar to the previous one. Keeping the notation there, we will point out only the necessary changes. Let us first discretize the integral on the right-hand side of \eqref{phi-psi-cont}:
  $$\int_0^1\big(\e^{-\a}\|\wt\p_\e*x\|_p\big)^q\,\frac{d\e}\e\approx
  \sum_{j=0}^\8 2^{jq\a}\int_{2^{-j-1}}^{2^{-j}}\|\wt\p_\e*x\|_p^q \,\frac{d\e}\e\,.$$
Now for $j\ge0$ and $2^{-j-1}<\e\le 2^{-j}$, we transfer \eqref{h} to the present setting:
 $$
 \p^{(\e)}(\xi)\f^{(j+k)}(\xi)
 =2^{\a_1k}\big[\,\frac{\p(2^{-j}\cdot 2^{j}\e \xi)}{|2^{-j}\xi|^{\a_1}}\,h^{(j+K)}(\xi)\big]\rho^{(j+k)}(\xi).
 $$
We then must  estimate the $L_1$-norm of the inverse Fourier transform of the function in the brackets. It is equal to
 $$\big\|\F^{-1}\big(I_{-\a_1}\p(2^{j}\e \cdot)h^{(K)}\big)\big\|_1
 =\d^{-\a_1}\big\|\F^{-1}\big(I_{-\a_1}\p h(\d 2^{-K}\cdot)\big)\big\|_1\,,$$
where $\d=2^{-j}\e^{-1}$. The last norm is estimated as follows:
 \be\begin{split}
 \big\|\F^{-1}\big(I_{-\a_1}\p h(\d 2^{-K}\cdot)\big)\big\|_1
 &\le \big\|\F^{-1}\big(I_{-\a_1}\p h\big)\big\|_1 +\big\|\F^{-1}\big(I_{-\a_1}\p\,[h- h(\d 2^{-K}\cdot)]\big)\big\|_1\\
 &\le \big\|\F^{-1}\big(I_{-\a_1}\p h\big)\big\|_1 +\sup_{1\le\d\le2}\big\|\F^{-1}\big(I_{-\a_1}\p\,[h- h(\d 2^{-K}\cdot)]\big)\big\|_1\,.
 \end{split}\ee
Note that the above supremum is finite since $I_{-\a_1}\p[h- h(\d 2^{-K}\cdot)]$ is a compactly supported infinitely differentiable function and its inverse Fourier transform depends continuously on $\d$.  It follows that for $k\le K$ and $2^{-j-1}\le\e\le 2^{-j}$
 \beq\label{Kke}
 2^{j\a}\|\wt\p_\e*\wt\f_{j+k}*x\|_p\les 2^{k(\a_1-\a)}\big(2^{(j+k)\a}\|\wt \rho_{j+k}*x\|_p\big),
 \eeq
which is the analogue of \eqref{Kk}. Thus, we get the continuous analogue of the final inequality of step~1 in the preceding proof.

We can make similar modifications in step~2, and then show the second part. Hence, we have proved
 $$\Big(\int_0^1\big(\e^{-\a}\|\wt\p_\e*x\|_p\big)^q\,\frac{d\e}\e\Big)^{\frac1q}
 \les  \|x\|_{B_{p,q}^\a}\,.$$

To show the converse inequality, we proceed as in step~3 above. By \eqref{psi}, there exists a constant $a>2$ such that $\p>0$ on $\{\xi: a^{-1}\le|\xi|\le a\}$. Assume also $a\le 2\sqrt2$. For $j\ge0$ let $R_j=(a^{-1}2^{-j-1},\, a 2^{-j+1}]$. The $R_j$'s are disjoint subintervals of $(0,\,1]$. Now we slightly modify \eqref{hbis} as follows:
 \beq\label{hter}
 \f^{(j)}=\f^{(j)}h^{(j+K)}=\frac{\f^{(j)}}{\p^{(\e)}}\,h^{(j+K)} \p^{(\e)}\,,\quad \e\in R_j\,.
 \eeq
Then
 $$\big\|\F^{-1}\big(\frac{\f^{(j)}}{\p^{(\e)}}\big)\big\|_1=\big\|\F^{-1}\big(\frac{\f(2^{-j}\e^{-1}\cdot)}{\p}\big)\big\|_1
 \le\sup_{2a^{-1}\le\d\le a2^{-1}}\big\|\F^{-1}\big(\frac{\f^{(\d)}}{\p}\big)\big\|_1<\8.$$
Like in step~3, we deduce
 \be\begin{split}
  \|x\|_{B_{p, q}^\a}
  &\les \Big(\sum_{j=0}^\8\big(2^{j\a}\int_{R_j} \|\wt h_{j+K}*\wt\p_{\e}*x\|_p\big)^q \,\frac{d\e}\e\Big)^{\frac1q}\\
  &\les  \Big(\int_0^1\big(\e^{-\a}\|\wt\p_\e*x\|_p\big)^q\,\frac{d\e}\e\Big)^{\frac1q}
  +\Big(\sum_{j=0}^\8\big(2^{j\a}\int_{R_j} \|\wt h_{j+K}*\wt\p_{\e}*x\|_p\big)^q \,\frac{d\e}\e\Big)^{\frac1q}\,.
 \end{split}\ee
The remaining of the proof follows step~3 with necessary modifications as in the first part.
\end{proof}

\begin{rk}\label{charat c0}
 Theorems~\ref{g-charct-Besov} and \ref{g-charct-Besov-cont} admit analogous characterizations for $B_{p, c_0}^\a(\T^d_\t)$ too. For example, a distribution $x$ on $\T^d_\t$ belongs to $B_{p, c_0}^\a(\T^d_\t)$  iff
  $$\lim_{\e\to0}\frac{\big\|\wt\p_\e*x\big\|_p}{\e^{\a}}=0.$$
This easily follows from Theorem~\ref{g-charct-Besov-cont} for $q=\8$. The same remark applies to the characterizations by the Poisson, heat semigroups and differences in the next two sections.
 \end{rk}

%%%%%%%%%%%%%%%%%%%%%%%%%%%%%%%%%%%%%%%%%%%%%%%%%%%%%%%%%%%%%%%%%%%%%%%%
%%%%%%%%%%%%%%%%%%%%%%%%%%%%%%%%%%%%%%%%%%%%%%%%%%%%%%%%%%%%%%%%%%%%%%%%

\section{The characterizations by Poisson and heat semigroups}
\label{The characterizations by Poisson and heat semigroups}

%%%%%%%%%%%%%%%%%%%%%%%%%%%%%%%%%%%%%%%%%%%%%%%%%%%%%%%%%%%%%%%%%%%%%%%%
%%%%%%%%%%%%%%%%%%%%%%%%%%%%%%%%%%%%%%%%%%%%%%%%%%%%%%%%%%%%%%%%%%%%%%%%

We now concretize the general characterization in the previous section  to the case of Poisson and heat kernels. We begin with the Poisson case. Recall that $\mathrm{P}$ denotes the Poisson kernel of $\real^d$ and
 $$\wt{\mathrm{P}}_\e(x)=\wt{\mathrm{P}}_\e*x=\sum_{m\in\ent^d}e^{-2\pi\e|m|}\wh x(m)U^m\,.$$
So for any positive integer $k$, the $k$th derivation relative to $\e$ is given by
 $$ \frac{\partial^k}{\partial\e^k}\,\wt{\mathrm{P}}_\e(x)=\sum_{m\in\ent^d}(-2\pi |m|)^ke^{-2\pi\e|m|}\wh x(m)U^m\,.$$
The inverse of the $k$th derivation is the $k$th integration $\mathcal{I}^k$ defined for $x$ with $\wh x(0)=0$ by
 \be\begin{split}
 \mathcal{I}^k_\e\,\wt{\mathrm{P}}_\e(x)
 &=\int_\e^\8\int_{\e_{k}}^\8\cdots\int_{\e_2}^\8\wt{\mathrm{P}}_{\e_1}(x)d\e_1\cdots d\e_{k-1}d\e_k\\
 &=\sum_{m\in\ent^d\setminus\{0\}}(2\pi|m|)^{-k}e^{-2\pi\e|m|}\wh x(m)U^m\,.
 \end{split}\ee
In order to simplify the presentation, for any $k\in\ent$, we define
 $$ \mathcal{J}^k_\e=\frac{\partial^k}{\partial\e^k}\;\text{ for }\; k\ge0\quad\text{and}\quad
  \mathcal{J}^k_\e=\mathcal{I}^{-k}_\e\;\text{ for }\; k<0.$$

It is worth to point out that all concrete characterizations in this section in terms of integration operators are new even in the classical case. Also, compare the following theorem with \cite[Section~2.6.4]{HT1992}, in which $k$ is  assumed to  be a positive integer in the Poisson characterization, and a nonnegative integer in the heat characterization.

\begin{thm}\label{Poisson charct-Besov}
 Let $1\le p, q\le\8$, $\a\in\real$ and $k\in\ent$ such that $k>\a$. Then for any distribution $x$ on $\T^d_\t$, we have
 $$\|x\|_{B_{p,q}^\a}\approx\Big(|\wh x(0)|^q+\int_0^1\e^{q(k-\a)}\big\|\mathcal{J}^k_\e\,\wt{\mathrm{P}}_\e(x)\big\|_p^q\,\frac{d\e}\e\Big)^{\frac1q}\,.$$
 \end{thm}

\begin{proof}
 Recall that $\mathrm{P}=\mathrm{P}_1$. Thanks to $\wh{\mathrm{P}}(\xi)=e^{-2\pi|\xi|}$,  we introduce the function $\p(\xi)=(-{\rm sgn}(k)2\pi|\xi|)^ke^{-2\pi|\xi|}$. Then
 $$\p(\e\xi)=\e^k\,\mathcal{J}^k_\e\,e^{-2\pi\e|\xi|}=\e^k\,\mathcal{J}^k_\e\,\wh{\mathrm{P}}_\e(\xi).$$
It follows that for $x\in B_{p,q}^\a(\T^d_\t)$,
 $$\wt\p_\e*x=\e^k\,\mathcal{J}^k_\e\,\wt{\mathrm{P}}_\e*x
 =\e^k\,\mathcal{J}^k_\e\,\wt{\mathrm{P}}_\e(x)\,.$$
Thus by Theorem~\ref{g-charct-Besov-cont}, it remains to check that $\p$ satisfies \eqref{psi} for some $\a_0<\a<\a_1$. It is clear that the last condition there is verified for any $\a_0$. For the second one, choosing $k=\a_1>\a$, we have  $I_{-\a_1}h\,\p=(-{\rm sgn}(k)2\pi)^k\, h\,\wh{\mathrm{P}}$. So
 $$\big\|\F^{-1}\big(I_{k-\a_1}h\,\p \big)\big\|_1\le (2\pi)^k\, \big\|\F^{-1}(h)\big\|_1\,
 \big\|\mathrm{P}\big\|_1<\8\,.$$
The theorem is thus proved.
  \end{proof}

There exists an analogous characterization in terms of the heat kernel. Let $\mathrm{W}_\e$ be the heat semigroup of $\real^d$:
 $$
 \mathrm{W}_\e(s)=\frac1{(4\pi\e)^{\frac d2}}\,e^{-\frac{|s|^2}{4\e}}\,.
 $$
As usual, let $\wt{ \mathrm{W}}_\e$ be the periodization of $\mathrm{W}_\e$. Given a distribution $x$ on $\T^d_\t$ let
 $$\wt{ \mathrm{W}}_\e(x)=\wt{ \mathrm{W}}_\e*x=\sum_{m\in\ent^d}\wh{ \mathrm{W}}(\sqrt\e\, m)\wh x(m)U^m\,,$$
 where $\mathrm{W}=\mathrm{W}_1$. Recall that
  $$\wh{ \mathrm{W}}(\xi)=e^{-4\pi^2|\xi|^2}\,.$$

 \begin{thm}\label{Heat charct-Besov}
 Let $1\le p, q\le\8$, $\a\in\real$ and $k\in\ent$ such that $k>\frac\a2$. Then for any distribution $x$ on $\T^d_\t$,
 $$
 \|x\|_{B_{p,q}^\a}\approx
 \Big(|\wh x(0)|^q+\int_0^1\e^{q(k-\frac\a2)}\big\|\mathcal{J}^k_\e\,\wt{\mathrm{W}}_\e(x)\big\|_p^q\,\frac{d\e}\e\Big)^{\frac1q}\,.
 $$
 \end{thm}

\begin{proof}
 This proof is similar to and simpler than the previous one.  This time, the function $\p$ is defined by $\p(\xi)=(-{\rm sgn}(k)4\pi^2|\xi|^2)^{k}e^{-4\pi^2|\xi|^2}$. Clearly, it satisfies \eqref{psi} for $2k=\a_1>\a$ and any $\a_0<\a$.  Thus Theorem~\ref{g-charct-Besov-cont} holds for this choice of $\p$. Note that a simple change of variables shows that the integral in \eqref{phi-psi-cont} is equal to
 $$2^{-\frac1q}\,\Big(\int_0^1 \e^{-\frac{\a q}2}\big\|\wt\p_{\sqrt\e}*x\big\|_p^q\,\frac{d\e}\e\Big)^{\frac1q}\,.
 $$
 Then using the identity
 $$\p(\sqrt\e\,\xi)=\e^k\,\mathcal{J}^k_\e\,\wh{\mathrm{W}}_\e(\xi),$$
we obtain the desired assertion.
  \end{proof}

Now we wish to formulate Theorems~\ref{Poisson charct-Besov} and \ref{Heat charct-Besov} directly in terms of the  circular Poisson and heat semigroups  of $\T^d$. Recall that  $\mathbb{P}_r$ denote the circular Poisson kernel of  $\T^d$ introduced by \eqref{circular P} and
the Poisson integral of a distribution $x$ on $\T^d_\t$ is defined by
 $$\mathbb{P}_r(x) = \sum_{m \in \mathbb{Z}^d } \wh{x} ( m ) r^{|m|} U^{m}, \quad 0 \le r < 1.$$
Accordingly, we introduce the circular heat kernel $\mathbb{W}$ of $\T^d$:
 \beq\label{circular W}
 \mathbb{W}_r(z)= \sum_{m \in \mathbb{Z}^d }r^{|m|^2} z^{m}, \quad z\in\T^d,\;0 \le r < 1.
 \eeq
Then for $x\in{\mathcal S}'(\T^d_\t)$ we put
 $$\mathbb{W}_r(x)= \sum_{m \in \mathbb{Z}^d } \wh{x} ( m ) r^{|m|^2} U^{m}, \quad 0 \le r < 1.$$
As before, $\mathcal{J}^k_r$ denotes  the $k$th derivation $\frac{\partial^k}{\partial r^k}$ if $k\ge0$ and the $(-k)$th integration $\mathcal{I}^{-k}_r$ if $k<0$:
 $$\mathcal{J}^k_r\,\mathbb{P}_r(x) = \sum_{m \in \ent^d}C_{m, k} \wh{x} (m) r^{|m|-k} U^{m}\,,$$
where
 $$C_{m, k}=|m|\cdots(|m|-k+1)\;\text{ if }\;k\ge0 \quad\text{ and }\quad
 C_{m, k}=\frac1{(|m|+1)\cdots(|m|-k)}\;\text{ if }\;k<0.$$
 $\mathcal{J}^k_r\,\mathbb{W}_r(x)$ is defined similarly. Since $|m|$ is not necessarily an integer, the coefficient $C_{m, k}$ may not vanish for $|m|<k$ and $k\ge2$. In this case, the corresponding term in $\mathcal{J}^k_r\,\mathbb{P}_r(x)$ above cause a certain problem of integrability since $r^{(|m|-k)q}$ is integrable on $(0,\, 1)$ only when $(|m|-k)q>-1$. This explains why we will remove all these terms from  $\mathcal{J}^k_r\,\mathbb{P}_r(x)$ in the following theorem. However, this difficulty does not occur for the heat semigroup.

 \smallskip

The following is new even in the classical case, that is, in the case of $\t=0$.

 \begin{thm}\label{circular-charct-Besov}
 Let $1\le p, q\le\8$, $\a\in\real$ and $k\in\ent$. Let $x$ be a distribution on $\T^d_\t$.
 \begin{enumerate}[\rm(i)]
 \item If $k>\a$, then
  $$
 \|x\|_{B_{p,q}^\a}\approx
 \Big(\max_{|m|<k}|\wh x(m)|^q+\int_0^1(1-r)^{q(k-\a)}\big\|\mathcal{J}^k_r\,{\mathbb{P}}_r(x_k)\big\|_p^q\,\frac{d r}{1-r}\Big)^{\frac1q}\,,
 $$
where $\displaystyle x_k=x-\sum_{|m|<k}\wh x(m)U^m$.
  \item If $k>\frac\a2$, then
 $$
 \|x\|_{B_{p,q}^\a}\approx
 \Big(\max_{|m|^2<k}|\wh x(m)|^q+\int_0^1(1-r)^{q(k-\frac\a2)}\big\|\mathcal{J}^k_r\,{\mathbb{W}}_r(x)\big\|_p^q\,\frac{d r}{1-r}\Big)^{\frac1q}\,.
 $$
 \end{enumerate}
 \end{thm}

\begin{proof}
 We consider only the case of the Poisson kernel. Fix $x\in B_{p,q}^\a(\T^d_\t)$ with $\wh x(0)=0$. We first claim that for any $0<\e_0<1$,
  $$\Big(\int_0^{1}\e^{q(k-\a)}\big\|\mathcal{J}^k_\e\,\wt{\mathrm{P}}_\e(x)\big\|_p^q\,\frac{d\e}\e\Big)^{\frac1q}
  \approx \Big(\int_0^{\e_0}\e^{q(k-\a)}\big\|\mathcal{J}^k_\e\,\wt{\mathrm{P}}_\e(x)\big\|_p^q\,\frac{d\e}\e\Big)^{\frac1q}\,.$$
Indeed, since
  $$\mathcal{J}^k_\e\,\wt{\mathrm{P}}_\e(x)
  =\sum_{m\in\ent^d\setminus\{0\}} (-{\rm sgn}(k)2\pi|m|)^k e^{-2\pi \e|m|} \wh x(m) U^m\,,$$
we have
 \be\begin{split}
  \Big(\int_{0}^{\e_0}\e^{q(k-\a)}\big\|\mathcal{J}^k_\e\,\wt{\mathrm{P}}_\e(x)\big\|_p^q\,\frac{d\e}\e\Big)^{\frac1q}
  &\ge \sup_{m\in\ent^d\setminus\{0\}}(2\pi|m|)^k | \wh x(m)|\Big(\int_{0}^{\e_0}\e^{q(k-\a)}e^{-2\pi \e|m|q}\,\frac{d\e}\e\Big)^{\frac1q}\\
  &\gtrsim\sup_{m\in\ent^d\setminus\{0\}} |m|^\a |\wh x(m)|\,.
  \end{split}\ee
 On the other hand, by triangular inequality,
  \be\begin{split}
  \Big(\int_{\e_0}^{1}\e^{q(k-\a)}\big\|\mathcal{J}^k_\e\,\wt{\mathrm{P}}_\e(x)\big\|_p^q\,\frac{d\e}\e\Big)^{\frac1q}
  &\le\sum_{m\in\ent^d\setminus\{0\}} (2\pi|m|)^k | \wh x(m)|\Big(\int_{\e_0}^{1}\e^{q(k-\a)}e^{-2\pi \e|m|q}\,\frac{d\e}\e\Big)^{\frac1q}\\
  &\les\sup_{m\in\ent^d\setminus\{0\}}|m|^\a |\wh x(m)|\, \sum_{m\in\ent^d} |m|^{k-\a} e^{-2\pi \e_0|m|}\\
  &\les \sup_{m\in\ent^d\setminus\{0\}}|m|^\a |\wh x(m)|\,.
  \end{split}\ee
We then deduce the claim.

 Similarly, we can show that for any $0<r_0<1$,
  $$\Big(\int_0^1(1-r)^{q(k-\a)}\big\|\mathcal{J}^k_r\,{\mathbb{P}}_r(x_k)\big\|_p^q\,\frac{d r}{1-r}\Big)^{\frac1q}
  \approx \Big(\int_{r_0}^1(1-r)^{q(k-\a)}\big\|\mathcal{J}^k_r\,{\mathbb{P}}_r(x_k)\big\|_p^q\,\frac{d r}{1-r}\Big)^{\frac1q}\,.$$
Now we use the relation $r=e^{-2\pi\e}$. If $\e_0>0$ is sufficiently small, then
 $$1-r\approx \e\; \text{ for }\; \e\in(0, \,\e_0).$$
So
 $$\Big(\int_0^{\e_0}\e^{q(k-\a)}\big\|\mathcal{J}^k_\e\,\wt{\mathrm{P}}_\e(x)\big\|_p^q\,\frac{d\e}\e\Big)^{\frac1q}
 \approx\Big(\sup_{0<|m|<k}|\wh x(m)|^q+\int_{r_0}^1(1-r)^{q(k-\a)}\big\|\mathcal{J}^k_r\,{\mathbb{P}}_r(x_k)\big\|_p^q\,\frac{d r}{1-r}\Big)^{\frac1q}\,.$$
Combining this with Theorem~\ref{Poisson charct-Besov}, we get the desired assertion.
 \end{proof}

%%%%%%%%%%%%%%%%%%%%%%%%%%%%%%%%%%%%%%%%%%%%%%%%%%%%%%%%%%%%%%%%%%%%%%%%
%%%%%%%%%%%%%%%%%%%%%%%%%%%%%%%%%%%%%%%%%%%%%%%%%%%%%%%%%%%%%%%%%%%%%%%%

\section{The characterization by differences}
\label{The characterization by differences}

%%%%%%%%%%%%%%%%%%%%%%%%%%%%%%%%%%%%%%%%%%%%%%%%%%%%%%%%%%%%%%%%%%%%%%%%
%%%%%%%%%%%%%%%%%%%%%%%%%%%%%%%%%%%%%%%%%%%%%%%%%%%%%%%%%%%%%%%%%%%%%%%%

In this section we show the quantum analogue of the classical characterization of Besov spaces by differences. Recall that  $\o_{p}^k(x,\e)$ is the $L_p$-modulus of smoothness of $x$ introduced in Definition~\ref{def diff}. The result of this section is the following:

\begin{thm}\label{diff-Besov}
Let $1\leq p,q \leq \8$ and $0<\a <n$ with $n\in \mathbb{N}$. Then for any distribution $x$ on $\T^d_\t$,
 \beq\label{diff-Besov bis}
 \big\|x\big\|_{B_{p,q}^\a}\approx \Big(|\wh x(0)|^q+\int_0^1 \e^{-q\a} \o_{p}^n(x,\e)^q \frac{d\e}{\e}\Big)^{\frac1q}\,.
 \eeq
\end{thm}

\begin{proof}
 We will derive the result from Theorem~\ref{g-charct-Besov-cont}, or more precisely, from its proof.  Since $\a>0$,  we take $\a_0=0$ and $\a_1=n$ in that theorem. Recall that $d_u(\xi)=e^{2\pi{\rm i}u\cdot\xi}-1$. Then the last condition of \eqref{psi} with $\p=d_u^n$  is satisfied uniformly in $u$ since
 $$\big\|\F^{-1}(d_u^n(2^{j}\cdot)\f)\big\|_1=\big\|\D_{2^ju}^n\F^{-1}(\f)\big\|_1\le 2^n\big\|\F^{-1}(\f)\big\|_1\,.$$
We will use a variant of the second one (which is not necessarily verified). To this end, let us come back to \eqref{h} and rewrite it as follows:
 \be\begin{split}
 \p^{(j)}(\xi)\f^{(j+k)}(\xi)
 &=2^{nk}\,\frac{\p^{(j)}(\xi)}{(2^{-j}u\cdot\xi)^{n}}\,h^{(j+K)}(\xi) (2^{-j-k}u\cdot\xi)^{n}\f^{(j+k)}(\xi)\\
 &=2^{nk}\eta^{(j)}(\xi)\rho^{(j+k)}(\xi),
 \end{split}\ee
where $\eta$  and $\rho$ are now defined by
 $$\eta(\xi)=\frac{\p(\xi)}{(u\cdot\xi)^{n}}\,h^{(K)}(\xi) \;\text{ and }\; \rho(\xi)=(u\cdot\xi)^{n}\f(\xi).$$
The second condition of \eqref{psi} becomes the requirement that \
 $$\sup_{u\in\real^d, |u|\le1}\,\big\|\F^{-1}(\eta)\big\|_1<\8.$$
The crucial point here is that $\p(\xi)=d_u^n(\xi)=(u\cdot\xi)^n\zeta(u\cdot\xi)$, where $\zeta$ is an analytic function on $\real$. This shows that the above supremum is finite.

However, the first condition of \eqref{psi}, the Tauberian condition is not verified for a single $d_u^n$. We will return back to this point later. For the moment, we just observe that the Tauberian condition  has not been used in steps~1 and 2 of the proof of Theorem~\ref{g-charct-Besov}. Reexamining those two steps with $\p=d_u^n$, we see that all estimates there can be made independent of $u$. For instance, \eqref{Kke} now becomes (with $\a_1=n$)
 $$2^{j\a}\|\D^n_{\e u}\wt\f_{j+k}*x\|_p\les 2^{(\a_1-\a)k}\big(2^{(j+k)\a}\|\wt \rho_{j+k}*x\|_p\big),$$
where the new function $\rho$ is defined as above.
Thus taking the supremum over all $u$ with $|u|\le1$, we get
 $$2^{j\a}\o_{p}^n(x,\e)\les 2^{k(\a_1-\a)}\big(2^{ (j+k)\a}\|\wt \rho_{(j+k)}*x\|_p\big).$$
Therefore,  by Lemma~\ref{IJ}, we obtain
 $$\Big(\int_0^1 \e^{-q\a} \o_{p}^n(x,\e)^q \frac{d\e}{\e}\Big)^{\frac1q}\les \big\|x\big\|_{B_{p,q}^\a}\,.$$

The reverse inequality requires necessarily a Tauberian-type condition. Although a single $d_u^n$ does not satisfy it, a finite family of $d_u^n$'s does satisfy this condition, which we precise below. Choose a $\frac12$-net $\{v_\el\}_{1\le \el\le L}$ of the unit sphere of $\real^d$. Let $u_\el=4^{-1}v_\el$ and
 $$\Om_\el=\big\{\xi : 2^{-1}\le|\xi|\le2,\; \big|\frac{\xi}{|\xi|}-v_\el\big|\le 2^{-1}\big\}.$$
Then the union of the $\Om_\el$'s is equal to $\{\xi : 2^{-1}\le|\xi|\le2\}$ and $|d_{u_\el}^n|>0$ on $\Om_\el$. So the family $\{d_{u_\el}^n\}_{1\le \el\le L}$  satisfies the following Tauberian-type condition:
 $$\sum_{\el=1}^L|d_{u_\el}^n|>0\;\text{ on }\; \{\xi : 2^{-1}\le|\xi|\le2\}.$$
Now we reexamine step~3 of the proof of Theorem~\ref{g-charct-Besov}. To adapt it to the present setting, by an appropriate partition of unity,  we decompose $\f$ into a sum of infinitely differentiable functions, $\f=\f_1+\cdots+\f_L$ such that ${\rm supp}\, \f_\el\subset\Om_\el$. Accordingly, for every $j\ge0$, let
 $$\f^{(j)}=\sum_{\el=1}^L \f_\el^{(j)}\,.$$
Then we write the corresponding \eqref{hter} with $(\f_\el, d_{u_\el}^n)$ in place of $(\f, \p)$ for every $\el\in\{1,\cdots, L\}$. Arguing as in step~3 of the proof of Theorem~\ref{g-charct-Besov}, we get
 $$\big\|x\big\|_{B_{p,q}^\a}\les
 \Big(|\wh x(0)|^q+\sum_{\el=1}^L\int_0^1\e^{-q\a}\big\|\wt{(d_{u_\el}^n)}_\e*x\big\|_p^q\,\frac{d\e}\e\Big)^{\frac1q}\,.$$
Since $\wt{(d_{u_\el}^n)}_\e*x=\D_{\e u_\el}^nx$, we deduce
 \be\begin{split}
 \big\|x\big\|_{B_{p,q}^\a}
 &\les
 \Big(|\wh x(0)|^q+\int_0^1\e^{-q\a}\sup_{1\le \el\le L}\big\|\D_{\e u_\el}^nx\big\|_p^q\,\frac{d\e}\e\Big)^{\frac1q}\\
 &\les \Big(|\wh x(0)|^q+\int_0^1 \e^{-q\a}\o_{p}^n(x,\e)^q \frac{d\e}{\e}\Big)^{\frac1q}\,.
 \end{split}\ee
Thus  the theorem is proved.
 \end{proof}

As a byproduct, the preceding theorem implies that  the right-hand side of \eqref{diff-Besov bis} does not depend on $n$ with $n>\a$, up to equivalence. This fact admits a direct simple proof and is an immediate consequence of the following analogue of Marchaud's classical inequality which is of interest in its own right.

  \begin{prop}\label{Marchaud}
  For any positive integers $n$ and $N$ with $n<N$ and  for any $\e>0$, we have
  $$
  2^{n-N}\o_{p}^N(x,\e)\le \o_{p}^n(x,\e)\les \e^n\int_{\e}^\8\frac{\o_{p}^N(x,\d)}{\d^n}\,\frac{d\d}\d\,.
  $$
 \end{prop}

\begin{proof}
 The argument below is standard.
Using the identity $\D_u^N=\D_u^{N-n}\, \D_u^n $, we get
 $$\big\|\D_u^N(x)\big\|_p\le 2^{N-n}\big\|\D_u^n(x)\big\|_p\,,$$
whence the lower estimate. The upper one is less obvious. By elementary calculations, for any $u\in\real^d$, we have
 $$
  d_{2u}^n
  =2^nd_{u}^n +\big[\sum_{j=0}^n\left(\begin{array}{c}  n\\j \end{array}\right) e_{ju}-2^n\big]\,d_{u}^n
  =2^n d_{u}^n+\sum_{j=0}^n\left(\begin{array}{c}  n\\j \end{array}\right) \sum_{i=0}^{j-1}e_{iu}d_{u}^{n+1}\,.
 $$
In terms of Fourier multipliers, this means
 $$\D_{2u}^n
  =2^n \D_{u}^n +\sum_{j=0}^n\left(\begin{array}{c}  n\\j \end{array}\right) \sum_{i=0}^{j-1}T_{iu}\D_{u}^{n+1}\,.
 $$
It then follows that
 $$\o_{p}^n(x,\e)\le \frac n2\, \o^{n+1}_{p}(x,\e) +  2^{-n}\o_{p}^n(x,2\e).$$
Iterating this inequality yields
 $$ \o_{p}^n(x,\e)\le \frac n2\, \sum_{j=1}^{J-1}\o^{n+1}_{p}(x, 2^j\e) +  2^{-Jn}\o_{p}^n(x,2^J\e),$$
from which we deduce the desired inequality for $N=n+1$ as $J\to\8$. Another iteration argument then yields the general case.
 \end{proof}

In view of  Definition~\ref {def diff} and Theorem~\ref{diff-Besov}, we introduce the following quantum analogue of the classical $L_p$-Zygmund class of order $\a$.  The case where $0<\a<1$ and $p=\8$ was already studied by Weaver \cite{Weaver1998}.

\begin{Def}
 Let $1\le p\le\8$, $\a>0$ and $n$ be the smallest integer greater than $\a$. The $L_p$-Zygmund class of order $\a$, $\Lambda_{p}^\a(\T^d_\t)$, is defined to be the space of all distributions $x$ such that 
  $$\sup_{\e>0}\frac{\o_p^n(x)}{\e^\a}<\8,$$
 equipped with the norm
 $$\|x\|=|\wh x(0)|+\sup_{\e>0}\frac{\o_p^n(x)}{\e^\a}\,.$$
The little $L_p$-Zygmund class of order $\a$, $\Lambda_{p, 0}^\a(\T^d_\t)$, is the subspace of $\Lambda_{p}^\a(\T^d_\t)$ consisting of all elements $x$ such that
 $$\lim_{\e\to0}\frac{\o_p^n(x)}{\e^\a}=0.$$
\end{Def}

\begin{rk}\label{Lip-Besov}
 Theorem~\ref{diff-Besov} shows that $B_{p,\8}^\a(\T^d_\t)=\Lambda_{p}^\a(\T^d_\t)$  and  $B_{p,c_0}^\a(\T^d_\t)=\Lambda_{p,0}^\a(\T^d_\t)$ with equivalent norms. Consequently, the integer $n$ in the above definition can be any integer greater than $\a$. On the other hand, by the reduction Theorem~\ref{Besov-isom}, if $\a$ is not an integer and if $k$ is the biggest integer less than $\a$, then 
 $$\Lambda_{p}^\a(\T^d_\t)=\big\{x\in \mathcal{S}'(\T^d_\t) : \d_j^kx\in {\rm Lip}_{p}^{\a-k}(\T^d_\t),\;j=1,\cdots, d.  \big\}.$$
A similar equality holds for the little $L_p$-Zygmund and Lipschitz classes of order $\a$.
  \end{rk}

%%%%%%%%%%%%%%%%%%%%%%%%%%%%%%%%%%%%%%%%%%%%%%%%%%%%%%%%%%%%%%%%%%%%%%%%
%%%%%%%%%%%%%%%%%%%%%%%%%%%%%%%%%%%%%%%%%%%%%%%%%%%%%%%%%%%%%%%%%%%%%%%%

\section{Limits of Besov norms}
\label{Limits of Besov norms}

%%%%%%%%%%%%%%%%%%%%%%%%%%%%%%%%%%%%%%%%%%%%%%%%%%%%%%%%%%%%%%%%%%%%%%%%
%%%%%%%%%%%%%%%%%%%%%%%%%%%%%%%%%%%%%%%%%%%%%%%%%%%%%%%%%%%%%%%%%%%%%%%%

In this section we consider the behavior of the right-hand side of \eqref{diff-Besov bis} as $\a\to n$. The study of this behavior is the subject of several recent publications in the classical setting; see, for instance, \cite{Beckner2014, Beckner2015, KMX2005, MS2002, HT2011}. It originated  from  \cite{BBM2001} in which Bourgain, Br\'ezis and Mironescu proved that for any $1\le p<\8$ and any $f\in C_0^\8(\real^d)$
 $$\lim_{\a\to 1}\Big((1-\a)\int_{\real^d \times \real^d} \frac{|f(s)-f(t)|^p}{|s-t|^{\a p+d}}ds\,dt \Big)^{\frac1p}
 =C_{p, d} \|\nabla f(t)\|_p.$$
It is well known that
 $$\Big(\int_{\real^d \times \real^d} \frac{|f(s)-f(t)|^p}{|s-t|^{\a p+d}}ds\,dt \Big)^{\frac1p}\approx
 \Big(\int_0^\8 \sup_{u\in\real^d, |u|\le\e}\big\|\D_u f\big\|_p^p\,\frac{d\e}\e \Big)^{\frac1p}\,.$$
The right-hand side is the norm of $f$ in the Besov space $B_{p,p}^1(\real^d)$.  Higher order extensions have been obtained in \cite{KMX2005, HT2011}.

\smallskip

The main result of the present section is the following quantum extension of these results. Let
\beq\label{besov diff norm}
\|x\|_{B^{\a,\o} _{p,q}} =\Big(\int_0^1 \e^{-\a q} \o_{p}^k(x,\e)^q \frac{d\e}{\e}\Big)^{\frac1q}\,.
\eeq

\begin{thm}\label{BBM}
Let $1\le p\le\8$, $1\le q<\8$ and $0<\a<k$ with $k$ an integer. Then for $x\in  W_p^k(\T^d_\t)$,
 $$\lim_{\a\to k} (k-\a)^{\frac1q}\|x\|_{B^{\a,\o} _{p,q}} \approx q^{-\frac1q} |x|_{W_p^k}$$
with relevant constants depending only on $d$ and $k$.
\end{thm}

\begin{proof}
The proof is easy by using  the results of section~\ref{Lipschitz classes}.  Let $x\in  W_p^k(\T^d_\t)$ with $\wh x(0)=0$. Let $A$ denote the limit in Lemma~\ref{sup=lim}. Then
 $$\int_0^1  \e ^{-\a q} \o^k_p(x,\e)^q\frac{d\e }{\e }\le A^q\int_0^1 \e ^{(k-\a)q}\frac{d\e }{\e }\,,$$
whence
 $$\limsup_{\a\to k}(k-\a)\int_0^1  \e ^{-\a q} \o_p(x,\e)^q\frac{d\e }{\e }\le \frac{A^q}q\,.$$
Conversely, for any $\eta>0$, choose $\d\in (0,1)$ such that
 $$\frac{\o^k_p(x,\e)}{\e^k}\ge A-\eta,\;\;\forall \e\le \d.$$
Then
 $$
(k-\a)\int_0^1  \e ^{-\a q} \o^k_p(x,\e)^q\frac{d\e }{\e }\ge  \frac{(A-\eta)^q}q\, \d^{(k-\a)q},
 $$
which implies
 $$\liminf_{\a\to k}(k-\a)\int_0^1  \e ^{-\a q} \o^k_p(x,\e)^q\frac{d\e }{\e }\ge  \frac{(A-\eta)^q}q\,.$$
Therefore,
 $$\lim_{\a\to k}(k-\a)\int_0^1  \e ^{-\a q} \o^k_p(x,\e)^q\frac{d\e }{\e }=\frac{1}q\,\lim_{\e\to0}\frac{\o^k_p(x,\e)}{\e^k}\,.$$
So Theorem~\ref{lim=nabla} implies the desired assertion.
 \end{proof}

\begin{rk}
 We will determine later the behavior of $\|x\|_{B^{\a,\o} _{p,q}}$ when $\a\to0$; see Corollary~\ref{limit Besov 0} below.
 \end{rk}

%%%%%%%%%%%%%%%%%%%%%%%%%%%%%%%%%%%%%%%%%%%%%%%%%%%%%%%%%%%%%%%%%%%%%%%%
%%%%%%%%%%%%%%%%%%%%%%%%%%%%%%%%%%%%%%%%%%%%%%%%%%%%%%%%%%%%%%%%%%%%%%%%

\section{The link with the classical Besov spaces}
\label{The link with the classical Besov spaces}

%%%%%%%%%%%%%%%%%%%%%%%%%%%%%%%%%%%%%%%%%%%%%%%%%%%%%%%%%%%%%%%%%%%%%%%%
%%%%%%%%%%%%%%%%%%%%%%%%%%%%%%%%%%%%%%%%%%%%%%%%%%%%%%%%%%%%%%%%%%%%%%%%

Like for the Sobolev spaces on $\T^d_\t$, there exists a strong link between $B_{p,q}^\a(\T^d_\t)$ and the classical vector-valued Besov spaces on $\T^d$. Let us give a precise definition of the latter spaces. We maintain the assumption and notation on $\f$ in section~\ref{Definitions and basic properties: Besov}.  In particular, $f\mapsto \wt\f_k*f$ is the Fourier multiplier on $\T^d$ associated to $\f(2^{-k}\cdot)$:
 $$\wt\f_k*f=\sum_{m\in\ent^d}\f(2^{-k}m)\wh f(m)z^m$$
for any $f\in \mathcal{S}'(\T^d; X)$. Here $X$ is a Banach space.

\begin{Def}
Let $1\leq p,q \leq \8$ and $\a\in \real$. Define
$$B^\a_{p,q} (\T^d; X)=\big\{f\in \mathcal{S}'(\T^d; X) : \|f\|_{B^\a_{p,q}} < \8 \big\},$$
where
 $$\|f\|_{B^\a_{p,q}} =\Big(\|\wh f(0)\|_X^q + \sum_{k\ge 0} 2^{\a k q } \big\| \wt\f_k * f\big\|_{L_p(\T^d; X)}^q\Big)^{\frac{1}{q}}\,.$$
\end{Def}

These vector-valued Besov spaces have been largely studied in literature. Note that almost all publications concern only the case of $\real^d$, but the periodic theory is parallel (see, for instance, \cite{Flett1971, Taibleson1964}; see also \cite{AB2004} for the vector-valued case).  $B^\a_{p,q} (\real^d; X)$ is defined in the same way with the necessary modifications among them the main difference concerns the term $\|\wh f(0)\|_X$ above which is replaced  by $\big\|\phi*f\big\|_{L_p(\real^d; X)}$, where $\phi$ is the function whose  Fourier transform is equal to $1-\sum_{k\ge0}\f(2^{-k}\cdot)$.

All results proved in the previous sections remain valid in the present vector-valued setting with essentially the same proofs for any Banach space $X$, except Theorem~\ref{Sobolev-Besov} whose vector-valued version holds only if $X$ is isomorphic to a Hilbert space. On the other hand, the duality assertion in Proposition~\ref{Besov-P} should be slightly modified by requiring that $X^*$ have the Radon-Nikodym property.

Let us state the vector-valued analogue of Theorem~\ref{circular-charct-Besov}. As said before, it is new even in the scalar case. The circular Poisson and heat semigroups are extended to the present case too. For any $f\in \mathcal{S}'(\T^d; X) $,
 $$
 \mathbb{P}_r(f) (z)= \sum_{m \in \ent^d}  r^{|m|} \wh f(m)z^{m}\;\text{ and }\;
 \mathbb{W}_r(f)(z)= \sum_{m \in \ent^d }r^{|m|^2}\wh f(m) z^{m}, \quad z\in\T^d,\;0 \le r < 1.
 $$
The operator $\mathcal{J}^k_r$ has the same meaning as before, for instance, in the Poisson case, we have
 $$\mathcal{J}^k_r\,\mathbb{P}_r(f)= \sum_{m \in \ent^d}C_{m, k} \wh{f} (m) r^{|m|-k} z^{m}\,,$$
where
 $$C_{m, k}=|m|\cdots(|m|-k+1)\;\text{ if }\;k\ge0 \quad\text{ and }\quad
 C_{m, k}=\frac1{(|m|+1)\cdots(|m|-k)}\;\text{ if }\;k<0.$$

 \begin{thm}\label{circular-charct-Besov-vu}
 Let $1\le p, q\le\8$, $\a\in\real$ and $k\in\ent$. Let $X$ be a Banach space.
 \begin{enumerate}[\rm(i)]
 \item If $k>\a$, then for any $f\in B_{p,q}^\a(\T^d; X)$,
  $$
 \|f\|_{B_{p,q}^\a}\approx
 \Big(\sup_{|m|<k}\|\wh f(m)\|_X^q+\int_0^1(1-r)^{(k-\a)q}\big\|\mathcal{J}^k_r\,{\mathbb{P}}_r(f_k)\big\|_{L_p(\T^d; X)}^q\,\frac{d r}{1-r}\Big)^{\frac1q}\,,
 $$
where $\displaystyle f_k=f-\sum_{|m|<k}\wh f(m)z^m$.
  \item If $k>\frac\a2$, then for any $f\in B_{p,q}^\a(\T^d; X)$,
 $$
 \|f\|_{B_{p,q}^\a}\approx
 \Big(\sup_{|m|^2<k}\|\wh f(m)\|_X^q+\int_0^1(1-r)^{(k-\frac\a2)q}\big\|\mathcal{J}^k_r\,{\mathbb{W}}_r(f)\big\|_{L_p(\T^d; X)}^q\,\frac{d r}{1-r}\Big)^{\frac1q}\,.
 $$
 \end{enumerate}
 \end{thm}

The following transference result from $\T^d_\t$ to $\T^d$ is clear. It can be used to prove a majority of the preceding results on $\T^d_\t$, under the hypothesis that the corresponding results in the vector-valued case on $\T^d$ are known.

\begin{prop}\label{TransBesov}
Let $1\leq p, q\leq \8$ and $\a\in\real$. The map $x \mapsto \wt{x}$ in Corollary \ref{prop:TransLp} is an isometric embedding from $B_{p,q}^\a(\T^d_{\t})$  into $B_{p,q}^\a(\T^d; L_p(\T_{\t}^d))$ with $1$-complemented range.
\end{prop}

\begin{rk}\label{os-Besov}
 As a subspace of $\el_q^\a(L_p(\T^d_\t))$ (see the proof of Proposition~\ref{Besov-P} for the definition of this space), $B_{p,q}^\a(\T^d_{\t})$ can be equipped with a  natural operator space structure in Pisier's sense \cite{Pisier1998}. Moreover, in the spirit of the preceding vector-valued case,  we can also introduce the vector-valued quantum Besov spaces. Given an operator space $E$, $B_{p,q}^\a(\T^d_{\t}; E)$ is defined exactly as in the scalar case; it is a subspace of $\el_q^\a(L_p(\T^d_\t; E))$. Then all results of this chapter are extended to this vector-valued case, except the duality in Proposition~\ref{Besov-P} and Theorem~\ref{Sobolev-Besov}.
 \end{rk}

%%%%%%%%%%%%%%%%%%%%%%%%%%%%%%%%%%%%%%%%%%%%%%%%%%%%%%%%%%%%%%%%%%%%%%%%
%%%%%%%%%%%%%%%%%%%%%%%%%%%%%%%%%%%%%%%%%%%%%%%%%%%%%%%%%%%%%%%%%%%%%%%%

{\Large\part{Triebel-Lizorkin spaces}\label{Triebel-Lizorkin spaces}}

\setcounter{section}{0}

%%%%%%%%%%%%%%%%%%%%%%%%%%%%%%%%%%%%%%%%%%%%%%%%%%%%%%%%%%%%%%%%%%%%%%%%
%%%%%%%%%%%%%%%%%%%%%%%%%%%%%%%%%%%%%%%%%%%%%%%%%%%%%%%%%%%%%%%%%%%%%%%%

This chapter is devoted to the study of Triebel-Lizorkin spaces. These spaces are much subtler than Besov spaces even in the classical setting. Like Besov spaces, the classical Triebel-Lizorkin spaces $F_{p, q}^\a(\real^d)$ have three parameters, $p, q$ and $\a$. The difference is that the $\el_q$-norm is now taken before the $L_p$-norm. Namely, $f\in F_{p, q}^\a(\real^d)$ iff $\big\|\big(\sum_{k\ge0}2^{k\a q}|\f_k*f|^q\big)^{\frac1q}\big\|_p$ is finite. Besides the usual subtlety just mentioned, more difficulties appear in the  noncommutative setting. For instance, a first elementary one concerns the choice of the internal $\el_q$-norm. It is a well-known fact in the noncommutative integration that  the simple replacement of the usual absolute value by the modulus of operators does not give a norm except for $q=2$. Alternately, one could use Pisier's definition of $\el_q$-valued noncommutative $L_p$-spaces in the category of operator spaces. However, we will not study the latter choice and will confine ourselves only to the case $q=2$, by considering the column and row norms (and their mixture) for the internal $\el_2$-norms. This choice is dictated by the theory of noncommutative Hardy spaces. In fact, we will show that the so-defined Triebel-Lizorkin spaces on $\T^d_\t$ are isomorphic to  the Hardy spaces developed in \cite{CXY2012}.

Another  difficulty is linked with the frequent use of maximal functions in the commutative case. These functions play a crucial role in the classical theory. However, the pointwise analogue of maximal functions is no longer available in the present setting, which makes our study harder than the classical case. We have already encountered this difficulty in the study of Besov spaces. It is much more substantial now. Instead, our development will rely heavily on the theory of Hardy spaces developed in \cite{XXX} through a Fourier multiplier theorem that is proved in the first section. It is this multiplier theorem  which clears the obstacles on our route. After the definitions and basic properties, we prove some characterizations of the quantum Triebel-Lizorkin spaces. Like in the Besov case, they are better than the classical ones even in the commutative case. We conclude the chapter with a short section on the operator-valued Triebel-Lizorkin spaces on $\T^d$ (or $\real^d$). These spaces are interesting in view of the theory of operator-valued Hardy spaces.

Throughout this chapter, we will use the notation introduced in the previous one. In particular,  $\f$ is a function satisfying \eqref{LP dec}, $\f^{(k)}=\f(2^{-k}\cdot)$ and $\wh\f_k=\f^{(k)}$.

\bigskip
%%%%%%%%%%%%%%%%%%%%%%%%%%%%%%%%%%%%%%%%%%%%%%%%%%%%%%%%%%%%%%%%%%%%%%%%
%%%%%%%%%%%%%%%%%%%%%%%%%%%%%%%%%%%%%%%%%%%%%%%%%%%%%%%%%%%%%%%%%%%%%%%%

\section{A multiplier theorem}
\label{A multiplier theorem}

%%%%%%%%%%%%%%%%%%%%%%%%%%%%%%%%%%%%%%%%%%%%%%%%%%%%%%%%%%%%%%%%%%%%%%%%
%%%%%%%%%%%%%%%%%%%%%%%%%%%%%%%%%%%%%%%%%%%%%%%%%%%%%%%%%%%%%%%%%%%%%%%%

The following multiplier result will play a crucial role in this chapter.

\begin{thm}\label{Hormander}
 Let  $\s\in\real$ with $\s>\frac{d}2$. Assume that $(\phi_j)_{j \ge0}$ and $(\rho_j)_{j \ge0}$ are two sequences of continuous functions on $\real^d\setminus\{0\}$ such that
   \beq\label{psi-rho}
  \left \{ \begin{split}
  & \mathrm{supp}(\phi_j\rho_j)\subset \{\xi: 2^{j-1}\leq |\xi|\le 2^{j+1}\},\; \forall j\ge0,\\
  & \underset{\substack{j\geq 0 \\ -2\leq k \leq 2}}{\sup} \big\|   {\phi} _j (2^{j+k}\cdot)\varphi \big\| _{H_2^\sigma(\mathbb{R}^d)}<\8.
  \end{split} \right.
  \eeq
 \begin{enumerate}[\rm(i)]
 \item Let $1<p<\8$.  Then for any distribution $x$ on $\T^d_\t$,
  $$
 \big\|\big(\sum_{j\ge0}2^{2j\a}|\wt\phi_j*\wt\rho_j*x|^2\big)^{\frac12}\big\|_{L_p(\T^d_\t)}\les \underset{\substack{j\geq 0 \\ -2\leq k \leq 2}}{\sup} \big\|   {\phi} _j (2^{j+k}\cdot)\varphi \big\| _{H_2^\sigma}\,
 \big\|\big(\sum_{j\ge0}2^{2j\a}|\wt\rho_j*x|^2\big)^{\frac12}\big\|_{L_p(\T^d_\t)}\,,
 $$
where the constant depends only on $p, \s, d$ and $\f$.
 \item Assume, in addition, that $\rho_j=\wh{\rho(2^{-j}\cdot)}$ for some Schwartz  function with
  $\mathrm{supp}(\rho)=\{\xi: 2^{-1}\le |\xi|\le2\}.$
 Then the above inequality also holds for $p=1$  with  relevant constant depending additionally on $\rho$.
 \end{enumerate}
 \end{thm}

The remainder of this section is devoted to the proof of the above theorem. As one can guess, the proof is based on the Calder\'on-Zygmund theory. We require several lemmas. 
The first one is  an elementary inequality.

\begin{lem}\label{CS Sobolev}
 Assume that $f: \real^d\to \el_2$ and $g:\real^d\to\com$ satisfy
 $$f\in H_2^\s(\real^d;  \el_2)\;\text{ and }\; \int_{\real^d}(1+|s|^2)^\s|\F^{-1}(g)(s)|ds<\8.$$
 Then $fg\in H_2^\s(\real^d; \el_2)$ and
  $$\big\|fg\big\|_{H_2^\s(\real^d;  \el_2)}\le \big\|f\big\|_{H_2^\s(\real^d;  \el_2)}\,  \int_{\real^d}(1+|s|^2)^\s|\F^{-1}(g)(s)|ds.$$
 \end{lem}

\begin{proof}
 The norm $\|\cdot\|$ below is that of $\el_2$. By the Cauchy-Schwarz inequality, for $s\in\real^d$, we have
  $$ \big\|\F^{-1}(fg)(s)\big\|^2
    = \big\|\F^{-1}(f)*\F^{-1}(g)(s)\big\|^2
    \le \big\| \F^{-1}(g)\big\|_{1}\int_{\real^d}\big\|\F^{-1}(f)(s-t)\big\|^2\, \big|\F^{-1}(g)(t)\big|dt.$$
It then follows that
  \be\begin{split}
 \big\|fg\big\|_{H_2^\s(\real^d; \el_2)}^2
 &= \int_{\real^d} (1+|s|^2)^{\s} \big\|\F^{-1}(fg)(s)\big\|^2ds\\
 &\le \big\| \F^{-1}(g)\big\|_{1}\int_{\real^d}(1+|s|^2)^\s \int_{\real^d}\big\|\F^{-1}(f)(s-t)\big\|^2\,\big|\F^{-1}(g)(t)\big|dt\,ds\\
 &\le \big\| \F^{-1}(g)\big\|_{1} \int_{\real^d}\int_{\real^d}(1+|s-t|^2)^\s \big\|\F^{-1}(f)(s-t)\big\|^2ds \,(1+|t|^2)^\s |\F^{-1}(g)(t)|dt\\
 &\le  \big\|f\big\|_{H_2^\s(\real^d; \el_2)}^2\Big(\int_{\real^d}(1+|t|^2)^\s|\F^{-1}(g)(t)|dt\Big)^2\,.
  \end{split}\ee
Thus the assertion is proved.
 \end{proof}

The following lemma is a well-known result in harmonic analysis, which asserts that every H\"ormander multiplier is a Calder\'on-Zygmund  operator. Note that the usual H\"ormander condition is expressed in terms of partial derivatives up to order $[\frac d2]+1$, while the condition below, in terms of the potential Sobolev space $H^\s_2(\real^d)$, is not commonly used  (it is explicitly stated on page~263 of \cite{Stein1993}).  Combined with the previous lemma, the standard argument as described in \cite[p.~211-214]{gar-rubio}, \cite[p.~245-247]{Stein1993} or \cite[p. 161-165]{HT1978} can be easily adapted to the present setting.

\begin{lem}\label{HS}
 Let $\phi=(\phi_j)_{j\ge0}$ be a sequence of continuous functions on $\real^d\setminus\{0\}$, viewed as a function from $\real^d$ to $\el_2$. Assume that
  \beq\label{HS condition}
  \sup_{k\in\ent} \big\|\phi(2^k\cdot)\f\big\|_{H_2^\s(\real^d; \el_2)}<\8.
  \eeq
 Let $\mathsf{k}=(\mathsf{k}_j)_{j\ge0}$ with $\mathsf{k}_j=\F^{-1}(\phi_j)$. Then $\mathsf{k}$ is a Calder\'on-Zygmund kernel with values in $\el_2$, more precisely,
 \begin{enumerate}[$\bullet$]
 \item $\displaystyle \big\|\wh{\mathsf{k}}\big\|_{L_\8(\real^d; \el_2)}\les \sup_{k\in\ent} \big\|\phi(2^k\cdot)\f\big\|_{H_2^\s(\real^d; \el_2)}$;
 \item $\displaystyle \sup_{t\in\real^d}\int_{|s|>2|t|}\|\mathsf{k}(s-t)-\mathsf{k}(s)\|_{\el_2}ds\les  \sup_{k\in\ent} \big\|\phi(2^k\cdot)\f\big\|_{H_2^\s(\real^d; \el_2)}$.
 \end{enumerate}
The relevant constants depend  only on $\f$, $\s$ and $d$.
 \end{lem}

The above kernel $\mathsf{k}$ defines a Calder\'on-Zygmund operator on $\real^d$. But we will consider only the periodic case, so we  need to periodize $\mathsf{k}$:
  $$\wt{\mathsf{k}}(s)=\sum_{m\in\ent^d} \mathsf{k} (s+m).$$
 By a slight abuse of notation, we use  $\wt{\mathsf{k}}_j$ to denote the Calder\'on-Zygmund operator on $\T^d$ associated to $\wt{\mathsf{k}}_j$ too:
 $$\wt{\mathsf{k}}_j(f)(s)=\int_{\mathbb{I}^d}\wt{\mathsf{k}}_j(s-t) f(t)dt,$$
 where we have identified $\T$ with  $\mathbb{I}=[0,\, 1)$. $\wt{\mathsf{k}}_j$ is the Fourier multiplier on $\T^d$ with symbol $\phi_j$: $f\mapsto \wt\phi_j*f$.

We have $\wh{\wt{\mathsf{k}}}=\wh{\mathsf{k}}\big|_{\ent^d}$.  If $\phi$ satisfies \eqref{HS condition}, then
 Lemma~\ref{HS} implies
   \beq\label{HS per}
  \left \{ \begin{split}
  & \big\|\wh{\wt{\mathsf{k}}}\big\|_{\el_\8(\ent^d; \el_2)}<\8,\\
  & \sup_{t\in\mathbb{I}^d}\int_{\{s\in\mathbb{I}^d: |s|>2|t|\}}\|\wt{\mathsf{k}}(s-t)-\wt{\mathsf{k}}(s)\|_{\el_2}ds<\8.
  \end{split} \right.
  \eeq

Now let $\M$ be a von Neumann algebra equipped with a  normal faithful tracial state $\tau$, and let $\N=L_\8(\T^d)\overline\ot\M$, equipped with the tensor trace. The following lemma should be known to experts; it is closely related to similar results of  \cite{HLMP2014, MP2009, Parcet2009}, notably to \cite[Lemma~2.3]{JMP2014}. Note that the sole difference between the following condition \eqref{HS condition per} and  \eqref{HS condition} is that the supremum in \eqref{HS condition} runs over all integers while the one below is restricted to nonnegative integers.

\begin{lem}\label{CZD}
 Let $\phi=(\phi_j)_{j\ge0}$ be a sequence of continuous functions on $\real^d\setminus\{0\}$ such that
 \beq\label{HS condition per}
    \|\phi\|_{h_2^\s}=\sup_{k\ge0} \big\|\phi(2^k\cdot)\f\big\|_{H_2^\s(\real^d; \el_2)}<\8.
  \eeq
Then for  $1<p<\8$ and any finite sequence $(f_j)\subset L_p(\N)$,
  $$\big\|\big(\sum_{j\ge0} |\wt\phi_j*f_j|^2\big)^{\frac12}\big\|_p\les \|\phi\|_{h_2^\s}\,
   \big\|\big(\sum_{j\ge0} |f_j|^2\big)^{\frac12}\big\|_p\,.$$
 The relevant constant depends only on $p$, $\f$, $\s$ and $d$.
\end{lem}

 \begin{proof}
The argument below is standard.  First, note that the Fourier multiplier on $\T^d$ with symbol $\phi_j$ does not depend on the values of $\phi_j$ in the open unit ball of $\real^d$. So letting $\eta$ be an infinitely differentiable function on $\real^d$ such that $\eta(\xi)=0$ for $|\xi|\le \frac12$ and  $\eta(\xi)=1$ for $|\xi|\ge 1$, we see that $\phi_j$ and $\eta\phi_j$ induce the same Fourier multiplier on $\T^d$ (restricted to distributions with vanishing Fourier coefficients at the origin). On the other hand, it is easy to see that \eqref{HS condition per} implies that the sequence $(\eta\phi_j)_{j\ge0}$ satisfies \eqref{HS condition} with $(\eta\phi_j)_{j\ge0}$  in place of $\phi$. Thus replacing $\phi_j$ by $\eta\phi_j$ if necessary, we will assume that $\phi$ satisfies the stronger condition \eqref{HS condition}.

We will use the Calder\'on-Zygmund theory and consider $\wt{\mathsf{k}}$ as a diagonal matrix with diagonal entries $(\wt{\mathsf{k}}_j)_{j\ge1}$. The Calder\'on-Zygmund operator associated to  $\wt{\mathsf{k}}$ is thus the convolution operator:
   $$\wt{\mathsf{k}}(f)(s)=\int_{\mathbb{I}^d} \wt{\mathsf{k}}(s-t) f(t)dt$$
 for any finite sequence $f=(f_j)$ (viewed as a column matrix). Then the assertion to prove amounts to the boundedness of  $\wt{\mathsf{k}}$ on $L_p(\N; \el_2^c)$.

We first show that $\wt{\mathsf{k}}$ is bounded from  $L_\8(\N; \el_2^c)$ into $\BMO^c(\T^d, B(\el_2)\overline\ot\M)$. Let $f$ be a finite sequence in $L_\8(\N; \el_2^c)$, and let $Q$ be a cube of $\mathbb I^d$ whose center is denoted by $c$. We decompose $f$ as $f=g+h$ with  $g=f\un_{\wt Q}$, where $\wt Q=2Q$, the cube with center $c$ and twice the side length of $Q$.
Setting
 $$a=\int_{\mathbb{I}^d\setminus\wt Q}\wt{\mathsf{k}}(c-t)f(t)dt,$$
we have
 $$\wt{\mathsf{k}}(f)(s)-a=\wt{\mathsf{k}}(g)(s)+\int_{\mathbb{I}^d}(\wt{\mathsf{k}}(s-t)-\wt{\mathsf{k}}(c-t))h(t)dt.$$
 Thus
 $$\frac1{|Q|}\int_Q|\wt{\mathsf{k}} (f)(s)-a|^2ds
 \le 2(A+B),$$
 where
 \be\begin{split}
 A &=\frac1{|Q|}\int_Q|\wt{\mathsf{k}} (g)(s)|^2ds, \\
 B&=\frac1{|Q|}\int_Q \Big|\int_{\mathbb{I}^d}(\wt{\mathsf{k}} (s-t)-\wt{\mathsf{k}}(c-t))h(t)dt\Big|^2ds.
  \end{split}\ee
The first term $A$ is easy to estimate. Indeed, by \eqref{HS per} and the Plancherel formula,
 \be\begin{split}
  |Q|A
  &\le\int_{\mathbb{I}^d}|\wt{\mathsf{k}}(g)(s)|^2ds
 =\sum_{m\in\ent^d}\big|\wh{\wt{\mathsf{k}}}(m)\wh g(m)\big|^2\\
 &=\sum_{m\in\ent^d}\wh g(m)^*\big[\,\wh{\wt{\mathsf{k}}} (m)^*\wh{\wt{\mathsf{k}}} (m)\big]\wh g(m)
 \le\sum_{m\in\ent^d}\|\wh{\wt{\mathsf{k}}}(m)\|_{B(\el_2)}^2|\wh g(m)|^2 \\
 &\le \big\|\wh{\wt{\mathsf{k}}}\big\|_{\el_\8(\ent^d; \el_\8)} \int_{\mathbb{I}^d}|g(s)|^2ds\\
 &\le \big\|\wh{\wt{\mathsf{k}}}\big\|_{\el_\8(\ent^d; \el_2)} \int_{\mathbb{I}^d}|g(s)|^2ds\les|\wt Q|\, \|f\|^2_{L_\8(\N; \el_2^c)}\,,
 \end{split}\ee
whence
 $$\|A\|_{B(\el_2)\overline\ot\M}\les \|f\|^2_{L_\8(\N; \el_2^c)}\,.$$

 To estimate $B$,  let $h=(h_j)$. Then by \eqref{HS per}, for any $s\in Q$ we have
  \be\begin{split}
  \Big|\int_{\mathbb{I}^d}(\wt{\mathsf{k}}(s-t)-\wt{\mathsf{k}}(c-t))h(t)dt\Big|^2
  &=\sum_j \Big|\int_{\mathbb{I}^d}(\wt{\mathsf{k}}_j(s-t)-\wt{\mathsf{k}}_j(c-t))h_j(t)dt\Big|^2\\
 &\les \sum_j \int_{\mathbb{I}^d\setminus\wt Q} |\wt{\mathsf{k}}_j(s-t)-\wt{\mathsf{k}}_j(c-t)|\,|h_j(t)|^2dt\\
 &\les \int_{\mathbb{I}^d\setminus\wt Q}\|\wt{\mathsf{k}}(s-t)-\wt{\mathsf{k}}(c-t)\|_{\el_\8}\sum_j |h_j(t)|^2dt\\
 &\les \|f\|^2_{L_\8(\N; \el_2^c)}  \int_{\mathbb{I}^d\setminus\wt Q}\|\wt{\mathsf{k}}(s-t)-\wt{\mathsf{k}}(c-t)\|_{\el_2}dt\\
 &\les \|f\|^2_{L_\8(\N; \el_2^c)} \,.
  \end{split}\ee
Thus
  \be\begin{split}
  \|B\|_{B(\el_2)\overline\ot\M}
  &\le \frac1{|Q|}\int_Q \Big\|\int_{\mathbb{I}^d}(\wt{\mathsf{k}} (s-t)-\wt{\mathsf{k}}(c-t))h(t)dt\Big\|_{B(\el_2)\overline\ot\M}^2\,ds\\
  &= \frac1{|Q|}\int_Q \Big\|\, \Big|\int_{\mathbb{I}^d}(\wt{\mathsf{k}}(s-t)-\wt{\mathsf{k}}(c-t))h(t)dt\Big|^2\Big\|_{B(\el_2)\overline\ot\M}\,ds\\
  &\les \|f\|^2_{L_\8(\N; \el_2^c)}\,.
   \end{split}\ee
Therefore, $\wt{\mathsf{k}}$ is bounded from  $L_\8(\N; \el_2^c)$ into $\BMO^c(\T^d, B(\el_2)\overline\ot\M)$.

We next show that $\wt{\mathsf{k}}$ is bounded from  $L_\8(\N; \el_2^c)$ into $\BMO^r(\T^d, B(\el_2)\overline\ot\M)$.  Let $f, Q$ and $a$ be as above. Now we have to estimate
 $$\Big\|\frac1{|Q|}\int_Q\big|\big[\,\wt{\mathsf{k}} (f)(s)-a\big]^*\big|^2ds\Big\|_{B(\el_2)\overline\ot\M}\,.$$
We will use the same decomoposition $f=g+h$. Then
 $$\frac1{|Q|}\int_Q\big|\big[\,\wt{\mathsf{k}} (f)(s)-a\big]^*\big|^2ds
 \le 2(A'+B'),$$
 where
 \be\begin{split}
 A' &=\frac1{|Q|}\int_Q\big|\big[\,\wt{\mathsf{k}} (g)(s)\big]^*\big|^2ds, \\
 B'&=\frac1{|Q|}\int_Q \Big|\int_{\mathbb{I}^d}\big[(\wt{\mathsf{k}} (s-t)-\wt{\mathsf{k}}(c-t))h(t)\big]^*dt\Big|^2ds.
  \end{split}\ee
The estimate of $B'$ can be reduced to that of $B$ before. Indeed,
  \be\begin{split}
  \|B'\|_{B(\el_2)\overline\ot\M}
  &\le \frac1{|Q|}\int_Q \Big\|\int_{\mathbb{I}^d}\big[(\wt{\mathsf{k}} (s-t)-\wt{\mathsf{k}}(c-t))h(t)\big]^*dt\Big\|_{B(\el_2)\overline\ot\M}^2ds\\
  &= \frac1{|Q|}\int_Q \Big\|\Big[\int_{\mathbb{I}^d}(\wt{\mathsf{k}} (s-t)-\wt{\mathsf{k}}(c-t))h(t)dt\Big]^*\Big\|_{B(\el_2)\overline\ot\M}^2ds\\
  &=\frac1{|Q|}\int_Q \Big\|\int_{\mathbb{I}^d}(\wt{\mathsf{k}} (s-t)-\wt{\mathsf{k}}(c-t))h(t) dt\Big\|_{B(\el_2)\overline\ot\M}^2ds\\
  &\les \|f\|^2_{L_\8(\N; \el_2^c)}\,.
  \end{split}\ee
However, $A'$ needs a different argument. Setting $g=(g_j)$, we have
 $$\|A'\|_{B(\el_2)\overline\ot\M}
 =\sup\Big\{\frac1{|Q|}\int_Q \tau\big[\sum_{i, j}\wt{\mathsf{k}_i} (g_i)(s)\wt{\mathsf{k}_j} (g_j)(s)^*a_j^*a_i \big]ds\Big\},$$
where the supremum runs over all $a=(a_i)$ in the unit ball of $\el_2(L_2(\M))$. Considering $a_i$ as a constant function on $\mathbb{I}^d$, we can write
 $$a_i\wt{\mathsf{k}_i} (g_i)=\wt{\mathsf{k}_i} (a_ig_i).$$
Thus
 $$\int_Q \tau\big[\sum_{i, j}\wt{\mathsf{k}_i} (g_i)(s)\wt{\mathsf{k}_j} (g_j)(s)^*a_j^*a_i \big]ds
 =\int_Q \big\|\sum_{i}\wt{\mathsf{k}_i} (a_ig_i)(s) \big\|_{L_2(\M)}^2ds.$$
So by the Plancherel formula,
  \be\begin{split}
  \int_Q \big\|\sum_{i}\wt{\mathsf{k}_i} (a_ig_i)(s) \big\|_{L_2(\M)}^2ds
 &\le\int_{\mathbb{I}^d} \big\|\sum_{i}\wt{\mathsf{k}_i} (a_ig_i)(s) \big\|_{L_2(\M)}^2ds\\
 &=\sum_{m\in\ent^d} \big\|\sum_{i}\wh{\wt{\mathsf{k}_i}}(m) \,a_i\, \wh{g_i}(m) \big\|_{L_2(\M)}^2.
  \end{split}\ee
On the other hand, by the Cauchy-Schwarz inequality, \eqref{HS per} and the Plancherel formula once more, we have
  \be\begin{split}
  \sum_{m\in\ent^d} \big\|\sum_{i}\wh{\wt{\mathsf{k}_i}}(m) \,a_i\, \wh{g_i}(m) \big\|_{L_2(\M)}^2
  &\le \sum_{m\in\ent^d}\|\wh{\wt{\mathsf{k}}}(m)\|_{\el_2}^2 \,\sum_{i}\tau(|a_i \wh{g_i}(m) |^2)\\
  &\les \sum_{i}\tau\big[a_i\, \sum_{m\in\ent^d}\wh{g_i}(m) \wh{g_i}(m)^*\,a_i^*\big]\\
  &= \sum_{i}\tau\big[a_i\,\int_{\mathbb{I}^d} g_i(s) g_i(s) ^*ds\,a_i^*\big]\\
   &= \sum_{i}\tau\big[a_i\,\int_{\wt Q} f_i(s) f_i(s) ^*ds\,a_i^*\big]\\
  &\le |\wt Q| \sum_{i}\tau\big[a_i\,\|f_i\|_{L_\8(\N)}^2\, a_i^*\big]\\
  &\les |Q| \, \|f\|^2_{L_\8(\N; \el_2^c)}\sum_{i}\tau(|a_i|^2)\\
  &\le |Q|\,  \|f\|^2_{L_\8(\N; \el_2^c)}\,.
  \end{split}\ee
Combining the above estimates, we get the desired estimate of $A'$:
 $$\|A'\|_{B(\el_2)\overline\ot\M}\les \|f\|^2_{L_\8(\N; \el_2^c)}\,.$$
Thus, $\wt{\mathsf{k}}$ is bounded from  $L_\8(\N; \el_2^c)$  into $\BMO^r(\T^d, B(\el_2)\overline\ot\M)$, so is it  from  $L_\8(\N; \el_2^c)$ into $\BMO(\T^d,B(\el_2)\overline\ot\M)$.

It is clear that $\wt{\mathsf{k}}$ is bounded from $L_2(\N; \el_2^c)$ into $L_2(B(\el_2)\overline\ot\N)$. Hence, by interpolation via \eqref{column interpolation} and  Lemma~\ref{H-BMO},  $\wt{\mathsf{k}}$ is bounded from $L_p(\N; \el_2^c)$ into $L_p(B(\el_2)\overline\ot\N)$ for any $2<p<\8$. This is the announced assertion  for $2\le p<\8$. The case $1<p<2$ is obtained by duality.
 \end{proof}

\begin{rk}
 In the commutative case, i.e., $\M=\com$, it is well known that the conclusion of the preceding lemma holds under the following weaker assumption on $\phi$:
  \beq\label{DS condition per}
 \sup_{k\ge0} \Big(\int_{\real^d}(1+|s|^2)^\s\big\|\F^{-1}(\phi(2^k\cdot)\f)(s)\big\|_{\el_\8}^2ds\Big)^{\frac12}<\8.
  \eeq
Like at the beginning of the preceding proof, this assumption can be strengthened to
  $$
  \sup_{k\in\ent} \Big(\int_{\real^d}(1+|s|^2)^\s\big\|\F^{-1}(\phi(2^k\cdot)\f)(s)\big\|_{\el_\8}^2ds\Big)^{\frac12}<\8.
  $$
 Then if we consider $\mathsf{k}=(\mathsf{k}_j)_{j\ge0}$ as a kernel with values in $\el_\8$, Lemma~\ref{HS} admits the following $\el_\8$-analogue:
 \begin{enumerate}[$\bullet$]
 \item $\displaystyle \big\|\wh{\mathsf{k}}\big\|_{L_\8(\real^d; \el_\8)}<\8$;
 \item $\displaystyle \sup_{t\in\real^d}\int_{|s|>2|t|}\|\mathsf{k}(s-t)-\mathsf{k}(s)\|_{\el_\8}ds<\8$.
 \end{enumerate}
Transferring this to the periodic case, we have
  \begin{enumerate}[$\bullet$]
 \item $\displaystyle\big\|\wh{\wt{\mathsf{k}}}\big\|_{\el_\8(\ent^d; \el_\8)}<\8$;
 \item $\displaystyle  \sup_{t\in\mathbb{I}^d}\int_{\{s\in\mathbb{I}^d: |s|>2|t|\}}\|\wt{\mathsf{k}}(s-t)-\wt{\mathsf{k}}(s)\|_{\el_\8}ds<\8$.
 \end{enumerate}
The last two properties of the kernel $\wt{\mathsf{k}}$ are exactly what is needed for the estimates of $A$ and $B$ in the proof of Lemma~\ref{CZD}, so the conclusion holds when $\M=\com$. However, we do not know whether Lemma~\ref{CZD} remains true when \eqref{HS condition per} is weakened to \eqref{DS condition per}.
\end{rk}

\begin{lem}\label{CZH}
 Let $\phi=(\phi_j)_{j\ge0}$ be a sequence of continuous functions on $\real^d\setminus\{0\}$ satisfying \eqref{HS condition per}.
Then for  $1\le p\le2$ and any $f\in \H_p^c(\T^d, \M)$,
 $$\big\|\big(\sum_{j\ge0}|\wt\phi_j*f|^2\big)^{\frac12}\big\|_{L_p(\N)}\les\|\phi\|_{h_2^\s}\, \|f\|_{\H_p^c}\,.$$
The relevant constant depends only on $\f$, $\s$ and $d$.
 \end{lem}

\begin{proof}
Like in the  proof of Lemma~\ref{CZD}, we can assume, without loss of generality, that $\phi$ satisfies \eqref{HS condition}. We use again the Calder\'on-Zygmund theory. Now we view $\wt{\mathsf{k}}=(\wt{\mathsf{k}}_j)_{j\ge0}$ as a column matrix and the  associated Calder\'on-Zygmund operator $\wt{\mathsf{k}}$ as defined on
$L_p(\N)$:
 $$\wt{\mathsf{k}}(f)(s)=\int_{\mathbb{I}^d} \wt{\mathsf{k}}(s-t) f(t)dt.$$
Thus $\wt{\mathsf{k}}$ maps functions to sequences of functions. We have to show that $\wt{\mathsf{k}}$ is bounded from $\H_p^c(\T^d, \M)$ to $L_p(\N; \el_2^c)$ for $1\le p\le2$. This is trivial for $p=2$. So by Lemma~\ref{q-H-BMO} via interpolation, it suffices to consider the case $p=1$. The argument below  is based on the atomic decomposition of $\H_1^c(\T^d, \M)$ obtained in  \cite{CXY2012} (see also \cite{Mei2007}). Recall that an $\M^c$-atom  is a function $a\in L_1(\M; L^c_2(\T^d))$ such that
 \begin{enumerate}[$\bullet$]
 \item $a$ is supported by a cube $Q\subset \T^d\approx \mathbb{I}^d$;
 \item $\int_Q a(s) ds =0$;
 \item $\tau\big[\big(\int_Q|a(s)|^2ds\big)^{\frac12}\big] \leq |Q|^{-\frac12}$.
 \end{enumerate}
Thus we need only to show that for any atom $a$
 $$\|\wt{\mathsf{k}}(a)\|_{L_1(\N; \el_2^c)}\les 1.$$
Let $Q$ be the supporting cube of $a$. By translation invariance of the operator $\wt{\mathsf{k}}$, we can assume that $Q$ is centered at the origin.  Set $\wt Q=2Q$ as before. Then
 \beq\label{1-norm of a}
 \|\wt{\mathsf{k}}(a)\|_{L_1(\N; \el_2^c)}\le \|\wt{\mathsf{k}}(a)\un_{\wt Q}\|_{L_1(\N; \el_2^c)}
 +\|\wt{\mathsf{k}}(a)\un_{\mathbb{I}^d\setminus\wt Q}\|_{L_1(\N; \el_2^c)}\,.
 \eeq
 The operator convexity of the square function $x\mapsto |x|^2$ implies
  $$\int_{\wt Q}|\mathsf{k}(a)(s)|ds\le |\wt Q|^{\frac12}\Big(\int_{\wt Q}|\wt{\mathsf{k}}(a)(s)|^2ds\Big)^{\frac12}.$$
However, by the Plancherel formula,
 \be\begin{split}
 \int_{\wt Q}|\wt{\mathsf{k}}(a)(s)|^2ds
 &\le \int_{\mathbb{I}^d}|\wt{\mathsf{k}}(a)(s)|^2ds
  =\sum_{m\in\ent^d}|\wh{\wt{\mathsf{k}}(a)}(m)|^2
 =\sum_{m\in\ent^d} |\wh{\wt{\mathsf{k}}}(m)\wh a(m)|^2\\
&\le \big\|\wh{\wt{\mathsf{k}}}\big\|_{\el_\8(\ent^d; \el_2)} \sum_{m\in\ent^d} |\wh a(m)|^2
=\big\|\wh{\wt{\mathsf{k}}}\big\|_{\el_\8(\ent^d; \el_2)} \int_{Q}|a(s)|^2ds\,.
\end{split}\ee
Therefore,  by  \eqref{HS per}
 $$ \|\wt{\mathsf{k}}(a)\un_{\wt Q}\|_{L_1(\N; \el_2^c)}=\tau \int_{\wt Q}|\wt{\mathsf{k}}(a)(s)|ds
 \les |\wt Q|^{\frac12}\tau\big[\big(\int_Q|a(s)|^2ds\big)^{\frac12}\big]\les 1.$$
This is the desired estimate of  the first term of the right-hand side of \eqref{1-norm of a}. For the second, since $a$ is of vanishing mean, for every $s\not\in\wt Q$ we can write
  $$\wt{\mathsf{k}}(a)(s)=\int_{Q}[\,\wt{\mathsf{k}}(s-t)-\wt{\mathsf{k}}(s)]a(t)dt.$$
Then by the Cauchy-Schwarz inequality via the operator convexity of the square function $x\mapsto |x|^2$, we have
 $$|\wt{\mathsf{k}}(a)(s)|^2
 \le \int_{Q}\|\wt{\mathsf{k}}(s-t)-\wt{\mathsf{k}}(s)\|_{\el_2} dt \,\cdot\,\int_{Q}\|\wt{\mathsf{k}}(s-t)-\wt{\mathsf{k}}(s)\|_{\el_2}\,|a(t)|^2dt.$$
Thus by  \eqref{HS per},
  \be\begin{split}
 &\|\wt{\mathsf{k}}(a)\un_{\mathbb{I}^d\setminus\wt Q}\|_{L_1(\N; \el_2^c)}
 = \tau\int_{\mathbb{I}^d\setminus\wt Q}|\wt{\mathsf{k}}(a)(s)|ds\\
 &\les \tau\Big[\big( \int_{Q}\int_{\mathbb{I}^d\setminus\wt Q}\|\wt{\mathsf{k}}(s-t)-\wt{\mathsf{k}}(s)\|_{\el_2} ds\,dt\big)^{\frac12}\cdot
 \big( \int_{Q}\int_{\mathbb{I}^d\setminus\wt Q}\|\wt{\mathsf{k}}(s-t)-\wt{\mathsf{k}}(s)\|_{\el_2}\,|a(t)|^2ds\,dt\big)^{\frac12}\Big] \\
 &\les |Q|^{1/2}\,\tau\big[\big(\int_Q|a(s)|^2ds\big)^{\frac12}\big]\les 1.
 \end{split}\ee
Hence  the desired assertion is proved.
 \end{proof}

By transference, the previous  lemmas imply the following. According to our convention used in the previous chapters, the map  $x\mapsto\wt\phi*x$ denotes the Fourier multiplier associated to $\phi$ on $\T^d_\t$.

\begin{lem}\label{multiplier DH}
 Let $\phi=(\phi_j)_j$ satisfy \eqref{HS condition per}.
   \begin{enumerate}[\rm(i)]
  \item For $1<p<\8$ we have
 $$\big\|\big(\sum_{j\ge0} |\wt\phi_j*x_j|^2\big)^{\frac12}\big\|_p\les\|\phi\|_{h_2^\s}\,
   \big\|\big(\sum_{j\ge0} |x_j|^2\big)^{\frac12}\big\|_p\,,\quad x_j\in L_p(\T^d_\t)$$
with relevant constant depending only on $p$, $\f$, $\s$ and $d$.
  \item For $1\le p\le2$ we have
 $$\big\|\big(\sum_{j\ge0}|\wt\phi_j* x|^2\big)^{\frac12}\big\|_{p}
 \les \|\phi\|_{h_2^\s}\,\|x\|_{\H_p^c}\,,\quad x\in \H_p^c(\T^d_\t)$$
 with relevant constant depending only on $\f$, $\s$ and $d$.
 \end{enumerate}
  \end{lem}

We are now ready to prove Theorem~\ref{Hormander}.

\begin{proof}[Proof of  Theorem~\ref{Hormander}]
 Let $\zeta_j=\phi_j (\f^{(j-1)}+\f^{(j)}+\f^{(j+1)})$. By \eqref{3-supports} and the support assumption on $\phi_j\rho_j$, we have
  $$\phi_j \rho_{j}=\zeta_j \rho_{j}\,,\;\text{ so }\; \wt\phi_j*\wt\rho_j*x=\wt\zeta_j*\wt\rho_j*x$$
 for any distribution $x$ on $\T^d_\t$. We claim that $\zeta=(\zeta_j)_{j\ge0}$ satisfies \eqref{HS condition per} in place of $\phi$. Indeed, given $k\in\nat_0$, by the support assumption on $\f$ in \eqref{LP dec}, the sequence $\zeta(2^k\cdot)\f=(\zeta_j(2^k\cdot)\f)_{j\ge0}$ has at most five nonzero terms of indices $j$ such that $k-2\le j\le k+2$. Thus
  $$\big\|\zeta(2^k\cdot)\f\big\|_{H_2^\s(\real^d;\el_2)}\le \sum_{j=k-2}^{k+2} \big\|\zeta_j(2^k\cdot)\f\big\|_{H_2^\s(\real^d)}\,.$$
 However, by Lemma~\ref{CS Sobolev},
   $$\big\|\zeta_j(2^k\cdot)\f\big\|_{H_2^\s(\real^d)}\les  \sup_{-2\leq k \leq 2} \big\|   {\phi} _j (2^{j+k}\cdot)\varphi  \big\| _{H_2^\sigma(\mathbb{R}^d)}
\,,\quad k-2\le j\le k+2,$$
 where the relevant constant depends only on $d, \s$ and $\f$. Therefore, the second condition of \eqref{psi-rho} yields the claim.

Now applying Lemma~\ref{multiplier DH} (i) with $\zeta_j$ instead of $\phi_j$ and $x_j=2^{j\a}\wt\f_j*x$, we prove part (i) of the theorem.

To show part (ii), we need the characterization of $\H_1^c(\T^d_\t)$ by discrete square functions stated in Lemma~\ref{Hp-discrete} with  $\p=I_{-\a}\rho$.
 Let  $x$ be a distribution on $\T^d_\t$ with $\wh x(0)=0$ such that
  $$\big\|\big(\sum_{j\ge0}2^{j\a}|\wt\rho_j*x|^2\big)^{\frac12}\big\|_1<\8.$$
 Let $y=I^\a(x)$. Then the discrete square function of  $y$  associated to $\psi$  is given by
 $$s_\p^{c}(y)^2=\sum_{j\ge0}|\wt\p_j*y|^2=\sum_{j\ge0}2^{j\a}|\wt\rho_j*x|^2\,.$$
So $y\in \H_1^c(\T^d_\t)$ and
 $$\|y\|_{ \H_1^c}\approx \big\|\big(\sum_{j\ge0}2^{j\a}|\wt\rho_j*x|^2\big)^{\frac12}\big\|_1\,.$$
We want to apply Lemma~\ref{multiplier DH} (ii) to $y$ but with a different multiplier in place of $\phi$. To that end, let $\eta_j=2^{j\a}I_{-\a}\phi_j$ and $\eta=(\eta_j)_{j\ge0}$. We claim  that $\eta$ satisfies \eqref{psi-rho} too. The support condition of \eqref{psi-rho} is obvious for $\eta$. To prove the second one, by \eqref{3-supports}, we write
 $$\eta_j(2^j\xi)\f(\xi)=|\xi|^{-\a}\f(\xi)\phi_j(2^j\xi)=|\xi|^{-\a}[\f(2^{-1}\xi)+\f(\xi)+\f(2\xi)]\f(\xi)\phi_j(2^j\xi).$$
Since $I_{-\a}(\f^{(-1)}+\f+ \f^{(1)})$ is an infinitely differentiable function with compact support,
 $$\int_{\real^d} (1+|s|^2)^\s\big|\F^{-1}(I_{-\a}(\f^{(-1)}+\f+ \f^{(1)}))(s)\big|ds<\8.$$
 Thus by Lemma~\ref{CS Sobolev},
 $$\big\|\eta_j(2^{j+k}\cdot)\f\big\|_{H^s_2(\real^d)}\les \big\|\phi_j(2^{j+k}\cdot)\f\big\|_{H^s_2(\real^d)}\,,$$
whence the claim.

As in the first part of the proof, we define a new sequence $\zeta$ by setting $\zeta_j= \eta_j\rho_j$. Then the new sequence $\zeta$ satisfies \eqref{HS condition per} too and
 $$\sup_{k\ge0}\big\|\zeta_j(2^k\cdot)\f\big\|_{H_2^\s(\real^d)}\les \underset{\substack{j\geq 0 \\ -2\leq k \leq 2}}{\sup} \big\|   {\eta} _j (2^{j+k}\cdot)\varphi \big\| _{H_2^\sigma(\mathbb{R}^d)}
 \les \underset{\substack{j\geq 0 \\ -2\leq k \leq 2}}{\sup} \big\|   {\phi} _j (2^{j+k}\cdot)\varphi \big\| _{H_2^\sigma(\mathbb{R}^d)}\,.$$
 On the other hand, we have
  $$2^{j\a}\wt\phi_j* \wt\rho_j*x=\wt\zeta_j* y\,.$$
 Thus we can apply Lemma~\ref{multiplier DH} (ii)  to $y$ with this new $\zeta$ instead of $\phi$, and as before,  we get
  \be\begin{split}
  \big\|\big(\sum_{j\ge0}2^{2j\a}|\wt\phi_j* \wt\rho_j*x|^2\big)^{\frac12}\big\|_{1}
  &= \big\|\big(\sum_{j\ge0}|\wt\zeta_j* y|^2\big)^{\frac12}\big\|_{1}\\
 &\les \sup_{k\ge0}\big\|\zeta(2^k\cdot)\f\big\|_{H_2^\s(\real^d; \el_2)} \, \|y\|_{\H_p^c}\\
 &\les \underset{\substack{j\geq 0 \\ -2\leq k \leq 2}}{\sup} \big\|   {\phi} _j (2^{j+k}\cdot)\varphi \big\| _{H_2^\sigma(\mathbb{R}^d)} \,\big\|\big(\sum_{j\ge0}2^{j\a}|\wt\rho_j*x|^2\big)^{\frac12}\big\|_1\,.
 \end{split}\ee
Hence the proof  of the theorem is complete.
 \end{proof}

 %%%%%%%%%%%%%%%%%%%%%%%%%%%%%%%%%%%%%%%%%%%%%%%%%%%%%%%%%%%%%%%%%%%%%%%%
%%%%%%%%%%%%%%%%%%%%%%%%%%%%%%%%%%%%%%%%%%%%%%%%%%%%%%%%%%%%%%%%%%%%%%%%

\section{Definitions and basic properties}
\label{Definitions and basic properties: Triebel}

%%%%%%%%%%%%%%%%%%%%%%%%%%%%%%%%%%%%%%%%%%%%%%%%%%%%%%%%%%%%%%%%%%%%%%%%
%%%%%%%%%%%%%%%%%%%%%%%%%%%%%%%%%%%%%%%%%%%%%%%%%%%%%%%%%%%%%%%%%%%%%%%%

 As said at the beginning of this chapter, we consider the Triebel-Lizorkin spaces on $\T^d_\t$ only for $q=2$. In this case, there exist three different families of spaces according to the three choices of the internal $\el_2$-norms.

\begin{Def}\label{q-Triebel}
Let $1\leq p<\8$ and $\a\in \real$.
\begin{enumerate}[(i)]
\item The  column Triebel-Lizorkin space $F^{\a, c}_{p} (\T^d_\t)$ is defined by
$$F^{\a, c}_{p} (\T^d_\t)=\big\{x\in \mathcal{S}'(\T^d_\t) : \|x\|_{F^{\a, c}_{p}} < \8 \big\},$$
where
 $$\|x\|_{F^{\a, c}_{p}} =|\wh x(0)| + \big\|\big(\sum_{k\ge 0} 2^{2k\a } | \wt\f_k * x|^2\big)^{\frac{1}{2}}\big\|_p\,.$$
\item The row space $F^{\a, r}_{p} (\T^d_\t)$ consists  of all $x$ such that $x^*\in F^{\a, c}_{p} (\T^d_\t)$, equipped with the norm $\|x\|_{F^{\a, r}_{p}} =\|x^*\|_{F^{\a, c}_{p}} $.
\item The mixture space $F^{\a}_{p} (\T^d_\t)$ is defined to be
  \be
  F^{\a}_{p} (\T^d_\t)=
 \left \{ \begin{split}
 &F^{\a, c}_{p} (\T^d_\t)+ F^{\a, r}_{p} (\T^d_\t) & \textrm{ if }\; 1\le p<2,\\
 &F^{\a, c}_{p} (\T^d_\t)\cap F^{\a, r}_{p} (\T^d_\t) & \textrm{ if }\; 2\le  p<\8,
 \end{split} \right.
 \ee
equipped with
  \be
  \|x\|_{F^{\a}_p} =
 \left \{ \begin{split}
 &\inf\big\{\|y\|_{F^{\a, c}_p} +\|z\|_{F^{\a, r}_p} : x=y+z \big\} & \textrm{ if }\; 1\le p<2,\\
 &\max(\|x\|_{F^{\a, c}_p}\,, \;\|x\|_{F^{\a, r}_p})   & \textrm{ if }\; 2\le  p<\8.
 \end{split} \right.
 \ee
 \end{enumerate}
\end{Def}

 In the sequel, we will concentrate our study only on the column Triebel-Lizorkin spaces. All results will admit the row and mixture analogues. The following shows that $F^{\a, c}_{p} (\T^d_\t)$ is independent of the choice of the function $\f$.

\begin{prop}\label{Triebel-I}
Let $\p$ be a  Schwartz function satisfying the same condition \eqref{LP dec} as $\f$. Let $\wh\p_k=\p^{(k)}=\p(2^{-k}\cdot)$.
  Then
  $$\|x\|_{F^{\a, c}_p}\approx |\wh x(0)| +  \big\|\big(\sum_{k\ge 0} 2^{2k\a} |\wt\psi_k * x|^2\big)^{\frac12}\big\|_p\,.$$
\end{prop}

\begin{proof}
 Fix a distribution $x$ on $\T^d_\t$ with $\wh x(0)=0$. By the support assumption on $\p^{(k)}$ and  \eqref{3-supports}, we have (with $\wt\f_{-1}=0$)
 $$\wt\psi_k*x=\sum_{j=-1}^{1}\wt\psi_k*\wt\f_{k+j}*x.$$
Thus by Theorem~\ref{Hormander},
  \be\begin{split}
  \big\|\big(\sum_{k\ge 0} 2^{2k\a} | \wt\psi_k * x|^2\big)^{\frac12}\big\|_p
 &\le \sum_{j=-1}^{1} \big\|\big(\sum_{k\ge 0} 2^{2k\a} | \wt\psi_k *\wt\f_{k+j}* x|^2\big)^{\frac12}\big\|_p\\
 &\les\big\| \big(\sum_{k\ge 0} 2^{2k\a} | \wt\f_k * x|^2\big)^{\frac12}\big\|_p\,.
 \end{split}\ee
Changing the role of $\f$ and $\p$, we get the reverse inequality.
\end{proof}

\begin{prop}\label{Triebel-P}
Let $1\leq p<\8$ and $\a\in \mathbb{R}$.
\begin{enumerate}[\rm(i)]
\item $F_{p}^{\a, c} (\T^d_\t)$ is a Banach space.
\item $F_{p}^{\a, c} (\T^d_\t)\subset F_{p}^{\b, c} (\T^d_\t)$  for $\b<\a$.
\item $\mathcal{P}_{\theta}$ is dense in $F_{p}^{\a, c} (\T^d_\t)$ .
\item $F_{p}^{0, c} (\T^d_\t)=\H_{p}^{c} (\T^d_\t)$.
\item $B_{p,\, \min(p, 2)}^{\a} (\T^d_\t)\subset F_{p}^{\a, c} (\T^d_\t)\subset B_{p,\, \max(p, 2)}^{\a} (\T^d_\t)$.
\end{enumerate}
\end{prop}

\begin{proof}
 (i) is proved as in the case of Besov spaces; see the corresponding proof of Proposition~\ref{Besov-P}. (ii) is obvious. To show (iii), we use the Fej\'er means as in the proof of Proposition~\ref{Sobolev-P}. We need one more property of those means, that is, they are completely contractive. So they are also contractive on $L_p(B(\el_2)\overline\ot\T^d_\t)$, in particular, on the column subspace $L_p(\T^d_\t; \el_2^c)$ too. We then deduce that $F_N$ is contractive on $F_{p}^{\a, c} (\T^d_\t)$ and $\lim_{N\to\8}F_N(x)=x$ for every $x\in F_{p}^{\a, c} (\T^d_\t)$.

(iv) has been already observed during the proof of  Theorem~\ref{Hormander}. Indeed, for any distribution $x$ on $\T^d_\t$, the square function associated to $\f$ defined in  Lemma~\ref{Hp-discrete} is given by
 $$s_\f^c(x)=\big(\sum_{k\ge 0}  | \wt\f_k * x|^2\big)^{\frac{1}{2}}\,.$$
Thus $\|x\|_{\H_p^c}\approx \|x\|_{F_{p}^{0, c} }$.

(v) follows from the following well-known property:
 $$\el_2(L_p(\T^d_\t))\subset L_p(\T^d_\t;\el_2^c)\subset \el_p(L_p(\T^d_\t))$$
 are contractive inclusions for $2\le p\le\8$; both inclusions are reversed for $1\le p\le2$. Note that the first inclusion is an immediate consequence of the triangular inequality of $L_{\frac p2}(\T^d_\t)$, the second is proved by complex interpolation.
  \end{proof}

The following is the Triebel-Lizorkin analogue of Theorem ~\ref{Besov-isom}. We keep the notation introduced before that theorem.

\begin{thm}\label{Triebel-isom}
Let $1\leq p< \8$ and $\a\in\real$.
 \begin{enumerate}[\rm(i)]
 \item For any $\b\in\real$, both $J^{\b}$ and $I^\b$ are isomorphisms between $F_{p}^{\a, c} (\T^d_\t)$ and $F_{p}^{\a -\b, c}(\T^d_\t)$. In particular,  $J^{\a}$ and $I^\a$ are isomorphisms between $F_{p}^{\a, c} (\T^d_\t)$ and $\H_p^c(\T^d_\t)$.
 \item Let $a\in\real_+^d$. If $x\in F_{p}^{\a, c} (\T^d_\t)$, then $D^ax\in F_{p}^{\a-|a|_1, c} (\T^d_\t)$ and
   $$\|D^ax\|_{F_{p}^{\a-|a|_1, c}}\les \|x\|_{F_{p}^{\a, c}}\,.$$
 \item Let $\b>0$. Then $x\in F_{p}^{\a, c} (\T^d_\t)$ iff $D_i^\b x\in F_{p}^{\a-\b, c} (\T^d_\t)$ for all $i=1,\cdots,  d$. Moreover, in this case,
  $$\|x\|_{F_{p}^{\a, c}}\approx |\wh x(0)|+\sum_{i=1}^d \|D_i^\b x\|_{F_{p}^{\a-\b, c}}\,.$$
  \end{enumerate}
\end{thm}

\begin{proof}
  (i)  Let $x \in F_{p}^{\a, c} (\T^d_\t)$ with $\wh x(0)=0$.  By Theorem~\ref{Hormander},
 \be\begin{split}
 \|J^{\b}x\|_{F_{p}^{\a -\b, c}}
 &= \big\|\big(\sum_{k\ge0} 2^{2k(\a -\b)} |J^{\b}*\wt\f_k\ast x|^2\big)^{\frac12}\big\|_p\\
 &\les\sup_{k\ge0} 2^{-k\b}\| J_\b(2^k\cdot)\f\|_{H_2^\s(\real^d)}\,\big\|\big(\sum_{k\ge0} 2^{2k\a} |\wt\f_k\ast x|^2\big)^{\frac12}\big\|_p\,.
 \end{split}\ee
However, it is easy to see that all partial derivatives of the function $2^{-k\b}J_\b(2^k\cdot)\f$, of order less than a fixed integer, are bounded uniformly in $k$. It then follows that
 $$\sup_{k\ge0} 2^{-k\b}\| J_\b(2^k\cdot)\f\|_{H_2^\s(\real^d)}<\8.$$
Thus $\|J^{\b}x\|_{F_{p}^{\a -\b, c}}\les  \|x\|_{F_{p}^{\a , c}}$. So $J^{\b}$ is bounded from $F_{p}^{\a, c} (\T^d_\t)$ to $F_{p}^{\a -\b, c}(\T^d_\t)$, its inverse, which is $J^{-\b}$, is bounded too. $I^\b$ is handled similarly.

  If $\b=\a$, then $F_{p}^{\a -\b, c}(\T^d_\t)=F_{p}^{0, c}(\T^d_\t)=\H_p^c(\T^d_\t)$ by Proposition~\ref{Triebel-P} (iv).

(ii) This proof is similar to the previous one  by replacing $J^\b$ by $D^a$ and using Lemma~\ref{weighted Bessel}.

(iii) One implication is contained in (ii). To show the other, we follow the proof of Theorem~\ref{Besov-isom} (iii) and keep the notation there.   Since
  $$\f_k= \sum_{i=1}^d\chi_i D_{i,\b}\f_k,$$
 by Theorem~\ref{Hormander},
 \be\begin{split}
 \|x\|_{F_{p}^{\a , c}}
 &\le \sum_{i=1}^d \big\|\big(\sum_{k\ge0} 2^{2k\a} |\wt\chi_i *\wt\f_k* D_i^\b x|^2\big)^{\frac12}\big\|_p\\
 &\les \sum_{i=1}^d \sup_{k\ge0} 2^{k\b}\| \chi_i(2^k\cdot)\f\|_{H_2^\s(\real^d)}\, \big\|\big(\sum_{k\ge0} 2^{2k(\a-\b)} |\wt\f_k* D_i^\b x|^2\big)^{\frac12}\big\|_p\,.
 \end{split}\ee
However,
 $$2^{k\b}\| \chi_i(2^k\cdot)\f\|_{H_2^\s(\real^d)}
 =\| \p\f\|_{H_2^\s(\real^d)}\,,$$
where
 $$\p(\xi)=\frac1{\chi(2^k\xi_1)|\xi_1|^\b+\cdots+\chi(2^k\xi_d)|\xi_d|^\b}\,\frac{\chi(2^k\xi_i)|\xi_i|^\b}{(2\pi{\rm i}\xi_i)^\b}\,.$$
As all partial derivatives of $\p\f$, of order less than a fixed integer, are bounded uniformly in $k$, the norm of $\p\f$ in $H_2^\s(\real^d)$ are controlled by a constant independent of $k$. We then deduce
 $$\|x\|_{F_{p}^{\a , c}}\les \sum_{i=1}^d \big\|\big(\sum_{k\ge0} 2^{2k(\a-\b)} |\wt\f_k* D_i^\b x|^2\big)^{\frac12}\big\|_p
 =\sum_{i=1}^d \|D_i^\b x\|_{F_{p}^{\a-\b, c}}\,.$$
The theorem is thus completely proved. \end{proof}

\begin{cor}\label{Triebel=Sobolev}
 Let $1<p<\8$ and $\a\in\real$. Then $F_{p}^{\a} (\T^d_\t)=H_{p}^{\a} (\T^d_\t)$ with equivalent norms.
 \end{cor}

\begin{proof}
 Since $J^\a$  is an isomorphism from $F_{p}^{\a} (\T^d_\t)$ onto $F_{p}^{0} (\T^d_\t)$, and from  $H_{p}^{\a} (\T^d_\t)$ onto $H_{p}^{0} (\T^d_\t)$, it suffices to consider the case $\a=0$. But then $H_{p}^{0} (\T^d_\t)=L_p (\T^d_\t)$ by definition, and $F_{p}^{0} (\T^d_\t)= \H_p(\T^d_\t)$ by Proposition~\ref{Triebel-P}. It remains to apply Lemma~\ref{q-H-BMO} to conclude $F_{p}^{0} (\T^d_\t)=H_{p}^{0} (\T^d_\t)$.
\end{proof}

We now discuss the duality of $F_{p}^{\a, c} (\T^d_\t)$. For this we need to define $F_{\8}^{\a, c} (\T^d_\t)$ that is excluded from the definition at the beginning of the present section. Let $\el_2^\a$ denote the Hilbert space of all complex sequences $a=(a_k)_{k\ge0}$ such that
 $$\|a\|=\big(\sum_{k\ge0}2^{2k\a}|a_k|^2\big)^{\frac12}<\8.$$
Thus $L_p(\T^d_\t; \el_2^{\a, c})$ is the column subspace of $L_p(B(\el_2^{\a})\overline\ot \T^d_\t)$.

\begin{Def}\label{F infity}
 For $\a\in\real$ we define  $F_{\8}^{\a, c} (\T^d_\t)$ as the space of all distributions $x$ on $\T^d_\t$ that admit a representation of the form
  $$x=\sum_{k\ge0} \wt\f_k*x_k\;\text{ with }\; (x_k)_{k\ge0} \in L_\8(\T^d_\t; \el_2^{\a, c}),$$
 and endow it with the norm
  $$\|x\|_{F_{\8}^{\a, c} }=|\wh x(0)|+ \inf\big\{ \big\|\big(\sum_{k\ge0} 2^{2k\a} |\wt\f_k\ast x_k|^2\big)^{\frac12}\big\|_\8\big\},$$
 where the infimum runs over all representations of $x$ as above.
 \end{Def}

 \begin{prop}\label{Triebel-dual}
 Let $1\le p<\8$ and $\a\in\real$. Then the dual space of $F_{p}^{\a, c} (\T^d_\t)$ coincides isomorphically with   $F_{p'}^{-\a, c} (\T^d_\t)$.
 \end{prop}

 \begin{proof}
  For simplicity, we will consider only  distributions with vanishing Fourier coefficients at $m=0$. We view $F_{p}^{\a, c} (\T^d_\t)$  as an isometric  subspace of $ L_p(\T^d_\t; \el_2^{\a, c})$ via $x\mapsto (\wt\f_k*x)_{k\ge0}$. Then the dual space of $F_{p}^{\a, c} (\T^d_\t)$ is identified with the following quotient of the latter:
  $$G_{p'}=\big\{y=\sum_{k\ge0} \wt\f_k*y_k:  (y_k)_{k\ge0} \in L_{p'}(\T^d_\t; \el_2^{-\a, c})\big\},$$
 equipped with the quotient norm
  $$\|y\|=\inf\big\{\big\| (y_k)\big\|_{L_{p'}(\T^d_\t; \el_2^{-\a, c})}\,:\, y=\sum_{k\ge0} \wt\f_k*y_k\big\}.$$
 The duality bracket is given by $\la x, y\ra=\tau(xy^*)$. If $p=1$, then  $G_{p'}=F_{\8}^{-\a, c} (\T^d_\t)$ by definition. It remains to show that $G_{p'}=F_{p'}^{-\a, c} (\T^d_\t)$ for $1<p<\8$. It is clear that $F_{p'}^{-\a, c} (\T^d_\t)\subset G_{p'}$, a contractive inclusion. Conversely, let $y\in G_{p'}$ and $y=\sum \wt\f_k*y_k$ for some $ (y_k)_{k\ge0} \in L_{p'}(\T^d_\t; \el_2^{-\a, c})$. Then
  $$\wt\f_k*y=\wt\f_k*\wt\f_{k-1}*y_{k-1}+ \wt\f_k*\wt\f_{k}*y_{k}+\wt\f_k*\wt\f_{k+1}*y_{k+1}\,.$$
 Therefore, by Lemma~\ref{multiplier DH},
  \be\begin{split}
  \big\|\big(\sum_{k\ge0} 2^{2k\a} |\wt\f_k*y|^2\big)^{\frac12}\big\|_{p'}
  &\le \sum_{j=-1}^1\big\|\big(\sum_{k\ge0} 2^{-2k\a} |\wt\f_k*\wt\f_{k+j}*y_{k+j}|^2\big)^{\frac12}\big\|_{p'}\\
  &\les \big\|\big(\sum_{k\ge0} 2^{-2k\a} |y_{k}|^2\big)^{\frac12}\big\|_{p'}\,.
  \end{split}\ee
Thus $y\in F_{p'}^{-\a, c} (\T^d_\t)$ and $\|y\|_{F_{p'}^{-\a, c}}\les \|y\|_{G_{p'}}$.
 \end{proof}

 \begin{rk}\label{Triebel-BMO}
 (i) The above proof shows that $F_{p}^{\a, c} (\T^d_\t)$ is a complemented subspace of $ L_p(\T^d_\t; \el_2^{\a, c})$ for $1<p<\8$.

 (ii) By duality,  Propositions~\ref{Triebel-I}, \ref{Triebel-P}  and Theorem~\ref{Triebel-isom} remain valid for $p=\8$, except the density of $\mathcal P_\t$. In particular, $F_{\8}^{0, c} (\T^d_\t)=\BMO^c(\T^d_\t)$.
  \end{rk}

We conclude this section with the following  Fourier multiplier theorem, which is an immediate consequence of Theorem~\ref{Hormander} for $p<\8$. The case $p=\8$ is obtained by duality. In the case of $\a=0$, this result is to be compared with Lemma~\ref{q-multiplier} where more smoothness of $\phi$ is assumed.

\begin{thm}\label{Triebel-Hormander}
 Let $\phi$ be a continuous function on $\real^d\setminus\{0\}$ such that
  $$\sup_{k\ge0}\big\|\phi(2^k\cdot)\,\f\big\|_{H_2^\s(\real^d)}<\8$$
 for some $\s>\frac{d}2$. Then $\phi$ is a bounded Fourier multiplier on $F_{p}^{\a, c} (\T^d_\t)$ for all $1\le p\le\8$ and $\a\in\real$. In particular, $\phi$ is a bounded Fourier multiplier on  $\H_{p}^{c} (\T^d_\t)$  for $1\le p<\8$ and on  $\BMO^{ c} (\T^d_\t)$.
 \end{thm}

%%%%%%%%%%%%%%%%%%%%%%%%%%%%%%%%%%%%%%%%%%%%%%%%%%%%%%%%%%%%%%%%%%%%%%%%
%%%%%%%%%%%%%%%%%%%%%%%%%%%%%%%%%%%%%%%%%%%%%%%%%%%%%%%%%%%%%%%%%%%%%%%%

\section{A general characterization}

%%%%%%%%%%%%%%%%%%%%%%%%%%%%%%%%%%%%%%%%%%%%%%%%%%%%%%%%%%%%%%%%%%%%%%%%
%%%%%%%%%%%%%%%%%%%%%%%%%%%%%%%%%%%%%%%%%%%%%%%%%%%%%%%%%%%%%%%%%%%%%%%%

In this  section we give a general characterization of Triebel-Lizorkin spaces on $\T^d_\t$ in the same spirit as that given in section~\ref{A general characterization-Besov} for Besov spaces.

 Let $\a_0, \a_1, \s\in\real$ with $\s>\frac{d}2$. Let $h$ be a Schwartz function satisfying \eqref{hH}. Assume that $\p$ is an infinitely differentiable function on $\real^d\setminus\{0\}$ such that
   \beq\label{psi-T}
  \left \{ \begin{split}
  &\displaystyle |\p|>0 \;\text{ on }\; \{\xi: 2^{-1}\leq |\xi|\leq 2\},\\
  &\displaystyle \int_{\real^d}(1+|s|^2)^\s\big|\F^{-1}(\p hI_{-\a_1})(s)\big| ds<\8, \\
  &\displaystyle \sup_{k\in\nat_0}2^{-k\a_0}\big\|\F^{-1}(\p(2^{k}\cdot) \f)\big\|_{H^\s_2(\real^d)}<\8.
  \end{split} \right.
  \eeq
Writing $\f=\f(\f^{(-1)}+\f+\f^{(1)})$ and using Lemma~\ref{CS Sobolev}, we have
 $$\big\|\F^{-1}(\p(2^{k}\cdot) \f)\big\|_{H^\s_2(\real^d)}\les \int_{\real^d}(1+|s|^2)^\s\big|\F^{-1}(\p(2^{k}\cdot) \f)(s)\big| ds.$$
So the third condition of \eqref{psi-T} is weaker than the corresponding one assumed in \cite[Theorem~2.4.1]{HT1992}. On the other hand, consistent with Theorem~\ref{g-charct-Besov} but contrary to \cite[Theorem~2.4.1]{HT1992}, our following theorem does not require that $\a_1>0$.

\begin{thm}\label{g-charct-Triebel}
 Let $1\le p<\8$ and $\a\in\real$. Assume that $\a_0<\a<\a_1$ and  $\p$ satisfies \eqref{psi-T}. Then for any distribution
$x$ on $\T^d_\t$, we have
 \beq\label{phi-psi-Triebel}
 \|x\|_{F_{p}^{\a, c}}\approx |\wh x(0)|+
 \big\|\big(\sum_{k\ge0}2^{2k\a}|\wt\p_k*x|^2\big)^{\frac12}\big\|_p\,.
 \eeq
  The equivalence is understood in the sense that whenever one side is finite, so is the other, and the two are then equivalent with constants independent of $x$.
 \end{thm}

\begin{proof}
 Although it resembles, in form, the proof of Theorem~\ref{g-charct-Besov}, the proof given below is harder and subtler than  the  Besov space case. The key new ingredient is Theorem~\ref{Hormander}. The main differences will already appear in the first part of the proof, which is an adaptation  of step~1 of the  proof of Theorem~\ref{g-charct-Besov}.  In the following, we will fix $x$ with $\wh x(0)=0$. By approximation, we can assume that $x$ is a polynomial. We will denote the right-hand side of \eqref{phi-psi-Triebel} by $\|x\|_{F_{p, \p}^{\a,c}}$.

Given   a positive integer $K$, we write, as before
 $$\p^{(j)}=\sum_{k=0}^\8\p^{(j)}\f^{(k)}
 =\sum_{k=-\8}^K\p^{(j)}\f^{(j+k)}+ \sum_{k=K}^\8\p^{(j)}\f^{(j+k)}\,.$$
Then
 \beq\label{split-T}
 \|x\|_{F_{p,\p}^{\a,c}}
 \le \mathrm{I}+  \mathrm{II},
 \eeq
 where
  \be\begin{split}
   \mathrm{I}&=\sum_{k\le K}\big\|\big(\sum_j 2^{2j \a} |\wt\p_j *\wt\f_{j+k}*x|^2\big)^{\frac12}\big\|_p\,,\\
  \mathrm{II}&=\sum_{k>K}\big\|\big(\sum_j 2^{2j \a} |\wt\p_j *\wt\f_{j+k}*x|^2\big)^{\frac12}\big\|_p\,.
 \end{split}\ee
The estimate of the term I corresponds to step~1 of the proof of Theorem~\ref{g-charct-Besov}. We use again \eqref{h} with $\eta$ and $\rho$ defined there. Then applying Theorem~\ref{Hormander} twice, we have
 \be\begin{split}
 \mathrm{I}
 &=\sum_{k\le K}2^{k(\a_1-\a)}\big\|\big(\sum_j 2^{2(j+k) \a} |\wt\eta_{j} *\wt\rho_{j+k}*x|^2\big)^{\frac12}\big\|_p\\
 &=\sum_{k\le K}2^{k(\a_1-\a)}\big\|\big(\sum_j 2^{2j \a} |\wt\eta_{j-k} *\wt\rho_{j}*x|^2\big)^{\frac12}\big\|_p\\
 &\les \sum_{k\le K+2}2^{k(\a_1-\a)}\big\|\eta^{(-k)}\f\big\|_{H_2^\s}\, \big\|\big(\sum_j 2^{2j \a} |\wt\rho_{j}*x|^2\big)^{\frac12}\big\|_p\\
 &\les \big\|I_{\a_1}\f\big\|_{H_2^\s} \sum_{k\le K+2}2^{k(\a_1-\a)}\big\|\eta^{(-k)}\f\big\|_{H_2^\s}\, \big\|\big(\sum_j 2^{2j \a} |\wt\f_{j}*x|^2\big)^{\frac12}\big\|_p\\
 &= \big\|I_{\a_1}\f\big\|_{H_2^\s} \sum_{k\le K+2}2^{k(\a_1-\a)}\big\|\eta^{(-k)}\f\big\|_{H_2^\s}\, \big\|x\|_{F_{p}^{\a,c}} \,.
 \end{split}\ee
Being an infinitely differentiable function with compact support, $I_{\a_1}\f$ belongs to $H_2^\s(\real^d)$, that is, $\big\|I_{\a_1}\f\big\|_{H_2^\s}<\8$. Next, we must estimate $\big\|\eta^{(-k)}\f\big\|_{H_2^\s}$ uniformly in $k$. To that end,   for $s\in\real^d$, using
 \be\begin{split}
 \big|\F^{-1}(\eta^{(-k)}\f)(s)\big|^2
  &=  \Big|\int_{\real^d}\F^{-1}(\eta)(t)*\F^{-1}( \f)(s-2^kt)dt \Big|^2\\
  &\le\big\|\F^{-1}(\eta)\big\|_{1} \int_{\real^d}\big|\F^{-1}(\eta)(t)\big|\, \big|\F^{-1}( \f)(s-2^kt)\big|^2dt \,,
\end{split}\ee
for $k\le K+2$, we have
  \be\begin{split}
  \big\|\eta^{(-k)}\f\big\|_{H_2^\s}^2
  &=\int_{\real^d}(1+|s|^2)^\s  \big|\F^{-1}(\eta^{(-k)}\f)(s)\big|^2ds \\
  &\le \big\|\F^{-1}(\eta)\big\|_{1} \int_{\real^d}(1+|s|^2)^\s   \int_{\real^d}\big|\F^{-1}(\eta)(t)\big|\,\big|\F^{-1}( \f)(s-2^kt)\big|^2dt ds\\
  &\les \big\|\F^{-1}(\eta)\big\|_{1} \int_{\real^d}(1+|2^k t|^2)^\s  \big|\F^{-1}(\eta)(t)\big|
  \int_{\real^d}(1+|s-2^kt|^2)^\s\big|\F^{-1}( \f)(s-2^kt)\big|^2dsdt\\
  &\le 2^{K\s}\big\|\F^{-1}(\eta)\big\|_{1} \int_{\real^d}(1+|t|^2)^\s  \big|\F^{-1}(\eta)(t)\big| dt
   \int_{\real^d}(1+|s|^2)^\s\big|\F^{-1}( \f)(s)\big|^2ds\\
  &\le c_{\f, \s, K}  \Big(\int_{\real^d}(1+|t|^2)^\s \big|\F^{-1}(\eta)(t)\big| dt\Big)^2\,.
   \end{split}\ee
In order to return back from $\eta$ to $\p$,  write
   $$\eta=I_{-\a_1} \p h+ I_{-\a_1} \p (h^{(K)}-h).$$
 Note that
   \beq\label{diff psi-T}
    \int_{\real^d}(1+|t|^2)^\s \big|\F^{-1}(I_{-\a_1} \p (h^{(K)}-h))(t)\big| dt= c_{\p, h, \a_1, \s, K} <\8
    \eeq
 since $I_{-\a_1} \p (h^{(K)}-h)$ is an infinitely differentiable function with compact support. We then deduce
   $$\int_{\real^d}(1+|t|^2)^\s \big|\F^{-1}(\eta)(t)\big|dt
   \les\int_{\real^d}(1+|t|^2)^\s \big|\F^{-1}(I_{-\a_1} \p h)(t)\big| dt.$$
 The term on the right-hand side is the second condition of \eqref{psi-T}. Combining the preceding inequalities, we obtain
   $$\mathrm{I}\les\int_{\real^d}(1+|t|^2)^\s \big|\F^{-1}(I_{-\a_1} \p h)(t)\big| dt\,\|x\|_{F_{p}^{\a,c}}\,.$$

The second term II on the right-hand side of \eqref{split-T} is easier to estimate. Using \eqref{H}, Theorem~\ref{Hormander}  and arguing as in the preceding part for the term I,   we obtain
  \be\begin{split}
  \mathrm{II}
  \les  \big\|I_{\a_0}\f\big\|_{H_2^\s}\sum_{k>K-2}2^{-2k \a} \big\|I_{-\a_0}\p(2^k\cdot)H\f\big\|_{H_2^\s} \,\|x\|_{F_{p}^{\a,c}}\\
  \les \sum_{k>K-2}2^{-2k \a} \big\|I_{-\a_0}\p(2^k\cdot)H\f\big\|_{H_2^\s} \,\|x\|_{F_{p}^{\a,c}}\,,
   \end{split}\ee
 where $H=\f(2^{-1}\cdot)+\f+ \f(2\,\cdot)$. To treat the last Sobolev norm, noting that $I_{-\a_0}H$ is an infinitely differentiable function with compact support, by Lemma~\ref{CS Sobolev}, we have
  $$  \big\|I_{-\a_0}\p(2^k\cdot)H\f\big\|_{H_2^\s}\le \big\|\p(2^k\cdot) \f\big\|_{H_2^\s}  \int_{\real^d}(1+|t|^2)^\s\big|\F^{-1}(I_{-\a_0}H)(t)\big| dt
  \les \big\|\p(2^k\cdot) \f\big\|_{H_2^\s} \,.$$
 Therefore,
  \beq\label{II}
  \begin{split}
  \mathrm{II}
  &\les  \sup_{k> K-2} 2^{-k\a_0}\big\|\p(2^k\cdot) \f\big\|_{H_2^\s}\sum_{k>K-2}2^{2k (\a_0-\a)} \,\|x\|_{F_{p}^{\a,c}}\\
  &\le c\,\sup_{k>K-2} 2^{-k\a_0}\big\|\p(2^k\cdot) \f\big\|_{H_2^\s}\,\frac{2^{(\a_0-\a)K}}{1- 2^{\a_0-\a}}\,\|x\|_{F_{p}^{\a,c}}
  \end{split}\eeq
 with some constant $c$ independent of $K$. Putting  this estimate together with that of I, we finally get
 $$\|x\|_{F_{p, \p}^{\a,c}}\les \,\|x\|_{F_{p}^{\a,c}}\,.$$

Now we show the reverse inequality by following step~3 of the proof of Theorem~\ref{g-charct-Besov} (recalling that $\l=1-h$). By \eqref{hbis} and Theorem~\ref{Hormander},
  \be\begin{split}
  \|x\|_{F_{p}^{\a,c}}
  &\les \big\|\p^{-1} \f^2\big\|_{H_2^\s} \big\|\big(\sum_{j=0}^\82^{2j\a} |\wt h_{j+K}*\wt\p_j*x|^2\big)^{\frac12}\big\|_p\\
  &\les  \big\|\big(\sum_{j=0}^\82^{2j\a} |\wt h_{j+K}*\wt\p_j*x|^2\big)^{\frac12}\big\|_p\\
  &\le \|x\|_{F_{p,\p}^{\a,c}} +\big\|\big(\sum_{j=0}^\82^{2j\a} |\wt \l_{j+K}*\wt\p_j*x|^2\big)^{\frac12}\big\|_p \,.
  \end{split}\ee
 Then combining the arguments in step~3 of the proof of Theorem~\ref{g-charct-Besov} and \eqref{II} with $\l^{(K)}\psi$ in place of $\p$, we deduce
  $$\big\|\big(\sum_{j=0}^\82^{2j\a} |\wt \l_{j+K}*\wt\p_j*x|^2\big)^{\frac12}\big\|_p
  \le c\,\sup_{k>K-2} 2^{-k\a_0}\big\|\l(2^{k-K}\cdot) \p(2^k\cdot) \f\big\|_{H_2^\s}\,\frac{2^{(\a_0-\a)K}}{1- 2^{\a_0-\a}}\,\|x\|_{F_{p}^{\a,c}}\,.$$
To remove $\l(2^{k-K}\cdot)$ from the above Sobolev norm, by triangular inequality, we have
   $$\big\|\l(2^{k-K}\cdot) \p(2^k\cdot) \f\big\|_{H_2^\s}
   \le \big\| \p(2^k\cdot) \f\big\|_{H_2^\s} +\big\|h(2^{k-K}\cdot) \p(2^k\cdot) \f\big\|_{H_2^\s}\,.$$
By the support assumption on $h$ and $\f$, $h(2^{k-K}\cdot)\f\neq0$ only for $k\le K+2$, so the second term on the right hand side above matters only for $k=K+1$ and $k=K+2$. But for these two values of $k$, by Lemma~\ref{CS Sobolev}, we have
 $$\big\|h(2^{k-K}\cdot) \p(2^k\cdot) \f\big\|_{H_2^\s}\le c'\big\|\p(2^k\cdot) \f\big\|_{H_2^\s}\,,$$
where $c'$ depends only on $h$. Thus
 $$\big\|\l(2^{k-K}\cdot) \p(2^k\cdot) \f\big\|_{H_2^\s}
   \le (1+c')\big\| \p(2^k\cdot) \f\big\|_{H_2^\s} \,.$$
Putting together all estimates so far obtained, we deduce
 $$ \|x\|_{F_{p}^{\a,c}}
 \le \|x\|_{F_{p,\p}^{\a,c}} +c\,(1+c')\,\sup_{k>K-2} 2^{-k\a_0}\big\|\p(2^k\cdot) \f\big\|_{H_2^\s}\,\frac{2^{(\a_0-\a)K}}{1- 2^{\a_0-\a}}\,\|x\|_{F_{p}^{\a,c}}\,.$$
So if $K$ is chosen sufficiently large, we finally obtain
 $$ \|x\|_{F_{p}^{\a,c}} \les \|x\|_{F_{p,\p}^{\a,c}} \,,$$
which finishes the proof of the theorem.
 \end{proof}

\begin{rk}
  Note that we have used the  infinite differentiability of $\p$ only  to insure  \eqref{diff psi-T}, which  holds whenever $\p$ is continuously differentiable up to order $[\frac{3d}2]+1$. More generally, we need only to assume that there exists $\s>\frac{3d}2 +1$ such that $\p \eta\in H_2^\s(\real^d)$ for any compactly supported infinite differentiable function $\eta$ which vanishes in a neighborhood of the origin.
 \end{rk}

Like in the case of Besov spaces, Theorem~\ref{g-charct-Triebel} admits the following continuous version.

 \begin{thm}\label{g-charct-Triebel-cont}
 Under the assumption of the previous theorem,  for any distribution
$x$ on $\T^d_\t$,
 $$
 \|x\|_{F_{p}^{\a, c}}\approx |\wh x(0)|+
 \Big\|\Big(\int_0^1\e^{-2\a}|\wt\p_\e*x|^2\,\frac{d\e}\e\Big)^{\frac12}\Big\|_p\,.
 $$
 \end{thm}

\begin{proof}
 This proof is very similar to that of  Theorem~\ref{g-charct-Triebel}. The main idea is, of course, to discretize the continuous square  function:
 $$\int_0^1\e^{-2\a}|\wt\p_\e*x|^2\,\frac{d\e}\e\approx
  \sum_{j=0}^\8 2^{2 j\a}\int_{2^{-j-1}}^{2^{-j}}|\wt\p_\e*x|^2 \,\frac{d\e}\e\,.$$
 We can further discretize the internal integrals on the right-hand side. Indeed, by approximation and assuming that $x$ is a polynomial, each internal integral can be approximated uniformly by discrete sums.
Then we follow  the proof of Theorem~\ref{g-charct-Besov-cont} with necessary modifications as in the preceding proof. The only difference is that when Theorem~\ref{Hormander} is applied,
 the $L_1$-norm of the inverse Fourier transforms of the various functions in consideration there must be replaced by the two norms of these functions appearing in \eqref{psi-T}. We omit the details.
\end{proof}

%%%%%%%%%%%%%%%%%%%%%%%%%%%%%%%%%%%%%%%%%%%%%%%%%%%%%%%%%%%%%%%%%%%%%%%%
%%%%%%%%%%%%%%%%%%%%%%%%%%%%%%%%%%%%%%%%%%%%%%%%%%%%%%%%%%%%%%%%%%%%%%%%

\section{Concrete characterizations}

%%%%%%%%%%%%%%%%%%%%%%%%%%%%%%%%%%%%%%%%%%%%%%%%%%%%%%%%%%%%%%%%%%%%%%%%
%%%%%%%%%%%%%%%%%%%%%%%%%%%%%%%%%%%%%%%%%%%%%%%%%%%%%%%%%%%%%%%%%%%%%%%%

This section concretizes the general characterization in the previous one in terms of  the  Poisson and heat kernels. We keep the notation introduced in section~\ref{The characterizations by Poisson and heat semigroups}.

The following result improves  \cite[Section~2.6.4]{HT1992} at two aspects even in the classical case: Firstly, in addition to derivation operators, it can also use integration operators (corresponding to negative $k$); secondly, \cite[Section~2.6.4]{HT1992} requires $k>d+\max(\a, 0)$ for the Poisson characterization while we only need $k>\a$.

\begin{thm}\label{PH charct-Triebel}
 Let $1\le p<\8$ and $\a\in\real$.
 \begin{enumerate}[\rm (i)]
 \item Let $k\in\ent$ such that $k>\a$. Then for any distribution $x$ on $\T^d_\t$,
 $$
 \|x\|_{F_{p}^{\a, c}}\approx
 |\wh x(0)|+\Big\|\Big(\int_0^1\e^{2(k-\a)}\big|\mathcal{J}^k_\e\,\wt{\mathrm{P}}_\e(x)\big|^2\,\frac{d\e}\e\Big)^{\frac12}\Big\|_p\,.
 $$
 \item  Let $k\in\ent$ such that $k>\frac{\a}2$. Then for any distribution $x$ on $\T^d_\t$,
 $$
 \|x\|_{F_{p}^{\a, c}}\approx
 |\wh x(0)|+\Big\|\Big(\int_0^1\e^{2(k-\frac\a2)}\big|\mathcal{J}^k_\e\,\wt{\mathrm{W}}_\e(x)\big|^2\,\frac{d\e}\e\Big)^{\frac12}\Big\|_p\,.
 $$
 \end{enumerate}
 \end{thm}

The preceding theorem can be formulated directly  in terms of the circular  Poisson and heat semigroups  of $\T^d_\t$.
The proof of the following result is similar to that of Theorem~\ref{circular-charct-Besov}, and  is left to the reader.

 \begin{thm}\label{circular-charct-Triebel}
 Let $1\le p<\8$, $\a\in\real$ and $k\in\ent$.
 \begin{enumerate}[\rm(i)]
 \item If $k>\a$, then for any distribution $x$ on $\T^d_\t$,
  $$
 \|x\|_{F_{p}^{\a, c}}\approx
 \max_{|m|<k}|\wh x(m)|+ \Big\|\Big(\int_0^1(1-r)^{2(k-\a)}\big|\mathcal{J}^k_r\,{\mathbb{P}}_r(x_k)\big|^2\,\frac{dr}{1-r}\Big)^{\frac12}\Big\|_p\,,
 $$
where $\displaystyle x_k=x-\sum_{|m|<k}\wh x(m)U^m$.
  \item If $k>\frac{\a}2$, then for any any distribution $x$ on $\T^d_\t$,
 $$
 \|x\|_{F_{p}^{\a, c}}\approx
 \max_{|m|^2<k}|\wh x(m)|+\Big\|\Big(\int_0^1(1-r)^{2(k-\frac\a2)}\big|\mathcal{J}^k_r\,{\mathbb{W}}_r(x)\big|^2\,\frac{dr}{1-r}\Big)^{\frac12}\Big\|_p\,.
 $$
 \end{enumerate}
 \end{thm}

\medskip\begin{proof}[The proof of Theorem \ref{PH charct-Triebel}] Similar to the Besov case, the proof of (ii) is done by choosing $\a_1=2k>\a$. But (i) is much subtler. We will  first prove (i) under the stronger assumption that $k>d+\a$, the remaining case being postponed. The proof in this case is similar to and a little bit harder than the proof of Theorem~\ref{Poisson charct-Besov}. Let again $\p(\xi)=(-{\rm sgn}(k)2\pi|\xi|)^ke^{-2\pi|\xi|}$. As in that proof, it remains to show that $\p$ satisfies the second condition of \eqref{psi-T} for some $\a_1>\a$ and $\s>\frac{d}2$. Since $k>d+\a$, we can choose $\a_1$ such that $\a<\a_1<k-d$. We claim that  $I_{k-\a_1}h\,\wh{\mathrm{P}}\in H_2^{\s_1}(\real^d)$ for every  $\s_1\in (\frac{d}2,\,k-\a_1+\frac{d}2)$. Indeed, this is a variant of Lemma~\ref{weighted Bessel} with $a=k-\a_1$  and $\rho=h\,\wh{\mathrm{P}}$. The difference is that this function $\rho$ is not infinitely differentiable at the origin. However, the claim is true if $\s_1$ is an integer. Then by complex interpolation as in the proof of that lemma, we deduce the claim in the general case. Now choose $\s$ such that $\frac{d}2<\s<\frac{1}2\,(\s_1-\frac{d}2)$ and set $\eta=\s_1-2\s$. Then $\eta>\frac{d}2$, and by the Cauchy-Schwarz inequality, we have
   \be\begin{split}
   \int_{\real^d}(1+|s|^2)^\s\big|\F^{-1}\big(I_{k-\a_1}h\,\wh{\mathrm{P}}\big)(s)\big|ds
   &\le \Big( \int_{\real^d}(1+|s|^2)^{2\s+\eta}\big|\F^{-1}\big(I_{k-\a_1}h\,\wh{\mathrm{P}}\big)(s)\big|^2ds\Big)^{\frac12}\\
   &\les \big\|\F^{-1}\big(I_{k-\a_1}h\,\wh{\mathrm{P}}\big)\big\|_{H_2^{\s_1}(\real^d)}\,.
   \end{split}\ee
Therefore, the second condition of \eqref{psi-T} is verified. This shows part (i) in the case $k>d+\a$.
\end{proof}

To deal with the remaining case $k>\a$, we need the following:

\begin{lem}
 Let $1\le p<\8$ and $k, \el\in\ent$ such that $\el>k>\a$. Then for any distribution $x$ on $\T^d_\t$ with $\wh x (0)=0,$
 $$
  \Big\|\Big(\int_0^1\e^{2(k-\a)}\big|\mathcal{J}^k_\e\,\wt{\mathrm{P}}_\e(x)\big|^2\,\frac{d\e}\e\Big)^{\frac12}\Big\|_p
  \approx \Big\|\Big(\int_0^1\e^{2(\el-\a)}\big|\mathcal{J}^\el_\e\,\wt{\mathrm{P}}_\e(x)\big|^2\,\frac{d\e}\e\Big)^{\frac12}\Big\|_p\,.
 $$
\end{lem}

\begin{proof}
By induction, it suffices to consider the case $\el=k+1$. We first show the lower estimate:
 $$ \Big\|\Big(\int_0^1\e^{2(k-\a)}\big|\mathcal{J}^k_\e\,\wt{\mathrm{P}}_\e(x)\big|^2\,\frac{d\e}\e\Big)^{\frac12}\Big\|_p
  \les \Big\|\Big(\int_0^1\e^{2(k+1-\a)}\big|\mathcal{J}^\el_\e\,\wt{\mathrm{P}}_\e(x)\big|^2\,\frac{d\e}\e\Big)^{\frac12}\Big\|_p\,.$$
To that end, we use
 $$\mathcal{J}^k_\e\,\wt{\mathrm{P}}_\e(x)=-\mathrm{sgn}(k)\int_{\e}^\8\mathcal{J}^{k+1}_\d\,\wt{\mathrm{P}}_\d(x) d\d.$$
Choose $\b\in(0,\, k-\a)$. By the Cauchy-Schwarz inequality via the operator convexity of the function $t\mapsto t^2$, we obtain
 $$\big|\mathcal{J}^k_\e\,\wt{\mathrm{P}}_\e(x)\big|^2
 \le \frac{\e^{-2\b}}{2\b}\,\int_{\e}^\8\d^{2(1+\b)}\big|\mathcal{J}^{k+1}_\d\,\wt{\mathrm{P}}_\d(x)\big|^2 \,\frac{d\d}{\d}.$$
It then follows that
 \be\begin{split}
 \int_0^1\e^{2(k-\a)}\big|\mathcal{J}^k_\e\,\wt{\mathrm{P}}_\e(x)\big|^2\,\frac{d\e}\e
 &\le \frac{1}{2\b}\,\int_0^\8\d^{2(1+\b)}\big|\mathcal{J}^{k+1}_\d\,\wt{\mathrm{P}}_\d(x)\big|^2 \,\frac{d\d}{\d}\,\int_0^\d \e^{2(k-\a-\b)}\,\frac{d\e}{\e}\\
 &= \frac{1}{4\b(k-\a-\b)}\,\int_0^\8\d^{2(k+1-\a)}\big|\mathcal{J}^{k+1}_\d\,\wt{\mathrm{P}}_\d(x)\big|^2 \,\frac{d\d}{\d}\,.
  \end{split}\ee
Therefore,
  \be\begin{split}
  \Big\|\Big(\int_0^1\e^{2(k-\a)}\big|\mathcal{J}^k_\e\,\wt{\mathrm{P}}_\e(x)\big|^2\,\frac{d\e}\e\Big)^{\frac12}\Big\|_p
  &\les \Big\|\Big(\int_0^\8\d^{2(k+1-\a)}\big|\mathcal{J}^{k+1}_\d\,\wt{\mathrm{P}}_\d(x)\big|^2\,\frac{d\d}\d\Big)^{\frac12}\Big\|_p\\
  &\les \Big\|\Big(\int_0^1\d^{2(k+1-\a)}\big|\mathcal{J}^{k+1}_\d\,\wt{\mathrm{P}}_\d(x)\big|^2\,\frac{d\d}\d\Big)^{\frac12}\Big\|_p\,,
  \end{split}\ee
as desired.

The upper estimate is harder. This time, writing
 $\wt{\mathrm{P}}_{\e_1+\e_2}= \wt{\mathrm{P}}_{\e_1}\,*\, \wt{\mathrm{P}}_{\e_2}$, we have
   \be\begin{split}
   \Big(\d^{k+1}\mathcal{J}^{k+1}_\d\,\wt{\mathrm{P}}_{\d}\Big)\Big|_{\d=2\e}
   &={\rm sgn}(k)2^{k+1}\,\e^{k+1}\frac{\partial}{\partial \e}\wt{\mathrm{P}}_{\e}\,* \,\mathcal{J}^k_\e\,\wt{\mathrm{P}}_{\e}\\
   &={\rm sgn}(k)2^{k+1}\,\e^{k}\,\wt\phi_{\e}\,* \,\mathcal{J}^k_\e\,\wt{\mathrm{P}}_{\e}\,,
    \end{split}\ee
 where $\phi(\xi)=-2\pi |\xi|\,e^{-2\pi|\xi|}.$
Thus
 \be\begin{split}
 \Big\|\Big(\int_0^1\e^{2(k+1-\a)}\big|\mathcal{J}^{k+1}_\e\,\wt{\mathrm{P}}_\e(x)\big|^2\,\frac{d\e}\e\Big)^{\frac12}\Big\|_p
 &=\Big\|\Big(\int_0^{\frac12}\Big|\Big(\d^{k+1-\a}\mathcal{J}^{k+1}_\d\,\wt{\mathrm{P}}_{\d}\Big)\Big|_{\d=2\e}\big|^2\,\frac{d\e}\e\Big)^{\frac12}\Big\|_p\\
 &=2^{k+1-\a}\,\Big\|\Big(\int_0^{\frac12}\e^{2(k-\a)}\big| \wt\phi_{\e}\,* \,\mathcal{J}^k_\e\,\wt{\mathrm{P}}_{\e} (x)\big|^2\,\frac{d\e}\e\Big)^{\frac12}\Big\|_p\\
 &\le2^{k+1-\a}\,\Big\|\Big(\int_0^1\e^{2(k-\a)}\big| \wt\phi_{\e}\,* \,\mathcal{J}^k_\e\,\wt{\mathrm{P}}_{\e} (x)\big|^2\,\frac{d\e}\e\Big)^{\frac12}\Big\|_p\;.
   \end{split}\ee
Now our task is to remove $\phi_\e$ from the integrand on the right-hand side in the spirit of Theorem~\ref{Hormander}. To that end, we will use a multiplier theorem analogous to Lemma~\ref{multiplier DH}. Let $H=L_2((0,\,1), \frac{d\e}\e)$ and define the $H$-valued kernel $\mathsf{k}$ on $\real^d$ by $\mathsf{k}(s)=\big(\phi_\e(s)\big)_{0<\e<1}$. It is a well-known elementary fact that this is a Calder\'on-Zygmund kernel, namely,
  \begin{enumerate}[$\bullet$]
 \item $\displaystyle \big\|\wh{\mathsf{k}}\big\|_{L_\8(\real^d; H)}<\8$;
 \item $\displaystyle \sup_{t\in\real^d}\int_{|s|>2|t|}\|\mathsf{k}(s-t)-\mathsf{k}(s)\|_{H}ds<\8$.
 \end{enumerate}
 Thus by Lemma~\ref{multiplier DH} (i) (more exactly,  following its proof), we obtain that the singular integral operator associated to $\mathsf{k}$ is bounded on $L_p(\T^d_\t; H^c)$ for any $1<p<\8$; consequently,
  \beq\label{k+1 leq k}
  \Big\|\Big(\int_0^1\e^{2(k-\a)}\big|\wt\phi_\e\,* \,\mathcal{J}^k_\e\,\wt{\mathrm{P}}_\e(x)\big|^2\,\frac{d\e}\e\Big)^{\frac12}\Big\|_p
  \les \Big\|\Big(\int_0^1\e^{2(k-\a)}\big|\mathcal{J}^k_\e\,\wt{\mathrm{P}}_\e(x)\big|^2\,\frac{d\e}\e\Big)^{\frac12}\Big\|_p\,,\eeq
 whence
 $$ \Big\|\Big(\int_0^1\e^{2(k+1-\a)}\big|\mathcal{J}^{k+1}_\e\,\wt{\mathrm{P}}_\e(x)\big|^2\,\frac{d\e}\e\Big)^{\frac12}\Big\|_p
  \les \Big\|\Big(\int_0^1\e^{2(k-\a)}\big|\mathcal{J}^k_\e\,\wt{\mathrm{P}}_\e(x)\big|^2\,\frac{d\e}\e\Big)^{\frac12}\Big\|_p\,.$$
Thus the lemma is proved for $1<p<\8$.

The case $p=1$ necessitates  a separate argument like Lemma~\ref{multiplier DH}. We will require a more characterization of $\H_1^c(\T^d_\t)$  which is a complement to Lemma~\ref{Hp-discrete}. It is the following equivalence proved in \cite{XXX}: \\
\indent Let  $\b>0$. Then for a distribution $x$ on $\T_\t^d$ with $\wh x (0)=0$, we have
 \beq\label{H1via Poisson}
 \|x\|_{\mathcal{H}^{c}_1}\approx\Big\|\Big(\int_0^1 \big|(I^{\b }\mathrm{P})_{\e} *x\big|^2\frac{d\e}{\e}\Big)^{\frac12}\Big\|_1\, .
 \eeq
Armed with this characterization, we can easily complete the proof of the lemma. Indeed, 
 $$(\wt {I^{k-\a}\mathrm{P}})_\e* (I^\a x)=(-{\rm sgn}(k)2\pi)^{-k}\,\e^{k-\a}\mathcal{J}^{k}_\e \wt {\mathrm{P}}_\e (x)\;.$$
Thus by \eqref{H1via Poisson}  with $\b=k-\a$,
 $$\Big\|\Big(\int_0^1\e^{2(k-\a)}\big|\mathcal{J}^k_\e\,\wt{\mathrm{P}}_\e(x)\big|^2\,\frac{d\e}\e\Big)^{\frac12}\Big\|_1\approx\|I^\a x\|_{\mathcal{H}^{c}_1}\;.$$
 It then remains to apply  Lemma~\ref{multiplier DH} (ii)  to $I^\a x$ to conclude that \eqref{k+1 leq k}  holds for $p=1$ too, so the proof of the lemma is complete.
  \end{proof}

\begin{proof}[End of the proof of Theorem \ref{PH charct-Triebel}] The preceding lemma shows that the norm in the right-hand side of the equivalence in part (i) is independent of $k$ with $k>\a$. As (i) has been already proved to be true for $k>d+\a$, we deduce the assertion  in full generality.
 \end{proof}

We end this section with a Littlewood-Paley type characterizations of Sobolev spaces. The following is an immediate consequence of Corollary~\ref{Triebel=Sobolev} and the characterizations proved previously in this chapter.

\begin{prop}
Let  $\psi$ satisfy \eqref{psi-T}, $k>\a$ and $1<p<\8$. Then for any distribution on $\T^d_\t$,
  \be\begin{split}
&\|x\|_{H^{\a}_p} \approx |\wh x(0)|+\\
&\;\;\;\;\;\; \left \{ \begin{split}
 &\inf\Big\{
 \big\|\big(\sum_{k\ge0}\big(2^{2k\a}|\wt\p_k*y|^2\big)^{\frac12}\big\|_p +
 \big\|\big(\sum_{k\ge0}\big(2^{2k\a}|(\wt\p_k*z)^*|^2\big)^{\frac12}\big\|_p  \Big\} & \textrm{if } 1< p<2,\\
 &\max\Big\{
 \big\|\big(\sum_{k\ge0}\big(2^{2k\a}|\wt\p_k*x|^2\big)^{\frac12}\big\|_p,\;
 \big\|\big(\sum_{k\ge0}\big(2^{2k\a}|(\wt\p_k*x)^*|^2\big)^{\frac12}\big\|_p \Big\}   & \textrm{if } 2\le  p<\8;
 \end{split} \right.
 \end{split}
 \ee
and
  \be\begin{split}
&\|x\|_{H^{\a}_p} \approx |\wh x(0)|+\\
& \left \{ \begin{split}
 &\inf\Big\{
 \Big\|\Big(\int_0^1\e^{2(k-\a)}\big|\mathcal{J}^k_\e\,\wt{\mathrm{P}}_\e(y)\big|^2\,\frac{d\e}\e\Big)^{\frac12}\Big\|_p +
 \Big\|\Big(\int_0^1\e^{2(k-\a)}\big|\big(\mathcal{J}^k_\e\,\wt{\mathrm{P}}_\e(z)\big)^*\big|^2\,\frac{d\e}\e\Big)^{\frac12}\Big\|_p \Big\} & \textrm{if } 1< p<2,\\
 &\max\Big\{
\Big\|\Big(\int_0^1\e^{2(k-\a)}\big|\mathcal{J}^k_\e\,\wt{\mathrm{P}}_\e(x)\big|^2\,\frac{d\e}\e\Big)^{\frac12}\Big\|_p,\;
 \Big\|\Big(\int_0^1\e^{2(k-\a)}\big|\big(\mathcal{J}^k_\e\,\wt{\mathrm{P}}_\e(x)\big)^*\big|^2\,\frac{d\e}\e\Big)^{\frac12}\Big\|_p \Big\}   & \textrm{if } 2\le  p<\8.
 \end{split} \right.
 \end{split}
 \ee
The above infima are taken above all decompositions $x=y+z$.
\end{prop}

%%%%%%%%%%%%%%%%%%%%%%%%%%%%%%%%%%%%%%%%%%%%%%%%%%%%%%%%%%%%%%%%%%%%%%%%
%%%%%%%%%%%%%%%%%%%%%%%%%%%%%%%%%%%%%%%%%%%%%%%%%%%%%%%%%%%%%%%%%%%%%%%%

\section{Operator-valued Triebel-Lizorkin spaces}

%%%%%%%%%%%%%%%%%%%%%%%%%%%%%%%%%%%%%%%%%%%%%%%%%%%%%%%%%%%%%%%%%%%%%%%%
%%%%%%%%%%%%%%%%%%%%%%%%%%%%%%%%%%%%%%%%%%%%%%%%%%%%%%%%%%%%%%%%%%%%%%%%

Unlike Sobolev and Besov spaces, the study of vector-valued Triebel-Lizorkin spaces in the classical setting does not allow one to handle their counterparts in quantum tori by means of transference. Given a Banach space $X$, a straightforward way of defining the $X$-valued Triebel-Lizorkin spaces on $\T^d$ is as follows: for $1\le p<\8$, $1\le q\le\8$ and $\a\in\real$, an $X$-valued distribution $f$ on $\T^d$ belongs to $F_{p, q}^\a(\T^d;X)$ if
 $$\|f\|_{F_{p, q}^\a}=\|\wh f(0)\|_X+\big\|\big(\sum_{k\ge0}2^{qk\a}\|\wt\f_k* f\|_{X}^q\big)^{\frac1q}\big\|_{L_p(\T^d)}<\8.$$
A majority of the classical results on Triebel-Lizorkin spaces can be proved to be true in this vector-valued setting with essentially the same methods.  Contrary to the Sobolev or Besov case,  the  space $F_{p, 2}^\a(\T^d; L_p(\T^d_\t))$ is very different from the previously studied space $F_{p, 2}^{\a, c}(\T^d_\t)$. This explains why  the transference method is not efficient here.

However, there exists another way of defining  $F_{p, 2}^\a(\T^d;X)$. Let $(r_k)$ be a Rademacher sequence, that is, an independent sequence of random variables on a probability space  $(\Omega, P)$, taking only two values $\pm1$ with equal probability. We define $F_{p, {\rm rad}}^\a(\T^d;X)$ to be the space of all $X$-valued distributions $f$ on $\T^d$ such that
 $$\|f\|_{F_{p, {\rm rad}}^\a}=\|\wh f(0)\|_X+\big\|\sum_{k\ge0} r_k\,2^{k\a}\,\wt\f_k* f\big\|_{L_p(\Omega\times\T^d; X)}<\8.$$
It seems that these spaces $F_{p, {\rm rad}}^\a(\T^d;X)$ have never been studied so far in literature. They might be worth to be investigated. If $X$ is a Banach lattice of finite concavity, then by the Khintchine inequality,
 $$\|f\|_{F_{p, {\rm rad}}^\a}\approx \|\wh f(0)\|_X+
 \big\|\big(\sum_{k\ge0} 2^{2k\a}\,|\wt\f_k* f|^2\big)^{\frac12}\big\|_{L_p(\T^d; X)}\,.$$
This norm resembles, in form, more the previous one $\|f\|_{F_{p, 2}^\a}$. Moreover, in this case, one can also define a similar space by replacing the internal $\el_2$-norm by any $\el_q$-norm.

But what we are interested in here is the noncommutative case, where $X$ is a noncommutative $L_p$-space, say, $X=L_p(\T^d_\t)$. Then by the noncommutative Khintchine inequality \cite{LPP1991}, we can show that for $2\le p<\8$ (assuming $\wh f(0)=0$),
 \be\begin{split}
 \|f\|_{F_{p, {\rm rad}}^\a}
 \approx
 \max\Big\{\big\|\big(\sum_{k\ge0} 2^{2k\a}\,|\wt\f_k* f|^2\big)^{\frac12}\big\|_{p}\,,\;
 \big\|\big(\sum_{k\ge0} 2^{2k\a}\,|(\wt\f_k* f)^*|^2\big)^{\frac12}\big\|_{p}\Big\}.
 \end{split}\ee
Here $\|\,\|_p$ is the norm of $L_p(\T^d; L_p(\T^d_\t))$. Thus the right hand-side is closely related to the norm of  $F_{p}^\a(\T^d_\t)$ defined in section~\ref{Definitions and basic properties: Triebel}. In fact, if $x\mapsto\wt x$ denotes the transference map introduced in Corollary~\ref{prop:TransLp}, then for $1<p<\8$, we have
 $$ \|x\|_{F_{p}^\a(\T^d_\t)}\approx  \|\wt x\|_{F_{p, {\rm rad}}^\a(\T^d; L_p(\T^d_\t))}\,.$$
This shows that if one wishes to  treat Triebel-Lizorkin  spaces on $\T^d_\t$ via transference, one should first investigate the spaces $F_{p, {\rm rad}}^\a(\T^d; L_p(\T^d_\t))$. The latter  ones are as hard to deal with as $F_{p}^\a(\T^d_\t)$.

We would like to point out, at this stage, that the method we have developed in this chapter applies as well to $F_{p, {\rm rad}}^\a(\T^d; L_p(\T^d_\t))$. In view of operator-valued Hardy spaces, we will call $F_{p, {\rm rad}}^\a(\T^d; L_p(\T^d_\t))$ an operator-valued Triebel-Lizorkin  space on $\T^d$. We can define similarly its column and row counterparts. We will give below an outline of these operator-valued Triebel-Lizorkin  spaces in the light of the development made in the previous sections. A  systematic study will be given elsewhere. In the remainder of this section, $\M$ will denote a finite von Neumann algebra $\M$ with a faithful normal tracial state $\tau$ and $\N=L_\8(\T^d)\overline\ot\M$.

 \begin{Def}
Let $1\leq p<\8$ and $\a\in \real$.
The  column operator-valued Triebel-Lizorkin space $F^{\a, c}_{p} (\T^d, \M)$ is defined to be
$$F^{\a, c}_{p} (\T^d, \M)=\big\{f\in \mathcal{S}'(\T^d; L_1(\M)) : \|f\|_{F^{\a, c}_{p}} < \8 \big\},$$
where
 $$\|f\|_{F^{\a, c}_{p}} =\|\wh f(0)\|_{L_p(\M)} + \big\|\big(\sum_{k\ge 0} 2^{2k\a } | \wt\f_k * f|^2\big)^{\frac{1}{2}}\big\|_{L_p(\N)}\,.$$
\end{Def}

The main ingredient for the study of these spaces is still a multiplier result like Theorem~\ref{Hormander} that is restated as follows:

\begin{thm}\label{Hormanderbis}
 Assume that $(\phi_j)_{\ge0}$ and $(\rho_j)_{\ge0}$ satisfy \eqref{psi-rho} with some $\s>\frac{d}2$.
 \begin{enumerate}[\rm(i)]
 \item Let $1<p<\8$.  Then for any  $f\in \mathcal{S}'(\T^d; L_1(\M))$,
  $$
 \big\|\big(\sum_{j\ge0}2^{2j\a}|\wt\phi_j*\wt\rho_j*f^2|^{\frac12}\big\|_{L_p(\N)}\les   \underset{\substack{j\geq 0 \\ -2\leq k \leq 2}}{\sup} \big\|   {\phi} _j (2^{j+k}\cdot)\varphi \big\| _{H_2^\sigma}\,
 \big\|\big(\sum_{j\ge0}2^{2j\a}|\wt\rho_j*f|^2\big)^{\frac12}\big\|_{L_p(\N)}\,.
 $$
 \item If $\rho_j=\wh{\rho(2^{-j}\cdot)}$ for some Schwartz   function $\rho$  with
  $\mathrm{supp}(\rho)=\{\xi: 2^{-1}\le |\xi|\le2\}.$
 Then the above inequality holds for $p=1$ too.
 \end{enumerate}
 \end{thm}

The proof of Theorem~\ref{Hormander}  already gives the above result.  Armed with this multiplier theorem, we can check that all results proved in the previous sections admit operator-valued analogues with the same proofs. For instance, the dual space of $F^{\a, c}_{1} (\T^d, \M)$ can be described as a space $F^{-\a, c}_{\8} (\T^d, \M)$ analogous to the one defined in Definition~\ref{F infity}. However, following the $H_1$-BMO duality developed in the theory of operator-valued Hardy spaces in \cite{XXX}, we can show the following nicer characterization of the latter space in the style of Carleson measures:

\begin{thm}
 A distribution $f\in \mathcal{S}'(\T^d; L_1(\M))$ with $\wh f(0)=0$  belongs to $F^{\a, c}_{\8} (\T^d, \M)$ iff
  $$\sup_Q\Big\|\frac1{|Q|}\int_Q\sum_{k\ge\log_2(l(Q))} 2^{2k\a } | \wt\f_k * f(s)|^2 ds\Big\|_{\M}<\8,$$
 where the supremum runs over all cubes of $\T^d$, and where $l(Q)$ denotes the side length of $Q$.
 \end{thm}

The characterizations of Triebel-Lizorkin spaces given in the previous two sections can be transferred to the present setting too. Let us formulate only the analogue of  Theorem~\ref{circular-charct-Triebel}.

\begin{thm}
 Let $1\le p<\8$, $\a\in\real$ and $k\in\ent$.
 \begin{enumerate}[\rm(i)]
 \item If $k>\a$, then for any $f\in \mathcal{S}'(\T^d; L_1(\M))$,
  $$
 \|f\|_{F_{p}^{\a, c}}\approx
 \max_{|m|<k}\|\wh f(m)\|_{L_p(\M)}+ \Big\|\Big(\int_0^1(1-r)^{2(k-\a)}\big|\mathcal{J}^k_r\,{\mathbb{P}}_r(f_k)\big|^2\,\frac{dr}{1-r}\Big)^{\frac12}\Big\|_{L_p(\N)}\,,
 $$
where $\displaystyle f_k=f-\sum_{|m|<k}\wh f(m)U^m$.
  \item If $k>\frac\a2$, then for any $f\in \mathcal{S}'(\T^d; L_1(\M))$,
 $$
 \|f\|_{F_{p}^{\a, c}}\approx
 \max_{|m|^2<k}\|\wh f(m)\|_{L_p(\M)}+\Big\|\Big(\int_0^1(1-r)^{2(k-\frac\a2)}\big|\mathcal{J}^k_r\,{\mathbb{W}}_r(f)\big|^2\,\frac{dr}{1-r}\Big)^{\frac12}\Big\|_{L_p(\N)}\,.
 $$
 \end{enumerate}
 \end{thm}

%%%%%%%%%%%%%%%%%%%%%%%%%%%%%%%%%%%%%%%%%%%%%%%%%%%%%%%%%%%%%%%%%%%%%%%%
%%%%%%%%%%%%%%%%%%%%%%%%%%%%%%%%%%%%%%%%%%%%%%%%%%%%%%%%%%%%%%%%%%%%%%%%

{\Large\part{Interpolation}}
\setcounter{section}{0}

%%%%%%%%%%%%%%%%%%%%%%%%%%%%%%%%%%%%%%%%%%%%%%%%%%%%%%%%%%%%%%%%%%%%%%%%
%%%%%%%%%%%%%%%%%%%%%%%%%%%%%%%%%%%%%%%%%%%%%%%%%%%%%%%%%%%%%%%%%%%%%%%%

Now we study the interpolation of the various spaces introduced in the preceding three chapters. We start with the interpolation of Besov and Sobolev spaces. Like in the classical case, the interpolation of Besov spaces on $\T^d_\t$ is very simple. However, the situation of (fractional) Sobolev  spaces is much more delicate. Recall that the complex interpolation problem of the classical couple $(W^k_1(\real^d),\, W^k_\8(\real^d))$ remains always open (see \cite[p.~173]{HN}). We show in the first section some partial results on the interpolation of $W^k_p(\T^d_\t)$ and $H^\a_p(\T^d_\t)$. The main result there concerns the Hardy-Sobolev spaces $W^k_{\H_1}(\T^d_\t)$ and $H^\a_{\H_1}(\T^d_\t)$, that is, when the $L_1$-norm is replaced by the nicer $\H_1$-norm on $\T^d_\t$. The spaces $W^k_{\BMO}(\T^d_\t)$ and $H^\a_{\BMO}(\T^d_\t)$ are also considered. The most important problem left unsolved in the first section is to transfer DeVore and Scherer's theorem on the real interpolation of $(W^k_1(\real^d),\, W^k_\8(\real^d))$ to the quantum setting. The main result of the second section characterizes the K-functional of the couple $(L_p(\T^d_\t),\, W_p^k(\T^d_\t))$ by the $L_p$-modulus of smoothness, thereby extending a theorem of Johnen and Scherer to the quantum tori. This result is closely related to the limit theorem of Besov spaces proved in section~\ref{Limits of Besov norms}.  The last short section contains some simple results on the interpolation of Triebel-Lizorkin spaces.

\bigskip

%%%%%%%%%%%%%%%%%%%%%%%%%%%%%%%%%%%%%%%%%%%%%%%%%%%%%%%%%%%%%%%%%%%%%%%%
%%%%%%%%%%%%%%%%%%%%%%%%%%%%%%%%%%%%%%%%%%%%%%%%%%%%%%%%%%%%%%%%%%%%%%%%

\section{Interpolation of Besov and Sobolev spaces}
\label{Interpolation of Besov and Sobolev spaces}

%%%%%%%%%%%%%%%%%%%%%%%%%%%%%%%%%%%%%%%%%%%%%%%%%%%%%%%%%%%%%%%%%%%%%%%%
%%%%%%%%%%%%%%%%%%%%%%%%%%%%%%%%%%%%%%%%%%%%%%%%%%%%%%%%%%%%%%%%%%%%%%%%

This section collects some results on the interpolation of Besov and Sobolev spaces. We start with the Besov spaces.

\begin{prop}\label{interpolation-Besov}
Let $0<\eta<1$.  Assume that $\alpha ,\alpha _0,\alpha _1\in\real$ and $p,p_0,p_1,q,q_0,q_1\in[1,\,\8]$  satisfy the constraints given in the formulas below. We have
\begin{enumerate}[\rm(i)]
\item $\displaystyle \big( B^{\a _0}_{p, q_0}(\T_{\t}^d),\, B^{\a _1}_{p,q_1}(\T_{\t}^d) \big)_{\eta,q}=B^{\a}_{p,q}(\T_{\t}^d),\quad \a_0\neq \a_1, \a=(1-\eta)\a_0+\eta\a_1$;

\item $\displaystyle\big( B^{\a }_{p, q_0}(\T_{\t}^d),\, B^{\a }_{p,q_1}(\T_{\t}^d) \big)_{\eta,q}=B^{\a}_{p,q}(\T_{\t}^d),\quad   \frac1q=\frac{1-\eta}{q_0}+\frac{\eta}{q_1}$;

\item $\displaystyle\big( B^{\a_0}_{p_0, q_0}(\T_{\t}^d),\, B^{\a_1}_{p_1,q_1}(\T_{\t}^d) \big)_{\eta,q}=B^{\a}_{p,q}(\T_{\t}^d),\quad   \a=(1-\eta)\a_0+\eta\a_1,\; \\
\frac1p=\frac{1-\eta}{p_0}+\frac{\eta}{p_1},\;\frac1q=\frac{1-\eta}{q_0}+\frac{\eta}{q_1}\,, \,p=q$;

\item $\displaystyle\big( B^{\a_0}_{p_0, q_0}(\T_{\t}^d),\, B^{\a_1}_{p_1,q_1}(\T_{\t}^d) \big)_{\eta}=B^{\a}_{p,q}(\T_{\t}^d),\quad  \a=(1-\eta)\a_0+\eta\a_1,\;
\frac1p=\frac{1-\eta}{p_0}+\frac{\eta}{p_1},\; \\ \frac1q=\frac{1-\eta}{q_0}+\frac{\eta}{q_1}\,,\; q<\8$.
\end{enumerate}
\end{prop}

\begin{proof}
We will use the embedding of $B_{p, q}^\a(\T^d_\t)$ into $\el^\a_q(L_p(\T^d_\t))$.
Recall that given a Banach space $X$, $\el_q^\a(X)$ denotes the weighted $\el_q$-direct sum of $(\com, X, X, \cdots)$, equipped with the norm
  $$\|(a, x_0,x_1,\cdots)\|=\Big(|a|^q+\sum_{k\ge0}2^{kq\a}\|x_k\|^q\Big)^{\frac1q}\,.$$
Then $B_{p, q}^\a(\T^d_\t)$ isometrically embeds into $\el^\a_q(L_p(\T^d_\t))$ via the map $\mathcal{I}$ defined by $\mathcal{I}x=(\wh x(0), \wt\f_0*x, \wt\f_1*x,\cdots)$. On the other hand, it is easy to check that the range of $\mathcal{I}$ is 1-complemented. Indeed, let $\mathcal{P}: \el^\a_q(L_p(\T^d_\t))\to B_{p, q}^\a(\T^d_\t)$ be defined by (with $\wt\f_{k}=0$ for $k\le-1$)
 $$\mathcal{P}(a, x_0, x_1, \cdots) =a+\sum_{k\ge0}(\wt\f_{k-1}+\wt\f_{k}+\wt\f_{k+1})*x_k.$$
Then by \eqref{3-supports}, $\mathcal{P}\mathcal{I}x=x$ for all $x\in B_{p, q}^\a(\T^d_\t)$. On the other hand, letting $y=\mathcal{P}(a, x_0, x_1, \cdots) $, we have
 $$\wt\f_j*y=\sum_{k=j-2}^{j+2}\f_j*(\wt\f_{k-1}+\wt\f_{k}+\wt\f_{k+1})*x_k,\quad j\ge0.$$
Thus we deduce that $\mathcal{P}$ is bounded with norm at most $15$.

Therefore, the interpolation of the Besov spaces is reduced to that of the spaces $\el^\a_q(L_p(\T^d_\t))$, which is well-known and is treated in \cite[Section~5.6]{BL1976}. Let us recall the results needed here. For  a Banach space $X$ and an interpolation couple  $(X_0,\, X_1)$  of Banach spaces, we have
\begin{enumerate}[$\bullet$]
\item $\displaystyle \big( \ell_{q_0}^{\a_0}(X) ,\, \ell_{q_1}^{\a_1}(X) \big)_{\eta, q}=\ell_q^\alpha (X),\quad \a_0\neq\a_1, \a=(1-\eta)\a_0+\eta\a_1$;
\item $\displaystyle \big( \ell_{q_0}^{\a_0}(X_0) ,\, \ell_{q_1}^{\a_1}(X_1) \big)_{\eta, q}=\ell_q^\alpha \big((X_0,\,X_1)_{\eta, q}\big),\quad  \a=(1-\eta)\a_0+\eta\a_1,\;\frac1q=\frac{1-\eta}{q_0}+\frac{\eta}{q_1}$;
\item $\displaystyle \big( \ell_{q_0}^{\a_0}(X_0) ,\, \ell_{q_1}^{\a_1}(X_1) \big)_{\eta}=\ell_q^\a\big((X_0,\,X_1)_{\eta}\big),\quad  \a=(1-\eta)\a_0+\eta\a_1,\;\frac1q=\frac{1-\eta}{q_0}+\frac{\eta}{q_1}\,,\; q<\8$.
\end{enumerate}
It is then clear that the interpolation formulas of the theorem follow from the above ones thanks to the complementation result proved previously.
 \end{proof}

\begin{rk}
 If $q=\8$, part (iv) holds for Calder\'on's second interpolation method, namely,
 $$\big( B^{\a_0}_{p_0, \8}(\T_{\t}^d),\, B^{\a_1}_{p_1,\8}(\T_{\t}^d) \big)^{\eta}=B^{\a}_{p,\8}(\T_{\t}^d),\quad  \a=(1-\eta)\a_0+\eta\a_1,\;
 \frac1p=\frac{1-\eta}{p_0}+\frac{\eta}{p_1}\,. $$
On the other hand, if one wishes to stay with the first complex  interpolation method in the case $q=\8$, one should replace $B^{\a}_{p, \8}(\T_{\t}^d)$ by $B^{\a}_{p, c_0}(\T_{\t}^d)$:
 $$\big( B^{\a_0}_{p_0, c_0}(\T_{\t}^d),\, B^{\a_1}_{p_1,c_0}(\T_{\t}^d) \big)_{\eta}=B^{\a}_{p,c_0}(\T_{\t}^d)\,.$$
 \end{rk}

Now we consider the potential Sobolev spaces. Since $J^\a$ is an isometry between $H_p^\a(\T^d_\t)$ and  $L_p(\T^d_\t)$ for all $1\le p\le\8$, we get immediately the following

\begin{rk}\label{inter potential-1a}
 Let $0<\eta<1$, $\a\in\real$, $1\le p_0, \, p_1\le\8$ and $\frac1p=\frac{1-\eta}{p_0}+\frac{\eta}{p_1}$. Then
 $$
 \big(H_{p_0}^\a(\T^d_\t),\, H_{p_1}^\a(\T^d_\t)\big)_\eta=H_{p}^\a(\T^d_\t)\;\text{ and}\;
 \big(H_{p_0}^\a(\T^d_\t),\, H_{p_1}^\a(\T^d_\t)\big)_{\eta, p}=H_{p}^\a(\T^d_\t)\,.$$
  \end{rk}

 The interpolation problem of the couple $\big(H_{p_0}^{\a_0}(\T^d_\t),\, H_{p_1}^{\a_1}(\T^d_\t)\big)$ for $\a_0\neq\a_1$ is delicate. At the time of this writing, we cannot, unfortunately, solve it completely. To our knowledge,  it seems that even in the commutative case, its interpolation spaces  by real or complex interpolation method have not been determined in full generality. We will prove some partial results.

 \begin{prop}
 Let $0<\eta<1$, $\a_0\neq \a_1\in\real$ and $1\le p, q\le\8$. Then
  $$\big(H_{p}^{\a_0}(\T^d_\t),\, H_{p}^{\a_1}(\T^d_\t)\big)_{\eta, q}=B_{p, q}^\a(\T^d_\t),\quad \a=(1-\eta)\a_0+\eta\a_1\,.$$
  \end{prop}

\begin{proof}
 The assertion follows from Theorem~\ref{Sobolev-Besov}, the reiteration theorem and Proposition~\ref{interpolation-Besov} (i).
\end{proof}

To treat the complex interpolation, we introduce the potential Hardy-Sobolev spaces.

\begin{Def}
For $\a\in\real$, define
 $$H_{\H_1}^\a(\T^d_\t)=\big\{x\in\mathcal{S}'(\T^d_\t)\,:\, J^\a x\in \H_1(\T^d_\t)\big\}\;\text{ with }\; \big\|x\big\|_{H_{\H_1}^\a}=\big\|J^\a x\big\|_{\H_1}\,.$$
We define $H_{\BMO}^\a(\T^d_\t)$  similarly.
\end{Def}

 \begin{thm}\label{complex inter-Bessel}
 Let $\a_0, \a_1\in\real$ and $1<p<\8$. Then
  $$\big(H_{\BMO}^{\a_0}(\T^d_\t) ,\, H_{\H_1}^{\a_1}(\T^d_\t)\big)_{\frac1p}=H_p^\a(\T^d_\t) ,\quad \a=(1-\frac1p)\a_0+\frac{\a_1}p\,.$$
  \end{thm}

We require the following result which extends Lemma~\ref{q-multiplier}(ii):

\begin{lem}\label{q-multiplier-H1-BMO}
 Let $\phi$ be a Mikhlin multiplier in the sense of Definition~\ref{Mikhlin-def}.  Then $\phi$ is a Fourier multiplier on both $\H_1(\T^d_\t)$ and $\BMO(\T^d_\t)$ with norms majorized by $c_d\|\phi\|_{\rm M}$.
\end{lem}

 \begin{proof}
  This is an immediate consequence of  Lemma~\ref{multiplier DH} (the sequence $(\phi_j)$ there becomes now the single function $\phi$). Indeed, by that Lemma, $\phi$ is a bounded Fourier multiplier on $\H_1(\T^d_\t)$, so by duality, it is bounded on $\BMO(\T^d_\t)$ too.
\end{proof}

We will use Bessel potentials of complex order. For $z\in\com$, define $J_z(\xi)=(1+|\xi|^2)^{\frac{z}2}$ and $J^z$ to be the associated Fourier multiplier.

\begin{lem}\label{complex Bessel}
 Let $t\in\real$. Then $J^{{\rm i}t}$ is bounded on both $\H_1(\T^d_\t)$ and $\BMO(\T^d_\t)$ with norms majorized by $c_d (1+|t|)^d$.
\end{lem}

 \begin{proof}
 One easily checks that $J_{{\rm i}t}$ is a Mikhlin multiplier and $\|J_{{\rm i}t}\|_{\rm M}\le c_d (1+|t|)^d$. Thus, the assertion follows from the previous lemma.
 \end{proof}

\begin{proof}[Proof of Theorem~\ref{complex inter-Bessel}]
Let $x\in H_{p}^{\a}(\T^d_\t)$ with norm less than $1$, that is, $J^\a x\in L_{p}(\T^d_\t)$ and $\|J^\a x\|_p<1$. By Lemma~\ref{q-H-BMO}, and  the definition of complex interpolation, there exists a continuous function $f$ from the strip $S=\{z\in\com\,:\, 0\le{\rm Re}(z)\le1\}$ to $\H_1(\T^d_\t)$, analytic in the interior, such that $f(\frac1p)=J^\a x$,
  $$\sup_{t\in\real}\big\|f({\rm i}t)\big\|_{\BMO}\le c\;\text{ and }\; \sup_{t\in\real}\big\|f(1+{\rm i}t)\big\|_{\H_1}\le c.$$
Define (with $\eta=\frac1p$)
 $$F(z)=e^{(z-\eta)^2}\,J^{-(1-z)\a_0-z\a_1}\, f(z),\quad z\in S.$$
Then for any $t\in\real$, by the preceding lemma,
 $$\big\|F({\rm i}t)\big\|_{H_{\BMO}^{\a_0}}=e^{-t^2+\eta^2}\,\big\|J^{{\rm i}t(\a_0-\a_1)}\,f({\rm i}t)\big\|_{\BMO}\le c'.$$
A similar estimate holds for the other extreme point $H_{\H_1}^{\a_1}(\T^d_\t)$. Therefore,
 $$x=F(\eta)\in \big(H_{\BMO}^{\a_0}(\T^d_\t) ,\, H_{\H_1}^{\a_1}(\T^d_\t)\big)_{\eta} \;\text{ with norm }\le c'.$$
We have thus proved
 $$H_p^\a(\T^d_\t) \subset\big(H_{\BMO}^{\a_0}(\T^d_\t) ,\, H_{\H_1}^{\a_1}(\T^d_\t)\big)_{\eta}\,.$$
Since the dual space of $\H_1(\T^d_\t)$ is $\BMO(\T^d_\t)$, we have
 $$H_{\H_1}^{\a_1}(\T^d_\t)^*=H_{\BMO}^{-\a_1}(\T^d_\t)\,.$$
Thus dualizing the above inclusion (for appropriate $\a_i$ and $p$), we get
 $$\big(H_{\BMO}^{\a_0}(\T^d_\t)\,,\,H_{\BMO}^{-\a_1}(\T^d_\t)^*\big)^{\eta}\subset H_p^\a(\T^d_\t)\,,$$
where $(\,\cdot\,\,\cdot\,)^{\eta}$ denotes Calder\'on's second complex interpolation method. However, by \cite{Berg1979}
 $$\big(H_{\BMO}^{\a_0}(\T^d_\t)\,,\,H_{\BMO}^{-\a_1}(\T^d_\t)^*\big)_{\eta}\subset \big(H_{\BMO}^{\a_0}(\T^d_\t)\,,\,H_{\BMO}^{-\a_1}(\T^d_\t)^*\big)^{\eta}
 \;\text{ isometrically}.$$
Since
 $$H_{\H_1}^{\a_1}(\T^d_\t)\subset H_{\BMO}^{-\a_1}(\T^d_\t)^*  \;\text{ isometrically},$$
we finally deduce
 $$\big(H_{\BMO}^{\a_0}(\T^d_\t)\,,\,H_{\H_1}^{\a_1}(\T^d_\t)\big)_{\eta}\subset H_p^\a(\T^d_\t)\,,$$
which concludes the proof of the theorem.
  \end{proof}

\begin{cor}
 Let $0<\eta<1$, $\a_0, \a_1\in\real$ and $1< p_0, p_1<\8$. Then
  $$\big(H_{p_0}^{\a_0}(\T^d_\t),\, H_{p_1}^{\a_1}(\T^d_\t)\big)_\eta=H_{p}^{\a}(\T^d_\t)\,,
  \quad \a=(1-\eta)\a_0+\eta\a_1\,,\; \frac1p=\frac{1-\eta}{p_0}+\frac{\eta}{p_1}\,.$$
 \end{cor}

\begin{proof}
The preceding proof works equally for this corollary. Alternately, in the case $p_0\neq p_1$, the corollary immediately follows from the previous theorem by reiteration. Indeed, if $p_0\neq p_1$, then for any $\a_0, \a_1\in\real$ there exist $\b_0, \b_1\in\real$ such that
 $$(1-\frac1{p_0})\b_0+ \frac1{p_0}\b_1=\a_0\;\text{ and }\; (1-\frac1{p_1})\b_0+ \frac1{p_1}\b_1=\a_1\,.$$
Thus the previous theorem implies
 $$\big(H_{\BMO}^{\b_0}(\T^d_\t)\,,\,H_{\H_1}^{\b_1}(\T^d_\t)\big)_{\frac1{p_j}}=H_{p_j}^{\a_j}(\T^d_\t),\quad j=0,1.$$
The corollary then follows from the reiteration theorem.
\end{proof}

It is likely that the above corollary still holds for all $1\le p_0, p_1\le\8$:

\begin{conjecture}
 Let  $\a_0, \a_1\in\real$ and $1<p<\8$. Then
  $$\big(H_{\8}^{\a_0}(\T^d_\t),\, H_{1}^{\a_1}(\T^d_\t)\big)_{\frac1p}=H_{p}^{\a}(\T^d_\t)\,,
  \quad \a=(1-\frac1p)\a_0+\frac{\a_1}p\,.$$
 \end{conjecture}

 By duality and Wolff's reiteration theorem \cite{Wolff1994}, the conjecture is reduced to showing that for any $0<\eta<1$ and $1<p_0<\8$,
 $$\big(H_{p_0}^{\a_0}(\T^d_\t),\, H_{1}^{\a_1}(\T^d_\t)\big)_\eta=H_{p}^{\a}(\T^d_\t)\,
  ,\quad \a=(1-\eta)\a_0+\eta\a_1\,,\; \frac1p=\frac{1-\eta}{p_0}+\frac{\eta}{1}\,.$$
 Since $H_{\H_1}^{\a_1}(\T^d_\t)\subset H_{1}^{\a_1}(\T^d_\t)$, Theorem~\ref{complex inter-Bessel} implies
  $$H_{p}^{\a}(\T^d_\t)\subset \big(H_{p_0}^{\a_0}(\T^d_\t),\, H_{1}^{\a_1}(\T^d_\t)\big)_\eta\,.$$
So the conjecture is equivalent to the validity of the converse inclusion.

\begin{rk}\label{complex inter-Bessel1}
 The proof of Theorem~\ref{complex inter-Bessel} shows that for $\a_0, \a_1\in\real$ and $0<\eta<1$,
 $$\big(H_{\H_1}^{\a_0}(\T^d_\t) ,\, H_{\H_1}^{\a_1}(\T^d_\t)\big)_{\eta}=H_{\H_1}^\a(\T^d_\t) ,\quad \a=(1-\eta)\a_0+\eta\a_1\,.$$
We do not know if this equality remains true for the couple $\big(H_1^{\a_0}(\T^d_\t) ,\, H_1^{\a_1}(\T^d_\t)\big)$.
\end{rk}

We conclude this section with a discussion on the interpolation of $\big(W_{p_0}^k(\T^d_\t),\,W_{p_1}^k(\T^d_\t)\big)$. Here, the most interesting case is, of course, that where $p_0=\8$ and $p_1=1$.
Recall that in the commutative case, the K-functional of $\big(W_{\8}^k(\T^d),\,W_{1}^k(\T^d)\big)$ is determined by DeVore and Scherer  \cite{DS1979}; however, determining the complex interpolation spaces of this couple is a longstanding open problem.

Note that if $1<p_0, p_1<\8$, $\big(W_{p_0}^k(\T^d_\t),\,W_{p_1}^k(\T^d_\t)\big)$ reduces to $\big(H_{p_0}^k(\T^d_\t),\,H_{p_1}^k(\T^d_\t)\big)$ by virtue of Theorem~\ref{q-Sobolev-Bessel}. So in this case, the interpolation problem is solved by the preceding results on potential Sobolev spaces. This reduction is, unfortunately, impossible when one of $p_0$ and $p_1$ is equal to $1$ or $\8$. However, in the spirit of potential Hardy-Sobolev spaces, it remains valid if we work with the Hardy-Sobolev spaces $W_{\BMO}^k(\T^d_\t)$ and $W_{\H_1}^k(\T^d_\t)$ instead of $W_{\8}^k(\T^d_\t)$ and $W_{1}^k(\T^d_\t)$, respectively. Here, the Hardy-Sobolev spaces are defined as they should be.

Using Lemma~\ref{q-multiplier-H1-BMO}, we see that the proof of Theorem~\ref{q-Sobolev-Bessel} remains valid for the Hardy-Sobolev spaces too. Thus we have the following:

\begin{lem}\label{Sobolev-potential Hardy}
For any $k\in\nat$, $W_{\BMO}^k(\T^d_\t)=H_{\BMO}^k(\T^d_\t)$ and $W_{\H_1}^k(\T^d_\t)=H_{\H_1}^k(\T^d_\t)$.
\end{lem}

 \begin{thm}\label{complex inter-Sobolev}
 Let $k\in\nat$ and $1<p<\8$. Then for $X=W_{\H_1}^{k}(\T^d_\t)$ or $X=W_{1}^{k}(\T^d_\t)$,
  $$\big(W_{\BMO}^{k}(\T^d_\t) ,\, X\big)_{\frac1p}=W_p^k(\T^d_\t)=\big(W_{\BMO}^{k}(\T^d_\t) ,\, X\big)_{\frac1p\,,p}\,.$$
 Consequently, for any $0<\eta<1$ and $1< p_0<\8$,
 $$\big(W_{p_0}^{k}(\T^d_\t) ,\, W_{1}^{k}(\T^d_\t)\big)_{\eta}=W_p^k(\T^d_\t)=\big(W_{p_0}^{k}(\T^d_\t) ,\, W_{1}^{k}(\T^d_\t)\big)_{\eta, p}
 ,\quad \frac1p=\frac{1-\eta}{p_0}+\frac{\eta}1\,.$$
  \end{thm}

\begin{proof}
 The first part for $X=W_{\H_1}^{k}(\T^d_\t)$ follows immediately from Remark~\ref{inter potential-1a}, Theorem~\ref{complex inter-Bessel}  and Lemma~\ref{Sobolev-potential Hardy}. Then by  the reiteration theorem, for any $1<p<\8$ and $0<\eta<1$, we get
  $$\big(W_{\BMO}^{k}(\T^d_\t) ,\, W_{p}^{k}(\T^d_\t)\big)_{\eta}=W_{q}^{k}(\T^d_\t)\;\text{ and }\; \big(W_{p}^{k}(\T^d_\t) ,\, W_{\H_1}^{k}(\T^d_\t)\big)_{\eta}=W_{r}^{k}(\T^d_\t),$$
 where $\frac1q=\frac{1-\eta}\8+\frac\eta{p}$ and  $\frac1r=\frac{1-\eta}p+\frac\eta{1}$. On the other hand, by the continuous inclusion $\H_1(\T^d_\t)\subset L_1(\T^d_\t)$, we have
  $$W_r^k(\T^d_\t)=\big(W_{p}^{k}(\T^d_\t) ,\, W_{\H_1}^{k}(\T^d_\t)\big)_{\eta}
  \subset \big(W_p^{k}(\T^d_\t) ,\, W_{1}^{k}(\T^d_\t)\big)_{\eta}\subset W_r^k(\T^d_\t),$$
the last inclusion above being trivial. Thus
   $$\big(W_{p}^{k}(\T^d_\t) ,\, W_{1}^{k}(\T^d_\t)\big)_{\eta}=W_{r}^{k}(\T^d_\t).$$
 Therefore, by Wolff's reiteration theorem \cite{Wolff1994}, we deduce the first part for $X=W_{1}^{k}(\T^d_\t)$. The second part follows from the first  by  the reiteration theorem.
 \end{proof}

   \begin{rk}
 The second part of the previous theorem had been proved by Marius Junge by a different method; he reduced it to the corresponding problem on $\H_1$ too.
 \end{rk}

 The main problem left open at this stage  is the following:

 \begin{problem}
 Does the second part of the previous theorem hold for $p_0=\8$?
 \end{problem}

%%%%%%%%%%%%%%%%%%%%%%%%%%%%%%%%%%%%%%%%%%%%%%%%%%%%%%%%%%%%%%%%%%%%%%%%
%%%%%%%%%%%%%%%%%%%%%%%%%%%%%%%%%%%%%%%%%%%%%%%%%%%%%%%%%%%%%%%%%%%%%%%%

\section{The K-functional of $(L_p,\, W_p^k)$}
\label{Kfunctional}

%%%%%%%%%%%%%%%%%%%%%%%%%%%%%%%%%%%%%%%%%%%%%%%%%%%%%%%%%%%%%%%%%%%%%%%%
%%%%%%%%%%%%%%%%%%%%%%%%%%%%%%%%%%%%%%%%%%%%%%%%%%%%%%%%%%%%%%%%%%%%%%%%

In this section we characterize the K-functional of the couple $(L_p(\T^d_\t),\, W_p^k(\T^d_\t))$ for any $1\le p\le\8$ and $k\in\nat$.  First, recall the definition of  the K-functional. For an interpolation couple $(X_0, \,X_1)$ of Banach spaces, we define
 $$K(x,\e;\, X_0, X_1)=\inf\big\{\|x_0\|_{X_0}+\e\|x_0\|_{X_1}\,:\,x=x_0+x_1, x_0\in X_0,\, x_1\in X_1\big\}$$
for $\e>0$ and $x\in X_0+X_1$. Since $W^k_p(\T^d_\t) \subset L_p(\T^d_\t)$ contractively, $K(x, \e;\, L_p(\T^d_\t),W^k_p(\T^d_\t))=\|x\|_{p}$ for $\e\ge1$; so only the case $\e<1$ is nontrivial. The following result is the quantum analogue of Johnen-Scherer's theorem for Sobolev spaces on $\real^d$ (see \cite{JS1976}; see also \cite[Theorem~5.4.12]{BS1988}).   Recall that $\o^k_p(x,\e)$ denotes the $k$th  order modulus of $L_p$-smoothness of $x$ introduced in section~\ref{The characterization by differences}.

\begin{thm}\label{K-functional}
 Let $1\leq p \leq \8$ and $k\in \mathbb{N}$. Then
 $$K(x,\e^k;\, L_p(\T^d_\t),W^k_p(\T^d_\t)) \approx \e^k|\wh x(0)|+ \o^k_p(x,\e),\quad 0<\e\le1$$
with relevant constants depending only on $d$ and $k$.
\end{thm}

\begin{proof}
 We will adapt the proof of \cite[Theorem~5.4.12]{BS1988}. Denote $K(x, \e;\, L_p(\T^d_\t),W^k_p(\T^d_\t))$ simply by $K(x, \e)$. It suffices to consider the elements of  $L_p(\T^d_\t)$ whose Fourier coefficients vanish at $m=0$. Fix such an element $x$. Let $x=y+z$ with $y\in L_p(\T^d_\t)$ and  $z\in W^k_p(\T^d_\t)$ (with vanishing Fourier coefficients at $0$). Then
by Theorem~\ref{lim=nabla},
 $$\o^k_p(x,\e)\le  \o^k_p(y,\e)+ \o^k_p(z,\e)\les \|y\|_p+\e^k|z|_{W^k_p}\,,$$
which implies
  $$\o^k_p(x,\e)\les K(x,\e^k).$$

 The converse inequality is harder. We have to produce an appropriate decomposition of $x$. To this end, let $\mathbb{I}=[0,\, 1)^d$ and define the required decomposition by
  $$y=(-1)^k\int_{\mathbb{I}} \cdots \int_{\mathbb{I}} \D^k_{\e u}(x)du_1\cdots du_k \;\text{ and }\; z=x-y,$$
where $u=u_1+\cdots+u_k$. Then
 $$\|y\|_p\le \int_{\mathbb{I}} \cdots \int_{\mathbb{I}} \|\D^k_{\e u}(x)\|_pdu_1\cdots du_k \le \o^k_p(x, k\sqrt d\,\e)\les\o^k_p(x, \e).$$
To handle $z$, using the formula
 $$\D_{\e u}^k=\sum_{j=0}^k (-1)^{k-j} \left(\begin{array}{c}  k\\j \end{array}\right)T_{j \e u}\,,$$
we rewrite $z$ as
 $$z=(-1)^{k+1}\sum_{j=1}^k (-1)^{k-j} \left(\begin{array}{c}  k\\j \end{array}\right) \int_{\mathbb{I}} \cdots \int_{\mathbb{I}} T_{j\e u}(x)du_1\cdots du_k .$$
All terms on the right-hand side are treated in the same way. Let us consider only the first one by setting
 $$z_1= \int_{\mathbb{I}} \cdots \int_{\mathbb{I}} T_{\e u}(x)du_1\cdots du_k.$$
Write each $u_i$ in  the canonical basis of $\real^d$:
 $$u_i=\sum_{j=1}^d u_{i,j}\mathbf{e}_j\,.$$
 We compute $\partial_1 z_1$ explicitly,  as example,  in the spirit of \eqref{derivation}:
 \be\begin{split}
  \partial_1 z_1=\frac1\e\,\sum_{i=1}^k\int_{\mathbb{I}} \cdots \int_{\mathbb{I}} \frac{\partial}{\partial u_{i, 1}} T_{\e u}(x)du_1\cdots du_k.
 \end{split}\ee
Integrating the partial derivative on the right-hand side with respect to $u_{i, 1}$ yields:
 $$\int_0^1\frac{\partial}{\partial u_{i, 1}} T_{\e u}(x)du_{i,1}
 =\D_{\e(\mathbf{e}_1+u-u_{i, 1}\mathbf{e}_1)}(x)=\int_0^1 \D_{\e(\mathbf{e}_1+u-u_{i, 1}\mathbf{e}_1)}(x)du_{i,1}\,,$$
where for the second equality, we have used the fact that $\D_{\e(\mathbf{e}_1+u-u_{i, 1}\mathbf{e}_1)}(x)$ is constant in $u_{i, 1}$. Thus
 $$\partial_1 z_1=\frac1\e\,\sum_{i=1}^k\int_{\mathbb{I}} \cdots \int_{\mathbb{I}}\D_{\e(\mathbf{e}_1+u-u_{i,1}\mathbf{e}_1)}(x) du_1\cdots du_k.$$
To iterate this formula, we use multi-index notation. For $n\in\nat$ let
 $$[[k]]^n=\big\{\underline{i}=(i_1,\cdots,i_n): 1\le i_\el\le k,\; \text{ all } i_\el\text{'s are distinct}\big\}.$$
Then for any $m_1\in\nat$ with $m_1\le k$, we have
 $$\partial_1^{m_1} z_1=\e^{-m_1}\sum_{\underline{i}^1\in [[k]]^{m_1}}\int_{\mathbb{I}} \cdots \int_{\mathbb{I}}\D^{m_1}_{\e u_{\underline{i}^1}}(x) du_1\cdots du_k\,,$$
where
 $$u_{\underline{i}^1}=\mathbf{e}_1+u-(u_{i^1_1,1}+\cdots+ u_{i^1_{m_1},1})\mathbf{e}_1\,.$$
Iterating this procedure, for any $m\in\nat_0^d$ with $|m|_1=k$, we get
 $$D^m z_1=\e^{-k}\sum_{\underline{i}^d\in [[k]]^{m_d}}\cdots \sum_{\underline{i}^1\in [[k]]^{m_1}}
 \int_{\mathbb{I}^k} \D^{m_d}_{\e u_{}\underline{i}^d}\cdots \D^{m_1}_{\e u_{}\underline{i}^1}(x) du_1\cdots du_k\,,$$
where the $u_{\underline{i}^j}$'s are defined by induction
 $$u_{\underline{i}^j}=\mathbf{e}_j+u_{\underline{i}^{j-1}}-(u_{i^j_1,j}+\cdots+ u_{i^j_{m_j},j})\mathbf{e}_j\,, \quad j=2,\cdots, d.$$
Thus we are in a position of appealing Lemma~\ref{k-diff} to conclude that
 $$\big\|D^m z_1\big\|_p\les \e^{-k}\o^k_p(x,\e),$$
whence
  $$|z|_{W_p^k}\les \e^{-k} \o^k_p(x,\e).$$
 Therefore, $K(x,\e^k)\les \o^k_p(x,\e)$.
 \end{proof}

\begin{rk}
 The preceding proof shows a little bit more: for any $x\in W_p^k(\T^d_\t)$ with $\wh x(0)=0$,
  $$\o^k_p(x, \e)\approx \big\{\|y\|_p+\e^k |z|_{W_p^k}\,:\, x=y+z,\, \wh y(0)=\wh z(0)=0\big\},\quad 0<\e\le1.$$
 In particular, this implies
  $$\|x\|_p\les \o^k_p(x, \e),$$
 which is the analogue for moduli of $L_p$-continuity of the inequality in Theorem~\ref{k-Sobolev} (the Poincar\'e inequality). On the other hand, together with Lemma~\ref{k-diff}, the above inequality provides an alternate proof of Theorem~\ref{k-Sobolev}.
  \end{rk}

The preceding theorem, together with  Theorem~\ref{diff-Besov} and the reiteration theorem, implies the following

\begin{cor}
 Let $0<\eta<1$,  $\a>0$, $k, k_0, k_1\in\nat$ and $1\leq p,q, q_1 \leq \8$. Then
 \begin{enumerate}[\rm(i)]
 \item $\displaystyle \big(L_p(\T^d_\t),\, W^k_p(\T^d_\t)\big)_{\eta,q}= B_{p,q}^{\eta k}(\T^d_\t)$;
 \item $\displaystyle \big(W^k_p(\T^d_\t), \,B_{p,q_1}^{\a}(\T^d_\t)\big)_{\eta,q}= B_{p,q}^{\b}(\T^d_\t)\,,\quad k\neq\a,\; \b=(1-\eta)k+\eta\a$;
\item $\displaystyle \big(W^{k_0}_p(\T^d_\t), \,W^{k_1}_p(\T^d_\t)\big)_{\eta,q}= B_{p,q}^{\a}(\T^d_\t)\,,\quad k_0\neq k_1,\; \a=(1-\eta)k_0+\eta k_1$.
\end{enumerate}
\end{cor}

We can also consider the complex interpolation of $\big(L_p(\T^d_\t),\, W^k_p(\T^d_\t)\big)$. If $1<p<\8$, this is reduced to that of $\big(L_p(\T^d_\t),\, H^k_p(\T^d_\t)\big)$; so by the result of the previous section, for any $0<\eta<1$,
 $$\big(L_p(\T^d_\t),\, W^k_p(\T^d_\t)\big)_\eta=H^{\eta k}_p(\T^d_\t).$$

\begin{problem}
Does the above equality hold for $p=1$? The problem is closely related to that in Remark~\ref{complex inter-Bessel1}.
\end{problem}

We conclude this section with a remark on the link between Theorem~\ref{BBM} and Theorem~\ref{K-functional}. The former can be easily deduced from the latter, by using the following elementary fact (see \cite{BL1976} p.~40): for any couple $(X_0, \,X_1)$ of Banach spaces and $x\in X_0\cap X_1$
 \be\begin{split}
 &\lim_{\eta\to1}\big(\eta (1-\eta)\big)^{\frac1q}\,\big\|x\big\|_{(X_0,\, X_1)_{\eta, q}}=q^{-\frac1q}\|x\|_{X_1}\,,\\
 &\lim_{\eta\to0}\big(\eta (1-\eta)\big)^{\frac1q}\,\big\|x\big\|_{(X_0,\, X_1)_{\eta, q}}=q^{-\frac1q}\|x\|_{X_0}\,.
\end{split}\ee
Here the norm of $(X_0,\, X_1)_{\eta, q}$ is that defined by the K-functional. Then Theorem~\ref{BBM} follows from Theorem~\ref{K-functional} and the first limit above. This is the approach adopted in \cite{KMX2005, Milman2005}. It also allows us to determine the other extreme case $\a=0$ in Theorem~\ref{BBM}, which  was done by  Maz'ya and Shaposhnikova \cite{MS2002} in the commutative case. Let us  record this result here.

\begin{cor}\label{limit Besov 0}
Let $1\le p\le\8$ and $1\le q<\8$. Then for $x\in  B_{p, q}^{\a_0}(\T^d_\t)$ with $\wh x(0)=0$ for some $\a_0>0$,
 $$\lim_{\a\to 0} \a^{\frac1q}\|x\|_{B^{\a,\o} _{p,q}} \approx q^{-\frac1q} \|x\|_{p}\,.$$
\end{cor}

%%%%%%%%%%%%%%%%%%%%%%%%%%%%%%%%%%%%%%%%%%%%%%%%%%%%%%%%%%%%%%%%%%%%%%%%
%%%%%%%%%%%%%%%%%%%%%%%%%%%%%%%%%%%%%%%%%%%%%%%%%%%%%%%%%%%%%%%%%%%%%%%%

\section{Interpolation of Triebel-Lizorkin spaces}

%%%%%%%%%%%%%%%%%%%%%%%%%%%%%%%%%%%%%%%%%%%%%%%%%%%%%%%%%%%%%%%%%%%%%%%%
%%%%%%%%%%%%%%%%%%%%%%%%%%%%%%%%%%%%%%%%%%%%%%%%%%%%%%%%%%%%%%%%%%%%%%%%

This short section contains some simple results on the interpolation of Triebel-Lizorkin spaces. They are similar to those for potential Sobolev spaces presented in section~\ref{Interpolation of Besov and Sobolev spaces}. It is surprising, however,  that the real interpolation spaces of $F_{p}^{\a, c}(\T^d_\t)$ for a fixed $p$ do not depend on the column structure.

\begin{prop}
 Let $1\le p, q\le\8$ and $\a_0, \a_1\in\real$ with $\a_0\neq\a_1$. Then
 $$ \big( F^{\a _0, c}_{p}(\T_{\t}^d),\, F^{\a _1, c}_{p}(\T_{\t}^d) \big)_{\eta,q}=B^{\a}_{p,q}(\T_{\t}^d),\quad  \a=(1-\eta)\a_0+\eta\a_1.$$
Similar statements hold for the row and mixture Triebel-Lizorkin spaces.
\end{prop}

\begin{proof}
The assertion is an immediate consequence of  Proposition~\ref{Triebel-P} (v) and  Proposition~\ref{interpolation-Besov} (i). Note, however, that Proposition~\ref{Triebel-P} (v) is stated for $p<\8$; but by  duality via Proposition~\ref{Triebel-dual}, it continues to hold for $p=\8$.
\end{proof}

On the other hand, the interpolation of $F_{p}^{\a, c}(\T^d_\t)$ for a fixed $\a$ is reduced to that of Hardy spaces by virtue of Proposition~\ref{Triebel-P} (iv) and Lemma~\ref{q-H-BMO}.

\begin{rk}
 Let $\a\in\real$ and $1<p<\8$. Then
 $$ \big( F^{\a , c}_{\8}(\T_{\t}^d),\, F^{\a , c}_{1}(\T_{\t}^d) \big)_{\frac1p}=F^{\a, c}_{p}(\T_{\t}^d)
 =\big( F^{\a , c}_{\8}(\T_{\t}^d),\, F^{\a , c}_{1}(\T_{\t}^d) \big)_{\frac1p\,, p}\,. $$
 \end{rk}

 \begin{prop}\label{complex inter-Triebel}
 Let $\a_0, \a_1\in\real$ and $1<p<\8$. Then
  $$\big(F^{\a_0 , c}_{\8}(\T^d_\t) ,\, F^{\a_1 , c}_{1}(\T^d_\t)\big)_{\frac1p}=F^{\a, c}_{p}(\T^d_\t) ,\quad \a=(1-\frac1p)\a_0+\frac{\a_1}p\,.$$
\end{prop}

\begin{proof}

This proof is similar to that of Theorem~\ref{complex inter-Bessel}.
Let $x$ be in the unit ball of $F^{\a, c}_{p}(\T^d_\t)$. Then by Proposition~\ref{Triebel-P}, $J^\a(x)\in \H^{c}_{p}(\T^d_\t)$. Thus by Lemma~\ref{q-H-BMO}, there exists a continuous function $f$ from the strip $S=\{z\in\com\,:\, 0\le{\rm Re}(z)\le1\}$ to $\H_1^c(\T^d_\t)$, analytic in the interior, such that $f(\frac1p)=J^\a (x)$ and such that
  $$\sup_{t\in\real}\big\|f({\rm i}t)\big\|_{\BMO^c}\le c,\;\; \sup_{t\in\real}\big\|f(1+{\rm i}t)\big\|_{\H_1^c}\le c.$$
Define
 $$F(z)=e^{(z-\frac1p)^2}\,J^{-(1-z)\a_0-z\a_1}\, f(z),\quad z\in S.$$
  By Remark~\ref{Triebel-BMO} and Lemma~\ref{complex Bessel}, for any $t\in\real$,
 $$\big\|F({\rm i}t)\big\|_{F^{\a , c}_{\8}}\approx e^{-t^2+\frac1{p^2}}\,\big\|J^{{\rm i}t(\a_0-\a_1)}\,f({\rm i}t)\big\|_{\BMO^c}\le c'.$$
Similarly,
 $$\big\|F(1+{\rm i}t)\big\|_{F^{\a , c}_{1}}\approx e^{-t^2+(1-\frac1p)^2}\,\big\|J^{{\rm i}t(\a_0-\a_1)}\,f(1+{\rm i}t)\big\|_{\H_1^c}\le c'.$$
Therefore,
 $$x=F(\frac1p)\in \big(F^{\a_0 , c}_{\8}(\T^d_\t) ,\, F^{\a_1 , c}_{1}(\T^d_\t)\big)_{\frac1p}\,,$$
whence
 $$F^{\a, c}_{p}(\T^d_\t)\subset \big(F^{\a_0 , c}_{\8}(\T^d_\t) ,\, F^{\a_1 , c}_{1}(\T^d_\t)\big)_{\frac1p}\,.$$
The converse inclusion is obtained by duality.
  \end{proof}

%%%%%%%%%%%%%%%%%%%%%%%%%%%%%%%%%%%%%%%%%%%%%%%%%%%%%%%%%%%%%%%%%%%%%%%%
%%%%%%%%%%%%%%%%%%%%%%%%%%%%%%%%%%%%%%%%%%%%%%%%%%%%%%%%%%%%%%%%%%%%%%%%

{\Large\part{Embedding}}
\setcounter{section}{0}

%%%%%%%%%%%%%%%%%%%%%%%%%%%%%%%%%%%%%%%%%%%%%%%%%%%%%%%%%%%%%%%%%%%%%%%%
%%%%%%%%%%%%%%%%%%%%%%%%%%%%%%%%%%%%%%%%%%%%%%%%%%%%%%%%%%%%%%%%%%%%%%%%

We consider the embedding problem in this chapter. We begin with  Besov spaces, then pass to Sobolev spaces. Our embedding theorem for Besov spaces is complete; however, the embedding problem of $W^1_1(\T^d_\t)$ is, unfortunately, left unsolved at the time of this writing. The last section deals with the compact embedding.
\bigskip

%%%%%%%%%%%%%%%%%%%%%%%%%%%%%%%%%%%%%%%%%%%%%%%%%%%%%%%%%%%%%%%%%%%%%%%%
%%%%%%%%%%%%%%%%%%%%%%%%%%%%%%%%%%%%%%%%%%%%%%%%%%%%%%%%%%%%%%%%%%%%%%%%

\section{Embedding of Besov spaces}

%%%%%%%%%%%%%%%%%%%%%%%%%%%%%%%%%%%%%%%%%%%%%%%%%%%%%%%%%%%%%%%%%%%%%%%%
%%%%%%%%%%%%%%%%%%%%%%%%%%%%%%%%%%%%%%%%%%%%%%%%%%%%%%%%%%%%%%%%%%%%%%%%

This section deals with the embedding  of Besov spaces. We will follow the semigroup approach developed by Varopolous \cite{Va1985} (see also \cite{Davies1989, VSCC1992}). This approach can be  adapted to the noncommutative setting, which has been done by Junge and Mei \cite{JM2010}. Here we can use either the circular Poisson or heat semigroup of $\T^d_\t$, already considered in section~\ref{The characterizations by Poisson and heat semigroups}. We choose to work with the latter. Recall that
for $x\in{\mathcal S}'(\T^d_\t)$,
 $$\mathbb{W}_r(x)= \sum_{m \in \mathbb{Z}^d } \wh{x} ( m ) r^{|m|^2} U^{m}, \quad 0 \le r < 1.$$
The following elementary lemma will be crucial.

\begin{lem}
 Let $1\le p \le p_1 \le \8$. Then
 \beq\label{RpqdforP}
\|\mathbb{W}_r(x)\|_{p_1} \les(1-r)^{\frac{d}2(\frac1{p_1}-\frac1{p})} \|x\|_p, \;\; x\in L_p(\T^d_\t),\; 0\le r<1.
\eeq
\end{lem}

\begin{proof}
 Consider first the case $p=1$ and $p_1=\8$. Then
 \be\begin{split}
\|\mathbb{W}_r(x)\|_{\8}
&\le \sum_{m\in\ent^d} r^{|m|^2} |\wh{x} ( m )|\le \|x\|_1 \sum_{m\in\ent^d} r^{|m|^2}\\
&= \|x\|_1 \sum_{k\geq 0}r^{k}\sum_{|m|^2=k} 1\les \|x\|_1 \sum_{k\geq 0}(1+k)^{\frac d2}r^{k}\\
&\approx (1-r)^{-\frac d2}\|x\|_1\,.
\end{split}\ee
The general case easily follows from this special one by interpolation. Indeed, the inequality just proved means that $\mathbb{W}_r$ is bounded from   $L_1(\T^d_\t)$ to  $L_\8(\T^d_\t)$ with norm controlled by $(1-r)^{-\frac d2}$. On the other hand, $\mathbb{W}_r$ is a contraction on $L_p(\T^d_\t)$ for $1\le p\le \8$. Interpolating these two cases, we get \eqref{RpqdforP} for $1<p<p_1=\8$. The remaining case $p_1<\8$ is treated similarly.
\end{proof}

The following is the main theorem of this section.

\begin{thm}\label{embedding-besov}
Assume that $1\leq p< p_1 \le\infty, 1\leq q\leq q_1 \leq \infty$ and $\a, \a_1\in\real$ such that $\alpha -\frac{d}p=\alpha _1-\frac{d}{p_1}$. Then we have the following continuous inclusion:
 $$B^\alpha _{p,q}(\mathbb{T}_{\theta}^d)\subset B^{\alpha _1}_{p_1,q_1}(\mathbb{T}^d_{\theta})\,\,.$$
\end{thm}

\begin{proof}
 Since $B^{\alpha _1}_{p_1,q}(\mathbb{T}^d_{\theta})\subset B^{\alpha _1}_{p_1,q_1}(\mathbb{T}^d_{\theta})$, it suffices to consider the case  $q=q_1$. On the other hand, by the lifting Theorem \ref{Besov-isom}, we can assume  $\max\{\alpha , \alpha _1\} < 0,$ so that we can take $k=0$ in Theorem \ref{circular-charct-Besov}. Thus, we are reduced to showing
 $$
\Big(\int_0^1 (1-r)^{-\frac{q\alpha _1}2} \big\|\mathbb{W}_r(x)\big\|_{p_1}^q\,  \frac{dr}{1-r}\Big)^{\frac1q}
\les \Big(\int_0^1 (1-r)^{-\frac{q\alpha}2} \big\|\mathbb{W}_r(x)\big\|_{p}^q\,  \frac{dr}{1-r}\Big)^{\frac1q}\,.
$$
To this end, we write $\mathbb{W}_r(x)=\mathbb{W}_{\sqrt r}\big(\mathbb{W}_{\sqrt r}(x)\big)$ and
apply \eqref{RpqdforP} to  get
 $$\big\|\mathbb{W}_r(x)\big\|_{p_1}
 \les (1-\sqrt r)^{\frac{d}2 (\frac{1}{p_1}-\frac{1}{p})} \big\|\mathbb{W}_{\sqrt r}(x)\big\|_p.$$
Thus
 \be\begin{split}
 \Big(\int_0^1 (1-r)^{-\frac{q\alpha _1}2} \big\|\mathbb{W}_{r}(x)\big\|_{p_1}^q\,  \frac{dr}{1-r}\Big)^{\frac1q}
 &\les \Big(\int_0^1 (1-r)^{-\frac{q\alpha _1}2} (1-\sqrt r)^{\frac{qd}2(\frac{1}{p_1}-\frac{1}{p})}  \big\|\mathbb{W}_{\sqrt r}(x)\big\|_p^q\, \frac{dr}{1-r}\Big)^{\frac1q} \\
 &= \Big(\int_0^1 (1-r^2)^{-\frac{q\alpha _1}2} (1-r)^{\frac{qd}2 (\frac{1}{p_1}-\frac{1}{p})}  \big\|\mathbb{W}_{r}(x)\big \|_p^q\,\frac{2rdr}{1-r^2}\Big)^{\frac1q}\\
 &\les \Big(\int_0^1 (1-r)^{-\frac{q\alpha}2} \big\|\mathbb{W}_{r}(x)\big \|_p^q\,\frac{dr}{1-r}\Big)^{\frac1q}\,,
 \end{split} \ee
as desired.
 \end{proof}

\begin{cor}\label{Besov-Lorentz}
Assume that $1\leq p< p_1 \le\infty, 1\le q \le\infty$ and $\alpha= d(\frac1p-\frac1{p_1})$. Then
 $$B^\alpha _{p,q}(\mathbb{T}_{\theta}^d)\subset L_{p_1,q}(\mathbb{T}^d_{\theta})\;\text{ if }\; p_1<\8\;\text{ and }\;
 B^\alpha _{p,1}(\mathbb{T}_{\theta}^d)\subset L_{\8}(\mathbb{T}^d_{\theta})\;\text{ if }\; p_1=\8\,.$$
\end{cor}

\begin{proof}
Applying the previous theorem to $\a_1=0$ and $q=q_1=1$, and by Theorem~\ref{Sobolev-Besov}, we get
  $$B^\alpha _{p,1}(\mathbb{T}_{\theta}^d)\subset B^0_{p_1,1}(\mathbb{T}^d_{\theta})\subset L_{p_1}(\mathbb{T}^d_{\theta})\,.$$
This gives the assertion in the case $p_1=\8$. For $p_1<\8$, we fix $p$ and choose two appropriate values of $\a$ (which give the two corresponding values of $p_1$); then we interpolate the resulting embeddings as above by real interpolation; finally, using \eqref{interpolation of Lp} and Proposition~\ref{interpolation-Besov}, we obtain the announced embedding for $p_1<\8$.
\end{proof}

The preceding corollary admits a self-improvement in terms of modulus of smoothness.

\begin{cor}
Assume that $1\leq p< p_1\le\infty,$  $\alpha= d(\frac1p-\frac1{p_1})$ and $k\in\nat$ such that $k>\a$. Then
 $$\o^k_{p_1}(x, \e)\les \int_0^\e \d^{-\a} \o^k_{p}(x, \d)\,\frac{d\d}\d\,,\quad 0<\e\le1.$$
\end{cor}

\begin{proof}
Without loss of generality, assume $\wh x(0)=0$. Then by the preceding corollary and Theorem~\ref{diff-Besov}, we have
 $$\|x\|_{p_1}\les \int_0^1 \d^{-\a} \o^k_{p}(x, \d)\,\frac{d\d}\d\,.$$
Now let $u\in\real^d$ with $|u|\le \e$. Noting that
$$\o^k_{p}(\D_u(x), \d)\le 2^k\min\big( \o^k_{p}(x, \e),\,  \o^k_{p}(x, \d)\big)\le  2^k \o^k_{p}(x, \min(\e,\,\d)),$$
we obtain
  \be\begin{split}
  \|\D_u(x)\|_{p_1}
  &\les \int_0^\e \d^{-\a} \o^k_{p}(x, \d)\,\frac{d\d}\d+  \e^{-\a}\o^k_{p}(x, \e)\\
  &\les \int_0^\e \d^{-\a} \o^k_{p}(x, \d)\,\frac{d\d}\d+  \int_{\frac\e2}^\e \d^{-\a} \o^k_{p}(x, \d)\,\frac{d\d}\d\\
  &\les \int_0^\e \d^{-\a} \o^k_{p}(x, \d)\,\frac{d\d}\d\,.
  \end{split} \ee
 Taking the supremum over all $u$ with $|u|\le \e$ yields the desired inequality.
 \end{proof}

\begin{rk}
 We will discuss the optimal order of the best constant of the embedding in Corollary~\ref{Besov-Lorentz} at the end of the next section.
\end{rk}

%%%%%%%%%%%%%%%%%%%%%%%%%%%%%%%%%%%%%%%%%%%%%%%%%%%%%%%%%%%%%%%%%%%%%%%%
%%%%%%%%%%%%%%%%%%%%%%%%%%%%%%%%%%%%%%%%%%%%%%%%%%%%%%%%%%%%%%%%%%%%%%%%

\section{Embedding of Sobolev spaces}

%%%%%%%%%%%%%%%%%%%%%%%%%%%%%%%%%%%%%%%%%%%%%%%%%%%%%%%%%%%%%%%%%%%%%%%%
%%%%%%%%%%%%%%%%%%%%%%%%%%%%%%%%%%%%%%%%%%%%%%%%%%%%%%%%%%%%%%%%%%%%%%%%

This section is devoted to the embedding of Sobolev spaces.  The following is our main theorem. Recall that $B^{\a_1}_{\8,\8}(\mathbb{T}^d_{\theta})$ in the second part below  is the quantum analogue of the classical Zygmund class of order $\a_1$ (see Remark~\ref{Lip-Besov}).

\begin{thm}\label{embedding-sobolev}
 Let  $\a,\a_1\in\real$ with $\a>\a_1$.
 \begin{enumerate}[\rm (i)]
 \item If  $1<p <p_1<\8$ are such that $\a -\frac dp=\alpha _1-\frac d{p_1}$, then
$$H^\a _p(\T^d_\t)\subset H^{\a_1} _{p_1}(\T^d_\t) \;\text{ continuously}.$$
In particular, if additionally $\a=k$ and $\a_1=k_1$ are nonnegative integers,  then
 $$W^k_p(\T^d_\t)\subset W^{k_1} _{p_1}(\T^d_\t)  \;\text{ continuously}.$$
 \item If $1\le p<\8$ is such that $p(\a-\a_1)>d$ and $\a_1=\a-\frac{d}{p}$, then
 $$H^\alpha _p(\T^d_\t)\subset B^{\a_1}_{\8,\8}(\T^d_\t)  \;\text{ continuously}.$$
In particular, if additionally $\a=k\in\nat$, and if either $p>1$ or $p=1$ and $k$ is even, then
$$W^k _p(\T^d_\t)\subset B^{\a_1}_{\8,\8}(\T^d_\t)  \;\text{ continuously}.$$
\end{enumerate}
\end{thm}

\begin{proof}
 (i) By Theorem \ref{q-Sobolev-Bessel}, the embedding of $W^k_p(\mathbb{T}_{\theta}^d)$ is a special case of that of $H^\alpha _p(\mathbb{T}_{\theta}^d).$
 Thus we just deal with the potential spaces $H^\alpha _p(\mathbb{T}_{\theta}^d)$. On the other hand, by the lifting property of potential Sobolev spaces, we can assume $\a_1=0$. By Theorem \ref{Sobolev-Besov} and Corollary~\ref{Besov-Lorentz}, we have
  $$ H_{p}^\alpha (\mathbb{T}^d_{\theta})\subset L_{q,\infty}(\mathbb{T}^d_{\theta}).$$
Now choose $0<\eta<1$ and two indices $s_0, s_1$ with $1<s_0, s_1<\frac d\a$ such that
 $$\frac1p=\frac{1-\eta}{s_0}+\frac\eta{s_1}\,.$$
 Let 
  $$\frac1{t_j}=\frac1{s_j}-\frac\a{d},\quad j=0, 1.$$ 
Then interpolating the above inclusions with $s_j$ in place of $p$ for $j=0, 1$,  using Remark~\ref{inter potential-1a} and \eqref{interpolation of Lp}, we get
 $$
  H_{p}^\a(\mathbb{T}^d_{\theta})
  =\big(  H_{s_0}^\a(\mathbb{T}^d_{\theta}),  H_{s_1}^\alpha (\mathbb{T}^d_{\theta}) \big)_{\eta, p}
  \subset \big( L_{t_0,\infty}(\mathbb{T}^d_{\theta}), L_{t_1,\infty}(\mathbb{T}^d_{\theta}) \big)_{\eta,p}=L_{p_1,p}(\mathbb{T}^d_{\theta})\subset L_{p_1}(\mathbb{T}^d_{\theta}).
 $$

(ii) By Theorems \ref{Sobolev-Besov} and \ref{embedding-besov}, we obtain
 $$H^\alpha _p(\mathbb{T}_{\theta}^d)\subset B^{\a}_{p,\8}(\mathbb{T}^d_{\theta})  \subset B^{\a_1}_{\8,\8}(\mathbb{T}^d_{\theta})\,.$$
If $k$ is even, $W^k _1(\T^d_\t)\subset H^k _1(\T^d_\t)$. Thus the theorem is proved.
\end{proof}

\begin{rk}
 The case $p\a=d$ with $\a_1=0$ is excluded from the preceding theorem. In this case, it is easy to see that $H^\alpha _p(\mathbb{T}_{\theta}^d)\subset L_{q}(\mathbb{T}^d_{\theta}) $ for any  $q<\8$. It is well known in the classical case that this embedding is false for $q=\8$. Consider, for instance, the ball $B=\{s\in\real^d: |s|\le \frac14\}$ and the function $f$ defined by $f(s)=\log\log(1+\frac1{|s|})$.
Then $f$ belongs to $W^d_1(B)$ but is unbounded on $B$. Now extending $f$ to a $1$-periodic function on $\real^d$  which is infinitely differentiable in $[-\frac12,\, \frac12]^d\setminus B$, we obtain a function in $W_1^d(\T^d)$ but unbounded on $\T^d$.
\end{rk}

\begin{rk}
 Part (ii) of the preceding theorem implies $W^d _p(\T^d_\t)\subset L_{\8}(\T^d_\t)$ for all $p>1$. In the commutative case, representing a function as an indefinite  integral of its derivatives, one easily checks that this embedding remains true for $p=1$. However, we do not know how to prove it in the noncommutative case. A related question concerns the embedding $W^k _p(\T^d_\t)\subset B^{\a_1}_{\8,\8}(\T^d_\t)$ in  the case of odd $k$ which is not covered by the same part (ii). 
\end{rk}

The quantum analogue of the Gagliardo-Nirenberg inequality can be also proved easily by interpolation.

\begin{prop}
Let $k\in \mathbb{N}, 1<r, p<\8, 1\le q<\8$ and $\beta\in \mathbb{N}_0^d$ with $0<|\beta|_1<k.$ If
 $$
 \eta=\frac{|\beta|_1}{k}\;\; \mbox{and} \;\; \frac{1}{r}= \frac{1-\eta}{q} + \frac{\eta}{p},
 $$
then for every $x\in W_p^k(\mathbb{T}_\theta^d)\cap L_q(\mathbb{T}_\theta^d)$,
 $$
 \|D^\beta x\|_r \les \|x\|_q^{1-\eta} \big(\sum_{|m|=k}\|D^m x\|_p\big)^\eta\,.
 $$
\end{prop}

\begin{proof} This inequality immediately follows from Theorem~\ref{complex inter-Sobolev}  and the well-known relation between real and complex interpolations:
  $$\big(L_q(\T^d_\t),\, W^k_p(\T^d_\t)\big)_{\eta, 1}\subset \big(L_q(\T^d_\t),\, W^k_p(\T^d_\t)\big)_{\eta}=W^{|\b|_1}_r(\T^d_\t).$$
It then follows that
  $$\|x\|_{W^{|\b|_1}_r}\les  \|x\|_q^{1-\eta}\, \|x\|_{W^{k}_p}^\eta\,.$$
 Applying this inequality to $x-\wh x(0)$ instead of $x$ and using Theorem~\ref{k-Sobolev}, we get the desired Gagliardo-Nirenberg inequality.
  \end{proof}

\n{\bf An alternate approach to Sobolev embedding.}  Note that the preceding proof of Theorem~\ref{embedding-sobolev} is based on Theorem~\ref{embedding-besov}, which is, in its turn, proved by Varopolous' semigroup approach. Varopolous initially developed his method for the Sobolev embedding, which was transferred to the noncommutative setting by Junge and Mei \cite{JM2010}. Our argument for the embedding of Besov spaces has followed this route. Let us now give an alternate proof of Theorem~\ref{embedding-sobolev} (i) by the same way. We state its main part as the following  lemma  that is  of interest in its own right.

\begin{lem}\label{weak embedding}
 Let $1\le p<q<\8$ such that $\frac1q=\frac1p-\frac 1d$. Then
 $$ W^1_p(\T^d_\t) \subset L_{q,\8}(\T^d_\t).$$
\end{lem}

\begin{proof}
We will use again the heat semigroup $\mathbb{W}_r$ of $\T^d_\t$. Recall that $\mathbb{W}_r=\wt{ \mathrm{W}}_\e$ with $r=e^{-4\pi^2\e}$, where $\wt{\mathrm{W}}_\e$ is the periodization of the usual heat kernel $\mathrm{W}_\e$ of $\real^d$ (see section~\ref{The characterizations by Poisson and heat semigroups}).
It is more convenient to work with  $\wt{ \mathrm{W}}_\e$. In the following, we assume  $x\in\mathcal{S}(\T^d_\t)$ and $\wh x(0)=0$. Let $\D_j=\Delta^{-1} \partial_j, \,1\le j\le d$.  Then
 $$\D^{-1}x=4\pi^2\int_0^{\8}\wt{ \mathrm{W}}_\e(x)d\e\;\text{ and }\; \D_j x =4\pi^2\int_0^{\infty} \wt{ \mathrm{W}}_\e( \partial_j x) \,d\e.$$
We claim that for any $1\le p \le\8$
 \beq\label{1to1}
 \|\wt{\mathrm{W}}_\e (\partial_j x)\|_{p}\les  \e^{-\frac{1}{2}}\|x\|_p
 \;\text{ and }\;\|\wt{\mathrm{W}}_\e (\partial_j x)\|_{\8} \les \e^{-\frac12(\frac dp+1)}\|x\|_p\,,\quad \e>0.
 \eeq
Indeed, in order to prove the first inequality, by  the transference method, it suffices to show a similar one for the Banach space valued heat semigroup of the usual $d$-torus. The latter immediately follows from the following standard estimate on the heat kernel $\mathrm{W}_\e$ of $\real^d$:
 $$\sup_{\e>0} \e^{\frac12}\int_{\real^d} \big|\nabla\mathrm{W}_\e(s)\big|\,ds<\8.$$
The second inequality of  \eqref{1to1} is proved in the same way as \eqref{RpqdforP}. First, for the case $p=1$, we have  (recalling that $\wh x(0)=0$)
 \be \begin{split}
 \|\wt{\mathrm{W}}_\e (\partial_jx)\|_{\8}
 &\leq 2\pi\sum_{m\in\ent^d\setminus\{0\}} |m_j| e^{-\e |m|^2} |\wh{x}(m)| \\
 &\leq 2\pi \|x\|_1 \sum_{m\in\ent^d\setminus\{0\}} |m_j| e^{-\e |m|^2}\\
 & \les e^{-\e}(1-e^{-\e})^{-\frac{d+1}{2}}\|x\|_1 \les \e^{-\frac{d+1}{2}}\|x\|_1.
  \end{split}\ee
 Interpolating this with the first inequality for $p=\8$, we get the second one in the general case.

Now let $\e>0$ and decompose $\D_j x$ into the following two parts:
 $$y=4\pi^2\int_\e^{\infty} \wt{\mathrm{W}}_\d (\partial_j x)\, d\d \;\text{ and }\; z=4\pi^2\int_0^\e \wt{\mathrm{W}}_\d (\partial_j x) \,d\d.$$
Then by \eqref{1to1},
 $$
 \|y\|_{\infty}\les \|x\|_p \int_\e^{\infty}  \d^{-\frac12(\frac dp +1)} \,d\d \approx  \e^{-\frac12(\frac dp -1)}\|x\|_p
 $$
and
 $$
 \|z\|_p\les  \|x\|_p \int_0^\e \d^{-\frac{1}{2}} \,d\d\approx \e^{\frac{1}{2}} \|x\|_p.
 $$
Thus for any  $t>0,$ choosing $\e$ such that $\e^{-\frac{d}{2p}}=t$, we deduce
 $$\|y\|_{\infty}+ t \|z\|_p \les  t^{1-\frac{p}{d}}\|x\|_p=t^{\eta}\|x\|_p\,,$$
where $\eta=1-\frac{p}{d}$.
It then follows that
 $$\|\D_j x\|_{q,\8}\approx \|\D_j x\|_{(L_\infty(\T^d_\theta),\, L_p(\T^d_\theta))_{\eta, \infty}}
 \les \|x\|_p\,.$$
 Since
  $$x=-\sum_{j=1}^d \Delta_j \partial_j x,$$
we finally get
 $$\|x\|_{q,\infty}\les\sum_{j=1}^d \|\Delta_j \partial_j x\|_{q,\infty}\les\sum_{j=1}^d\|\partial_j x\|_p =\|\nabla x\|_p.$$
Thus the lemma is proved.
\end{proof}

\begin{proof}[Alternate proof of Theorem~\ref{embedding-sobolev} {\rm(i)}]
 For $1<p<d$, choose $p_0, p_1$ such that $1<p_0<p<p_1<d$. Let $\frac1{q_i}=\frac1{p_i}-\frac1{d}$ for $i=0, 1$. Then by the previous lemma,
  $$W^1_{p_i}(\T^d_\t)\subset L_{q_i,\8}(\T^d_\t),\quad i=0, 1.$$
 Interpolating these two inclusions by real  method, we obtain
  $$W^1_{p}(\T^d_\t)\subset L_{q,p}(\T^d_\t).$$
 This is the embedding of Sobolev spaces in Theorem~\ref{embedding-sobolev} (i) for $k=1$. The case $k>1$ immediately follows by iteration. Then using real interpolation, we deduce the embedding of potential Sobolev spaces.
 \end{proof}

\n{\bf Sobolev embedding for $p=1$.} Now we discuss the case $p=1$ which  is not covered by Theorem~\ref{embedding-sobolev}. The main problem concerns the following:
 \beq\label{embedding 1}
 W^1_1(\T^d_\t)\subset L_{\frac d{d-1}}(\T^d_\t).
 \eeq
 At the time of this writing, we are unable, unfortunately, to prove it. However, Lemma~\ref{weak embedding} provides a weak substitute, namely,
 \beq\label{weak embedding1}
 W^1_1(\T^d_\t)\subset L_{\frac d{d-1},\8}(\T^d_\t).
 \eeq
 In the classical case, one can rather easily deduce \eqref{embedding 1} from \eqref{weak embedding1}. Let us explain the idea  coming from \cite[page~58]{VSCC1992}. It was kindly pointed out to us by Marius Junge. Let $f$ be a nice real function on $\T^d$ with $\wh f(0)=0$. For any $t\in\real$ let $f_t$ be the indicator function of the subset $\{f>t\}$. Then $f$ can be decomposed as an integral of the $f_t$'s:
 \beq\label{dec of f}
 f=\int_{-\8}^{+\8} f_t\,dt.
  \eeq
By triangular inequality (with $q=\frac d{d-1}$),
 $$ \|f\|_q\le\int_{-\8}^{+\8} \|f_t\|_q\,dt.$$
However,
 $$ \|f_t\|_q= \|f_t\|_{q,\8}\,\quad \forall\, t\in\real.$$
Thus by \eqref{weak embedding1} for $\t=0$,
 $$ \|f_t\|_q\les \|f_t\|_1+ \|\nabla f_t\|_1\,.$$
It comes now the crucial point which is the following
 \beq\label{dec of f'}
 \int_{-\8}^{+\8} \|\nabla f_t\|_1\,dt\les \|\nabla f\|_1 \,.
 \eeq
In fact, the two sides are equal in view of Sard's theorem. We then get the strong embedding \eqref{embedding 1} in the case $\t=0$.  Note that this proof yields a stronger embedding:
 \beq\label{embedding 1b}
 W^1_1(\T^d)\subset L_{\frac d{d-1}, 1}(\T^d).
 \eeq

The above decomposition of $f$ is not smooth in the sense that $f_t$ is not derivable even though $f$ is nice. In his proof of Hardy's inequality in Sobolev spaces, Bourgain \cite{bourg-Hardy} discovered independently the same decomposition but using nicer functions $f_t$ (see also \cite{Po1983}).  Using \eqref{embedding 1b} and the Hausdorff-Young inequality, Bourgain derived the following Hardy type inequality (assuming $d\ge3$):
 $$\sum_{m\in\ent^d}\frac{|\wh f(m)|}{(1+|m|)^{d-1}}\les \|f\|_{W^1_1(\T^d)}\,.$$

We have encountered difficulties in the attempt of extending this approach to the noncommutative case. Let us formulate the corresponding open problems explicitly as follows:

\begin{problem}\label{embedding p=1}
Let $d\ge2$.
\begin{enumerate}[\rm(i)]
\item Does one have the following embedding
 $$W^1_1(\T^d_\t)\subset L_{\frac d{d-1}}(\T^d_\t)\;\text{ or }\; W^1_1(\T^d_\t)\subset L_{\frac d{d-1}, 1}(\T^d_\t)\,?$$
\item Does one have the following inequality
 $$\sum_{m\in\ent^d}\frac{|\wh x(m)|}{(1+|m|)^{d-1}}\les \|x\|_{W^1_1(\T^d_\t)}\,?$$
\end{enumerate}
\end{problem}

By the previous discussion, part (i)  is reduced to a decomposition for operators in $W^1_1(\T^d_\t)$  of the form \eqref{dec of f} and satisfying  \eqref{dec of f'}. One could attempt to do this by transference by first considering operator-valued functions on $\real^d$. With this in mind, the following observation, due to Marius Junge, might be helpful.

Given an interval $I=[s,\, t]\subset\real$ and an element $a\in L_1(\T^d_\t)$, we have
 $$\partial(\un_{I} \ot a)= \d_{s}\ot a-\d_{t}\ot a,$$
where $\partial$ denotes the distribution derivative relative to $\real$ and $\d_{s}$ is the Dirac measure at $s$.
Let $\|\cdot\|_L$ denote the norm of the dual space $C_0(\real; \mathcal{A}_\t)^*$, which contains $L_1(\mathbb{R}; L_1(\T_{\theta}^d))$ isometrically. If $f$ is a (nice) linear combination of  $\un_{I} \ot a$'s, then we have the desired decomposition of $f$. Indeed, assume $f=\sum_{i=1}^n \a_i \un_{I_i}\otimes e_i$, where
$\a_i  \in \real_+$ and the $\un_{I_i}\otimes e_i$'s are pairwise disjoint projections of $L_\8(\real)\overline\ot\T^d_\t$. Let $f_t=\un_{(t,\infty)}(f)$. Then
 $$f=\int_0^\infty f_t \,dt.$$
So for any $q\ge1$,
 $$ \|f\|_{q}\leq \int_0^\infty \|f_t\|_{q} dt.$$
On the other hand, by writing explicitly $f_t$ for every $t$, one easily checks
$$ \|\partial f\|_L =\int_0^{\infty} \|\partial f_t\|_L\, dt.$$
By iteration, the above decomposition can be extended to higher dimensional case for all functions $f$ of the form $\sum_{i=1}^n \a_i \un_{R_i}\otimes e_i$, where
$\a_i  \in \real_+$, $R_i$'s are rectangles (with sides parallel to the axes) and  $\un_{R_i}\otimes e_i$'s are pairwise disjoint projections of $L_\8(\real^d)\overline\ot\T^d_\t$.

The next  idea would be to apply Lemma~\ref{weak embedding} to these special functions. Then two difficulties come up to us, even in the commutative case. The first is that these functions do not belong to $W_1^1$; this difficulty can be resolved quite easily by regularization. The second one, substantial, is the density of these functions, more precisely, of suitable regularizations of them, in $W_1^1$.

\medskip\n{\bf Uniform Besov embedding.} We end this section with a discussion on the link between a certain uniform embedding of Besov spaces and the embedding of Sobolev spaces. Let $0<\a<1$, $1\le p<\8$ with $\a p<d$ and $ \frac1r=\frac1p-\frac\a{d}$. Then
 \beq\label{uniform embedding of besov}
 \|x\|_r^p\le c_{d, p}\,\frac{\a(1-\a)}{(d-\a p)}\, \|x\|_{B^{\a, \o}_{p, p}}^p\,, \quad x\in B^\a_{p, p}(\T^d_\t)\,,
 \eeq
where $ \|x\|_{B^{\a, \o}_{p, p}}$ is the Besov norm defined by \eqref{besov diff norm}. In the commutative case, this inequality is proved in  \cite{BBM2002} for $\a$ close to $1$ and in \cite{MS2002} for general $\a$. One can show that \eqref{uniform embedding of besov} is essentially equivalent to the embedding of $W^1_p(\T^d_\t)$ into $L_q(\T^d_\t)$ (or $L_{q,p}(\T^d_\t)$) for $d>p$ and $\frac1q=\frac1p-\frac1d$. Indeed, assume\eqref{uniform embedding of besov}. Then taking limit in both sides of  \eqref{uniform embedding of besov} as $\a\to1$, by Theorem~\ref{BBM}, we get
 $$\|x\|_q\les |x|_{W^1_p}$$
for all $x\in W^1_p(\T^d_\t)$ with $\wh x(0)=0$. Conversely, if $W^1_p(\T^d_\t)\subset L_q(\T^d_\t)$,  then
 $$ \big(L_p(\T^d_\t),\, W^1_p(\T^d_\t)\big)_{\a, p}\subset \big(L_p(\T^d_\t),\, L_q(\T^d_\t)\big)_{\a, p}\,.$$
Theorem~\ref{K-functional} implies that
 $$\big(L_p(\T^d_\t),\, W^1_p(\T^d_\t)\big)_{\a, p}\subset B^\a_{p, p}(\T^d_\t)$$
with relevant constant depending only on $d$, here $B^\a_{p, p}(\T^d_\t)$ being equipped with the norm $\|\,\|_{B^{\a, \o}_{p, p}}$. On the other hand, By a classical  result  of Holmstedt \cite{Hol1970} on real interpolation of $L_p$-spaces (which readily extends to the noncommutative case, as observed in \cite[Lemma~3.7]{JX2007}),
 $$\big(L_p(\T^d_\t),\, L_q(\T^d_\t)\big)_{\a, p}\subset L_{r,p}(\T^d_\t)$$
with the inclusion constant uniformly controlled by $\a^{\frac1p}(1-\a)^{\frac1q}$. We then deduce
 $$\|x\|_{r, p}\les\a^{\frac1p}(1-\a)^{\frac1q}\|x\|_{B^{\a, \o}_{p, p}}.$$
This implies a variant of \eqref{uniform embedding of besov} since $L_{r,p}(\T^d_\t)\subset L_{r}(\T^d_\t)$.

Since we have proved the embedding $W^1_p(\T^d_\t)\subset L_q(\T^d_\t)$ for $p>1$, \eqref{uniform embedding of besov} holds for $p>1$. Let us record this explicitly  as follows:

\begin{prop}
 Let $0<\a<1$, $1< p<\8$ with $\a p<d$ and $ \frac1r=\frac1p-\frac\a{d}$. Then
  $$
 \|x\|_r\les\,\big(\a(1-\a)\big)^{\frac1p}\|x\|_{B^{\a, \o}_{p, p}}^p\,, \quad x\in B^\a_{p, p}(\T^d_\t)
 $$
with relevant constant independent  of $\a$.
\end{prop}

In the case $p=1$, Problem~\ref{embedding p=1} (i) is equivalent to \eqref{uniform embedding of besov} for $p=1$ and $\a$ close to $1$.

%%%%%%%%%%%%%%%%%%%%%%%%%%%%%%%%%%%%%%%%%%%%%%%%%%%%%%%%%%%%%%%%%%%%%%%%
%%%%%%%%%%%%%%%%%%%%%%%%%%%%%%%%%%%%%%%%%%%%%%%%%%%%%%%%%%%%%%%%%%%%%%%%

\section{Compact embedding}

%%%%%%%%%%%%%%%%%%%%%%%%%%%%%%%%%%%%%%%%%%%%%%%%%%%%%%%%%%%%%%%%%%%%%%%%
%%%%%%%%%%%%%%%%%%%%%%%%%%%%%%%%%%%%%%%%%%%%%%%%%%%%%%%%%%%%%%%%%%%%%%%%

This section deals with the compact embedding. The case $p=2$ for potential Sobolev spaces was solved by Spera \cite{Spera1992}:

\begin{lem}\label{spera}
The embedding $H^{\alpha_1}_2(\mathbb{T}_{\theta}^d) \hookrightarrow H^{\alpha_2}_2(\mathbb{T}_{\theta}^d)$ is compact for $\alpha_1>\alpha_2 \geq0$.
\end{lem}

We will require the following real interpolation result on compact operators, due to Cwikel \cite{Cwikel1992}.

\begin{lem}\label{cwikel}
Let $(X_0,\, X_1)$ and $(Y_0,\, Y_1)$ be two interpolation couples of Banach spaces, and let $T: X_j \to Y_j$ be a bounded linear operator, $j=0, 1$. If $T: X_0 \to Y_0 $ is compact, then $T: (X_0,\, X_1)_{\eta, p} \to (Y_0,\, Y_1)_{\eta, p} $ is compact too for  any $0<\eta<1$ and $1\le p\le\8$.
\end{lem}

\begin{thm} 
Assume that $1\leq p<p_1 \le\infty, \, 1\le p^*<p_1,\, 1\leq q\leq q_1 \leq \infty$ and $\alpha -\frac{d}p=\alpha _1-\frac{d}{p_1}$. Then the embedding
 $B^\a _{p,q}(\T_{\t}^d) \hookrightarrow  B^{\a_1}_{p^*,q_1}(\T^d_{\t})\,$
is compact.
\end{thm}

\begin{proof}
 Without loss of generality, we can assume $q=q_1$. First consider the case $p=2$. Choose $t$ sufficiently close to $q$ and $0<\eta<1$ such that
  $$\frac1q=\frac{1-\eta}2+\frac{\eta}{t}\,.$$
 Then by Proposition~\ref{interpolation-Besov},
  $$B^\a _{2,q}(\T_{\t}^d)=\big(B^\a _{2,2}(\T_{\t}^d),\, B^\a _{2,t}(\T_{\t}^d)\big)_{\eta, q}\,.$$
By Lemma~\ref{spera}, $B^\a _{2,2}(\T_{\t}^d)\hookrightarrow B^{\a_1} _{2,2}(\T_{\t}^d)$ is compact. On the other hand, by Theorem~\ref{embedding-besov}, $B^\a _{2,t}(\T_{\t}^d)\hookrightarrow B^{\a_1} _{p_1,t}(\T_{\t}^d)$ is continuous. So by Lemma~\ref{cwikel},
 $$B^\a _{2,q}(\T_{\t}^d)\hookrightarrow \big(B^{\a_1} _{2,2}(\T_{\t}^d),\, B^{\a_1} _{p_1,t}(\T_{\t}^d)\big)_{\eta, q}\;\text{ is compact}.$$
However, by the proof of Proposition~\ref{interpolation-Besov} and \eqref{interpolation of Lp}, we have
 $$\big(B^{\a_1} _{2,2}(\T_{\t}^d),\, B^{\a_1} _{p_1,t}(\T_{\t}^d)\big)_{\eta, q}\subset \el^{\a_1}_q\big((L_2(\T^d_\t),\, L_{p_1}(\T^d_\t)_{\eta, q}\big)
 = \el^{\a_1}_q(L_{s, q}(\T^d_\t)),$$
where $s$ is determined by
 $$\frac1{s}=\frac{1-\eta}2+\frac{\eta}{p_1}=\frac1{p_1}+\frac{(1-\eta)(\a-\a_1)}d\,.$$
Note that  $\eta$ tends to $1$ as $t$ tends to $q$. Thus we can choose $\eta$ so that $s>p^*$. Then $L_{s, q}(\T^d_\t)\subset L_{p^*}(\T^d_\t)$. Thus the desired assertion for $p=2$ follows.

The case $p\neq2$ but $p>1$ is dealt with similarly. Let $t$ and $\eta$ be as above. Choose $r<p$ ($r$ close to $p$). Then
 $$\big(B^{\a} _{2,2}(\T_{\t}^d),\, B^{\a} _{r,t}(\T_{\t}^d)\big)_{\eta, q}\subset\el^{\a}_q\big((L_2(\T^d_\t),\, L_{r}(\T^d_\t)_{\eta, q}\big)
 = \el^{\a}_q(L_{p_0, q}(\T^d_\t)),$$
where $p_0$ is determined by
 $$\frac1{p_0}=\frac{1-\eta}2+\frac{\eta}{r}\,.$$
If $\eta$ is sufficiently close to $1$, then $p_0<p$ that we will assume. Thus $L_{p}(\T^d_\t)\subset L_{p_0, q}(\T^d_\t)$. It then follows that
 $$B^{\a} _{p,q}(\T_{\t}^d)\subset\big(B^{\a} _{2,2}(\T_{\t}^d),\, B^{\a} _{r,t}(\T_{\t}^d)\big)_{\eta, q}\,.$$
The rest of the proof is almost the same as the case $p=2$, so is omitted.

The remaining case $p=1$ can be easily reduced to the previous one. Indeed, first embed $B^\a _{p,q}(\T_{\t}^d)$ into $B^{\a_2} _{p_2,q}(\T_{\t}^d)$ for some $\a_2\in (\a,\, \a_1)$ ($\a_2$ close to $\a$) and $p_2$ determined by $\alpha -\frac{d}p=\alpha _2-\frac{d}{p_2}$. Then by the previous case, the embedding $B^{\a_2} _{p_2,q}(\T_{\t}^d) \hookrightarrow  B^{\a_1}_{p^*,q_1}(\T^d_{\t})\,$ is compact, so we are done.
\end{proof}

\begin{thm}
 Let  $1<p <p_1<\8$ and $\a,\a_1\in\real$.
 \begin{enumerate}[\rm (i)]
 \item If $\alpha -\frac{d}p=\alpha _1-\frac{d}{p_1}$, then
$H^\a _p(\T^d_\t)\hookrightarrow H^{\a_1} _{p^*}(\T^d_\t)$ is compact for $p^*<p_1$.
In particular, if additionally $\a=k$ and $\a_1=k_1$ are nonnegative integers,  then
 $W^k _p(\T^d_\t)\hookrightarrow W^{k_1} _{p^*}(\T^d_\t)$ is compact.
 \item If $p(\a-\a_1)>d$ and $\a^*<\a_1=\a-\frac{d}{p}$, then
 $H^\alpha _p(\T^d_\t)\hookrightarrow  B^{\a^*}_{\8,\8}(\T^d_\t)$ is compact.
In particular, if additionally $\a=k\in\nat$, then
$W^k _p(\T^d_\t)\hookrightarrow  B^{\a^*}_{\8,\8}(\T^d_\t)$ is compact.
\end{enumerate}
\end{thm}

\begin{proof}
Based on the preceding theorem, this proof is similar to that of Theorem~\ref{embedding-sobolev} and left to the reader.
 \end{proof}

%%%%%%%%%%%%%%%%%%%%%%%%%%%%%%%%%%%%%%%%%%%%%%%%%%%%%%%%%%%%%%%%%%%%%%%%
%%%%%%%%%%%%%%%%%%%%%%%%%%%%%%%%%%%%%%%%%%%%%%%%%%%%%%%%%%%%%%%%%%%%%%%%

{\Large\part{Fourier multiplier}}
\setcounter{section}{0}

%%%%%%%%%%%%%%%%%%%%%%%%%%%%%%%%%%%%%%%%%%%%%%%%%%%%%%%%%%%%%%%%%%%%%%%%
%%%%%%%%%%%%%%%%%%%%%%%%%%%%%%%%%%%%%%%%%%%%%%%%%%%%%%%%%%%%%%%%%%%%%%%%

This chapter deals with Fourier multipliers on Sobolev, Besov and Triebel-Lizorkin spaces on $\T^d_\t$. The first section concerns the Sobolev spaces. Its main result is the analogue for $W_p^k(\T^d_\t)$ of \cite[Theorem~7.3]{CXY2012} (see also Lemma~\ref{cb-multiplier}) on c.b. Fourier multipliers on $L_p(\T^d_\t)$; so the space of c.b. Fourier multipliers on $W_p^k(\T^d_\t)$ is independent of $\t$. The second section turns to Besov spaces on which Fourier multipliers behave better. We extend some classical results to the present setting. We show that the space of c.b. Fourier multipliers on $B_{p,q}^\a(\T^d_\t)$ does not depend on $\t$ (nor on $q$ or $\a$). We also prove that a function on $\ent^d$ is a Fourier multiplier on $B_{1,q}^\a(\T^d_\t)$  iff it is the Fourier transform of an element of $B_{1,\8}^0(\T^d)$. The last section deals with Fourier multipliers on Triebel-Lizorkin spaces.

\bigskip

%%%%%%%%%%%%%%%%%%%%%%%%%%%%%%%%%%%%%%%%%%%%%%%%%%%%%%%%%%%%%%%%%%%%%%%%
%%%%%%%%%%%%%%%%%%%%%%%%%%%%%%%%%%%%%%%%%%%%%%%%%%%%%%%%%%%%%%%%%%%%%%%%

\section{Fourier multipliers on Sobolev spaces}

%%%%%%%%%%%%%%%%%%%%%%%%%%%%%%%%%%%%%%%%%%%%%%%%%%%%%%%%%%%%%%%%%%%%%%%%
%%%%%%%%%%%%%%%%%%%%%%%%%%%%%%%%%%%%%%%%%%%%%%%%%%%%%%%%%%%%%%%%%%%%%%%%

We now investigate Fourier multipliers on Sobolev spaces. We refer to \cite{Po1982, BoPo1987} for the study of Fourier multipliers on the classical Sobolev spaces. If $X$ is a Banach space of distributions on $\T^d_\t$, we denote by $\mathsf{M} (X)$ the space of bounded Fourier multipliers on $X$; if $X$ is further equipped with an operator space structure, $\mathsf{M}_{\mathrm{cb}}(X)$ is the space of c.b. Fourier multipliers on $X$. These spaces are endowed with their natural norms. Recall that the Sobolev spaces $W^k_p(\T^d_{\t})$,  $H^\a_p(\T^d_{\t})$ and the Besov $B_{p, q}^\a(\T^d_{\t})$  are equipped with their natural operator space structures as defined in Remarks~\ref{os-Sobolev} and \ref{os-Besov}.

The aim of this section is to extend  \cite[Theorem~7.3]{CXY2012} (see also Lemma~\ref{cb-multiplier}) on c.b. Fourier multipliers on $L_p(\T^d_\t)$ to Sobolev spaces.
Inspired by Neuwirth and Ricard's transference theorem \cite{NR2011}, we will relate Fourier multipliers with Schur multipliers. Given a distribution $x$ on $\T^d_\t$, we write its matrix in the basis $(U^m)_{m\in\ent^d}$:
 $$[x]=\Big( \langle xU^n,\;U^m\rangle \Big)_{m,n\in \mathbb{Z}^d}
 =\Big(\wh{x}(m-n)e^{\mathrm{i} n \tilde{\theta} (m-n)^t}\Big )_{m,n\in\mathbb{Z}^d}\,.$$
Here  $k^t$ denotes the transpose of $k=(k_1,\dots,k_d)$ and  $\tilde{\theta}$ is the following $d\times d$-matrix deduced from the skew symmetric matrix $\theta$:
 \[\tilde{\theta}=-2\pi\begin{pmatrix}
 0 & \theta_{12} & \theta_{13} &\dots & \theta_{1d}\\
 0 & 0           & \theta_{23} &\dots & \theta_{2d}\\
 \vdots&\vdots&\vdots&\vdots&\vdots\\
 0 & 0           & 0           &\dots & \theta_{d-1,d}\\
 0 & 0           & 0           &\dots & 0
 \end{pmatrix}.\]
Now let $\phi:\ent^d\to\com$ and $M_\phi$ be the associated Fourier multiplier on $\T^d_\t$. Set $\mathring{\phi}=\big(\phi_{m-n}\big)_{m, n\in\ent^d}$.  Then
 \beq\label{F-S}
 \big[M_{\phi}x\big]=\big(\phi_{m-n}\wh{x}(m-n)e^{\mathrm{i} n\tilde{\theta}(m-n)^t}\big)_{m,n\in\mathbb{Z}^d}
 =S_{\mathring{\phi}}([x]),
 \eeq
where $S_{\mathring{\phi}}$ is the Schur multiplier with symbol $\mathring{\phi}$.

According to the definition of $W_p^k(\T^d_\t)$, for any matrix $a=(a_{m, n})_{m,n\in\mathbb{Z}^d}$ and $\el\in\nat_0^d$ define
 $$D^\el a=\big((2\pi{\rm i}(m-n))^\el a_{m, n}\big)_{m,n\in\mathbb{Z}^d}\,.$$
If $x$ is a distribution on $\T^d_\t$, then clearly
 $$\big[M_{\phi}D^\el x\big]=S_{\mathring{\phi}}\big(D^\el[x]\big).$$
 We introduce the space
 $$S_{p}^k=\Big\{a=(a_{m, n})_{m,n\in\mathbb{Z}^d} \,:\, D^\el a\in S_p(\ell_2(\mathbb{Z}^d)),\; \forall \, \el \in \mathbb{N}_0^d, 0\leq |\el |_1 \leq k \Big\}$$
and endow it with the norm
 $$\|a\|_{S_{p}^k}=\Big(\sum_{0\leq |\el|_1\leq k}\|D^\el a\|_{S_p}^p\Big)^{\frac1p}.$$
 Then $S_{p}^k$ is a closed subspace of the $\el_p$-direct sum of $L$ copies of $S_p(\ell_2(\mathbb{Z}^d))$ with $L=\sum_{0\leq |\el |_1 \leq k}\,1$. The latter direct sum is equipped with its natural operator space structure, which induces an operator space structure on $S_p^k$ too.

 If $\psi=(\psi_{m, n})_{m,n\in\ent^d}$ is a complex matrix, its associated Schur multiplier $S_\psi$ on $S_p^k$ is defined by $S_\psi a=(\psi_{m, n}\,a_{m, n})_{m,n\in\ent^d}$.  Let $\mathsf{M}_{\mathrm{cb}}(S_p^k)$ denote the space of all c.b. Schur multipliers on $S_p^k$, equipped with the natural norm.

\medskip

\begin{thm}\label{q-multiplier-Sobolev}
Let $1\leq p \leq \8$ and $k\in \mathbb{N}$. Then
$$\mathsf{M}_{\mathrm{cb}} (W^k_p (\T^d_{\t}))
 =\mathsf{M}_{\mathrm{cb}} (S^k_p)\;\text{ with equal norms}.$$
Consequently,
 $$\mathsf{M}_{\mathrm{cb}} (W^k_p (\T^d_{\t}))
 =\mathsf{M}_{\mathrm{cb}} (W^k_p (\T^d))\;\text{ with equal norms}.$$
\end{thm}

\begin{proof}
This proof  is an adaptation of that of \cite[Theorem~7.3]{CXY2012}.
We start with an elementary observation. Let $V=\mathrm{diag} (\cdots,U^n,\cdots)_{n\in\ent^d}$.  For any  $a=(a_{m, n})_{m,n\in\mathbb{Z}^d}\in B (\ell_2(\mathbb{Z}^d))$,  let  $x=V(a\ot 1_{\T^d_{\t}})V^*\in B(\ell_2(\mathbb{Z}^d))\overline\ot\T^d_{\t}$, where $1_{\T^d_{\t}}$ denotes the unit of $\T^d_\t$. Then
 $$x=(U^ma_{mn}U^{-n})_{m,n \in\mathbb{Z}^d}
 =\sum_{m,n}a_{m, n}e_{m,n}\otimes U^mU^{-n}
 =\sum_{m,n}a_{m, n}e_{m, n}\otimes e^{- \mathrm{i} n\tilde{\theta}m^{t}}U^{m-n},$$
where $(e_{m, n})$ are the canonical matrix units of $B(\ell_2(\mathbb{Z}^d))$. So,
 $$[x]=\big(a_{m, n}e_{m, n}\big)_{m,n\in\mathbb{Z}^d}\,,$$
a matrix with entries in $B (\ell_2(\mathbb{Z}^d))$.
  Since $V$ is unitary, we have
 $$ \big\|x\big\|_{L_p(B(\ell_2(\mathbb{Z}^d))\overline\ot\T^d_{\t})}=\big\|a\ot 1_{\T^d_{\t}}\big\|_{L_p(B(\ell_2(\mathbb{Z}^d))\overline\ot\T^d_{\t})}
 =\big\|a\big\|_{S_p(\ell_2(\mathbb{Z}^d))}.$$
Similarly, for $\el\in\nat_0^d$,
 \beq\label{2 norms}
  \big\|D^\el x\big\|_{L_p(B(\ell_2(\mathbb{Z}^d))\overline\ot\T^d_{\t})}=\big\|D^\el a\big\|_{S_p(\ell_2(\mathbb{Z}^d))}\,.
  \eeq

Now suppose that $\phi \in \mathsf{M}_{\mathrm{cb}} (W^k_p (\T^d_{\t})).$ For $a=(a_{m, n})_{m,n\in
\mathbb{Z}^d}\in B (\ell_2(\mathbb{Z}^d))$,  define  $x=V(a\ot 1_{\mathbb{T}^d_{\theta}})V^*$ as above. Then by \eqref{F-S},
for $\el\in\nat_0^d$,
 $$(\mathrm{Id}_{B(\ell_2(\mathbb{Z}^d))}\overline\otimes M_{\phi} ) (D^\el x)
 =V (S_{\mathring{\phi}}(D^\el a) \ot 1_{\mathbb{T}^d_{\theta}})V^{*}\,.$$
It then follows from \eqref{2 norms} that
\be\begin{split}
\|S_{\mathring{\phi}}(a)\|_{S_p^k}
&=\big[\sum_{|\ell|_1\le k} \|(\mathrm{Id}_{B(\ell_2(\mathbb{Z}^d))}\otimes M_{\phi} ) (D^\ell x)\|_{L_p(B(\ell_2(\mathbb{Z}^d))\overline\ot\T^d_{\t})}^p\big]^{\frac1p}\\
&\leq \|\phi\|_{\mathsf{M}_{\mathrm{cb}} (W^k_p (\mathbb{T}^d_{\theta}))} \big[\sum_{|\ell|_1\le k}\|D^\ell x\|_{L_p(B(\ell_2(\mathbb{Z}^d))\overline\ot\T^d_{\t})}^p\big]^{\frac1p}\\
&=\|\phi\|_{\mathsf{M}_{\mathrm{cb}} (W^k_p (\mathbb{T}^d_{\theta}))} \|a\|_{S_p^k}.
\end{split}\ee
Therefore, $\mathring{\phi}$ is a bounded Schur multiplier on $S_p^k$.
Considering matrices $a=(a_{m, n})_{m,n\in\mathbb{Z}^d}$ with entries in $S_{p}$, we show in the same way that  $M_{\mathring{\phi}}$ is c.b. on
$S_p^k$, so $\mathring{\phi}$ is a c.b. Schur multiplier on $S_p^k$ and
 $$\|\mathring{\phi}\|_{\mathsf{M}_{\mathrm{cb}} (S_{p}^k)} \le \|\phi\|_{\mathsf{M}_{\mathrm{cb}} (W^k_p (\mathbb{T}^d_{\theta}))}.$$

To show the converse direction, introducing the following Folner sequence of $\mathbb{Z}^d$:
 $$Z_N=\{-N,\dots,-1,0,1,\dots,N\}^d\subset\mathbb{Z}^d,$$
we define two maps $A_N$ and $B_N$ as follows:
 $$A_N:\; \mathbb{T}^d_{\theta}\to B (\ell_2^{|Z_N|}) \quad \text{with}\quad x\mapsto P_N([x]),$$
where $P_N: \;  B (\ell_2(\mathbb{Z}^d))\rightarrow  B (\ell_2^{|Z_N|})$ with $(a_{m, n})\mapsto
(a_{m, n})_{m,n\in Z_N}$; and
 $$B_N:\;  B (\ell_2^{|Z_N|})\to \mathbb{T}^d_{\theta}\quad \text{with}\quad
  e_{m, n}\mapsto \frac{1}{|Z_N|}e^{-\mathrm{i} n \tilde{\theta} (m-n)^t}U^{m-n}.$$
Here $ B (\ell_2^{|Z_N|})$ is endowed with the normalized trace. Both $A_N, B_N$ are  unital, completely positive
and trace preserving, so extend to  complete contractions between the corresponding $L_p$-spaces.
Moreover,
$$\lim_{N\rightarrow\infty}B_N\circ A_N(x)=x \;\;\mbox{in}\;\; L_p(\mathbb{T}^d_{\theta}), \;\; \forall x\in
L_p(\mathbb{T}^d_{\theta}).$$
If we define  $S_{p}^k(\ell_2^{|Z_N|})$ as before  for $S_p^k$ just replacing $S_p(\el_2(\ent^d))$ by $S_{p}(\ell_2^{|Z_N|})$,  we see that $A_N$  extends to a complete contraction from
$W^k_p(\mathbb{T}^d_{\t})$ into $S_{p}^k(\ell_2^{|Z_N|})$,  while $B_N$ a complete contraction from $S_{p}^k(\ell_2^{|Z_N|})$
into $W^k_p(\mathbb{T}^d_{\t}).$

Now assume that  $\mathring{\phi}$ is a c.b. Schur multiplier on $S_{p}^k$, then it is also a c.b. Schur multiplier on $S_{p}^k(\ell_2^{|Z_N|})$. We want to prove that $M_{\phi}$ is c.b. on $W^k_p(\mathbb{T}^d_{\t})$. For any $x\in L_p(B(\ell_2(\ent^d))\overline\ot\T^d_\t)$,
\be\begin{split}
\big\|{\rm Id}\ot M_{\phi}(x)\big\|_{L_p(B(\ell_2(\ent^d))\overline\ot\T^d_\t)}
 &=\lim_{N}\big\|\big({\rm Id}\ot B_N\big)\circ\big({\rm Id}\ot A_N\big)
 \big({\rm Id}\ot M_{\phi}(x)\big)\big\|_{L_p(B(\ell_2(\ent^d))\overline\ot\T^d_\t)}\\
 &=\lim_{N}\big\|\big({\rm Id}\ot B_N\big)\circ \big({\rm Id}\ot S_{\mathring{\phi}}\big)\big({\rm Id}\ot A_N(x)\big)\big\|_{L_p(B(\ell_2(\ent^d))\overline\ot\T^d_\t)}\\
 &\le \limsup_N\big\|{\rm Id}\ot S_{\mathring{\phi}}({\rm Id}\ot A_N(x))\big\|_{S_p(\ell_2(\mathbb{Z}^d);S_{p}^k(\ell_2^{|Z_N|}))}\\
 &\leq \limsup_N \big\|S_{\mathring{\phi}}\big\|_{\mathrm{cb}}\,\big\|{\rm Id}\ot A_N(x)\big\|_{S_p(\ell_2(\mathbb{Z}^d);S_{p}^{k}(\ell_2^{|Z_N|}))}\\
 & \leq \big\|S_{\mathring{\phi}}\big\|_{\mathrm{cb}}\,\big\|x\big\|_{L_p(B(\ell_2(\mathbb{Z}^d))\overline\ot\T^d_{\t})},
\end{split}
\ee
where in the second equality we have used the fact that
 $${\rm Id}\ot A_N({\rm Id}\ot M_{\phi}(x))={\rm Id}\ot S_{\mathring{\phi}}({\rm Id}\ot A_N(x)),$$
which follows from  \eqref{F-S}. Therefore, $M_{\phi}$ is c.b. on $W^k_p(\mathbb{T}^d_{\theta})$ and
 $$ \|\phi\|_{\mathsf {M}_{\mathrm{cb}} (W^k_p (\mathbb{T}^d_\t))}\le \|\mathring{\phi}\|_{\mathsf{M}_{\mathrm{cb}} (S_{p}^k)} .$$
The theorem is thus proved.
\end{proof}

 \begin{rk}
Let $1\leq p\leq \8$ and  $\alpha \in \real$. Since $J^\a$ is a complete isometry from $H^\alpha _p(\T^d_\t)$ onto $L_p(\T^d_\t)$, we have
$$\mathsf{M}_{\mathrm{cb}} (H^\a _p (\T^d_{\t}))
 =\mathsf{M}_{\mathrm{cb}} (L_{p} (\T^d_\t)) \;\text{ with equal norms}.$$
Thus, by Lemma~\ref{cb-multiplier}
$$\mathsf{M}_{\mathrm{cb}} (H^\a_p (\T^d_{\t}))
 =\mathsf{M}_{\mathrm{cb}} (H^\a _p (\T^d)) \;\text{ with equal norms}.$$
\end{rk}

Note that the proof of Theorem~\ref{q-Sobolev-Bessel} shows that $W^k _p (\T^d_{\t})=H^k _p (\T^d_{\t})$ holds completely isomorphically for $1<p<\8$.  Thus the above remark implies

\begin{cor}
Let $1< p< \8$ and  $k \in \nat$.
 $$\mathsf{M}_{\mathrm{cb}} (W^k _p (\T^d_{\t}))
 =\mathsf{M}_{\mathrm{cb}} (L_{p} (\T^d)) \;\text{ with equivalent norms}.$$
 \end{cor}

Clearly, the above equality still holds for $p=1$ or $p=\8$ if $d=1$ (the commutative case) since then $W^k_p(\T)=L_p(\T)$ for all $1\le p\le\8$ by the (complete) isomorphism
 $$L_p(\T)\ni x\mapsto \wh x(0)+ \sum_{m\in\ent\setminus\{0\}}\frac1{(2\pi{\rm i} m)^k}\,\wh x(m) z^m\in W^k_p(\T)\,.$$
However, this is no longer the case as soon as $d\ge2$, as proved by Poornima \cite{Po1982} in the commutative case for $\real^d$.  Poornima's example comes from Ornstein \cite{Or1962} which is still valid for our setting. Indeed, by  \cite{Or1962}, there exists a distribution $T$ on $\T^2$ which is not a measure and such that $T=\partial_1 \mu_0$, $\partial_1 T=\partial_2 \mu_1$ and $\partial_2 T=\partial_1 \mu_2$ for three measures $\mu_i$ on $\T^2$. $T$ induces a Fourier multiplier on $\T^2_\t$,  which is defined by the Fourier transform of $T$ and is denoted by $x\mapsto T* x$. Then for any $x\in W_1^1(\T^2_\t)$,
 \be\begin{split}
 &T*x=\partial_1\mu_0 *x=\mu_0 *\partial_1x\in L_1(\T^2_\t),\\
 &\partial_1 T*x=\partial_1\mu_1 *x=\mu_1 *\partial_2x\in L_1(\T^2_\t),\\
 &\partial_2 T*x=\partial_2\mu_2 *x=\mu_2 *\partial_1x\in L_1(\T^2_\t).
 \end{split}\ee
Thus $T*x\in W_1^1(\T^2_\t)$, so the Fourier multiplier induced by $T$ is bounded on $W_1^1(\T^2_\t)$. We show in the same way that it is c.b. too. Since $T$ is not a measure, it does not belong to $\mathsf{M}(L_1(\T^2))$.

%%%%%%%%%%%%%%%%%%%%%%%%%%%%%%%%%%%%%%%%%%%%%%%%%%%%%%%%%%%%%%%%%%%%%%%%
%%%%%%%%%%%%%%%%%%%%%%%%%%%%%%%%%%%%%%%%%%%%%%%%%%%%%%%%%%%%%%%%%%%%%%%%

\section{Fourier multipliers on Besov spaces}

%%%%%%%%%%%%%%%%%%%%%%%%%%%%%%%%%%%%%%%%%%%%%%%%%%%%%%%%%%%%%%%%%%%%%%%%
%%%%%%%%%%%%%%%%%%%%%%%%%%%%%%%%%%%%%%%%%%%%%%%%%%%%%%%%%%%%%%%%%%%%%%%%

It is  well known that  in the classical setting,  Fourier multipliers behave better on Besov spaces than on $L_p$-spaces. We will see that this fact remains true in the quantum case. We maintain the notation introduced in section~\ref{Definitions and basic properties: Besov}. In particular, $\f$ is a function satisfying \eqref{LP dec} and $\f^{(k)}(\xi)=\f(2^{-k}\xi)$ for $k\in\nat_0$. As usual, $\f^{(k)}$ is viewed as a function on $\ent^d$ too.

\smallskip

The following is the main result of this section.  Compared with the corresponding result in the classical case (see, for instance, Section~2.6 of \cite{HT1983}), our result is more precise since it gives a characterization of Fourier multipliers on $B_{p, q}^\a(\T^d_\t)$ in terms of those on $L_p(\T^d_\t)$.

\begin{thm}\label{q-multiplier-Besov}
 Let $\a\in\real$ and $1\le p, q\le\8$. Let $\phi: \ent^d\to\com$. Then $\phi$ is a Fourier multiplier on $B_{p, q}^\a(\T^d_\t)$ iff the $\phi \f^{(k)}$'s are Fourier multipliers on $L_p(\T^d_\t)$ uniformly in $k$. In this case, we have
  $$\big\|\phi\big\|_{\mathsf{M}(B_{p, q}^\a(\T^d_\t))}\approx |\phi(0)|+ \sup_{k\ge0} \big\|\phi\f^{(k)}\big\|_{\mathsf{M}(L_{p}(\T^d_\t))}$$
with relevant constants depending only on $\a$.
A similar c.b. version holds too.
\end{thm}

\begin{proof}
 Without loss of generality, we assume that $\phi(0)=0$ and all elements $x$ considered below have vanishing Fourier coefficients at the origin. Let $\phi\in \mathsf{M}(B_{p, q}^\a(\T^d_\t))$ and $x\in L_p(\T^d_\t)$. Then $y=(\wt\f_{k-1}+\wt\f_{k}+\wt\f_{k+1})*x\in B_{p, q}^\a(\T^d_\t)$ and
 $$\|y\|_{B_{p, q}^\a}\le c_\a  2^{k \a} \|x\|_p\;\text{ with }\; c_\a=9\cdot 4^{|\a|}\,.$$
So
  $$\|M_{\phi}(y)\|_{B_{p, q}^\a}\le\|\phi\big\|_{\mathsf{M}(B_{p, q}^\a(\T^d_\t))} \|y\|_{B_{p, q}^\a}\le c_{\a} 2^{k\a}\big\|\phi\big\|_{\mathsf{M}(B_{p, q}^\a(\T^d_\t))} \|x\|_p\,.$$
On the other hand, by \eqref{3-supports}, $\wt \f_k*M_{\phi}(y)=M_{\phi  \f^{(k)}}(x)$ and
 $$\|M_{\phi}(y)\|_{B_{p, q}^\a}\ge 2^{k\a} \|\wt \f_k*M_{\phi}(y)\|_p=2^{k\a} \|M_{\phi \f^{(k)}}(x)\|_p\,.$$
It then follows that
 $$\|M_{\phi \f^{(k)}}(x)\|_p\le c_\a \big\|\phi\big\|_{\mathsf{M}(B_{p, q}^\a(\T^d_\t))} \|x\|_p\,,$$
whence
 $$\sup_{k\ge0}   \big\|\phi\f^{(k)}\big\|_{\mathsf{M}(L_{p}(\T^d_\t))}\le c_\a \big\|\phi\big\|_{\mathsf{M}(B_{p, q}^\a(\T^d_\t))}\,.$$
Conversely, for $x\in B_{p, q}^\a(\T^d_\t)$,
 \be\begin{split}
 \|\wt\f_k*M_{\phi}(x)\|_p
 &= \|M_{\phi\f^{(k)}}\big((\wt\f_{k-1}+\wt\f_{k}+\wt\f_{k+1})*x\big)\|_p\\
 &\le   \big\|\phi\f^{(k)}\big\|_{\mathsf{M}(L_{p}(\T^d_\t))}\|(\wt\f_{k-1}+\wt\f_{k}+\wt\f_{k+1})*x\|_p\,.
 \end{split}\ee
We then deduce
 $$\big\|M_{\phi}(x)\big\|_{B_{p, q}^\a}\le 3\cdot 2^{|\a|} \sup_{k\ge0}\big\|\phi\f^{(k)}\big\|_{\mathsf{M}(L_{p}(\T^d_\t))}\big\|x\big\|_{B_{p, q}^\a}\,,$$
 which implies
  $$\big\|\phi\big\|_{\mathsf{M}(B_{p, q}^\a(\T^d_\t))}\le 3\cdot 2^{|\a|} \sup_{k\ge0}\big\|\phi\f^{(k)}\big\|_{\mathsf{M}(L_{p}(\T^d_\t))}\,.$$
Thus the assertion concerning bounded multipliers is proved.

The preceding argument can be modified to work in the c.b. case too.   First note that for $k\ge0$, $\f^{(k)}$ is a c.b. Fourier multiplier on $L_p(\T^d_\t)$ for all $1\le p\le\8$ with c.b. norm $1$, that is, the map $x\mapsto \wt\f_k*x$ is c.b. on $L_p(\T^d_\t)$.  So for any $x\in S_q[L_p(\T^d_\t)]$ (the $L_p(\T^d_\t)$-valued Schatten $q$-class),
 $$\big\|(\mathrm{Id}_{S_q}\ot M_{\f^{(k)}})(x)\big\|_{S_q[L_p(\T^d_\t)]}\le \|x\|_{S_q[L_p(\T^d_\t)]}\,.$$
Now let $\phi\in \mathsf{M}_{\mathrm{cb}}(B_{p, q}^\a(\T^d_\t))$ and $x\in S_q[L_p(\T^d_\t)]$. Define $y$ as above:  $y=(\wt\f_{k-1}+\wt\f_{k}+\wt\f_{k+1})*x$. Then for $k-2\le j\le k+2$,
 $$\|\wt\f_j*y\|_{S_q[L_p(\T^d_\t)]}\le 3\,\|x\|_{S_q[L_p(\T^d_\t)]}\,.$$
It thus follows that
 \be\begin{split}
 \big\|(\mathrm{Id}_{S_q}\ot M_{\phi})(y)\big\|_{S_q[B_{p, q}^\a(\T^d_\t)]}
 &\le \|\phi\|_{\mathsf{M}_{\mathrm{cb}}(B_{p, q}^\a(\T^d_\t))}\, \|y\|_{S_q[B_{p, q}^\a(\T^d_\t)]}\\
 &\le \|\phi\|_{\mathsf{M}_{\mathrm{cb}}(B_{p, q}^\a(\T^d_\t))}\, \sum_{j=k-2}^{k+2} 2^{j\a} \|\wt\f_j*y\|_{S_q[L_p(\T^d_\t)]}\\
 &\le c_\a 2^{k\a} \|\phi\|_{\mathsf{M}_{\mathrm{cb}}(B_{p, q}^\a(\T^d_\t))}\, \|x\|_{S_q[L_p(\T^d_\t)]}\,.
  \end{split}\ee
 Then as before, we deduce
   $$\sup_{k\ge0}   \big\|\phi\f^{(k)}\big\|_{\mathsf{M}_{\mathrm{cb}}(L_{p}(\T^d_\t))}
   \le c_\a \big\|\phi\big\|_{\mathsf{M}_{\mathrm{cb}}(B_{p, q}^\a(\T^d_\t))}\,.$$
 To show the converse inequality, assume
  $$ \sup_{k\ge0}   \big\|\phi\f^{(k)}\big\|_{\mathsf{M}_{\mathrm{cb}}(L_{p}(\T^d_\t))} \le1.$$
Then for $x\in S_q[B_{p, q}^\a(\T^d_\t)]$,
 \be\begin{split}
 \|\wt\f_k*M_{\phi}(x)\|_{S_q[L_p(\T^d_\t)]}
 &\le \big\|\phi\f^{(k)}\big\|_{\mathsf{M}_{\mathrm{cb}}(L_{p}(\T^d_\t))}\|(\wt\f_{k-1}+\wt\f_{k}+\wt\f_{k+1})*x\|_{S_q[L_p(\T^d_\t)]}\\
 &\le \|(\wt\f_{k-1}+\wt\f_{k}+\wt\f_{k+1})*x\|_{S_q[L_p(\T^d_\t)]}\,.
 \end{split}\ee
Therefore,
 \be\begin{split}
 \|M_{\phi}(x)\|_{S_q[B_{p, q}^\a(\T^d_\t)]}
 &=\Big(\sum_{k\ge0} \big(2^{k\a} \|\wt\f_k*M_{\phi}(x)\|_{S_q[L_p(\T^d_\t)]}\big)^q\Big)^{\frac1q}\\
 &\le\Big(\sum_{k\ge0} \big(2^{k\a}\|(\wt\f_{k-1}+\wt\f_{k}+\wt\f_{k+1})*x\|_{S_q[L_p(\T^d_\t)]}\big)^q\Big)^{\frac1q}\\
 &\le  3\cdot 2^{|\a|} \|x\|_{S_q[B_{p, q}^\a(\T^d_\t)]}\,.
 \end{split}\ee
We thus get the missing  converse inequality, so the theorem is proved.
 \end{proof}

The following is an immediate consequence of the preceding theorem.

\begin{cor}\label{multiplier-Besov properties}
\begin{enumerate}[\rm (i)]
\item $\mathsf{M}(B_{p, q}^\a(\T^d_\t))$ is independent of $\a$ and $q$, up to equivalent norms.
\item $\mathsf{M}(B_{p, \8}^0(\T^d_\t))=\mathsf{M}(B_{p', \8}^0(\T^d_\t))$, where $p'$ is the conjugate index of $p$.
\item $\mathsf{M}(B_{p_0, \8}^0(\T^d_\t))\subset \mathsf{M}(B_{p_1, \8}^0(\T^d_\t))$ for $1\le p_0< p_1\le2$.
\item $\mathsf{M}(L_{p}(\T^d_\t))\subset \mathsf{M}(B_{p, q}^\a(\T^d_\t))$.
\end{enumerate}
\noindent Similar statements hold for the spaces $\mathsf{M}_{\mathrm{cb}}(B_{p, q}^\a(\T^d_\t))$.
\end{cor}

Theorem~\ref{q-multiplier-Besov} and Lemma~\ref{cb-multiplier} imply the following:

\begin{cor}
$\mathsf{M}_{\mathrm{cb}}(B_{p, q}^\a(\T^d_\t))=\mathsf{M}_{\mathrm{cb}}(B_{p, q}^\a(\T^d))$ with equivalent norms.
\end{cor}

Let $\F(B_{1,\8}^0(\T^d))$ be the space of all Fourier transforms of functions in $B_{1,\8}^0(\T^d)$ (a commutative Besov space), equipped with the norm $\|\wh f\|=\|f\|_{B_{1,\8}^0}$.

\begin{cor}\label{multiplier B1}
$\mathsf{M}_{\mathrm{cb}}(B_{1, q}^\a(\T^d_\t))=\F(B_{1,\8}^0(\T^d))$ with equivalent norms.
\end{cor}

\begin{proof}
 Let $\phi\in \mathsf{M}_{\mathrm{cb}}(B_{1, \8}^0(\T^d_\t))$ and $f$ be the distribution on $\T^d$ such that $\wh f=\phi$. By Theorem~\ref{q-multiplier-Besov} and Lemma~\ref{cb-multiplier}, we have
  $$\sup_{k\ge0} \big\|\phi\f^{(k)}\big\|_{\mathsf{M}(L_{1}(\T^d))}<\8.$$
 Recall that the Fourier transform of $\f_k$ is $\f^{(k)}$ and $\wt\f_k$ is the periodization of $\f_k$. So
  $$\|\wt\f_k\|_{L_1(\T^d)}=\|\f_k\|_{L_1(\real^d)}=\|\f\|_{L_1(\real^d)}\,.$$
Noting that by \eqref{3-supports},
 $\wt\f_k*f=M_{\phi\f^{(k)}}(\wt\f_{k-1}+\wt\f_{k}+\wt\f_{k+1}),$
we get
  $$\|\wt\f_k*f\|_1\le  \big\|\phi\f^{(k)}\big\|_{\mathsf{M}(L_1(\T^d))} \|\wt\f_{k-1}+\wt\f_{k}+\wt\f_{k+1}\|_1
  \le 3\|\f\|_{L_1(\real^d)}\, \big\|\phi \f^{(k)}\big\|_{\mathsf{M}(L_{1}(\T^d))}\,,$$
 whence
  $$\|f\|_{B_{1,\8}^0}\le 3\|\f\|_{L_1(\real^d)}\,\sup_{k\ge0}\big\|\phi\f^{(k)}\big\|_{\mathsf{M}(L_1(\T^d))}\,.$$
 Conversely, assume $\phi=\wh f$ with $f\in B_{1,\8}^0(\T^d)$. Let $g\in B_{1,\8}^0(\T^d)$. Then
  \be\begin{split}
  \|\wt\f_k*M_{\phi}(g)\|_1
  &=\|\wt\f_k*f*g\|_1=\|\wt\f_k*f*(\wt\f_{k-1}+\wt\f_{k}+\wt\f_{k+1})*g\|_1\\
  &\le \|\wt\f_k*f\|_1\,\|(\wt\f_{k-1}+\wt\f_{k}+\wt\f_{k+1})*g\|_1\\
  &\le 3\|f\|_{B_{1,\8}^0}\, \|g\|_{B_{1,\8}^0}\,.
  \end{split}\ee
 Thus $M_{\phi}(g)\in B_{1,\8}^0(\T^d)$ and
  $$\|M_{\phi}(g)\|_{B_{1,\8}^0}\le 3\|f\|_{B_{1,\8}^0}\, \|g\|_{B_{1,\8}^0}\, ,$$
which implies that $\phi$ is a Fourier multiplier on $B_{1, \8}^0(\T^d_\t)$ and
  $$\|\phi\|_{\mathsf{M}(B_{1, \8}^0(\T^d_\t))} \le 3\|f\|_{B_{1,\8}^0}\,.$$
Considering $g$ with values in $S_\8$, we show that $\phi$ is c.b. too. Alternately, since $\mathsf{M}(L_1(\T^d))= \mathsf{M}_{\mathrm{cb}}(L_1(\T^d))$,
Theorem~\ref{q-multiplier-Besov} yields $\mathsf{M}(B_{1, \8}^0(\T^d))= \mathsf{M}_{\mathrm{cb}}(B_{1, \8}^0(\T^d))$, which allows us to conclude the proof too.
\end{proof}

We have seen previously that every bounded (c.b.) Fourier multiplier on $L_p(\T^d_\t)$ is a bounded (c.b.) Fourier multiplier on $B_{p, q}^\a(\T^d_\t)$.  Corollary~\ref{multiplier B1} shows that the converse is false for $p=1$. We now show that it also is false for any $p\neq2$.

\begin{prop}
 There exists a Fourier multiplier $\phi$ which is c.b. on $B_{p, q}^\a(\T^d_\t)$ for any $p, q$ and $\a$ but never belongs to $\mathsf{M}(L_p(\T^d_\t))$ for any $p\neq2$ and any $\t$.
 \end{prop}

\begin{proof}
 The example constructed by Stein and Zygmund \cite{SZ1967} for a similar circumstance can be shown to work for our setting too. Their example is a distribution on $\T$ defined as follows:
  $$\mu(z)=\sum_{n=2}^\8\frac1{\log n}\, (w_nz)^{2^n}D_n(w_nz) $$
for some appropriate $w_n\in\T$,  where
  $$D_n(z)=\sum_{j=0}^n z^j\,, \quad z\in\T.$$
Since $\|D_n\|_{L_1(\T)}\approx \log n$, we see that $\mu\in B_{1,\8}^0(\T)$. Considered as a distribution on $\T^d$, $\mu\in B_{1,\8}^0(\T^d)$ too. Thus by Corollaries~\ref{multiplier-Besov properties} and \ref{multiplier B1}, $\phi=\wh\mu$ belongs to $\mathsf{M}_{\mathrm{cb}}(B_{p, p}^\a(\T^d_\t))$ for any $p, q$ and $\a$. However, Stein and Zygmund proved that $\phi$ is not a Fourier multiplier on $L_p(\T)$ for any $p\neq2$ if the $w_n$'s are appropriately chosen. Consequently, $\phi$ cannot be  a Fourier multiplier on $L_p(\T^d_\t)$ for any $p\neq2$ and any $\t$ since $L_p(\T)$ isometrically embeds into  $L_p(\T^d_\t)$ by an embedding that is also a c.b. Fourier multiplier.
\end{proof}

We conclude this section with some comments on the vector-valued case. The proof of Theorem~\ref{q-multiplier-Besov} works equally for vector-valued Besov spaces. Recall that for an operator space $E$,  $B_{p, q}^\a(\T^d_\t; E)$ denotes the $E$-valued Besov space on $\T^d_\t$ (see Remark~\ref{os-Besov}).

\begin{prop}
 For any operator space $E$,
  $$\big\|\phi\big\|_{\mathsf{M}(B_{p, q}^\a(\T^d_\t; E))}\approx |\phi(0)|+ \sup_{k\ge0} \big\|\phi\f^{(k)}\big\|_{\mathsf{M}(L_{p}(\T^d_\t; E))}$$
 with equivalence constants depending only on $\a$.
  \end{prop}

 If $\t=0$, we go back to the classical vector-valued case. The above proposition explains the well-known fact mentioned at the beginning of this section that Fourier multipliers behave better on Besov spaces than on $L_p$-spaces. To see this, it is more convenient to write the above right-hand side in a different form:
 $$\big\|\phi\f^{(k)}\big\|_{\mathsf{M}(L_{p}(\T^d_\t; E))}=\big\|\phi(2^k\cdot)\f\big\|_{\mathsf{M}(L_{p}(\T^d_\t; E))}\,.$$
Thus if $\phi$ is homogeneous, the above multiplier  norm is independent of $k$, so $\phi$ is a Fourier multiplier on $B_{p, q}^\a(\T^d; E)$ for any $p, q, \a$ and any Banach space $E$. In particular, the Riesz transform is bounded on  $B_{p, q}^\a(\T^d; E)$.

The preceding characterization  of Fourier multipliers is, of course, valid for $\real^d$ in place of $\T^d$.  Let us record this here:

\begin{prop}
 Let $E$ be a Banach space. Then for any $\phi:\real^d\to\com$,
  $$\big\|\phi\big\|_{\mathsf{M}(B_{p, q}^\a(\real^d; E))}\approx \|\phi \psi\|_{\mathsf{M}(L_{p}(\real^d; E))}+
   \sup_{k\ge0} \big\|\phi(2^k\cdot)\f\big\|_{\mathsf{M}(L_{p}(\real^d; E))}\,,$$
 where $\p$ is defined by
 $$\p(\xi)=\sum_{k\ge 1}\f(2^{k}\xi),\quad \xi\in\real^d\,.$$
  \end{prop}

%%%%%%%%%%%%%%%%%%%%%%%%%%%%%%%%%%%%%%%%%%%%%%%%%%%%%%%%%%%%%%%%%%%%%%%%
%%%%%%%%%%%%%%%%%%%%%%%%%%%%%%%%%%%%%%%%%%%%%%%%%%%%%%%%%%%%%%%%%%%%%%%%

\section{Fourier multipliers on Triebel-Lizorkin spaces}

%%%%%%%%%%%%%%%%%%%%%%%%%%%%%%%%%%%%%%%%%%%%%%%%%%%%%%%%%%%%%%%%%%%%%%%%
%%%%%%%%%%%%%%%%%%%%%%%%%%%%%%%%%%%%%%%%%%%%%%%%%%%%%%%%%%%%%%%%%%%%%%%%

As we have seen in the chapter on Triebel-Lizorkin spaces, Fourier multipliers on such spaces are subtler than those on Sobolev and Besov spaces. Similarly to the previous two sections, our target here is to show that the c.b. Fourier multipliers on $F_{p}^{\a,c}(\T_\t^d)$ are independent of $\theta$. By definition,  $F_{p}^{\a,c}(\T_\t^d)$ can be viewed as a subspace of the column space $L_p(\T_\t^d; \ell_2^{\a,c})$, the latter  is the column subspace of $L_p(B (\el_2^\a)\overline \ot\T_\t^d)$. Thus  $F_{p}^{\a,c}(\T_\t^d)$ inherits  the natural operator space structure of $L_p(B (\el_2^\a)\overline \ot\T_\t^d)$.  Similarly, the row Triebel-Lizorkin space $F_{p}^{\a,r}(\T_\t^d)$ carries a natural operator space structure too. Finally, the mixture space $F_{p}^{\a}(\T_\t^d)$ is equipped with the sum or intersection operator space structure according to $p<2$ or $p\ge2$.  Note that according to this definition, even though it is a commutative function space, the space $F_p^\a(\T^d)$ (corresponding to $\t=0$) is endowed with three different operator space structures, the first two being defined by its embedding into $L_p(\T^d;\ell_2^{\a,c})$ and $L_p(\T^d;\ell_2^{\a,r})$, the third one being the mixture of these two. The resulting operator spaces are denoted by $F_p^{\a,c}(\T^d)$ , $F_p^{\a,r}(\T^d)$ and $F_p^{\a}(\T^d)$, respectively. Similarly, we introduce operator space structures on the Hardy space $\H_p^c(\T^d_\t)$, its row and mixture versions too.

The main result of this section is the following:

\begin{thm}\label{q-multiplier-Triebel}
 Let $1\leq p\le \8$ and $\a\in \real$. Then
\be\begin{split}\mathsf{M}_{\mathrm{cb}} (F^{\a,c}_p (\T^d_{\t}))&=\mathsf{M}_{\mathrm{cb}} (F^{\a,c}_p (\T^d))\;\text{ with equal norms},\\
\mathsf{M}_{\mathrm{cb}} (F^{\a,c}_p (\T^d_{\t}))
 &=\mathsf{M}_{\mathrm{cb}} (F^{0,c}_p (\T^d))\;\text{ with equivalent norms}.
 \end{split}\ee
Similar statements hold for the row and mixture spaces.
\end{thm}

We will show the theorem only in the case $p<\8$. The proof presented below can be easily modified to work for $p=\8$ too. Alternately, the case $p=\8$ can be also obtained by duality from the case $p=1$. Note, however, that this duality argument yields  only the first equality of the theorem with equivalent norms for $p=\8$.

 We adapt the proof of Theorem \ref{q-multiplier-Sobolev} to the present situation, by introducing the space
 $$S_{p}^{\a,c}=\Big\{a=(a_{m, n})_{m,n\in\mathbb{Z}^d} \,:\, \big(\sum_{k\ge 0}2^{2k\a}\big|\wt\f_k*a\big|^2\big)^{\frac12} \in S_p(\ell_2(\mathbb{Z}^d)) \Big\},$$
equipped with the norm
 $$\|a\|_{S_{p}^{\a,c}}=\big\| \big(\sum_{k\ge 0}2^{2k\a}\big|\wt\f_k*a\big|^2\big)^{\frac12}\big\|_{S_p},$$
where
 $$\wt\f_k*a= \big(\f(2^{-k}(m-n)) \,a_{m, n}\big)_{m,n\in\mathbb{Z}^d}\,.$$
 Then $S_{p}^{\a,c}$ is a closed subspace of the column subspace of $S_p(\ell_2^\a\ot\ell_2(\mathbb{Z}^d))$, which introduces a natural operator space structure on $S_{p}^{\a, c}$. Let $\mathsf{M}_{\mathrm{cb}}(S_p^{\a,c})$ denote the space of all c.b. Schur multipliers on $S_p^{\a,c}$, equipped with the natural norm.

\begin{lem}
Let $1\leq p< \8$ and $\a\in \mathbb{R}$. Then
$$\mathsf{M}_{\mathrm{cb}} (F^{\a,c}_p (\T^d_{\t}))
 =\mathsf{M}_{\mathrm{cb}} (S^{\a,c}_p)\;\text{ with equal norms}.$$
\end{lem}

\begin{proof}
This proof is similar to the one of Theorem \ref{q-multiplier-Sobolev}; we point out the necessary changes. Keeping the notation there, we have for $a=(a_{m, n})_{m,n\in\mathbb{Z}^d}\in S^{\a,c}_p$ and  $x=V(a\ot 1_{\T^d_{\t}})V^*\in B(\ell_2(\mathbb{Z}^d))\overline\ot\T^d_{\t}$
 $$ \big\| \big(\sum_k2^{2k\a} |\wt \f_k *x\big|^2\big)^{\frac12}\big\|_{L_p(B(\ell_2(\mathbb{Z}^d))\overline\ot\T^d_{\t})}=\big\|a\big\|_{S^{\a,c}_p}\, .$$
Suppose that $\phi \in \mathsf{M}_{\mathrm{cb}} (F^{\a,c}_p (\T^d_{\t})).$
It then follows that
\be\begin{split}
\big\|S_{\mathring{\phi}}(a)\big\|_{S_{p}^{\a,c}}
&=\big\|\big(\sum_k2^{2k\a} \big|\wt \f_k *(V(S_{\mathring{\phi}}(a)\ot 1_{\T^d_{\t}})V^*)\big|^2\big)^{\frac12} \big\|_{L_p(B(\ell_2(\mathbb{Z}^d))\overline\ot\T^d_{\t})}\\
&=\big\|\big(\sum_k 2^{2k\a}\big|M_\phi\big(\wt \f_k *(V(a\ot 1_{\T^d_{\t}})V^*)\big)\big|^2\big)^{\frac12} \big\|_{L_p(B(\ell_2(\mathbb{Z}^d))\overline\ot\T^d_{\t})}\\
&\leq\big\|\phi\big\|_{\mathsf{M}_{\mathrm{cb}}(F^{\a,c}_p (\T^d_{\t}))} \| x\|_{S_p[F^{\a,c}_p (\T^d_{\t})]}\\
&=\|\phi\|_{\mathsf{M}_{\mathrm{cb}}(F^{\a,c}_p (\T^d_{\t}))} \| a\|_{S_p^{\a,c}}.
\end{split}\ee
Therefore, $\mathring{\phi}$ is a bounded Schur multiplier on $S_p^{\a,c}$.
Considering matrices $a=(a_{m, n})_{m,n\in\mathbb{Z}^d}$ with entries in $B(\ell_2)$, we show in the same way that  $S_{\mathring{\phi}}$ is c.b. on
$S_p^{\a,c}$, so $\mathring{\phi}$ is a c.b. Schur multiplier on $S_p^{\a, c}$ and
 $$\|\mathring{\phi}\|_{\mathsf{M}_{\mathrm{cb}} (S_{p}^{\a,c})} \le \|\phi\|_{\mathsf{M}_{\mathrm{cb}} (F^{\a,c}_p (\mathbb{T}^d_{\theta}))}.$$

To show the opposite inequality, we just note that the contractive and convergence properties of the maps $A_N$ and $B_N$ introduced in the proof of Theorem \ref{q-multiplier-Sobolev} also hold on the corresponding $F^{\a,c}_p(\T^d_\t)$ or $S_p^{\a,c}$ spaces. To see this, we take $A_N$ for example. Since it is c.b. between the corresponding $L_p$-spaces, it is also c.b. from $L_p( B (\ell_2)\overline\ot\T_\t^d)$ to $L_p( B (\ell_2)\overline\ot B (\ell_2^{|Z_N|}))$. Applying this to the elements of the form
$$ \begin{pmatrix}
 \wt \varphi_{0} * x & 0           & 0           &\dots \\
2^{\a} \wt \varphi_1 * x  & 0           & 0           &\dots \\
 2^{2\a} \wt\varphi_{2} * x  & 0           & 0           &\dots \\
 \cdot           & \cdot &\cdot & \dots
 \end{pmatrix}
$$
we see that $A_N$ is completely contractive from $F^{\a,c}_p(\T_\t^d)$ to $S_p^{\a,c}( B (\ell_2^{|Z_N|}))$, the latter space being the finite dimensional analogue of $S_p^{\a,c}$. We then argue as in the proof of Theorem \ref{q-multiplier-Sobolev} to deduce the desired  opposite inequality.
\end{proof}

\begin{proof}[Proof of Theorem \ref{q-multiplier-Triebel}]
The first part is an immediate consequence of the previous lemma. For the second, we need the c.b. version of Theorem \ref{Triebel-isom} (i), whose proof is already contained in section \ref{A multiplier theorem}. To see this, we just note that, letting $\M=B (\ell_2(\mathbb{Z}^d))\overline \ot \T_\t^d$ and $\N=B (\ell_2(\mathbb{Z}^d))\overline \ot L_\8(\T^d)\overline \ot \T_\t^d$ in Lemma \ref{CZH} we obtain the c.b. version of Lemma \ref{multiplier DH}, and that, in the same way, the c.b. version of Lemma \ref{Hp-discrete} holds, i.e., for $x\in S_p[\mathcal{H}_p^c(\T^d_\t)],$
$$\|x\|_{S_p[\mathcal{H}_p^c(\T^d_\t)]}\approx \|\wh x(0)\|_{S_p}+ \| s^{c}_\p(x)\|_{S_p[L_p(\T^d_\t)]}\,.$$
Finally, the previous lemma and the c.b. version of Theorem \ref{Triebel-isom} (i) yield the desired conclusion.
\end{proof}

\begin{rk}
 The preceding theorem and the c.b. version of Theorem \ref{Triebel-isom} (i) show that
  $$\mathsf{M}_{\mathrm{cb}} (\H_p^c (\T^d_{\t}))=\mathsf{M}_{\mathrm{cb}} (\H_p^c (\T^d))\;\text{ with equivalent norms}.$$
In fact, using arguments similar to the proof of the preceding theorem, we can show that the above equality holds with equal norms.
\end{rk}

\bigskip

\n{\bf Acknowledgements.} We wish to thank Marius Junge for discussions on the embedding and interpolation of Sobolev spaces, Tao Mei for discussions on characterizations of Hardy spaces and Eric Ricard  for discussions on Fourier multipliers and comments. We are also grateful to Fedor Sukochev for suggestions and comments on a preliminary version of the paper. We acknowledge the financial supports of ANR-2011-BS01-008-01, NSFC grant (No. 11271292, 11301401 and 11431011). 

\bigskip

%%%%%%%%%%%%%%%%%%%%%%%%%%%%%%%%%%%%%%%%%%%%%%%%%%%%%%%%%
%%%%%%%%%%%%%%%%%%%%%%%%%%%%%%%%%%%%%%%%%%%%%%%%%%%%%%%%%

\end{document}